\documentclass[a4paper,reqno]{amsart}

\usepackage{amsmath}
\usepackage{amsthm}
\usepackage{amssymb}
\usepackage{color}
\usepackage{amscd} %diagrams

\usepackage{bbm} %bold numbers
\usepackage{mathtools,mathrsfs} %nice calligraphy

\usepackage{times} % set math font

\usepackage{verbatim} %DO NOT comment this line (comment environment)

\usepackage{enumitem}

\newtheoremstyle{newremark}
  {5pt}
  {5pt}
  {\rmfamily}
  {}
  {\rmfamily\bf}
  {.}
  {.5em}
  {}

\newtheorem{theorem}{Theorem}
\newtheorem{lemma}[theorem]{Lemma}
\newtheorem{corollary}[theorem]{Corollary}
\newtheorem{proposition}[theorem]{Proposition}

\theoremstyle{newremark}
\newtheorem{remark}[theorem]{Remark}
\newtheorem{definition}[theorem]{Definition}

\newtheorem*{definition*}{Definition} %no numbering for Theorem*
\newtheorem*{notations*}{Notations}

\numberwithin{theorem}{section}
\numberwithin{equation}{section}

\newcommand{\N}{\mathbb{N}} %natural numbers
\newcommand{\R}{\mathbb{R}} %real numbers

\def\XXint#1#2#3{{%
\setbox0=\hbox{$#1{#2#3}{\int}$}
\vcenter{\hbox{$#2#3$}}\kern-.5\wd0}}

\renewcommand{\leq}{\leqslant}
\renewcommand{\geq}{\geqslant}
\renewcommand{\subset}{\subseteq}

\newcommand{\LL}{\mathop{\hbox{\vrule height 6pt width .5pt depth 0pt
\vrule height .5pt width 3pt depth 0pt}}\nolimits}

\newcommand{\res}{\mathop{\hbox{\vrule height 7pt width .5pt depth 0pt
\vrule height .5pt width 6pt depth 0pt}}\nolimits}

\newcommand{\Om}{\Omega}
\newcommand{\eps}{\varepsilon}

\newcommand{\e}{{\rm e}}
\newcommand{\de}{{\rm d}}

\begin{document}

%=================
% TITLE AND AUTHOR
%=================

\title[\bf Asymptotics for the fractional Allen-Cahn equation]{Asymptotics for the fractional Allen-Cahn equation\\  
and stationary nonlocal minimal surfaces}

\author{Vincent Millot}
\address{Vincent Millot\\ Universit\'e Paris Diderot\\ 
 Lab. J.L. Lions (CNRS UMR 7598)\\ Paris, France}
\email{millot@ljll.univ-paris-diderot.fr}
\author{Yannick Sire}
\address{Yannick Sire\\ Johns Hopkins University \\  
Department of Mathematics\\ Baltimore, USA}
\email{sire@math.jhu.edu}
\author{Kelei Wang}
\address{Kelei Wang\\ School of Mathematics and Statistics\\
Wuhan University, China}
\email{wangkelei@whu.edu.cn}

%\date{\today}

%=========
% ABSTRACT
%=========

\begin{abstract}

This article is mainly devoted to the asymptotic analysis of a fractional version of the (elliptic) Allen-Cahn equation in a bounded domain $\Omega\subset\R^n$, with or without a source term in the right hand side of the equation (commonly called chemical potential). Compare to the usual Allen-Cahn equation, the Laplace operator is here replaced by the fractional Laplacian $(-\Delta)^s$ with $s\in(0,1/2)$, as defined in Fourier space. In the singular limit $\eps\to 0$, we show that arbitrary solutions with uniformly bounded energy converge both in the energetic and geometric sense to surfaces of prescribed nonlocal mean curvature in $\Omega$ whenever the chemical potential remains  bounded in suitable Sobolev spaces. With no chemical potential, the notion of surface of prescribed nonlocal mean curvature reduces to the stationary version of the nonlocal minimal surfaces  introduced by L.A.~Caffarelli, J.M. Roquejoffre, and O.~Savin~\cite{CRS}.  Under the same Sobolev regularity assumption on the chemical potential, we also prove that  surfaces of prescribed nonlocal mean curvature have a Minkowski codimension equal to one, and that the associated sets have a locally finite fractional $2s^\prime$-perimeter  in $\Omega$ for every $s^\prime\in(0,1/2)$.  
\end{abstract}

\maketitle

%==================
% TABLE OF CONTENTS
%==================

\tableofcontents

%%%%%%%%%%%%%%%%%%%%%%%%%%%%%%%%%%%%%%%%%%%%%%%%%%%%%%%%%%

\section{Introduction}

In the van der Waals-Cahn-Hilliard theory of phase transitions, two-phase systems are driven by energy functionals of the form 
\begin{equation}\label{EN0}
\int_{\Om} \eps|\nabla u|^2+\frac{1}{\eps}W(u) \,\de x \,,\qquad\eps\in(0,1)\,,
\end{equation}
where $u:\Om\subset\R^n\to\R$ is a normalized density distribution of the two phases, and the (smooth) potential $W:\R\to[0,\infty)$ has  exactly two global minima at $\pm 1$ with  $W(\pm1)=0$  (see e.g. \cite{Gurt}). Here and after  $\Omega$ denotes a smooth  and bounded open set in dimension $n\geq 2$. Critical points satisfy the so-called elliptic Allen-Cahn (or scalar Ginzburg-Landau) equation
\begin{equation}\label{AC0}
-\Delta u_\eps+\frac{1}{\eps^{2}}W^\prime(u_\eps)=0\quad\text{in $\Om$}\,. 
\end{equation}
When $\eps$ is small, a control on the potential  implies  that $u_\eps\simeq\pm 1$ away from a region whose volume is of order $\eps$. Formally, the transition layer from the phase $-1$ to the phase $+1$ has a characteristic width of order~$\eps$. It should take place along an hypersurface which is expected to be a critical point of the area functional, i.e., a minimal surface. More precisely, the region $\{u_\eps\simeq 1\}$, which is essentially delimited  by this hypersurface and the container $\Om$, should be a  stationary set in $\Om$ of the perimeter functional, at least as $\eps\to 0$. 

For energy minimizing solutions (under their own boundary condition), this picture has been  justified first in \cite{ModMort} through one of the first    examples of $\Gamma$-convergence. The  result shows that  if the energy is equibounded, then $u_\eps\to u_*$ in $L^1(\Om)$ as $\eps\to0$ for some function $u_*\in BV(\Om;\{\pm1\})$ (up to subsequences). The set $\{u_*=1\}$ minimizes (locally) its perimeter in $\Om$, and up to a multiplicative constant, the energy converges to the relative perimeter of $\{u_*=1\}$ in $\Omega$. 
The analogous analysis concerning  global minimization of the energy under a volume constraint has been addressed in \cite{Mod,Ster}. 

The case of general critical points has been treated more recently in \cite{HT}. It presents a slightly different feature. Namely, if the energy is equibounded, then the energy density converges in the sense of measures as $\eps\to0$ to a stationary integral $(n-1)$-varifold, i.e., a generalized minimal hypersurface with integer multiplicity.  The multiplicity of the limiting hypersurface comes from an eventual folding of the diffuse interface $\{|u_\eps|\lesssim 1/2\}$ as $\eps\to0$. In such a case, the interface between the two regions $\{u_*=1\}$ and $\{u_*=-1\}$ can be strictly smaller than the support of the  limiting varifold. In fact, the boundary  of the region $\{u_*=1\}$ corresponds to the set of points where the varifold has odd multiplicity. In particular, the perimeter of $\{u_*=1\}$ can be strictly smaller than the the limit of the energy. This energy loss effect is in strong analogy with the lack of strong compactness as $\eps\to 0$ of solutions of the (vectorial) Ginzburg-Landau system with a potential well $\{W=0\}$ given by a smooth and compact manifold $\mathcal{M}\subset\R^d$, see \cite{LW1,LW2}. 
\vskip3pt

In the last few years, there have been many studies on nonlocal or fractional  versions of equation \eqref{AC0} and energy \eqref{EN0} (see e.g. \cite{ALBBEL,ABScr,ABS,CC,CC2,CS1,CS2,CSM,PSV,SVold,SV1,SV2,SV}). Many of them are motivated by physical problems such as  
stochastic Ising models from statistical mechanics, or the Peirls-Nabarro model for dislocations in crystals \cite{GM,Im,ImSoug}. In this article, we consider one of the simplest fractional version of equation \eqref{AC0} where the Laplace operator is replaced by the fractional Laplacian $(-\Delta)^s$, i.e., the Fourier multiplier  of symbol $(2\pi|\xi|)^{2s}$, with exponent $s\in (0,1/2)$. In details, we are interested in the asymptotic behavior as $\varepsilon\to 0$ of weak solutions $v_\varepsilon:\R^n\to\R$ of the fractional Allen-Cahn equation 
\begin{equation}\label{GLIntro}
 (-\Delta)^{s} v_\varepsilon+\frac{1}{\varepsilon^{2s}}W^\prime(v_\varepsilon) =0\quad \text{in $\Omega$}\,,
 \end{equation}
 subject to an exterior Dirichlet condition of the form 
\begin{equation}\label{Dircondintro}
v_\eps= g_\eps \qquad\text{on $\R^n\setminus\Omega$}\,,
\end{equation}
where $g_\eps:\R^n\to\R$ is a given smooth and bounded function. 
For $s\in(0,1)$, the action of the integro-differential operator $ (-\Delta)^{s} $ on a smooth bounded function $v:\R^n\to\R$ is defined   by 
\begin{equation}\label{formpvfraclap}
 (-\Delta)^{s}  v(x):={\rm p.v.} \left(\gamma_{n,s}\int_{\mathbb{R}^n}\frac{v(x)-v(y)}{|x-y|^{n+2s}}\,\de y\right) \quad \text{with } \gamma_{n,s}:=s2^{2s}\pi^{-\frac{n}{2}}\frac{\Gamma\big(\frac{n+2s}{2}\big)}{\Gamma(1-s)} \,,
 \end{equation}
where the notation ${\rm p.v.}$ means that the integral  is taken in the {\sl Cauchy principal value} sense. 
In terms of distributions, the action of $ (-\Delta)^{s}  v$ on a test function $\varphi\in\mathscr{D}(\Omega)$ is defined by
\begin{multline}\label{tardtard}
\big\langle   (-\Delta)^{s}  v,\varphi\big\rangle_\Omega:= \frac{\gamma_{n,s}}{2} \iint_{\Omega\times\Omega}\frac{\big(v(x)-v(y)\big)(\varphi(x)-\varphi(y)\big)}{|x-y|^{n+2s}}\,\de x\de y\\
+\gamma_{n,s} \iint_{\Omega\times(\R^n\setminus\Omega)}\frac{\big(v(x)-v(y)\big)\varphi(x)}{|x-y|^{n+2s}}\,\de x\de y\,.
 \end{multline}
This formula defines indeed a distribution on $\Omega$ whenever $v\in L^2_{\rm loc}(\R^n)$ satisfies
\begin{multline}\label{defenergE}
\mathcal{E}(v,\Omega):=\frac{\gamma_{n,s}}{4}  \iint_{\Omega\times\Omega}\frac{|v(x)-v(y)|^2}{|x-y|^{n+2s}}\,\de x\de y\\
+\frac{\gamma_{n,s}}{2} \iint_{\Omega\times(\R^n\setminus\Omega)} \frac{|v(x)-v(y)|^2}{|x-y|^{n+2s}}\,\de x\de y<\infty\,.
\end{multline}
More precisely, if \eqref{defenergE} holds, then   $ (-\Delta)^{s} v$ belongs to $H^{-s}(\Omega)$. To include the Dirichlet condition \eqref{Dircondintro}, one considers the  restricted class of functions given by the affine space $H^{s}_{g_\eps}(\Omega):=g_\eps+H_{00}^{s}(\Omega)$. Since $\mathcal{E}(\cdot,\Omega)$ is exactly the quadratic form induced by \eqref{tardtard}, the functional $\mathcal{E}(\cdot,\Omega)$  can be  thought as {\sl fractional Dirichlet energy} in~$\Omega$ associated to $ (-\Delta)^{s} $. Integrating  the potential in \eqref{GLIntro}, we obtain the {\sl fractional Allen-Cahn energy} in $\Omega$ associated to equation \eqref{GLIntro}, i.e., 
\begin{equation}\label{defFGLenerg}
\mathcal{E}_\varepsilon(v,\Omega):=\mathcal{E}(v,\Omega)+ \frac{1}{\varepsilon^{2s}}\int_\Omega W(v)\,\de x\,.
\end{equation}
In this way, we define weak solutions of \eqref{GLIntro} as critical points  of  $\mathcal{E}_\varepsilon(\cdot,\Omega)$  
with respect to perturbations supported in $\Omega$. 

Concerning minimizers of  $\mathcal{E}_\varepsilon(\cdot,\Omega)$  over $H^s_{g_\eps}(\Omega)$, their asymptotic behavior as $\eps\to 0$ has been investigated quite recently in \cite{SV1} through a $\Gamma$-convergence analysis. 
The  result reveals a dichotomy between the two cases $s\geq 1/2$ and $s<1/2$. In the case $s\geq 1/2$, the normalized energies
$$\widetilde{\mathcal{E}}_\varepsilon(\cdot,\Om):=\begin{cases}
\eps^{2s-1}\mathcal{E}_\varepsilon(\cdot,\Om) & \text{if $s\in(1/2,1)$}\,,\\
|\ln \eps|^{-1} \mathcal{E}_\varepsilon(\cdot,\Om) & \text{if $s=1/2$}\,,
\end{cases}$$
$\Gamma\big(L^1(\Omega)\big)$-converge as $\eps\to 0$ to the functional $\widetilde{\mathcal{E}}_0(\cdot,\Om)$ defined on $BV(\Om;\{\pm 1\})$ by 
$$\widetilde{\mathcal{E}}_0(v,\Om):=\sigma\, {\rm Per}\big(\{v=1\},\Om\big)\,,$$
where $\sigma=\sigma(W,n,s)$ is a positive constant, and ${\rm Per}(E,\Om)$ denotes the distributional (relative) perimeter of the set $E$ in $\Om$. In other words, for $s\geq 1/2$, fractional Allen-Cahn  energies (and thus minimizers) behave essentially as in the classical case, and area-minimizing  hypersurfaces arise in the limit $\eps\to0$. 
For $s\in(0,1/2)$, the variational convergence of $\mathcal{E}_\varepsilon(\cdot,\Om)$ appears to be much simpler since $H^s$-regularity does not exclude (all) characteristic functions. In particular, there is no need in this case to normalize  $\mathcal{E}_\varepsilon(\cdot,\Om)$. 
Assuming that $g_\eps\to g$ in $L^1_{\rm loc}(\R^n\setminus\Omega)$ for some function $g$ satisfying $|g|=1$ a.e. in $\R^n\setminus\Omega$, the functionals $\mathcal{E}_\varepsilon(\cdot,\Om)$ (restricted to $H^{s}_{g_\eps}(\Omega)$) converge as $\eps\to 0$ both in the variational and pointwise sense to 
$$ \mathcal{E}_0(v,\Om):=\begin{cases}
\mathcal{E}(v,\Omega) & \text{if $v\in H^{s}_g(\Om;\{\pm1\})$}\,,\\
+\infty & \text{otherwise}\,.
\end{cases}$$
Now it is worth noting that 
\begin{equation}\label{idDirPersharpint}
 \mathcal{E}(v,\Omega)  = 2\gamma_{n,s} P_{2s}\big(\{v=1\},\Omega\big)\quad \forall v\in H^{s}_g(\Om;\{\pm1\})\,,
 \end{equation}
where $P_{2s}(E,\Omega)$ is the so-called {\sl fractional $2s$-perimeter} in $\Omega$ of a set $E\subset\R^n$, i.e., 
\begin{multline*}
P_{2s}(E,\Omega):=\int_{E\cap\Omega}\int_{E^c\cap\Omega}\frac{\de x\de y}{|x-y|^{n+2s}}+\int_{E\cap\Omega}\int_{E^c\setminus\Omega}\frac{\de x\de y}{|x-y|^{n+2s}}\\
+\int_{E\setminus\Omega}\int_{E^c\cap\Omega}\frac{\de x\de y}{|x-y|^{n+2s}}\,.
\end{multline*}
As a consequence of this $\Gamma$-convergence result, a sequence $\{v_\eps\}$ of minimizing solutions  of  \eqref{GLIntro}-\eqref{Dircondintro} with $s\in(0,1/2)$ converges as $\eps\to 0$ (up to subsequences) to some function 
$v_*\in H^{s}_g(\Om)$ of the form $v_*=\chi_{E_*} -\chi_{\R^n\setminus E_*}$, and the limiting set $E_*\subset \R^n$ is minimizing its $2s$-perimeter in $\Om$, i.e., 
\begin{equation}\label{defminimizingnonlocminsurf}
P_{2s}(E_*,\Om)\leq P_{2s}(F,\Om) \qquad \forall F\subset\R^n\,,\; F\setminus\Omega=E_*\setminus\Om\,.
\end{equation}
Sets satisfying the minimality condition \eqref{defminimizingnonlocminsurf} have been introduced  in \cite{CRS}. Their boundary $\partial E_*\cap \Omega$ are referred to as (minimizing) {\sl nonlocal ($2s$-)minimal surfaces} in~$\Omega$.  By the minimality condition \eqref{defminimizingnonlocminsurf}, the first inner variation of the $2s$-perimeter vanishes at $E_*$, i.e.,   
\begin{equation}\label{defstatnonlocminsurf}
\delta P_{2s}(E_*,\Om)[X]:=\left[\frac{\de}{\de t}  P_{2s}\big(\phi_t(E_*) , \Omega\big)\right]_{t=0}=0 
\end{equation}
 for any vector field $X\in C^1(\R^n;\R^n)$ compactly supported in $\Omega$, where $\{\phi_t\}_{t\in\R}$ denotes the flow generated by $X$. 
If the boundary $\partial E\cap \Omega$ of a set $E\subset \R^n$ is smooth enough (e.g. a $C^{2}$-hypersurface), the first variation of the $2s$-perimeter at $E$ can be computed explicitly (see e.g. \cite[Section 6]{FFMMM}), and it gives 
 \begin{equation}\label{smoothfirstcarper}
 \delta P_{2s}(E,\Omega)[X]= \int_{\partial E \cap \Omega} \mathrm{H}^{(2s)}_{\partial E}(x) \,X\cdot \nu_E \,\de\mathcal{H}^{n-1}\,,
 \end{equation}
 where $\nu_E$ denotes the unit exterior normal field on $\partial E$, and $\mathrm{H}^{(2s)}_{\partial E}$ is the so-called {\sl nonlocal} (or fractional) {\sl ($2s$-)mean curvature} of $\partial E$, defined by
$$ \mathrm{H}^{(2s)}_{\partial E}(x):={\rm p.v.}\left(\int_{\mathbb{R}^n}\frac{\chi_{\R^n\setminus E}(y)-\chi_E(y)}{|x-y|^{n+2s}}\,\de y\right)\,,\quad x\in\partial E\,.$$
(See \cite{AbaVal} for its geometric interpretation.) Therefore, a set $E_*$ whose boundary is a minimizing nonlocal  $2s$-minimal surface in $\Omega$ (i.e., such that \eqref{defminimizingnonlocminsurf} holds) satisfies in the weak sense  the Euler-Lagrange equation 
\begin{equation}\label{zeromeancurv}
\mathrm{H}^{(2s)}_{\partial E_*} = 0\quad \text{on $\partial E_*\cap\Omega$}\,.
\end{equation}
The weak sense here being precisely relation \eqref{defstatnonlocminsurf}. It has been proved in \cite{CRS} that minimizing nonlocal $2s$-minimal surfaces also satisfies \eqref{zeromeancurv} in a suitable viscosity sense.  This is one of the key ingredient in the regularity theory of \cite{CRS}. It  states  that a minimizing nonlocal minimal surface is a $C^{1,\alpha}$-hypersurface away from a (relatively) closed subset of Hausdorff dimension less than $(n-2)$.  Since then, the $C^{1,\alpha}$ regularity has been improved to $C^\infty$ in \cite{BFV}, and the size of the singular set reduced to $(n-3)$ in \cite{SV1bis}. Whether or not  the singular set can be further reduced remains an open question (see \cite{DDPW,FV} in this direction). 
 \vskip3pt

One of the main objective of this article is to extend the results of \cite{SV1} on the fractional Allen-Cahn equation \eqref{GLIntro} to the case of arbitrary critical points for $s\in(0,1/2)$, i.e., in the regime of nonlocal minimal surfaces. Since we do not assume any kind of minimality, the geometrical objects arising in the limit $\eps\to0$ are not  the ``minimizing'' nonlocal minimal surfaces of  \cite{CRS} (i.e., solutions of \eqref{defminimizingnonlocminsurf}). Our main theorem shows that the limiting equation is  in fact relation \eqref{defstatnonlocminsurf}, which can be interpreted as a weak formulation of the zero nonlocal $2s$-mean curvature equation \eqref{zeromeancurv}. 
We shall referred to as  {\sl stationary~nonlocal~$2s$-minimal surface} in~$\Omega$, the boundary $\partial E_*\cap\Omega$ of a set $E_*\subset\R^n$ satisfying relation \eqref{defstatnonlocminsurf} (i.e., a critical point under inner variations in $\Omega$ of the $2s$-perimeter). 
\vskip3pt

In all our results, we make use of the following set of structural assumptions on the double well potential $W:\R\to [0,\infty)$. 
\begin{enumerate}
\item[(H1)] $W\in C^2\big(\R;[0,\infty)\big)$.
\vskip3pt

\item[(H2)]  $\{W=0\}=\{\pm 1\}$ and $W^{\prime\prime}(\pm1)>0$.
\vskip3pt

\item[(H3)]  There exist $p\in (1,\infty)$  and  a constant  $\boldsymbol{c}_W>0$  such that for all $t\in\R$, 
$$\frac{1}{\boldsymbol{c}_W}\big(|t|^{p-1}-1\big) \leq |W^\prime(t)|\leq \boldsymbol{c}_W\big(|t|^{p-1}+1\big)\,.$$   
\end{enumerate}
Those assumptions are of course satisfied  by the prototypical potential $W(t)=(1-t^2)^2/4$. Notice that assumption (H3) implies that $W$ has a $p$-growth at infinity so that finite energy solutions of  \eqref{GLIntro} belongs to $L^p(\Omega)$. Assuming that (H1)-(H2)-(H3) hold, we will prove that any weak solution of \eqref{GLIntro}-\eqref{Dircondintro} actually belongs to $C^{1,\alpha}_{\rm loc}(\Omega)\cap C^0(\R^n)$ for some $\alpha\in(0,1)$.

 \begin{theorem}\label{main1new}
Assume that $s\in(0,1/2)$ and that {\rm (H1)-(H2)-(H3)} hold. Let $\Omega\subset \R^n$ be a smooth and bounded open set. For a given sequence $\eps_k\downarrow  0$, let $\{g_k\}_{k\in\mathbb{N}}\subset C^{0,1}_{\rm loc}(\R^n)$ be such that $\sup_k\|g_k\|_{L^\infty(\R^n\setminus\Omega)}<\infty$ and $g_k\to g$ in $L^1_{\rm loc}(\R^n\setminus \Omega)$ for a function $g$ satisfying $|g|=1$  a.e. in~$\R^n\setminus\Omega\,$. For each $k\in\mathbb{N}$, let $v_k\in H_{g_k}^s(\Omega)\cap L^p(\Omega)$ be a weak solution of 
 \begin{equation}\label{eqthm}
 \begin{cases}
 \displaystyle  (-\Delta)^{s} v_k+\frac{1}{\varepsilon_k^{2s}}W^\prime(v_k) =0 & \text{in $\Omega$}\,,\\
  v_k=g_k & \text{in $\R^n\setminus \Omega$}\,.
\end{cases}
\end{equation}
 If $\sup_k \mathcal{E}_{\varepsilon_k}(v_k,\Omega)<\infty$, then there exist a (not relabeled) subsequence and a set $E_*\subset \R^n$ of finite $2s$-perimeter in $\Omega$ such that 
  \begin{itemize}[leftmargin=25pt]
\item[ \rm  (i)] $v_k\to v_*:=\chi_{E_*}-\chi_{\R^n\setminus E_*}$ strongly in $H^{s^\prime}_{\rm loc}(\Omega)\cap L^2_{\rm loc}(\R^n)$ for every $s^\prime<\min(2s,1/2)$;  
\vskip5pt

\item[\rm (ii)] the set $E_*\cap \Omega$ is open; 
\vskip5pt

\item[(iii)] the boundary $\partial E_*\cap\Omega$ is a stationary nonlocal $2s$-minimal surface in $\Omega$ (i.e., relation \eqref{defstatnonlocminsurf} holds).  
\end{itemize}
 In addition, for every smooth open set $\Omega'\subset\Omega$ such that $\overline{\Omega'}\subset\Omega$, 
 \begin{itemize}[leftmargin=25pt]
\item[ \rm  (iv)] $ \mathcal{E}(v_k,\Omega') \to 2\gamma_{n,s}P_{2s}(E_*,\Omega^\prime)$; 
\vskip5pt

\item[ \rm  (v)] $\int_{\Omega^\prime}W(v_k)\,\de x=O(\varepsilon_k^{\min(4s,\alpha)})$ for every $\alpha\in(0,1)$; 
\vskip5pt

\item[\rm  (vi)] $\displaystyle \frac{-1}{\eps_k^{2s}}\,W^\prime(v_k)\to  \left(\frac{\gamma_{n,s}}{2}\int_{\R^n}\frac{|v_*(x)-v_*(y)|^2}{|x-y|^{n+2s}}\,\de y\right)v_*(x)$ strongly in  $H^{-s}(\Omega^\prime)$ and 
\vskip3pt

\noindent weakly in $L^{\bar p}(\Omega^\prime)$ for every $\bar p<1/2s\,$; 
\vskip5pt 

\item[\rm  (vii)] $v_k\to v_*$ in $C^{1,\alpha}_{\rm loc}(\Omega\setminus \partial E_*)$ for some $\alpha=\alpha(n,s)\in(0,1)$; 
\vskip5pt 

\item[\rm  (viii)] for each $t\in(-1,1)$, the level set $L^t_k:=\{v_k=t\}$ converges locally uniformly  in $\Omega$ to $\partial E_*\cap \Omega$, i.e., for every compact set $K\subset \Omega$ and every $r>0$,
$$L^t_k\cap K\subset \mathscr{T}_r(\partial E_*\cap \Omega) \quad \text{and}\quad \partial E_*\cap K\subset \mathscr{T}_r(L_k^t\cap \Omega)$$
whenever $k$ is large enough. Here, $\mathscr{T}_r(A)$ represents the open tubular neighborhood of radius $r$ of a set $A$. 
\end{itemize}
 \end{theorem}

Comparing this result to what is known on the classical Allen-Cahn equation \eqref{AC0}, we can now say that the main difference lies in the strong compactness of solutions (at and  above the energy regularity level), and the resulting continuity of the energy. In some sense, such compactness is not really surprising as one may guess that $H^{s^\prime}$-regularity with $s^\prime\in(0,1/2)$ is not strong enough to capture folding of interfaces. The key   argument in proving compactness in the energy space rests on the fractional scaling of the equation and the {\sl Marstrand's Theorem} (see e.g. \cite{Matti}), a purely measure theoretic result. In the same flavour,  strong convergence of solutions to the $p$-Ginzburg-Landau system (involving the $p$-Laplacian) towards stationary $p$-harmonic maps has been proved in \cite{W} for non-integer values of the exponent $p$. Compactness at the $H^{s^\prime}$-level with $s^\prime<\min(2s,1/2)$ is in turn a much more delicate issue. We establish such compactness combining fine elliptic estimates in the region $|v_k|\simeq 1$ together with quantitative estimates on the volume of the sublevel sets $\{|v_k| \lesssim 1/2\}$. To derive these volume estimates,  we apply the {\sl quantitative stratification principle} of singular sets  introduced in \cite{CheegNab} (in the context of harmonic maps and minimal currents) and generalized to an abstract framework in \cite{FMS}. We point out that this stratification principle does not apply verbatim to our setting since solutions of \eqref{GLIntro} are smooth, and non trivial adjustments have to be made. To the best of our knowledge, this is the first  time  that the quantitative stratification principle is applied to an Allen-Cahn (or Ginzburg-Landau) type equation. 

\begin{remark}
We emphasize that Theorem \ref{main1new} applies to minimizing solutions of  \eqref{eqthm} since the function $\chi_\Omega-g_k\chi_{\R^n\setminus\Omega}$ is an admissible competitor of uniformly bounded energy. In particular, this theorem extend the result of \cite{SV1} for $s\in(0,1/2)$ to arbitrary solutions (with uniformly  bounded energy) together with a full set of new estimates. However, if we assume that each $v_k$ is minimizing, i.e.,  $\mathcal{E}_{\varepsilon_k}(v_k,\Omega)\leq \mathcal{E}_{\varepsilon_k}(w,\Omega)$ for every $w\in H^s_{g_k}(\Omega)$, then \cite{SV1}  shows that the limiting set $E_*$ is a minimizing nonlocal minimal surface in $\Omega$ in the sense of~\cite{CRS}, i.e., $E_*$ satisfies \eqref{defminimizingnonlocminsurf}. 
\end{remark}

\begin{remark}
Non trivial examples of (entire) stationary nonlocal minimal surfaces have been constructed in \cite{DDPW}. These examples are nonlocal analogues of classical minimal surfaces such as catenoids, or Lawson cones (see also \cite{CFSMT,CFW} for Delaunay type surfaces with constant nonlocal mean curvature). It would be very interesting to construct solutions of the fractional Allen-Cahn equation concentrating as $\eps\to 0$ on such surfaces.  
\end{remark}

In proving Theorem \ref{main1new}, we actually investigate the more general case where \eqref{GLIntro} is replaced by 
\begin{equation}\label{nonhomeq}
 (-\Delta)^{s} v_\eps+\frac{1}{\varepsilon^{2s}}W^\prime(v_\eps) =f_\eps \quad \text{in $\Omega$}\,,
 \end{equation}
with a smooth right hand side $f_\eps$ controlled (with respect to $\eps$) in a suitable Sobolev space. Considering such inhomogeneous equation is a way to analyse the asymptotic behavior of an arbitrary sequence of (smooth) functions $v_\eps\in H^s_{g_\eps}(\Omega)$ satisfying $\mathcal{E}_\eps(v_\eps,\Omega)=O(1)$ and 
$$\big\| (-\Delta)^{s} v_\eps+\varepsilon^{-2s}W^\prime(v_\eps) \big\|_{W^{1,q}(\Omega)}=O(1) \quad \text{as $\eps\to 0$}\,,$$
for some suitable exponent  $q$. 

In the classical case $s=1$, such analysis has been pursued in  \cite{Ton1,Ton2} (in continuation to~\cite{HT}). For $s=1$, one considers a sequence $\{u_\eps\}$ of (uniformly bounded) smooth functions on $\Omega$ with uniformly bounded energy \eqref{EN0}, and satisfying 
\begin{equation}\label{expoton}
\|-\eps\Delta u_\eps+\eps^{-1}W^\prime(u_\eps)\|_{W^{1,q}(\Omega)}=O(1)\quad\text{for some $q>n/2$}\,.
\end{equation}
Under this assumption, there is still a well defined limiting interface as $\eps\to0$, which is given by an $(n-1)-$integral varifold with bounded first variation. In addition, the measure theoretic mean curvature of this varifold is given by the weak $W^{1,q}$-limit of  $-\eps\Delta u_\eps+\eps^{-2}W^\prime(u_\eps)$, and it belongs to $L^{r}$,  $r:=q(n-1)/(n-q)>(n-1)$, with respect to the $(n-1)$-dimensional measure on  the interface. The range of exponents in \eqref{expoton} thus leads to the maximal range of integrability exponents in Allard's regularity theory  \cite{All,Sim2}, and the limiting interface is (partially) regular, see \cite{Scha}.   

Considering the inhomogeneous equation \eqref{nonhomeq} (complemented with the exterior Dirichlet condition \eqref{Dircondintro}), we assume that $f_\eps\in C^{0,1}(\Omega)$ satisfies 
$$\eps^{2s}\|f_\eps\|_{L^\infty(\Omega)}+\|f_\eps\|_{W^{1,q}(\Omega)}=O(1) \quad \text{for some $q>n/(1+2s)$}\,.$$
In this setting, we have proved that the main conclusions in Theorem \ref{main1new} hold (see Theorem~\ref{main1} and Theorem \ref{thmlastone} for precise statements) with a limiting set $E_*$ satisfying 
\begin{equation}\label{weakPNMC}
\delta P_{2s}(E_*,\Omega)[X]=\frac{1}{\gamma_{n,s}}\int_{E_*\cap\Omega} {\rm div}(fX)\,\de x \qquad\forall X\in C^1_c(\Omega;\R^n)\,,
\end{equation}
where $f$ is the weak limit of $f_\eps$ in $W^{1,q}(\Omega)$ as $\eps\to 0$.  In view of \eqref{smoothfirstcarper}, the boundary of $E_*$ satisfies in the weak sense
\begin{equation}\label{presnonlocmceq}
{\rm H}^{(2s)}_{\partial E_*}=\frac{1}{\gamma_{n,s}} f\quad \text{on $\partial E_*\cap \Omega$}\,.
\end{equation}
We shall refer to this equation as the {\sl prescribed nonlocal ($2s$-)mean curvature equation} in $\Omega$, and to weak solutions as  {\sl surfaces of prescribed nonlocal ($2s$-)mean curvature}. 
\vskip3pt

Our analysis of the fractional Allen-Cahn equation naturally leads to the regularity problem for stationary nonlocal minimal surfaces, or more generally, for weak solutions of \eqref{presnonlocmceq} with $f\in W^{1,q}(\Omega)$ and $q>n/(1+2s)$. 
In this direction, we have obtained partial results (compare to \cite{CRS}), and some of the main conclusions can be summarized in the following theorem (see Section \ref{complrestprescrNMMC} for the complete set of results).  

\begin{theorem}\label{mainthm2}
For $s\in(0,1/2)$, let $E_*\subset\R^n$ be  a Borel set satisfying $P_{2s}(E_*,\Omega)<\infty$ and \eqref{weakPNMC} for some function $f\in W^{1,q}(\Omega)$ and $q>n/(1+2s)$. Then, 
\begin{itemize}[leftmargin=30pt]
\item[ \rm  (i)] $E_*\cap \Omega$ is (essentially) open; 
\vskip5pt

\item[ \rm  (ii)] if $\partial E_*\cap \Omega$ is not empty, it has a Minkowski codimension equal to 1;  
\vskip5pt

\item[ \rm  (iii)] $P_{2s^\prime}(E_*,\Omega^\prime)<\infty$ for every $s^\prime\in(0,1/2)$ and every open set $\Omega^\prime$ such that $\overline{\Omega^\prime}\subset \Omega$. 
\end{itemize}
\end{theorem}

This theorem is obtained through a blow-up analysis for solutions of \eqref{weakPNMC}. Such analysis rests on a preliminary result stating that solutions of \eqref{weakPNMC} are compact in the energy space. This is of course the sharp interface analogue of the compactness property for the fractional Allen-Cahn equation, and it  relies again on Marstrand's Theorem. Note that such compactness doesn't hold if $P_{2s}$ is replaced by the usual (distributional) perimeter of sets (see \cite{Scha}). With this compactness at hand, we have  applied the quantitative stratification principle of \cite{CheegNab,FMS} to solutions of  \eqref{weakPNMC}, leading to conclusions (ii) and (iii).

\begin{remark}
Theorem \ref{mainthm2} is new even in the case $f=0$, i.e., in the case of stationary nonlocal minimal surfaces. Whether or not solutions to \eqref{defstatnonlocminsurf} or \eqref{weakPNMC} are more regular (in the spirit of the minimizing case \cite{CRS}) remains an open question. Let us mention that, in the recent article \cite{cinserval}, it has been proved that (some) {\sl stable solutions} of \eqref{defstatnonlocminsurf} have locally finite perimeter in $\Omega$. In particular, their boundary are rectifiable.  Note that item (iii) in Theorem \ref{mainthm2} goes somehow in this direction. Indeed, if we knew that  $(1-2s^\prime)P_{2s^\prime}(E_*,\Omega^\prime)=O(1)$ as $s^\prime\uparrow 1/2$, then it would say that $E_*$ has finite perimeter in the open set $\Omega^\prime$ since $(1-2s^\prime)P_{2s}(\cdot,\Omega^\prime)$ converges to the usual perimeter functional as $s^\prime\to 1/2$, see \cite{ADPM,Dav}. Unfortunately, the bound  $P_{2s^\prime}(E_*,\Omega^\prime)<\infty$ is obtained by a compactness argument (hinged on the quantitative stratification principle), and no explicit dependence on $s^\prime$ seems to follow. 
\end{remark}

\begin{remark}
A set $E_*\subset \R^n$  satisfying 
\begin{multline}\label{nonlocalminf}
P_{2s}(E_*,\Omega)-\frac{1}{\gamma_{n,s}}\int_{E_*\cap\Omega}f\,\de x\leq P_{2s}(F,\Omega)-\frac{1}{\gamma_{n,s}}\int_{F\cap\Omega}f\,\de x \quad\\[5pt] 
\forall F\subset\R^n\,,\; F\setminus\Omega=E_*\setminus\Om\,,
\end{multline}
provides a solution of \eqref{weakPNMC}. It corresponds to a minimizing solution of the prescribed nonlocal $2s$-mean curvature equation. Since $f\in W^{1,q}(\Omega)$ with $q>n/(1+2s)$, we have $f\in L^r(\Omega)$ with $r:=nq/(n-q)>n/2s$. Hence we can apply in this case the  regularity theory for nonlocal {\sl almost minimal} surfaces of \cite{CapGui}. Combined with \cite{SV1bis},  it shows that  $\partial E_*\cap\Omega$ is a $C^{1,\alpha}$-hypersurface for every $\alpha<(1+2s-n/q)/(n+2s)$ away from a relatively closed subset of Hausdorff dimension less then $(n-3)$ (and discrete for $n=3$). 
\end{remark}

\begin{remark}
The notion of stationary nonlocal minimal surface is strongly related to {\sl stationary fractional $s$-harmonic maps} into a sphere. With this respect, this article is natural continuation to the analysis of the fractional Ginzburg-Landau equation and $1/2$-harmonic maps \cite{MilSir} by the two first authors. Fractional harmonic maps into a sphere were originally introduced in \cite{DaRi} for $s=1/2$ and $n=1$. A mapping $v:\R^n\to\mathbb{S}^{d-1}$ (of finite fractional Dirichlet energy) is called a weakly $s$-harmonic map in    $\Omega$ if  
$$\left[\frac{\de}{\de t} \mathcal E \left(\frac{v+t \varphi}{|v+t \varphi|} , \Omega\right)\right]_{t=0}=0 \qquad \forall \varphi\in\mathscr{D}(\Omega;\R^d)\,.$$
As shown in \cite{MilSir} for $s=1/2$, this condition leads (in the weak sense) to the Euler-Lagrange equation 
\begin{equation}\label{sharmmap}
(-\Delta)^sv(x)= \left(\frac{\gamma_{n,s}}{2}\int_{\R^n}\frac{|v(x)-v(y)|^2}{|x-y|^{n+2s}}\,\de y\right)v(x) \quad \text{in $\Omega$}\,.
\end{equation} 
For {\sl any} set $E\subset \R^n$ of finite $2s$-perimeter in $\Omega$, the function $v=\chi_E-\chi_{\R^n\setminus E}$ turns out to satisfy equation \eqref{sharmmap} (see Lemma \ref{lemeqharmmap}). In other words, if we identify $\{\pm 1\}$ with $\{\pm1\}\times\{0\}\subset \R\times\mathbb{R}^{d-1}$, the function $\chi_E-\chi_{\R^n\setminus E}$ is a weakly $s$-harmonic map into $\mathbb{S}^{d-1}$ in the open set $\Omega$ (explaining in particular item (vi) in Theorem \ref{main1new}). As a consequence, no regularity can be expected for weakly $s$-harmonic maps for $s<1/2$. This is of course in analogy with the non-regularity result of \cite{Riv} for usual weakly harmonic maps into a manifold (for $n\geq 3$).   
Stationary $s$-harmonic maps into $\mathbb{S}^{d-1}$ are defined as weakly $s$-harmonic maps  satisfying the additional stationarity condition $\delta\mathcal{E}(v,\Omega)=0$ (where $\delta\mathcal{E}(\cdot,\Omega)$ denotes the first inner variation of $\mathcal{E}(\cdot,\Omega)$). One may expect that, for such $s$-harmonic maps, some partial regularity holds (see \cite{DaRi,MilSir} in the case $s=1/2$). In  view of \eqref{idDirPersharpint}, if a set $E_*\subset\R^n$ satisfies \eqref{defstatnonlocminsurf} (i.e., whose boundary is a stationary nonlocal $2s$-minimal surface in $\Omega$), then the function $\chi_{E_*}-\chi_{\R^n\setminus E_*}$ is a stationary $s$-harmonic map in $\Omega$. It shows that, for general stationary $s$-harmonic maps into a sphere, the singular set (or discontinuity set) can have a positive $\mathcal{H}^{n-1}$-measure if $s<1/2$ (compare to the vanishing $\mathcal{H}^{n-1}$-measure of the singular set  for stationary $1/2$-harmonic maps, see \cite{MilSir}). 
\end{remark}

As it is customary by now, our analysis rely on the Caffarelli-Silvestre extension procedure~\cite{CaffSil} to the open upper half space $\R^{n+1}_+:=\R^n\times (0,\infty)$. This extension allows us to represent $(-\Delta)^s$ as the Dirichlet-to-Neumann operator associated to the degenerate elliptic operator $L_s:=-{\rm div}(z^{1-2s}\nabla\cdot)$  on $\R^{n+1}_+$, where $z\in(0,\infty)$ denotes the extension variable. In this way, we rewrite solutions to the fractional Allen-Cahn equation or the prescribed nonlocal $2s$-mean curvature equation as 
$L_s$-harmonic functions in $\R^{n+1}_+$ satisfying nonlinear boundary conditions.  In the spirit of \cite{CRS}, this extension leads to fundamental monotonicity formulas. All the functional and variational aspects surrounding the fractional Laplacian $(-\Delta)^s$ and the Caffarelli-Silvestre extension are presented in Section \ref{prelim}. In Section \ref{FractAC}, we prove some basic (but necessary) regularity estimates on solutions to the fractional Allen-Cahn equation and $L_s$-harmonic functions with Allen-Cahn degenerate boundary reaction. A first part of the asymptotic analysis as $\eps\to0$ is performed in Section \ref{epsregtitle} for  Allen-Cahn degenerate boundary reactions. Consequences for the fractional Allen-Cahn equation are then given in Section~\ref{FGLasymp}.  Section \ref{GHSNMC} is devoted to the analysis of surfaces of prescribed nonlocal mean curvature. Finally, we prove in Section \ref{imprsect} the aforementioned volume estimate on transition sets, and complete our asymptotic analysis of the fractional Allen-Cahn equation.

\subsection*{{Notation}}
Throughout the paper, $\R^n$ is identified with $\partial  \mathbb{R}^{n+1}_+=\R^n\times\{0\}$. More generally, sets $A\subset\mathbb{R}^n$ are identified with $A\times\{0\}\subset\partial  \mathbb{R}^{n+1}_+$. Points in $\mathbb{R}^{n+1}$ are written $\mathbf{x}=(x,z)$ with $x\in\mathbb{R}^n$ and $z\in\mathbb{R}$.  
We shall denote by $B_r(\mathbf{x})$ the open ball in $\mathbb{R}^{n+1}$ of radius $r$ centered at $\mathbf{x}=(x,z)$, while $D_r(x):= B_r(\mathbf{x})\cap\mathbb{R}^{n}$ is the open ball (or disc) in $\R^n$ centered at $x$. For an arbitrary set $G\subset  \mathbb{R}^{n+1}$, we write 
$$G^+:=G\cap \mathbb{R}^{n+1}_+\quad\text{ and }\quad\partial^+ G:=\partial G\cap \mathbb{R}^{n+1}_+\,.$$
If $G\subset\R^{n+1}_+$ is a bounded open set, we shall say that $G$ is  {\bf admissible} whenever 
\begin{itemize}
\item $\partial G$ is Lipschitz regular;  
\vskip2pt
\item the (relative) open set $\partial^0G\subset\R^n$ defined by 
$$\partial^0G:=\Big\{\mathbf{x}\in\partial G\cap\partial\R^{n+1}_+ : B^+_{r}(\mathbf{x})\subset G \text{ for some $r>0$}\Big \}\,,$$
is non empty and has Lipschitz boundary; 
\vskip2pt

\item $\partial G=\partial^+ G\cup\overline{\partial^0G}\,$.
\end{itemize}

Finally, we shall always denote by $C$ a generic positive constant which may only depend on the dimension $n$, and possibly changing from line to line. If a constant depends on additional given parameters, we shall write those parameters using the subscript notation.

%%%%%%%%%%%%%%%%%%%%%%%%%%%%%%%%%%%%%%%%%%%%%%%%%%%%%%%
%%%%%%%%%%%%%%%%%%%%%%%%%%%%%%%%%%%%%%%%%%%%%%%%%%%%%%%
   								 
\section{Functional spaces and the fractional Laplacian} \label{prelim}   
								 
%%%%%%%%%%%%%%%%%%%%%%%%%%%%%%%%%%%%%%%%%%%%%%%%%%%%%%%
%%%%%%%%%%%%%%%%%%%%%%%%%%%%%%%%%%%%%%%%%%%%%%%%%%%%%%%

\subsection{$H^{s}$-spaces for $s\in(0,1/2)$}
 
For an open set $\Omega\subset \mathbb{R}^n$,  the fractional Sobolev space $H^{s}(\Omega)$ is made of functions $v\in L^2(\Omega)$ such that\footnote{The normalization constant $\gamma_{n,s}$ is chosen in such a way that $\displaystyle [v]^2_{H^{s}(\R^n)}=\int_{\R^n}(2\pi|\xi|)^{2s}|\hat v|^2\,\de\xi\,$. }  
$$[v]^2_{H^{s}(\Omega)}:=\frac{\gamma_{n,s}}{2}\iint_{\Omega\times \Omega} \frac{|v(x)-v(y)|^2}{|x-y|^{n+2s}}\,\de x\de y<\infty\,,\quad \gamma_{n,s}:=s\,2^{2s}\pi^{-\frac{n}{2}}\frac{\Gamma\big(\frac{n+2s}{2}\big)}{\Gamma(1-s)} \,.$$ 
It is a separable Hilbert space normed by $\|\cdot\|^2_{H^{s}(\Omega)}:= \|\cdot\|^2_{L^2(\Omega)}+[\cdot]^2_{H^{s}(\Omega)}$. 
The space $ H^{s}_{\rm loc}(\Omega)$ denotes the class of functions whose restriction to any relatively compact  open subset $\Omega'$ of $\Omega$ belongs to $H^{s}(\Omega')$. 
The linear subspace $H^{s}_{00}(\Omega) \subset H^{s}(\mathbb{R}^n)$ is defined by 
$$H^{s}_{00}(\Omega):=\big\{v\in H^{s}(\mathbb{R}^n) :  v=0 \text{ a.e. in } \R^n\setminus\Omega\big\}\,. $$
Endowed with the  induced norm,   
 $H^{s}_{00}(\Omega)$ is also an Hilbert space, and for  $v\in H^{s}_{00}(\Omega)$, 
\begin{align}\label{defErond}
[v]^2_{H^{s}(\mathbb{R}^n)}&=2\mathcal{E}(v,\Omega) \\
\nonumber&=\frac{\gamma_{n,s}}{2}\iint_{\Omega\times\Omega}  \frac{|v(x)-v(y)|^2}{|x-y|^{n+2s}}\,\de x\de y 
+ \gamma_{n,s} \iint_{\Omega\times \Omega^{c}}  \frac{|v(x)|^2}{|x-y|^{n+2s}}\,\de x\de y \\
\nonumber &=[v]^2_{H^{s}(\Omega)} + \int_{\Omega} \boldsymbol{\rho}_\Omega(x) |v(x)|^2\,\de x\,,
\end{align}
where $\mathcal{E}(\cdot,\Omega)$ is  the {\it fractional Dirichlet energy} defined in \eqref{defenergE}, and 
$$ \boldsymbol{\rho}_\Omega(x) := \gamma_{n,s}\int_{\R^n\setminus \Omega}\frac{1}{|x-y|^{n+2s}}\,\de y\,.$$
Since $s\in(0,1/2)$, if $\Omega$ is bounded and its boundary is smooth enough (e.g. if  $\partial \Omega$ is Lipschitz regular), then 
$$ \int_{\Omega} \boldsymbol{\rho}_\Omega(x) |v(x)|^2\,\de x\leq C_\Omega \|v\|^2_{H^{s}(\Omega)}\qquad \forall v\in H^s(\Omega)\,,$$
for a constant $C_\Omega=C_\Omega(s)>0$. As a consequence, if $v\in H^s(\Omega)$ and $\widetilde v$ denotes the extension of $v$ by zero  outside $\Omega$, then  
$$ \|v\|_{H^{s}(\Omega)}\leq \|\widetilde v\|_{H^{s}(\R^n)}\leq (C_\Omega+1)^{\frac{1}{2}}\|v\|_{H^{s}(\Omega)}\,.$$
In particular, if $\partial \Omega$ is smooth enough, then $H^{s}_{00}(\Omega)=\big\{\widetilde v : v\in  H^s(\Omega)\big\}$ 
(see \cite[Corollary~1.4.4.5]{G}), and  (see \cite[Theorem~1.4.2.2]{G}) 
\begin{equation}\label{densitysmoothH1/200}
H^{s}_{00}(\Omega)= \overline{\mathscr{D}(\Omega)}^{\,\|\cdot\|_{H^{s}(\mathbb{R}^n)}}\,.
 \end{equation}
The topological dual space of $H^{s}_{00}(\Omega)$ is denoted by $H^{-s}(\Omega)$.
\vskip5pt

We are interested in the class of functions 
$$\widehat{H}^{s}(\Omega):=\Big\{v\in L^{2}_{\rm loc}(\mathbb{R}^n) : \mathcal{E}(v,\Omega)<\infty\Big\} \,.$$
The following properties hold for any bounded open subsets $\Omega$ and $\Omega'$ of $\R^n$: 
\begin{itemize}
\vskip5pt
\item $ \widehat{H}^{s}(\Omega)$ is a linear space; 
\vskip5pt
\item $ \widehat{H}^{s}(\Omega) \subset  \widehat{H}^{s}(\Omega')$ whenever $\Omega'\subset\Omega$, and    
$ \mathcal{E}(v,\Omega')\leq \mathcal{E}(v,\Omega)\,$;
\vskip5pt
\item $ \widehat{H}^{s}(\Omega)\cap H^{s}_{\rm loc}(\mathbb{R}^n) \subset  \widehat{H}^{s}(\Omega')\,$;
\vskip5pt
\item $H^{s}_{\rm loc}(\mathbb{R}^n)\cap L^\infty(\mathbb{R}^n) \subset  \widehat{H}^{s}(\Omega)\,$.
\end{itemize}
From Lemma \ref{adminHchap} below, it is straightforward to show that $\widehat{H}^{s}(\Omega)$ is actually  a Hilbert space for the scalar product induced by the norm $v\mapsto \big(\|v\|^2_{L^2(\Omega)}+ \mathcal{E}(v,\Omega)\big)^{1/2}$  (see e.g. \cite[proof of Lemma 2.1]{MilSir}). 

\begin{lemma}\label{adminHchap}
Let $x_0\in\Omega$ and $\rho>0$ be  such that $D_{\rho}(x_0)\subset\Omega$.  There exists a constant $C_\rho>0$, independent of~$x_0$, such that 
$$\int_{\R^n}\frac{|v(x)|^2}{(|x-x_0|+1)^{n+2s}}\,\de x\leq C_{\rho}\left(\mathcal{E}\big(v,D_{\rho}(x_0)\big)+\|v\|^2_{L^2(D_{\rho}(x_0))}\right)$$
for every $v\in  \widehat{H}^{s}(\Omega)$. 
\end{lemma}

\begin{remark}\label{remEDB}
If $v\in \widehat{H}^{s}(\Omega)$, then  
$v+ H^{s}_{00}(\Omega)\subset  \widehat{H}^{s}(\Omega)$.   
Conversely,  if 
$v=g$ a.e. in $\R^n\setminus\Omega$ for some functions $v$ and $g$ in $\widehat{H}^{s}(\Omega)$, then $v-g\in H^{s}_{00}(\Omega)$. 
As a consequence, for $g\in \widehat{H}^{s}(\Omega)$, 
$$H^{s}_g(\Omega):= \left\{v\in \widehat{H}^{s}(\Omega;\mathbb{R}^m) : v=g\, \text{ a.e. in $\R^n\setminus\Omega$}\right\}=g+H^{s}_{00}(\Omega)\,.$$ 
Note that  $H^{s}_g(\Omega)\subset  H^{s}_{\rm loc}(\mathbb{R}^n)$ whenever $g\in \widehat{H}^{s}(\Omega)\cap H^{s}_{\rm loc}(\mathbb{R}^n)$. 
\end{remark}

%%%%%%%%%%%%%%%%%%%%%%%%%%%%%%%%%%%%%%%%%%%%%%%%%%%%%%%

\subsection{The fractional Laplacian}

%%%%%%%%%%%%%%%%%%%%%%%%%%%%%%%%%%%%%%%%%%%%%%%%%%%%%%%

Let  $\Omega\subset \mathbb{R}^n$ be a bounded open set.  We define the fractional  Laplacian $(-\Delta)^s: \widehat{H}^{s}(\Omega)\to (\widehat{H}^{s}(\Omega))^\prime$ as the continuous linear operator induced by the quadratic form $\mathcal{E}(\cdot,\Omega)$. More precisely,  given a function $v\in \widehat{H}^{s}(\Omega)$, we define its {\it distributional fractional Laplacian} $ (-\Delta)^{s} v$ through 
its action on $ \widehat{H}^{s}(\Omega)$ by setting 
\begin{multline}\label{deffraclap}
\big\langle  (-\Delta)^{s} v, \varphi\big\rangle_\Omega:=\frac{\gamma_{n,s}}{2}\iint_{\Omega\times\Omega}  \frac{\big(v(x)-v(y)\big)\big(\varphi(x)-\varphi(y)\big)}{|x-y|^{n+2s}}\,\de x\de y\\ 
+ \gamma_{n,s} \iint_{\Omega\times \Omega^c}  \frac{\big(v(x)-v(y)\big)\big(\varphi(x)-\varphi(y)\big)}{|x-y|^{n+2s}}\,\de x\de y \,. 
\end{multline}
If $v$ is a smooth bounded function, then the distribution $ (-\Delta)^{s} v$ can be rewritten from \eqref{deffraclap} as a pointwise defined function which coincides with the one given by  formula~\eqref{formpvfraclap}. Notice also that the restriction of the linear form $ (-\Delta)^{s} v$ to the subspace $H^{s}_{00}(\Omega)$ belongs to $H^{-s}(\Omega)$ with the  estimate 
\begin{equation}\label{estinormH-1/2fraclap}
\| (-\Delta)^{s} v\|^2_{H^{-s}(\Omega)}\leq 2\mathcal E(v,\Omega)\,.
\end{equation}
In this way, $(-\Delta)^{s} v $ appears to be the first outer variation of   $\mathcal{E}(\cdot,\Omega)$ at $v$ with respect to pertubations supported in $\Omega$, i.e.,  
\begin{equation}\label{firstvarcalE}
\big\langle  (-\Delta)^{s} v, \varphi\big\rangle_\Omega=\left[\frac{\de}{\de t} \mathcal E (v+t \varphi , \Omega)\right]_{t=0} 
\end{equation}
for all $\varphi \in H^{s}_{00}(\Omega)$.

\begin{remark}
If $\Omega'\subset \Omega$ are two open sets and $v\in \widehat{H}^{s}(\Omega)$, then 
$$\big\langle  (-\Delta)^{s} v, \varphi\big\rangle_\Omega= \big\langle  (-\Delta)^{s} v, \varphi\big\rangle_{\Omega'}$$
for all $\varphi\in H^{s}_{00}(\Omega')$.
\end{remark}

%%%%%%%%%%%%%%%%%%%%%%%%%%%%%%%%%%%%%%%%%%%%%%%%%%%%%%%

 \subsection{Weighted Sobolev spaces}

%%%%%%%%%%%%%%%%%%%%%%%%%%%%%%%%%%%%%%%%%%%%%%%%%%%%%%%

For an open set $G\subset \R^{n+1}$, we define the weighted $L^2$-space  
$$L^2(G,|z|^a\de \mathbf{x}):= \Big\{u\in L^1_{\rm loc}(G) :  |z|^{\frac{a}{2}} u \in L^2(G)\Big\} \quad \text{with }a:=1-2s\,,$$
normed by 
$$\|u\|^2_{L^2(G,|z|^a\de \mathbf{x})}:=\int_G |z|^{a}|u|^2\,\de \mathbf{x}\,. $$
Accordingly, we introduce the weighted Sobolev space 
$$H^1(G,|z|^a\de \mathbf{x}):= \Big\{u\in L^2(G,|z|^a\de \mathbf{x}) : \nabla u \in L^2(G,|z|^a\de \mathbf{x})\Big\} \,,$$ 
normed by 
$$\|u\|_{H^1(G,|z|^a\de \mathbf{x})}:=\|u\|_{L^2(G,|z|^a\de \mathbf{x})}+\|\nabla u\|_{L^2(G,|z|^a\de \mathbf{x})}\,. $$
Both $L^2(G,|z|^a\de \mathbf{x})$ and $H^1(G,|z|^a\de \mathbf{x})$ are separable Hilbert spaces when equipped with the scalar product induced by their respective Hilbertian norm. 

If  $\Omega$ denotes a (relatively) open subset of $ \partial\R^{n+1}_+\simeq \R^n$ such that  $\Omega\subset\partial G$, we set 
$$L^2_{{\rm loc}}(G\cup\Omega, |z|^a\de \mathbf{x}):= \Big\{u\in L^1_{\rm loc}(G) :  |z|^{\frac{a}{2}} u \in L^2_{\rm loc}(G\cup\Omega)\Big\} \,,$$
and
$$ H^1_{{\rm loc}}(G\cup\Omega,|z|^a\de \mathbf{x}) := \Big\{u\in L^2_{{\rm loc}}(G\cup\Omega,|z|^a\de \mathbf{x}) : \nabla u \in L^2_{{\rm loc}}(G\cup\Omega,|z|^a\de \mathbf{x})\Big\}\,.$$

\begin{remark}\label{trace}
For a bounded admissible open set $G\subset \R^{n+1}_+$, the space $L^2(G,|z|^a\de \mathbf{x})$ embeds continuously into $L^\gamma(G)$  for every $1\leq \gamma<\frac{1}{1-s}$ by H\"older's inequality. In particular, 
\begin{equation}\label{contembedd}
H^1(G,|z|^a\de \mathbf{x})\hookrightarrow W^{1,\gamma}(G) 
\end{equation}
continuously for every $1<\gamma<\frac{1}{1-s}$. As a first consequence,  $H^1(G,|z|^a\de \mathbf{x})\hookrightarrow L^{1}(G)$ with compact embedding. Secondly, for such $\gamma$'s,  the compact  linear trace operator 
\begin{equation}\label{conttracW1}
u\in W^{1,\gamma}(G)\mapsto u_{|\partial^0 G}\in L^1(\partial^0G)
\end{equation}
induces a compact linear trace operator from $H^1(G,|z|^a\de \mathbf{x})$ into $L^1(\partial^0G)$,  extending the usual trace of smooth functions. We may denote by $u_{|\partial^0G}$ the trace of $u\in H^1(G,|z|^a\de \mathbf{x})$ on $\partial^0G$, or simply by $u$ if it is clear from the context. Finally, we write   
$H^1(G,|z|^a\de \mathbf{x})\cap L^p(\partial^0G)$ the class of functions $u\in H^1(G,|z|^a\de \mathbf{x})$ such that  $u_{|\partial^0G}\in L^p(\partial^0 G)$. 
\end{remark}

\begin{lemma}\label{poincare}
There exists a constant $\boldsymbol{\lambda}_{n,s}>0$ depending only on $n$ and $s$ such that for every $r>0$, and every $u\in H^1(B_r^+,|z|^a\de{\bf x})$, 
$$ \big\| u - [u]_r  \big\|_{L^1(D_r)} \leq \boldsymbol{\lambda}_{n,s}\,r^{\frac{n+2s}{2}}\|\nabla u\|_{L^2(B_r^+,|z|^a\de{\bf x})}\,,$$
where $[u]_r$ denotes the average of $u$ over $D_r$.  
\end{lemma}

\begin{proof} 
By scaling it suffices to consider the case $r=1$. We claim that there exists a constant $c_n>0$ such that for every $u\in W^{1,1}(B_1^+)$,  
\begin{equation}\label{poinc}
 \big\| u - [u]_1  \big\|_{L^1(D_1)} \leq c_n \int_{B_1^+}|\nabla u|\,\de \mathbf{x}\,.
 \end{equation}
Then the conclusion follows from H\"older's inequality. To prove \eqref{poinc} it is enough to consider functions $u\in W^{1,1}(B_1^+)$ satisfying $[u]_1=0$. Then 
we argue by contradiction assuming that there exists a sequence $\{u_k\}_{k\in\mathbb{N}}\subset W^{1,1}(B_1^+)$ such that $[u_k]_1=0$ and $\|u_k\|_{L^1(D_1)}> k \|\nabla u_k\|_{L^1(B^+_1)}$ for every $k\in\mathbb{N}$. Replacing $u_k$ by $u_k/\|u_k\|_{L^1(B_1^+)}$ if necessary, we can assume that $\|u_k\|_{L^1(B_1^+)}=1$ for each $k\in\mathbb{N}$. The trace operator being continuous, we can find a constant $\mathbf{t}_n>0$ such that 
$$\|u_k\|_{L^1(D_1)}\leq \mathbf{t}_n( \|\nabla u_k\|_{L^1(B^+_1)}+ \|u_k\|_{L^1(B_1^+)})\,. $$
Therefore $\|u_k\|_{L^1(D_1)}\leq 2\mathbf{t}_n$ whenever $k$ is large enough. Then $\|\nabla u_k\|_{L^1(B^+_1)}\leq 2\mathbf{t}_n/k$. By the compact embedding  $W^{1,1}(B_1^+)\hookrightarrow L^{1}(B_1^+)$ and the condition $[u_k]_1=0$, we deduce that  $u_k\to 0$ strongly in $W^{1,1}(B_1^+)$, which is in contraction with our normalization choice $\|u_k\|_{L^1(B_1^+)}=1$. 
\end{proof}

\begin{remark}[\bf Smooth approximation]\label{smoothapprox}
If $G\subset \R^{n+1}_+$ is an admissible bounded open set, any function $u\in H^1(G,|z|^a\de \mathbf{x})$ with compact support in $G\cup\partial^0 G$ can be approximated in the $H^1(G,|z|^a\de \mathbf{x})$-norm sense by a  sequence $\{u_k\}_{k\in\mathbb{N}}$ of smooth functions compactly supported in $G\cup\partial^0 G$. To construct such a sequence, one can proceed as follows. First notice that the set $\widetilde G:= \big\{(x,z)\in\R^{n+1} : (x,|z|)\in G\cup\partial^0 G\big\}$  is open in $\mathbb{R}^{n+1}$. The symmetrized function $\widetilde u(x,z):=u(x,|z|)$  then belongs to $H^1(\widetilde G,|z|^a\de \mathbf{x})$, and has compact support in $\widetilde G$. By classical (convolution) arguments,  we can find a sequence $\{\widetilde u_k\}_{k\in\mathbb{N}}$ of smooth functions with compact support in $\widetilde G$ converging to $\widetilde u$ in the $H^1(\widetilde G,|z|^a\de \mathbf{x})$-norm sense. Then we obtain the required sequence $\{u_k\}_{k\in\mathbb{N}}$  by considering the restriction of $\widetilde u_k$ to $G\cup\partial^0 G$. 

If the function $u\in H^1(G,|z|^a\de \mathbf{x})$ is compactly supported in $G\cup\overline\Omega$ for some smooth and bounded  open set $\Omega\subset \R^n$ such that $\overline\Omega\subset\partial^0 G$, the sequence $\{u_k\}_{k\in\mathbb{N}}$ can be chosen in such a way that each $u_k$ is compactly supported in $G\cup\Omega$. Indeed, by a diagonal argument, it is enough to show that $u$ can be approximated in the $H^1(G,|z|^a\de \mathbf{x})$-norm by a  sequence $\{\widehat u_k\}_{k\in\mathbb{N}}\subset H^1(G,|z|^a\de \mathbf{x})$ made of functions compactly supported in the set $G\cup\Omega$.  To this purpose, we first reduce the problem to the case of a bounded function $u$ through the usual truncation argument. From the smoothness assumption on $\partial \Omega$, and since $\partial \Omega$ is a set of codimension 2 in $\R^{n+1}$, it has a vanishing $H^1$-capacity in $\R^{n+1}$. Hence, 
we can find a sequence of cut-off functions $\zeta_k:\R^{n+1}\to[0,1]$ such that $\zeta_k=1$ in a neighborhood of $\partial\Omega$, $\zeta_k\to 0$ a.e. in $\R^{n+1}$ and $\zeta_k\to 0$ strongly in $H^1(\mathbb{R}^{n+1})$ (see e.g. \cite[Theorem 3, p.154]{EvGa}). Setting $\widehat u_k:=(1-\zeta_k)u$, we observe  that $\widehat u_k$ has compact support in $G\cup\Omega$, and   
$$\|\widehat u_k-u\|^2_{H^1(G,|z|^a\de \mathbf{x})}\leq  C\left(\int_G z^a\zeta^2_k|\nabla u|^2\,\de \mathbf{x}+\|u\|^2_{L^\infty(G)}\|\zeta_k\|^2_{H^1(G)}\right)\mathop{\longrightarrow}\limits_{k\to\infty} 0\,,$$
by dominated convergence.
\end{remark}

%%%%%%%%%%%%%%%%%%%%%%%%%%%%%%%%%%%%%%%%%%%%%%%%%%%%%%%

 \subsection{The Dirichlet-to-Neumann operator}

%%%%%%%%%%%%%%%%%%%%%%%%%%%%%%%%%%%%%%%%%%%%%%%%%%%%%%%

Consider the function $\mathbf{K}_{n,s}:\R^{n+1}_+\to [0,\infty)$ defined by 
$$\mathbf{K}_{n,s}(\mathbf{x}):=\sigma_{n,s}\,\frac{z^{2s}}{|\mathbf{x}|^{n+2s}}\,, \qquad \sigma_{n,s}:=\pi^{-\frac{n}{2}}\frac{\Gamma(\frac{n+2s}{2})}{\Gamma(s)}\,,$$
where $\mathbf{x}:=(x,z)\in\mathbb{R}^{n+1}_+:=\R^n\times(0,\infty)$. The choice of the constant $\sigma_{n,s}$ is made in such a way that
\footnote{Indeed, changing variables one obtains
\begin{multline*}
\int_{\R^n}\frac{1}{(|x|^2+1)^{\frac{n+2s}{2}}}\,\de x=|\mathbb{S}^{n-1}|\int_0^\infty\frac{r^{n-1}}{(r^2+1)^{\frac{n+2s}{2}}}\,\de r\\
=\frac{|\mathbb{S}^{n-1}|}{2}\int_0^\infty\frac{t^{\frac{n}{2}-1}}{(t+1)^\frac{n+2s}{2}}\,\de t= \frac{|\mathbb{S}^{n-1}|}{2}\,{\rm B}(n/2,s)\,,
\end{multline*}
where ${\rm B}(\cdot,\cdot)$ denotes the Euler Beta function.}   
$\int_{\R^n}\mathbf{K}_{n,s}(x,z)\,\de x=1$ for every $z>0$. 

As shown in  \cite{CaffSil}, the function $\mathbf{K}_{n,s}$ solves 
$$\begin{cases}
{\rm div}(z^{a}\nabla \mathbf{K}_{n,s})= 0 & \text{in $\R^{n+1}_+$}\,,\\
\mathbf{K}_{n,s}=\delta_0 & \text{on $\partial\R^{n+1}_+$}\,,
\end{cases}$$
where $\delta_0$ is the Dirac distribution at the origin. In other words, the function $\mathbf{K}_{n,s}$ can be interpreted as the {\it ``fractional Poisson kernel''} by analogy with the standard case $s=1/2$. 

From now on, for a measurable function $v$ defined over $\mathbb{R}^n$, we shall denote by 
$v^\e$ its extension to the half-space $\mathbb{R}^{n+1}_+$ given by the convolution (in the $x$-variables) of $v$ with the 
 fractional Poisson kernel $\mathbf{K}_{n,s}$,  i.e., 
\begin{equation}\label{poisson}
v^\e(x,z):= \sigma_{n,s}\int_{\mathbb{R}^n}\frac{z^{2s} v(y)}{(|x-y|^2+z^2)^{\frac{n+2s}{2}}}\,\de y\,.
\end{equation}
Notice that $v^\e$ is well defined if $v$ belongs to the Lebesgue space $L^q$ over $\R^n$ with respect to the probability measure 
\begin{equation}\label{defmeasm}
\mathfrak{m}:=\sigma_{n,s}(1+|y|^2)^{-\frac{n+2s}{2}}\,\de y
\end{equation}
for some $1\leq q \leq\infty$. In particular,  $v^\e$ can be defined  whenever 
$v\in \widehat{H}^{s}(\Omega)$ for some bounded  open set $\Omega\subset\R^n$ by Lemma~\ref{adminHchap}. 
Moreover, if  $v\in L^\infty(\R^n)$, then $v^\e\in L^\infty(\R_+^{n+1})$ and  
\begin{equation}\label{bdlinftyext}
\|v^\e\|_{L^\infty(\R_+^{n+1})} \leq \|v\|_{L^\infty(\R^n)}\,.
\end{equation}
For an admissible function $v$, the extension  $v^\e$ has a pointwise trace on $\partial\R^{n+1}_+=\R^n$ which is equal to $v$ at every Lebesgue point. 
 In addition, $v^\e$ solves the equation 
\begin{equation}\label{eqextharm}
\begin{cases} 
{\rm div}(z^{a}\nabla v^\e) = 0 & \text{in $\mathbb{R}_+^{n+1}$}\,,\\
v^\e = v  & \text{on $\partial\R^{n+1}_+$}\,. 
\end{cases}
\end{equation}
By analogy with the standard case $s=1/2$ (for which \eqref{eqextharm} reduces to the Laplace equation), we may say that $v^\e$ is the {\it fractional harmonic extension} of $v$. 
\vskip3pt

The following continuity property is elementary and can be obtained exactly as in \cite[Lemma~2.5]{MilSir}. 

\begin{lemma}\label{context}
For every $R>0$, the restriction operator $\mathfrak{R}_R:L^2(\R^n,\mathfrak{m})\to L^2(B_R^+,|z|^a\de\mathbf{x})$ defined by 
\begin{equation}\label{Pfrak}
\mathfrak{R}_R(v):={v^\e}_{|B_R^+} \,,
\end{equation}
is continuous. 
\end{lemma}

It has been proved in \cite{CaffSil} that $v^\e$  belongs to the weighted space $H^1(\mathbb{R}_+^{n+1},|z|^a\de\mathbf{x})$ whenever 
$v\in H^{s}(\mathbb{R}^n)$. In addition, the $H^{s}$-seminorm of $v$ coincides with the weighted $L^2$-norm of $\nabla v^\e$, extending a well known identity for $s=1/2$. 

\begin{lemma}[\cite{CaffSil}]\label{normexth1/2}
Let $v\in H^{s}(\mathbb{R}^n)$, and let $v^\e$ be its fractional harmonic extension to~$\mathbb{R}^{n+1}_+$ given by \eqref{poisson}. Then 
$v^\e$ belongs to $H^1(\mathbb{R}_+^{n+1},|z|^a\de \mathbf{x})$ and
\begin{align}
\nonumber [v]^2_{H^{s}(\mathbb{R}^n)} &= d_s\|\nabla v^\e\|^2_{L^2(\R_+^{n+1},|z|^a\de \mathbf{x})} \\
\label{isomtry}&=\inf\left\{ d_s\|\nabla u\|^2_{L^2(\R_+^{n+1},|z|^a\de \mathbf{x})} : u\in H^1(\mathbb{R}^{n+1}_+,|z|^a\de \mathbf{x})\,, \ u=v \text{ on $\mathbb{R}^n$} \right\}\,,
\end{align}
where $d_s:=2^{2s-1}\frac{\Gamma(s)}{\Gamma(1-s)}$.
\end{lemma}

\begin{remark}\label{remtrace}
Let  $G\subset \R^{n+1}_+$ be an admissible bounded  open set. For any function $u\in H^1(\mathbb{R}^{n+1}_+,|z|^a\de\mathbf{x})$ compactly supported in $G\cup\partial^0G$, 
the trace $u_{|\mathbb{R}^n}$  belongs to $H^s_{00}(\partial^0G)$. Indeed, if $u$ is smooth in $\overline{\mathbb{R}^{n+1}_+}$, then we can apply identity  \eqref{isomtry}. In the general case, it suffices to apply the approximation procedure in Remark \ref{smoothapprox} to reach the conclusion. 
\end{remark}

If $v\in  \widehat{H}^{s}(\Omega)$ for a bounded open set $\Omega\subset\R^n$,  
we have the following estimates on $v^\e$ extending Lemma \ref{normexth1/2} to the local setting. The proof follows closely the arguments in \cite[Lemma 2.7]{MilSir}, and we shall omit it. 

\begin{lemma}\label{hatH1/2toH1}
Let $\Omega\subset \mathbb{R}^n$ be a bounded open set. 
For every $v\in \widehat{H}^{s}(\Omega)$, the  extension $v^\e$ 
given by \eqref{poisson} belongs to $H^1_{{{\rm loc}}}(\R^{n+1}_+\cup\Omega,|z|^a\de \mathbf{x})\cap L^2_{{\rm loc}}\big(\overline{\R^{n+1}_+},|z|^a\de \mathbf{x}\big)$. In addition, for every $x_0\in\Omega$, $R>0$, and  $\rho>0$ 
such that $D_{3\rho}(x_0)\subset\Omega$, there exist constants $C_{s,R,\rho}>0$  and $C_{s,\rho}>0$, independent of $v$ and $x_0$, such that  
$$\big\|v^\e\big\|^2_{L^2(B_R^+(x_0),|z|^a\de \mathbf{x})}\leq C_{s,R,\rho} \left(\mathcal{E}\big(v,D_{2\rho}(x_0)\big)+\|v\|^2_{L^2(D_{2\rho}(x_0))}\right)\,, $$
and 
$$ \big\|\nabla v^\e\big\|^2_{L^2(B_\rho^+(x_0),|z|^a\de \mathbf{x})}\leq C_{s,\rho} \left(\mathcal{E}\big(v,D_{2\rho}(x_0)\big)+\|v\|^2_{L^2(D_{2\rho}(x_0))}\right)\,. $$
\end{lemma}

\begin{remark}\label{H1/2loctoH1loc}
By the previous lemma, for any $v\in \widehat H^{s}(\Omega)\cap H^{s}_{\rm loc}(\mathbb{R}^n)$, the fractional harmonic extension $v^\e$ 
belongs to $H^1_{{{\rm loc}}}(\overline{\R^{n+1}_+},|z|^a\de \mathbf{x})$, and for any $R>0$, 
$$\big\| v^\e\big\|^2_{H^1(B_R^+,|z|^a\de \mathbf{x})}\leq C_{s,R} \left(\mathcal{E}\big(v,D_{2R}\big)+\|v\|^2_{L^2(D_{2R})}\right)\,. $$
\end{remark}
\vskip5pt

If $v\in \widehat H^{s}(\Omega)$ for some bounded open set $\Omega\subset\R^n$ with Lipschitz boundary, the divergence free vector field $z^{a}\nabla v^\e$ admits a distributional normal trace on $\Omega$, that we denote by $\mathbf{\Lambda}^{(2s)}v$.  
More precisely, we define $\mathbf{\Lambda}^{(2s)} v$  through its action on a test function $\varphi\in \mathscr{D}(\Omega)$ by setting
\begin{equation}\label{defNeumOp}
\left\langle \mathbf{\Lambda}^{(2s)} v, \varphi\right\rangle_\Omega := \int_{\mathbb{R}^{n+1}_+}z^{a}\nabla v^\e\cdot\nabla\Phi\,\de \mathbf{x}\,,
\end{equation}
where $\Phi$ is any smooth extension of $\varphi$ compactly supported in  $\mathbb{R}_+^{n+1}\cup\Omega$. Note that the right hand side of \eqref{defNeumOp} 
is well defined by Lemma~\ref{hatH1/2toH1}. Using equation \eqref{eqextharm} and the divergence theorem, it is routine to check that 
the integral in \eqref{defNeumOp} does not depend on the choice of the extension $\Phi$. In the light of \eqref{densitysmoothH1/200} and Lemma \ref{normexth1/2}, we infer that $\mathbf{\Lambda}^{(2s)}:\widehat H^{s}(\Omega)\to H^{-s}(\Omega)$ defines a continuous linear operator. It can be thought as a {\it fractional  Dirichlet-to-Neumann operator}. Indeed, whenever $v$ is smooth, the distribution $\mathbf{\Lambda}^{(2s)}v$ is the pointwise defined function given by 
$$\mathbf{\Lambda}^{(2s)} v(x)=-\lim_{z\downarrow0}z^{a}\partial_z v^\e(x,z)=2s\, \lim_{z\downarrow0} \frac{v^\e(x,0)-v^\e(x,z)}{z^{2s}}$$  
for  $x\in \Omega$.  
\vskip5pt

In the case $\Omega=\R^n$, it has been proved in \cite{CaffSil} that $\mathbf{\Lambda}^{(2s)}$ coincides with $ (-\Delta)^{s} $, up to a constant multiplicative factor. In our localized  setting, this identity still holds, and it can be obtained essentially as in \cite[Lemma 2.9]{MilSir}.

\begin{lemma}\label{repnormderfraclap}
If $\Omega\subset \mathbb{R}^n$ is a bounded open set with  Lipschitz boundary, then 
$$ (-\Delta)^{s} = d_s \mathbf{\Lambda}^{(2s)} \text{ on $\widehat H^{s}(\Omega)$}\,.$$
\end{lemma}

A local counterpart of Lemma \ref{normexth1/2} concerning the minimality of $v^\e$ can be obtained from the above identity. This is the purpose of Corollary \ref{minenergdirchfrac} below, which is inspired from \cite[Lemma 7.2]{CRS}. 
From now on, we use the notation
\begin{equation}\label{defEbold}
{\bf E}(u,G):= \frac{d_s}{2}\int_{G}z^{a}|\nabla u|^2\,\de \mathbf{x}\,,
\end{equation}
for an open set $G\subset \R^{n+1}_+$ and $u\in H^1(G,|z|^a\de{\bf x})$. We shall refer to ${\bf E}(\cdot, G)$ as the {\it weighted Dirichlet energy} in the domain $G$.  

\begin{corollary}\label{minenergdirchfrac}
Let $\Omega\subset \mathbb{R}^n$ be a bounded open set, and $G\subset \R^{n+1}_+$ be an admissible bounded  open set  
such that $\overline{\partial^0 G}\subset \Omega$. 
Let $v\in \widehat H^{s}(\Omega)$, and let $v^\e$ be its fractional harmonic extension to $\mathbb{R}^{n+1}_+$ given by~\eqref{poisson}. Then, 
\begin{equation}\label{ineqenergDfrac}
{\bf E}(u,G)-{\bf E}(v^\e,G)
\geq  \mathcal{E}(u,\Omega) -\mathcal{E}(v,\Omega)
\end{equation}
for all $u\in H^1(G,|z|^a\de\mathbf{x})$ such that $u-v^\e$ is compactly supported in $G\cup \partial^0G$. 
In the right hand side of \eqref{ineqenergDfrac}, the trace of $u$ on $\partial^0G$ is extended by $v$ outside $\partial^0G$. 
\end{corollary}

\begin{proof}
Let $u\in H^1(G,|z|^a\de\mathbf{x})$ such that $u-v^\e$ is compactly supported in $G\cup\partial^0G$. We extend $u$ by 
$v^\e$ outside~$G$. Then $w:=u-v^\e\in H^1(\mathbb{R}^{n+1}_+,|z|^a\de\mathbf{x})$ and $w$ is compactly supported in $G\cup\partial^0G$. Hence  
$w_{|\R^n} \in H^{s}_{00}(\partial^0 G)$ by Remark \ref{remtrace}. Since $v\in \widehat H^{s}(\partial^0 G)$, we deduce from Remark \ref{remEDB}
that the trace of $u$ on $\mathbb{R}^n$ belongs to $ \widehat{H}^{s}(\partial^0 G)$. 

Using Lemma \ref{normexth1/2} and Lemma \ref{repnormderfraclap}, we  estimate
\begin{align}
\nonumber {\bf E}(u,G)-{\bf E}(v^\e,G)
& = \frac{d_s}{2}\int_{\R^{n+1}_+} z^a|\nabla w|^2\,\de \mathbf{x}+d_s \int_{\mathbb{R}^{n+1}_{+}}z^a \nabla v^\e \cdot \nabla w\,\de \mathbf{x} \\
\nonumber& = \frac{d_s}{2}\int_{\R^{n+1}_+} z^a|\nabla w|^2\,\de \mathbf{x} +  \big\langle  (-\Delta)^{s} v,w_{|\mathbb{R}^n}\big\rangle_{\partial^0 G} \\
\nonumber& \geq  \big[w_{|\mathbb{R}^n}\big]^2_{H^{s}(\mathbb{R}^n)}+  \big\langle  (-\Delta)^{s} v,w_{|\mathbb{R}^n}\big\rangle_{\partial^0 G} \\
& = \mathcal{E}(w_{|\mathbb{R}^n},\partial^0 G)+  \big\langle  (-\Delta)^{s} v,w_{|\mathbb{R}^n}\big\rangle_{\partial^0 G} \,. \label{estiequiv1}
\end{align}
Using the fact that $u_{|\mathbb{R}^n},v \in \widehat{H}^{s}(\partial^0 G)$, we derive  that 
\begin{equation}\label{estiequiv2}
 \mathcal{E}(w_{|\R^n},\partial^0G)=  \mathcal{E}(u_{|\mathbb{R}^n},\partial^0G)+ \mathcal{E}(v,\partial^0 G)- \big\langle  (-\Delta)^{s} v,u_{|\mathbb{R}^n}\big\rangle_{\partial^0 G}\,,
 \end{equation}
and 
\begin{equation}\label{estiequiv3}
\big\langle  (-\Delta)^{s} v,w_{|\mathbb{R}^n}\big\rangle_{\partial^0 G}= \big\langle  (-\Delta)^{s} v,u_{|\mathbb{R}^n}\big\rangle_{\partial^0 G}-  2\mathcal{E}(v,\partial^0G)\,.
\end{equation}
Gathering \eqref{estiequiv1}-\eqref{estiequiv2}-\eqref{estiequiv3} yields 
$${\bf E}(u,G)-{\bf E}(v^\e,G) \geq  \mathcal{E}(u_{|\R^n},\partial^0 G) -\mathcal{E}(v,\partial^0 G)\,.$$
Since $u_{|\R^n}=v$ outside $\partial^0 G$, we infer that 
$$\mathcal{E}(u_{|\R^n},\partial^0 G)- \mathcal{E}(v,\partial^0G)
=\mathcal{E}(u_{|\R^n},\Omega)- \mathcal{E}(v,\Omega)\,,$$
and the conclusion follows.  
\end{proof}

The crucial observation for us  is that \eqref{ineqenergDfrac} leads to a local representation (in terms of~$v^\e$) of the first inner variation of $\mathcal{E}(\cdot,\Omega)$ at a function $v\in\widehat H^s(\Omega)$. We recall that, given $X\in C^{1}(\R^n;\R^n)$ compactly supported in $\Omega$, the first inner variation $\delta\mathcal{E}(v,\Omega)$ evaluated at $X$ is defined by 
$$\delta\mathcal{E}(v,\Omega)[X] :=\left[ \frac{\de}{\de t} \,\mathcal{E}(v\circ\phi_{-t},\Omega)\right]_{t=0} \,,$$
where $\{\phi_t\}_{t\in\mathbb{R}}$ denotes the flow on $\R^n$ generated by $X$, i.e., for $x\in\R^{n}$, the map $t\mapsto \phi_t(x)$ is defined as the unique solution of the ordinary differential equation
$$
\begin{cases}
\displaystyle\frac{\de}{\de t} \phi_t(x)=  X\big(\phi_t(x)\big) \,,\\
\phi_0(x)=x\,.
\end{cases}
$$
Now we can state our representation result. 

\begin{corollary}\label{firstvargen}
Let $\Omega\subset \mathbb{R}^n$ be a bounded open set, and $G\subset \R^{n+1}_+$ be an admissible bounded  open set  
such that $\overline{\partial^0 G}\subset \Omega$. 
For each $v\in \widehat H^s(\Omega)$, and each $X\in C^1(\mathbb{R}^n;\mathbb{R}^n)$ compactly supported in $\partial^0 G$, we have 
\begin{multline*}
\delta\mathcal{E}(v,\Omega)[X]= \frac{d_s}{2}\int_{G}z^{a}\Big( |\nabla v^\e|^2{\rm div}\,\mathbf{X} 
-2(\nabla v^\e\otimes\nabla v^\e):\nabla\mathbf{X}
\Big)\,\de \mathbf{x}\\+\frac{d_sa}{2}\int_{G}z^{a-1} |\nabla v^\e|^2 {\mathbf X}_{n+1}\,\de \mathbf{x} \,,
\end{multline*}
where $\mathbf{X}=(\mathbf{X}_1,\ldots,\mathbf{X}_{n+1})\in C^1(\overline G;\mathbb{R}^{n+1})$ is any vector field compactly supported in $G\cup\partial^0 G$ satisfying $\mathbf{X}=(X,0)$ on $\partial^0 G$.
\end{corollary}

\begin{proof}
Let $\mathbf{X}=(\mathbf{X}_1,\ldots,\mathbf{X}_{n+1})\in C^1(\overline{G};\R^{n+1})$ be an arbitrary  vector field  compactly supported in $G\cup\partial^0G$ and satisfying $\mathbf{X}=(X,0)$ on $\partial^0G$. Then 
consider a compactly supported $C^1$-extension of $\mathbf{X}$ to the whole space $\R^{n+1}$, still denoted by $\mathbf{X}$, such that $\mathbf{X}=(X,0)$ on $\R^n$. We define $\{\Phi_t\}_{t\in\R}$   to be the flow on $\mathbb{R}^{n+1}$  generated by $\mathbf{X}$,  i.e., for $\mathbf{x}\in\mathbb{R}^{n+1}$, the map $t\mapsto \Phi_t(x)$ is defined as the unique solution of the differential equation
$$\begin{cases}
\displaystyle\frac{\de}{\de t} \Phi_t(\mathbf{x})= \mathbf{X}\big(\Phi_t(\mathbf{x})\big) \,,\\
\Phi_0(\mathbf{x})=\mathbf{x}\,.
\end{cases}
$$
Noticing that $\Phi_t=(\phi_t,0)$ on~$\R^n$ and that ${\rm supp}\big(\Phi_t-{\rm id}_{\mathbb{R}^{n+1}}\big)\cap\overline{\mathbb{R}^{n+1}_+}\subset G\cup\partial^0 G$, 
 we   infer from Corollary~\ref{minenergdirchfrac} that 
$${\bf E}(v^\e\circ\Phi_{-t},G)-{\bf E}(v^\e,G)
\geq \mathcal{E}\big(v\circ\phi_{-t},\Omega\big)-\mathcal{E}(v,\Omega)\,.$$
Dividing both sides of this inequality by $t\not=0$, and then letting $t\downarrow0$ and $t\uparrow0$, we obtain 
$$ \left[\frac{\de}{\de t} {\bf E}(v^\e\circ\Phi_{-t},G) 
\right]_{t=0}
=\left[\frac{\de}{\de t} \,\mathcal{E}\big(v\circ\phi_t,\Omega\big)\right]_{t=0} \,.$$
On the other hand, standard computations (see e.g. \cite[Chapter 2.2]{Sim}) yield 
\begin{multline}\label{tim1516new}
\left[\frac{\de}{\de t} {\bf E}(v^\e\circ\Phi_{-t},G) 
\right]_{t=0}
=\frac{d_s}{2}\int_{G}z^{a}\Big( |\nabla v^\e|^2{\rm div}\,\mathbf{X} 
-2(\nabla v^\e\otimes\nabla v^\e):\nabla\mathbf{X}
\Big)\,\de \mathbf{x}\\
+\frac{d_sa}{2}\int_{G}z^{a-1} |\nabla v^\e|^2 {\mathbf X}_{n+1}\,\de \mathbf{x} \,,
\end{multline}
and the conclusion follows.
\end{proof}

\begin{remark}\label{varuptobound}
For an admissible bounded open set $G\subset\R^{n+1}_+$ and $u\in H^1(G,|z|^a\de{\bf x})$, we can define {\it  the first inner variation up to the boundary $\partial^0G$  
of ${\bf E}(\cdot,G)$ 
at $u$} as 
$$\delta{\bf E}\big(u,G\cup\partial^0G\big)[{\bf X}]:=\left[ \frac{\de}{\de t} \,\mathbf{E}(u\circ\Phi_{-t},G)\right]_{t=0} \,,$$
where (as in the previous proof) $\{\Phi_t\}_{t\in\R}$ denotes the flow on $\R^{n+1}$ generated by a given vector field $\mathbf{X}=(X,\mathbf{X}_{n+1})\in C^1(\overline G;\mathbb{R}^{n+1})$  compactly supported in $G\cup\partial^0 G$ and satisfying $\mathbf{X}_{n+1}=0$ on $\partial^0 G$. Then, one obtains 
\begin{multline}\label{formulafirstvarbound}
\delta{\bf E}\big(u,G\cup\partial^0G\big)[{\bf X}]=\frac{d_s}{2}\int_{G}z^{a}\Big( |\nabla u|^2{\rm div}\,\mathbf{X} 
-2(\nabla u\otimes\nabla u):\nabla\mathbf{X}
\Big)\,\de \mathbf{x}\\+\frac{d_sa}{2}\int_{G}z^{a-1} |\nabla u|^2 {\mathbf X}_{n+1}\,\de \mathbf{x} \,.
\end{multline}
Hence, we can rephrased the conclusion of Corollary \ref{firstvargen} as $\delta\mathcal{E}(v,\Omega)=\delta{\bf E}\big(v^\e,G\cup\partial^0G\big)$. 
\end{remark}

%%%%%%%%%%%%%%%%%%%%%%%%%%%%%%%%%%%%%%%%%%%%%%%%%%%%%%%%%%%%%
%%%%%%%%%%%%%%%%%%%%%%%%%%%%%%%%%%%%%%%%%%%%%%%%%%%%%%%%%%%%%
															 
\section{The fractional Allen-Cahn equation: a priori estimates}\label{FractAC}    
															 
%%%%%%%%%%%%%%%%%%%%%%%%%%%%%%%%%%%%%%%%%%%%%%%%%%%%%%%%%%%%%
%%%%%%%%%%%%%%%%%%%%%%%%%%%%%%%%%%%%%%%%%%%%%%%%%%%%%%%%%%%%%

We consider in this section a bounded open set $\Omega\subset \mathbb{R}^n$ with (at least) Lipschitz boundary. We are interested  
in weak solutions  $v_\varepsilon\in \widehat H^{s}(\Omega)\cap L^p(\Omega)$ of the fractional Allen-Cahn  equation
\begin{equation}\label{eqfractGL}
 (-\Delta)^{s} v_\varepsilon+\frac{1}{\varepsilon^{2s}}W'(v_\varepsilon) =f\quad\text{in $\Omega$}\,,
\end{equation}
with a source term $f$ belonging to either $L^\infty(\Omega)$  or  $C^{0,1}(\Omega)$. 
The notion of weak solution is understood in the duality sense according to the formulation \eqref{deffraclap} of the fractional Laplacian, i.e., 
$$\big\langle  (-\Delta)^{s} v_\varepsilon, \varphi\big\rangle_{\Omega} + \frac{1}{\varepsilon^{2s}}\int_{\Omega}W'(v_\varepsilon)\,\varphi\,\de x=\int_\Omega f\varphi\,\de x \qquad 
\forall \varphi\in H^{s}_{00}(\Omega)\cap L^p(\Omega)\,.$$
Such solutions correspond  to  critical points in $\Omega$ of the functional 
$$\mathcal{F}_\varepsilon(v,\Omega):=\mathcal{E}_\varepsilon(v,\Omega)-\int_\Omega f v\,\de x\,,$$
where $\mathcal{E}_\varepsilon(\cdot,\Omega)$ is the  fractional Allen-Cahn energy in \eqref{defFGLenerg}. 
In other words, we are interested in maps $v_\varepsilon\in \widehat H^{s}(\Omega)\cap L^p(\Omega)$ satisfying  
\begin{equation}\label{ptcritAC}
\left[\frac{\de}{\de t} \mathcal{F}_\varepsilon(v_\varepsilon+t\varphi,\Omega) \right]_{t=0} =0\qquad 
\forall \varphi\in H^{s}_{00}(\Omega)\cap L^p(\Omega)\,.
\end{equation}

\begin{remark}
An elementary way to construct solutions of \eqref{eqfractGL} is of course to minimize $\mathcal{F}_\varepsilon(\cdot,\Omega)$ under an exterior Dirichlet condition. 
Indeed,  given $g\in \widehat H^{s}(\Omega)\cap L^p(\Omega)$, the minimization problem
\begin{equation}\label{GLminProb}
\min\Big\{ \mathcal{F}_\varepsilon(v,\Omega) : v\in H^{s}_{g} (\Omega)\cap L^p(\Omega) \Big\}\,, 
\end{equation}
is easily solved using the Direct Method of Calculus of Variations, and it obviously returns a solution of \eqref{ptcritAC}.   
\end{remark}

%%%%%%%%%%%%%%%%%%%%%%%%%%%%%%%%%%%%%%%%%%%%%%%%%%%%%%%

\subsection{Degenerate Allen-Cahn boundary reactions}

%%%%%%%%%%%%%%%%%%%%%%%%%%%%%%%%%%%%%%%%%%%%%%%%%%%%%%%

To obtain a priori estimates on weak solutions of \eqref{eqfractGL}, we rely on the fractional harmonic extension to $\R^{n+1}_+$ 
introduced in Section \ref{prelim}.  
According to Lemmas \ref{hatH1/2toH1} \& \ref{repnormderfraclap}, and \eqref{defNeumOp},  if $v_\varepsilon\in \widehat H^{s}(\Omega)\cap L^p(\Omega)$ is a weak solution of \eqref{eqfractGL}, then its 
fractional harmonic extension $v^\e_\varepsilon$ given by \eqref{poisson}  satisfies 
$$d_s \int_{\R^{n+1}_+} z^{a}\, \nabla v^\e_\varepsilon \cdot\nabla\phi\,\de \mathbf{x}+ \frac{1}{\varepsilon^{2s}}\int_\Omega W^\prime(v^\e_\varepsilon)\, \phi\,\de x=\int_\Omega f \phi\,\de x $$
for every smooth function $\phi :\overline{\mathbb{R}^{n+1}_+}\to\mathbb{R}$  compactly supported in $\R^{n+1}_+\cup\Omega$, or equivalently,   
for every $\phi\in H^1(\mathbb{R}_+^{n+1},|z|^a\de\mathbf{x})\cap L^p(\Omega)$  compactly supported in $\R^{n+1}_+\cup\Omega$ (by Remark~\ref{smoothapprox}). 
In particular, given an admissible bounded open set $G\subset \R^{n+1}_+$ such that $\overline{\partial^0G}\subset\Omega$, the extension $v_\eps^\e$ obviously satisfies     
\begin{equation}\label{varformbdgleq}
d_s \int_{G} z^{a}\, \nabla v^\e_\varepsilon \cdot\nabla\phi\,\de \mathbf{x}+ \frac{1}{\varepsilon^{2s}}\int_{\partial^0G} W^\prime(v^\e_\varepsilon)\, \phi\,\de x=\int_{\partial^0 G} f\phi\,\de x 
\end{equation}
 for every $\phi \in H^1(G,|z|^a\de \mathbf{x})\cap L^p(\partial^0G)$ compactly supported in $G\cup\partial^0G$. In other words, the extension $v^\e_\varepsilon$ is a critical point of the functional $\mathbf{F}_\eps(\cdot,G)$ defined on the weighted space 
 $ H^1(G,|z|^a\de \mathbf{x})\cap L^p(\partial^0G)$ by 
\begin{equation}\label{defEnergGLB}
\mathbf{F}_\varepsilon(u,G):=\mathbf{E}_\varepsilon(u,G)-\int_{\partial^0G} f u \,\de x\,,
\end{equation}
with
$$\mathbf{E}_\varepsilon(u,G):=\mathbf{E}(u,G)+ \frac{1}{\varepsilon^{2s}}\int_{\partial^0G} W(u)\,\de x\,, $$
where ${\bf E}(\cdot,G)$ is {\it the weighted Dirichlet energy} defined in \eqref{defEbold}.
\vskip3pt

In general, if a function $u_\eps$ is a critical point of $\mathbf{F}_\eps(\cdot,G)$ such that both $u_\eps$ and $z^{a}\partial_z u_\eps$ are continuous in  $G$ up to $\partial^0G$, then $u_\eps$ satisfies in the pointwise sense the Euler-Lagrange equation 
\begin{equation}\label{eqext}
\begin{cases}
{\rm div}(z^{a}\nabla u_\eps)= 0 & \text{in $G$}\,,\\[10pt]
\displaystyle d_s \boldsymbol{\partial}^{(2s)}_{z} u_\eps=\frac{1}{\varepsilon^{2s}} W^\prime (u_\eps)-f & \text{on $\partial^0G$}\,,
\end{cases}
\end{equation} 
where we have set for $\mathbf{x}=(x,0)\in\partial^0G$, 
$$\boldsymbol{\partial}^{(2s)}_{z} u_\eps(\mathbf{x}):=\lim_{z\downarrow0} z^{a}\partial_{z} u_\eps(x,z)\,.$$
We shall refer  to as {\it weak solution} of  equation \eqref{eqext} a critical point of $\mathbf{F}_\eps(\cdot,G)$.

%%%%%%%%%%%%%%%%%%%%%%%%%%%%%%%%%%%%%%%%%%%%%%%%%%%%%%%

\subsection{Regularity for degenerate boundary reactions} 

%%%%%%%%%%%%%%%%%%%%%%%%%%%%%%%%%%%%%%%%%%%%%%%%%%%%%%%

Our strategy now  consists in deriving a priori estimates for weak solutions of \eqref{eqext}. 
Concerning regularity, the starting point is the following linear estimate  given in \cite[proof of Lemma~4.5]{CS1}.  

\begin{lemma}[\cite{CS1}]\label{prereg}
Let $\mathbf{f}\in L^\infty(D_2)$ and $u\in H^1(B_2^+,|z|^a\de\mathbf{x})\cap L^\infty(B_2^+)$ be a weak solution of 
\begin{equation}\label{eqlin}
\begin{cases} 
{\rm div}(z^{a}\nabla u)=0 & \text{in $B_2^+$}\,,\\[8pt]
\displaystyle \boldsymbol{\partial}^{(2s)}_z u=\mathbf{f} & \text{on $D_{2}$}\,.
\end{cases}
\end{equation} 
There exist $\boldsymbol{\beta}_*=\boldsymbol{\beta}_*(n,s)\in(0,1)$, and a positive constant $\mathbf{c}_{n,s}$ depending only on $n$ and $s$ such that 
\begin{equation}\label{estischauderlin}
\|u\|_{C^{0,\boldsymbol{\beta}_*}(\overline{B_1}^+)}\leq  \mathbf{c}_{n,s}\big( \| \mathbf{f}\|_{L^\infty(D_2)}+\|u\|_{L^\infty(B_2^+)}\big)\,.
\end{equation}
In addition, if $\mathbf{f}\in C^{0,\sigma}(D_2)$ with $\sigma\in(0,1)$, then $z^{a}\partial_{z}u\in C^{0,\gamma}(\overline B_1^+)$ for some $\gamma\in(0,1)$.   
\end{lemma}

For $f\in C^{0,1}(D_2)$, bootstrapping estimate \eqref{estischauderlin} yields the following interior regularity  for bounded  weak solutions of \eqref{eqext}.

\begin{theorem}\label{regint}
Let  $f\in C^{0,1}(D_2)$ and $u_\varepsilon\in H^1(B_{2}^+,|z|^ad\mathbf{x})\cap L^\infty(B_{2}^+)$ be a weak solution of 
\begin{equation}\label{pasdidee}
\begin{cases} 
{\rm div}(z^{a}\nabla u_\varepsilon)=0 & \text{in $B_{2}^+$}\,,\\[8pt]
\displaystyle d_s\boldsymbol{\partial}^{(2s)}_z u_\varepsilon= \frac{1}{\varepsilon^{2s}}W^\prime(u_\varepsilon)-f & \text{on $D_{2}$}\,.
\end{cases}
\end{equation}
Then $u_\eps\in C^\infty(B_2^+)$, $u_\varepsilon\in C^{0,\boldsymbol{\beta}_*}\big(\overline B_1^+\big)$, $\nabla_xu_\eps\in C^{0,\boldsymbol{\beta}_*}(\overline B_1^+)$, and  $z^{a}\partial_{z}u_\eps\in C^{0,\gamma}(\overline B_1^+)$ for some $\gamma\in(0,1)$ (with $\boldsymbol{\beta}_*$  given by Lemma \ref{prereg}).
\end{theorem}

\begin{proof}
Regularity in the interior of the half ball $B_2^+$ follows from the usual elliptic theory. Then, to prove the announced regularity near $D_1$, we  first apply Lemma \ref{prereg} to  deduce that $u_\varepsilon\in C^{0,\boldsymbol{\beta}_*}_{\rm loc}\big( B_2^+\cup D_2\big)$ and  $z^{a}\partial_{z}u_\eps\in C^{0,\gamma}_{\rm loc}( B_2^+\cup D_2)$. Now it only remains to show that $\nabla_x u_\eps$ is H\"older continuous up to $D_1$. Denote by  $k_*\in\mathbb{N}$  the integer part of $1/\boldsymbol{\beta}_*$. Choosing the universal constant $\boldsymbol{\beta}_*$ slightly smaller if necessary, we may assume without loss of generality that $k_*<1/\boldsymbol{\beta}_*$. Then $(k_*+1)\boldsymbol{\beta}_*\in(1,2)$. 

Fix an arbitrary point $x_0\in \overline D_1$, and for $\mathbf{x}=(x,z)\in B^+_{1}\cup D_{1}$ define the translated function  
$$\bar u(\mathbf{x}):=u_\eps(x+x_0,z)\,.$$ 
Given a non vanishing  $h\in D_{1/8}$,  we set for $\mathbf{x}\in B^+_{7/8}\cup D_{7/8}$, 
\begin{equation}\label{defwh}
w_h(\mathbf{x}):=\frac{\bar u(x+h,z)-\bar u(\mathbf{x})}{|h|^{\boldsymbol{\beta}_*}} \,.
\end{equation}
Then $w_h\in H^1(B_{7/8}^+,|z|^a\de \mathbf{x})\cap L^\infty(B_{7/8}^+)$ and $\|w_h\|_{L^\infty(B^+_{7/8})}$ is bounded independently of~$h$. In addition, 
$w_h$ weakly  solves equation \eqref{eqlin} in $B^+_{7/8}$ with right hand side 
$$\mathbf{f}_h(x):=\frac{W^{\prime}\big(\bar u(x+h,0)\big)-W^{\prime}\big(\bar u(x,0)\big)}{\eps^{2s}\big(\bar u(x+h,0)-\bar u(x,0)\big)}\, w_h(x,0) - \frac{f(x_0+x+h)-f(x_0+x)}{|h|^{\boldsymbol{\beta}_*}}\,. $$
By assumption $W\in C^2(\mathbb{R})$ and $f\in C^{0,1}(D_2)$, so that 
$\|\mathbf{f}_h\|_{L^\infty(D_{7/8})}$ is bounded independently of $h$. Hence Lemma \ref{prereg} yields $w_h\in C^{0,\boldsymbol{\beta}_*}(\overline B^+_{7/16})$, and $\|w_h\|_{C^{0,\boldsymbol{\beta}_*}(B^+_{7/16})}$  is bounded independently of $h$. 
In particular, 
$$\frac{|w_h(x,z)-w_h(x-h,z)|}{|h|^{\boldsymbol{\beta}_*}} \leq C_1\qquad\forall (x,z)\in \overline D_{1/8}\times[0,1/8]\,,$$
for some constant $C_1$ independent of $h$. In view of the arbitrariness of $h$, we deduce that  
\begin{equation}\label{improvhold}
\sup_{x\in \overline{D}_{1/8}}  \big|\bar u(x+h,z)-2\bar u(x,z)+\bar u(x-h,z)\big|\leq C_1|h|^{2{\boldsymbol{\beta}_*}}
\end{equation}
for every $h\in \overline D_{1/8}$ and $z\in[0,1/8]$. 

Let us now fix a  cut-off function $\zeta\in C^\infty(\R^n;[0,1])$ such that $\zeta(x)=1$ for $|x|\leq 1/16$ and $\zeta(x)=0$ for $|x|\geq 1/8$. 
Given $z\in[0,1/8]$,  we define for $x\in\mathbb{R}^n$, 
$$\vartheta_{z}(x):=\zeta(x)\bar u(x,z)\,.$$
For $h\in\mathbb{R}^n$,  we denote by $D^2_h\vartheta_{z}$  the second order difference quotient of $\vartheta_{z}$ on $\R^n$ given   by 
$$D^2_h\vartheta_{z}(x):= \vartheta_{z}(x+h)-2\vartheta_{z}(x)+\vartheta_{z}(x-h)\,. $$
From \eqref{improvhold}, it is elementary to show that  
$$\|\vartheta_z\|_{L^\infty(\R^n)}+\sup_{|h|>0}\frac{\|D_h^2\vartheta_z\|_{L^\infty(\R^n)}}{|h|^{2{\boldsymbol{\beta}_*}}}\leq C_2\,, $$
for a constant $C_2$ independent of $z\in[0,1/8]$. 
\vskip3pt

We now have to distinguish two cases. 
\vskip3pt

\noindent{\it Case 1).} If $k_*=1$ (i.e., ${\boldsymbol{\beta}_*}>1/2$), then we infer from \cite[Proposition 9 in Chapter~V.4]{Stein} that $\vartheta_z\in C^{1,\boldsymbol{\alpha}_*}(\R^n)$ with $\boldsymbol{\alpha}_*=2\boldsymbol{\beta}_*-1$, and $\|\vartheta_z\|_{C^{1,\boldsymbol{\alpha}_*}(\R^n)}\leq \widetilde C_2$ for a constant $ \widetilde C_2$ independent of $z\in[0,1/8]$. As a consequence  
$\bar u(\cdot,z)\in C^{1,\boldsymbol{\alpha}_*}(D_{1/16})$, and $\|\bar u(\cdot,z)\|_{C^{1,\boldsymbol{\alpha}_*}(D_{1/16})}\leq \widetilde C_2$ for every $z\in[0,1/8]$. 

We fix $j\in\{1,\ldots,n\}$,  $\delta\in(0,1/32)$, and we define for $\mathbf{x}=(x,z)\in B^+_{1/32}\cup D_{1/32}$, 
$$\widetilde w_\delta(\mathbf{x}):=\frac{\bar u(x+\delta e_j,z)-\bar u(\mathbf{x})}{\delta} \,.$$
Then $\widetilde w_\delta \in   H^1(B_{1/32}^+,|z|^a\de\mathbf{x})\cap L^\infty(B_{1/32}^+)$ and $\|\widetilde w_\delta\|_{L^\infty((B^+_{1/32})}$ is bounded independently of $\delta$.  In addition, 
$\widetilde w_\delta$  weakly solves equation \eqref{prereg} in $B^+_{1/32}$ with right hand side 
$$\widetilde{\mathbf{f}}_\delta(x):=\frac{W^{\prime}(\bar u(x+\delta e_j,0))-W^{\prime}(\bar u(x,0))}{\eps^{2s}(\bar u(x+\delta e_j,0)-\bar u(x,0))}\, \widetilde w_\delta(x,0) - \frac{f(x_0+x+\delta e_j)-f(x_0+x)}{\delta} \,. $$
Again, since $W\in C^{2}(\mathbb{R})$ and $f\in C^{0,1}(D_2)$, we have $\widetilde{\mathbf{f}}_\delta\in L^\infty(D_{1/32})$ and $\|\widetilde{\mathbf{f}}_\delta\|_{L^\infty(D_{1/32})}$ is bounded independently of $\delta$. Then Lemma \ref{prereg} yields $\widetilde w_\delta\in C^{0,\boldsymbol{\beta}_*}(\overline B^+_{1/64})$, and 
$$\frac{|\widetilde w_\delta(\mathbf{x}_1)-\widetilde w_\delta(\mathbf{x}_2)|}{|\mathbf{x}_1-\mathbf{x}_2|^{\boldsymbol{\beta}_*}}\leq C_3 \quad\forall \mathbf{x}_1,\mathbf{x}_2\in \overline B^+_{1/64}\,, \mathbf{x}_1\not=\mathbf{x}_2\,,$$
for a constant $C_3$ independent of $\delta$. Letting $\delta\to 0$, we finally deduce that 
$$\frac{|\partial_j \bar u(\mathbf{x}_1)-\partial_j \bar u(\mathbf{x}_2)|}{|\mathbf{x}_1-\mathbf{x}_2|^{\boldsymbol{\beta}_*}}\leq C_3 \quad\forall \mathbf{x}_1,\mathbf{x}_2\in \overline B^+_{1/64}\,, \mathbf{x}_1\not=\mathbf{x}_2\,.$$
Since the index $j$ is arbitrary, it shows that $\nabla_x u_\eps$ is indeed of class $C^{0,\boldsymbol{\beta}_*}$ in a neighborhood of the point $(x_0,0)$.   
\vskip3pt

\noindent{\it Case 2).} We now assume that $k_*\geq 2$ (i.e., $\boldsymbol{\beta}_*<1/2$). Then we infer from \cite[Proposition~8 in Chapter~V.4]{Stein} that $\vartheta_z\in C^{0,2\boldsymbol{\beta}_*}(\R^n)$ and $\|\vartheta_z\|_{C^{0,2\boldsymbol{\beta}_*}(\R^n)}\leq \widehat C_2$ for a constant $ \widehat C_2$ independent of $z\in[0,1/8]$. As a consequence, for every $z\in[0,1/8]$, we have 
$\bar u(\cdot,z)\in C^{0,2\boldsymbol{\beta}_*}(D_{1/16})$, and $\|\bar u(\cdot,z)\|_{C^{0,2\boldsymbol{\beta}_*}(D_{1/16})}\leq \widehat C_2$. We then repeat the argument starting with the function $w_h$ given in \eqref{defwh} with $\boldsymbol{\beta}_*$ replaced by $2\boldsymbol{\beta}_*$ and the point $\mathbf{x}$ lying in a smaller half ball. After iterating $k_*$ times this procedure we are back to Case~1, and we conclude that   $\nabla_x u_\eps$ is of class $C^{0,\boldsymbol{\beta}_*}$ in a neighborhood of   $(x_0,0)$.
\end{proof} 
 
\begin{remark}\label{remclear} 
Note that for $\eps\geq 1/2$, Lemma \ref{prereg} also shows that any weak solution  $u_\eps\in H^1(B_{2}^+,|z|^ad\mathbf{x})\cap L^\infty(B_{2}^+)$ of \eqref{pasdidee} satisfies 
$$\|u_\eps\|_{C^{0,\boldsymbol{\beta}_*}(\overline{B_1}^+)}\leq \mathbf{c}_*$$ 
for some constant $\mathbf{c}_*>0$ depending only on $n$, $s$, $W$, $\|f\|_{L^\infty(D_2)}$, and $\|u_\eps\|_{L^\infty(B_2^+)}$.  
\end{remark}  

A fundamental consequence of the previous regularity result is that bounded weak solutions of \eqref{eqext} with $f\in C^{0,1}(\partial^0G)$ are stationary points of $\mathbf{F}_\eps(\cdot,G)$, i.e., critical points with respect to inner variations up to $\partial^0G$. In other words, we have 

\begin{corollary}\label{statAC}
Let $G\subset \mathbb{R}^{n+1}_+$ be an  admissible bounded  open set, and $f\in C^{0,1}(\partial^0G)$. If  $u_\eps\in H^1(G,|z|^a\de\mathbf{x})\cap L^\infty(G)$ is a weak solution of  \eqref{eqext}, then 
$$\delta{\bf E}\big(u_\eps,G\cup\partial^0G\big)[{\bf X}]
+\frac{1}{\eps^{2s}}\int_{\partial^0 G}W(u_\eps)\,{\rm div}X\,\de x=\int_{\partial^0G} u_\eps\,{\rm div}(fX)\,\de x$$
for every vector field $\mathbf{X}=(X,\mathbf{X}_{n+1})\in C^1(\overline G;\R^{n+1})$ compactly supported in $G\cup \partial^0G$ such that $\mathbf{X}_{n+1}=0$ on $\partial^0G$. 
\end{corollary}

\begin{proof} Let $\mathbf{X}=(X,\mathbf{X}_{n+1})\in C^1(\overline{G};\R^{n+1})$ be an arbitrary  vector field  compactly supported in $G\cup\partial^0G$ and satisfying $\mathbf{X}_{n+1}=0$ on $\partial^0G$. 
For $\delta\geq 0$, we set 
\begin{multline*}
V_\delta:= \frac{d_s}{2}\int_{G\cap\{z>\delta\}}z^{a}\Big( |\nabla u_\eps|^2{\rm div}\,\mathbf{X} 
-2(\nabla u_\eps\otimes\nabla u_\eps):\nabla\mathbf{X}
\Big)\,\de \mathbf{x}\\
+\frac{d_sa}{2}\int_{G\cap\{z>\delta\}}z^{a-1} |\nabla u_\eps|^2 {\mathbf X}_{n+1}\,\de \mathbf{x}\,,
\end{multline*}
so that $V_0=\lim_{\delta\downarrow0} V_\delta$. 

For each $\delta>0$ we can use equation \eqref{eqext} and integrate by parts to find  
\begin{multline*}
V_\delta=d_s\int_{G\cap\{z=\delta\}}\big(z^a\partial_z u_\eps\big)\big(X\cdot\nabla_x u_\eps\big)\,\de x+\frac{d_s\delta^{2s}}{2}\int_{G\cap\{z=\delta\}}\big|z^a\partial_zu_\eps\big|^2\frac{\mathbf{X}_{n+1}}{z}\,\de x\\
-\frac{d_s}{2}\int_{G\cap\{z=\delta\}}z^a|\nabla_xu_\eps|^2\mathbf{X}_{n+1}\,\de x\,.
\end{multline*}
By the regularity estimates in Theorem \ref{regint}, we can let $\delta\downarrow 0$ to derive
\begin{multline*}
V_0= \int_{\partial^0G}\big(\boldsymbol{\partial}^{(2s)}_z u_\eps\big)\big(X\cdot\nabla_x u_\eps\big)\,\de x\\
=\frac{1}{\eps^{2s}}\int_{\partial^0G}W^\prime(u_\eps)X\cdot\nabla_x u_\eps\,\de x-\int_{\partial^0G}  f X\cdot\nabla_x u_\eps\,\de x\,.
\end{multline*}
Integrating this last term by parts, we conclude that 
$$V_0=  -\frac{1}{\varepsilon^{2s}}\int_{\partial^0G} W(u_\eps)\,{\rm div} X\,\de x+ \int_{\partial^0G} u_\eps\,{\rm div}(fX)\, \de x\,,$$
which, in view of Remark \ref{varuptobound},  is the announced identity. 
\end{proof}

\subsection{Regularity and Maximum Principle for the fractional equation}
By estimate \eqref{bdlinftyext}, a bounded weak solution of the fractional equation \eqref{eqfractGL} yields  a bounded weak solution of \eqref{eqext} after extension. Hence Theorem \ref{regint} and Remark \ref{remclear} provide the following interior regularity for bounded weak solutions of  the fractional equation. 

 \begin{corollary}
 Let $v_\varepsilon \in  \widehat H^{s}(\Omega)\cap L^\infty(\R^n)$ be a weak solution of \eqref{eqfractGL} with $f\in L^\infty(\Omega)$. 
 Then $v_\varepsilon \in C^{0,\boldsymbol{\beta}_*}_{\rm loc}(\Omega)$ with $\boldsymbol{\beta}_*$  given by Lemma \ref{prereg}. In addition, if $f\in C^{0,1}(\Omega)$, then $v_\varepsilon \in C^{1,\boldsymbol{\beta}_*}_{\rm loc}(\Omega)$. 
 \end{corollary}
  
 The regularity issue then reduces to prove that a given weak solution of the fractional equation \eqref{eqfractGL} is bounded. If we complement  \eqref{eqfractGL} with a smooth exterior Dirichlet condition, this is indeed the case. 
  
 \begin{lemma}\label{Linftybd}
 Let $g \in C^{0,1}_{\rm loc}(\R^n)\cap L^{\infty}(\R^n)$, $f\in L^\infty(\Omega)$,  and let $v_\varepsilon \in H_{g}^{s}(\Omega)\cap L^p(\Omega)$ 
be a weak solution of \eqref{eqfractGL}. Then $v_\varepsilon \in L^\infty(\R^n)$. 
 \end{lemma}
 
Let us start with an elementary lemma concerning the potential $W$.

\begin{lemma}\label{lemstrutpot}
 Let $W:\R\to[0,\infty)$ satisfying {\rm (H1)-(H2)-(H3)}. Then, for all $\delta>0$, 
 \begin{equation}\label{estiW}
 W^\prime(t)t-\delta |t|\geq 0\quad\text{whenever }|t|\geq (1+\boldsymbol{c}_W\delta)^{\frac{1}{p-1}}\,.
 \end{equation}
\end{lemma} 
 
\begin{proof}
From the lower bound in (H3), it follows that $|W^\prime(t)|>0$ for $|t|>1$. Since $W$ achieves its minimum value at $\pm1$, we deduce that $W^\prime(t)\leq 0$ for $t\leq -1$, and $W^\prime(t)\geq 0$ for $t\geq 1$. Hence the lower bound in (H3) yields
$$W^\prime(t)t\geq \frac{1}{\boldsymbol{c}_W}\big(|t|^{p-1}-1\big)|t|\geq \delta|t| $$
for $|t|\geq (1+\boldsymbol{c}_W\delta)^{\frac{1}{p-1}}$. 
\end{proof} 
 
 \begin{proof}[Proof of Lemma \ref{Linftybd}]
We fix for the whole proof a radius $R>0$ such that $\overline\Omega\subset D_R$.  

\noindent{\it Step 1.}  By Remarks \ref{remEDB} \& \ref{H1/2loctoH1loc}, $v^\e_\varepsilon \in H^1_{{{\rm loc}}}(\overline{\R^{n+1}_+},|z|^a\de\mathbf{x})$ 
and $v^\e_\varepsilon$ weakly solves \eqref{eqext} with $G=B_R^+$.  By elliptic regularity 
 we have $v^\e_\varepsilon\in C^\infty(\R^{n+1}_+)$. 
 Since $g$ is locally Lipschitz continuous and ${\rm dist}(\partial^+B_R,\overline\Omega)>0$, we easily infer from formula  \eqref{poisson} that the function 
${\bf x}\in \partial^+B_R\mapsto |v_\eps^\e({\bf x})|+ z^{a}|\nabla v^\e_\varepsilon({\bf x})|$ is bounded. We set 
   $$M:=\| v^\e_\varepsilon\|_{L^\infty(\partial^+B_R)}+\left\|z^{a}\nabla v^\e_\varepsilon\right\|_{L^\infty(\partial^+B_R)}<\infty\,. $$
Let us consider a cut-off function $\chi_R\in C^\infty(\R;[0,1])$ such that $\chi_R(t)=1$ for $|t|\leq R$ and $\chi_R(t)=0$ for $|t|\geq 3R/2$. 
We introduce the scalar function
 $$\eta:= \chi_R(|\mathbf{x}|)\sqrt{|v^\e_\varepsilon|^2+\lambda^2} \in H^1(B^+_{2R},|z|^a\de\mathbf{x})\cap L^p(\Omega)\,,$$
 with
 $$\lambda:=\max\left(\big(1+\boldsymbol{c}_W\eps^{2s}\|f\|_{L^\infty(\Omega)}\big)^{\frac{1}{p-1}},1+\|g\|_{L^\infty(\R^n\setminus\Omega)}\right)\,,$$
and $\boldsymbol{c}_W$ the constant given in  assumption ({\rm H}3).

Fix a  nonnegative function $\phi \in C^1(\overline B^+_{2R})$ with compact support in 
$B^+_{2R}\cup\Omega$.  
Noticing that $ v^\e_\varepsilon /\eta \in H^1(B_R^+,|z|^a\de\mathbf{x})$,  we obtain
$$ \int_{B_R^+} z^{a}\nabla\eta \cdot \nabla \phi\,\de \mathbf{x}= 
 \int_{B_R^+} z^{a}\nabla v^\e_\varepsilon\cdot \nabla \left(\frac{  v^\e_\varepsilon }{\eta}\phi\right) \,\de \mathbf{x}
 - \int_{B_R^+} z^{a} \frac{\phi}{\eta}\left(1-\frac{(v^\e_\eps)^2}{\eta^2}\right) |\nabla v_\eps^\e|^2 \,\de \mathbf{x}\,.$$
On the other hand $\phi\geq 0$, so that 
 $$ \int_{B_R^+}  z^{a} \nabla\eta \cdot \nabla \phi\,\de \mathbf{x} \leq   \int_{B_R^+}  z^{a} \nabla v^\e_\varepsilon\cdot \nabla \left(\frac{  v^\e_\varepsilon }{\eta}\phi\right) \,\de \mathbf{x} \,.$$
 Using equation  \eqref{eqext}, we  infer that 
  \begin{equation}\label{inegKato}
 \int_{B_R^+}  z^{a} \nabla\eta \cdot \nabla \phi\,\de \mathbf{x} \leq 
 \int_{\partial^+{B_R} } z^{a} \frac{\partial v_\varepsilon^\e}{\partial\nu}v_\eps^\e\frac{\phi}{\eta}\,\de\mathscr{H}^n
 -\frac{1}{\varepsilon^{2s}}\int_\Omega \Big(W'(v_\eps^\e)v_\varepsilon^\e-\eps^{2s}f v^\e_\eps\Big) \frac{\phi}{\eta}\,\de x\,.
 \end{equation}
 Then we conclude by approximation  (see Remark \ref{smoothapprox}) that \eqref{inegKato} actually holds for any nonnegative $\phi \in H^1(B^+_{2R},|z|^a\de \mathbf{x})\cap L^p(\Omega)$ with compact support in $B^+_{2R}\cup \overline\Omega$. 
\vskip3pt
 
 \noindent{\it Step 2.} Given $T>0$ and $\gamma>0$, we  define  the functions
 $$\rho:=\max\{\eta-\sqrt{2}\lambda,0\} \,,\quad \rho_T:=\min(\rho,T)\,,\quad \psi_{T,\gamma}:= \rho^\gamma_T\rho\,,\quad \phi_{T,\gamma}:= \rho^{2\gamma}_T\rho\,.$$
 which all belong to $H^1(B_{2R}^+,|z|^a\de\mathbf{x})\cap L^p(\Omega)$. 
Setting $G_T:=\{0<\rho<T\}\cap B_R^+$,  
straightforward computations yield 
 $$\int_{B_R^+} z^{a} |\nabla \psi_{T,\gamma}|^2\,\de \mathbf{x}=\int_{B_R^+} z^{a} \rho_T^{2\gamma}|\nabla\eta|^2\,\de \mathbf{x}+ (\gamma^2+2\gamma)\int_{G_T} z^{a}\rho^{2\gamma}|\nabla\eta|^2\,\de \mathbf{x}\,,$$
 and 
 $$\int_{B_R^+} z^{a} \nabla\eta\cdot\nabla\phi_{T,\gamma}\,\de \mathbf{x}= \int_{B_R^+} z^{a} \rho_T^{2\gamma}|\nabla \eta|^2\,\de \mathbf{x}+ 2\gamma\int_{G_T} z^{a}\rho^{2\gamma}|\nabla \eta|^2\,\de \mathbf{x}\,.$$
 From this two last equalities, we  infer that 
 $$ \int_{B_R^+} z^{a} |\nabla \psi_{T,\gamma}|^2\,\de \mathbf{x}\leq (\gamma+1)\int_{B_R^+} z^{a} \nabla\eta\cdot\nabla\phi_{T,\gamma}\,\de \mathbf{x}\,.$$
Next we want to use $\phi_{T,\gamma}$ as a test function in  \eqref{inegKato}. To this purpose it is enough to show that $\rho$ has compact support in $B_{2R}^+\cup\overline\Omega$. Obviously, $\rho$ has compact support in $B_{2R}^+\cup D_{2R}$. 
 Since $v^\e_\eps=g_\eps$ on $D_{2R}\setminus\overline\Omega$, we have $|v_\eps^\e|\leq \lambda-1$ on $D_{2R}\setminus\overline\Omega$. Consider a point $\mathbf{x}_0=(x_0,0)$ with $x_0\in D_{2R}\setminus\overline\Omega$. From the smoothness of $g_\eps$ and \eqref{poisson}, we derive that $v_\eps^\e$ is continous at $\mathbf{x}_0$. Therefore there exists a radius $\delta>0$ such that $|v_\eps^\e|<\lambda$ in $\overline B_{\delta}^+(\mathbf{x}_0)$. Then $\rho=0$  in $\overline B_{\delta}^+(\mathbf{x}_0)$, and hence $\rho$ has compact support in $B_{2R}^+\cup\overline\Omega$. 
  
 Then, finally  using $\phi_{T,\gamma}$ as a test function in \eqref{inegKato},  we  deduce that
\begin{multline*}
\int_{B_R^+} z^{a} |\nabla \psi_{T,\gamma}|^2\,\de \mathbf{x}\leq (\gamma+1) 
 \int_{\partial^+B_R}  z^{a} \frac{\partial v_\varepsilon^\e}{\partial\nu}\frac{v_\eps^\e}{\eta}\, \rho_T^{2\gamma}\rho\,\de \mathscr{H}^n \\
 -\frac{\gamma+1}{\varepsilon^{2s}}\int_\Omega \Big(W'(v_\eps^\e)v_\varepsilon^\e -\eps^{2s}fv^\e_\eps\Big)\frac{\rho_T^{2\gamma}\rho}{\eta}\,\de x\,.
\end{multline*}
 Noting that $|v_\eps^\e|\geq \lambda$ on $\{\rho>0\}$, we have
 $$W'(v_\eps^\e)v_\eps^\e-\eps^{2s}fv^\e_\eps\geq W'(v_\eps^\e)v_\eps^\e-\eps^{2s}\|f\|_{L^\infty(\Omega)}|v^\e_\eps| \geq 0\quad\text{ on $\{\rho>0\}\cap \Omega$}\,,$$ 
by Lemma \ref{lemstrutpot}. Since $\rho\leq |v_\eps^\e|$, the previous estimate leads to
 $$\| \nabla( \rho_T^{\gamma}\rho)\|^2_{L^2(B^+_R,|z|^a\de\mathbf{x})}\leq (\gamma+1)\mathscr{H}^n(\partial^+B_R) M^{2\gamma+2} \,.$$
 By the continuous embedding \eqref{contembedd}, $\rho_T^{\gamma}\rho\in W^{1,1}(B_R^+)$. Moreover, since  $\rho_T^{\gamma}\rho$ vanishes on $D_R\setminus \overline\Omega$, we can apply the Poincar\'e inequality in \cite[Corollary 4.5.2]{Zi} and the continuity of the trace operator \eqref{conttracW1} to deduce that 
 $$\|\rho_T^\gamma\rho\|^2_{L^1(D_R)}\leq C_{R,\Omega} \| \nabla (\rho_T^{\gamma}\rho)\|^2_{L^1(B_R^+)}\,,$$
 for a constant $C_{R,\Omega}>0$ which only depends on $R$ and $\Omega$.  From the two previous inequality and \eqref{contembedd}, we derive 
 $$  \|\rho_T^\gamma\rho\|^2_{L^1(D_R)}\leq C_{s,R,\Omega} (\gamma+1)M^{2\gamma+2}\,.$$
 Next we let $T\to \infty$ in this last inequality to obtain 
 $$\|\rho\|^2_{L^{\gamma+1}(D_R)} \leq C_{s,R,\Omega}^{1/(\gamma+1)}(\gamma+1)^{1/(\gamma+1)}M^2\,.$$
 Letting now $\gamma\to\infty$ leads to $\|\rho\|_{L^\infty(D_R)}\leq M$, which in turn implies $v_\varepsilon\in L^\infty(\Omega)$. 
 Since $v_\varepsilon=g$ outside~$\Omega$, we have thus proved that $v_\varepsilon\in L^\infty(\R^n)$. 
 \end{proof}
  
In the case where  equation \eqref{eqfractGL} is complemented with a smooth exterior Dirichlet condition, weak solutions are thus bounded. Then we can apply  \cite[Theorem 2]{SerVal} to deduce continuity  across the boundary $\partial \Omega$, and finally obtain the following regularity result.  

\begin{theorem}\label{regDirich}
Assume that $\partial\Omega$ is smooth. Let $g\in   C^{0,1}_{\rm loc}(\R^n)\cap L^\infty(\R^n)$, $f\in L^\infty(\Omega)$,  
and let $v_\varepsilon \in H_{g}^{s}(\Omega)\cap L^p(\Omega)$ 
be a weak solution of \eqref{eqfractGL}. Then 
$v_\eps \in C^{0,\boldsymbol{\beta}_*}_{\rm loc}(\Omega)\cap C^0(\R^n)$ with $\boldsymbol{\beta}_*$  given by Lemma \ref{prereg}. 
\end{theorem}  

By means of the Hopf boundary lemma in \cite[Proposition 4.11]{CS1}, we now derive the following maximum principle for equation \eqref{eqfractGL}.  
 
\begin{corollary}\label{modless1}
Let $\Omega$, $g$, and $f$ be as in Theorem \ref{regDirich}. Let $v_\varepsilon \in H_{g}^{s}(\Omega)\cap L^p(\Omega)$ 
be a weak solution of~\eqref{eqfractGL}. Then,  
\begin{equation}\label{ptbd}
\|v_\varepsilon\|_{L^\infty(\R^n)}\leq \max\left(\big(1+\boldsymbol{c}_W\eps^{2s}\|f\|_{L^\infty(\Omega)}\big)^{\frac{1}{p-1}},\|g\|_{L^\infty(\R^n\setminus\Omega)}\right)\,,
\end{equation}
where $\boldsymbol{c}_W$ is the constant given in  assumption {\rm (H3)}.
\end{corollary}
 
\begin{proof}
We consider the  function $m_\varepsilon:=\lambda^2-|v^\e_\varepsilon|^2$ with $\lambda$ being the constant in the right hand side of \eqref{ptbd}. 
By Theorem  \ref{regDirich}, $m_\eps$ is continuous in $\overline\R^{n+1}_+$, and $z^a\partial_z m_\eps$ is continuous up to~$\Omega$. Moreover, 
$m_\varepsilon$ satisfies (in the pointwise sense)
$$\begin{cases} 
-{\rm div}\big(z^{a} \nabla m_\varepsilon\big)=2z^{a}|\nabla v^\e_\varepsilon|^2\geq 0 & \text{in $\R^{n+1}_+$}\,,\\
\displaystyle d_s\boldsymbol{\partial}^{(2s)}_z m_\varepsilon= -\frac{2}{\varepsilon^{2s}}W^\prime(v^\e_\eps)v^\e_\eps+2fv^\e_\eps & \text{on $\Omega$}\\
m_\varepsilon \geq 0 & \text{on $\R^n\setminus\Omega$}\,.
\end{cases}
$$
Assume that $m_\varepsilon$ achieves its minimum over $\R^n$ at a point $x_0\in\Omega$. Then $x_0$ is a point of maximum of $|v_\varepsilon|$, and hence $\mathbf{x}_0=(x_0,0)$  is an absolute minima of 
$m_\varepsilon$ over $\overline\R^{n+1}_+$ by \eqref{bdlinftyext}. If $m_\eps(\mathbf{x}_0)< 0$,  then $|v_\eps^\e(\mathbf{x}_0)|>\lambda$, and we obtain 
$\boldsymbol{\partial}^{(2s)}_z m_\varepsilon(\mathbf{x}_0)\leq 0$ from \eqref{estiW}. On the other hand, the strong maximum maximum principle of \cite[Corollary 2.3.10]{FKS} implies that $m_\eps>m_\eps(\mathbf{x}_0)$ in $\mathbb{R}^{n+1}_+$. Then, the 
Hopf boundary lemma of \cite[Proposition 4.11]{CS1} yields $\boldsymbol{\partial}^{(2s)}_z m_\varepsilon(\mathbf{x}_0)>0$ which gives a contradiction.  
\end{proof}

%%%%%%%%%%%%%%%%%%%%%%%%%%%%%%%%%%%%%%%%%%%%%%%%%%%%%%%%%%%%%%%%%%%%%%%%
%%%%%%%%%%%%%%%%%%%%%%%%%%%%%%%%%%%%%%%%%%%%%%%%%%%%%%%%%%%%%%%%%%%%%%%%
                                                              										                  
\section{Asymptotics for degenerate  Allen-Cahn boundary reactions}\label{epsregtitle}                         
                                                                                                              					            
%%%%%%%%%%%%%%%%%%%%%%%%%%%%%%%%%%%%%%%%%%%%%%%%%%%%%%%%%%%%%%%%%%%%%%%%
%%%%%%%%%%%%%%%%%%%%%%%%%%%%%%%%%%%%%%%%%%%%%%%%%%%%%%%%%%%%%%%%%%%%%%%%

In this section, our objective is to perform the asymptotic analysis as $\varepsilon\downarrow 0$ of the degenerate boundary reaction equation \eqref{equinG}. As described in Section \ref{FractAC}, any solution of the fractional Allen-Cahn equation yields a solution of \eqref{equinG} after applying  the extension procedure  \eqref{poisson}. Here again, the strategy is to  first  analyse equation  \eqref{equinG} and then to apply the results to the fractional equation. The main theorem here is Theorem \ref{asymptneum} below. Its application to the fractional equation will be the object of Section \ref{FGLasymp}.

 \begin{theorem}\label{asymptneum}
Let $G\subset\R^{n+1}_+$ be an admissible bounded  open set, and $\varepsilon_k\downarrow0$  a given sequence. Let $\{f_k\}_{k\in\mathbb{N}}\subset C^{0,1}(\partial^0G)$  
satisfying 
$$\sup_k\Big(\eps_k^{2s}\|f_k\|_{L^\infty(\partial^0G)} + \|f_k\|_{W^{1,q}(\partial^0G)}\Big)<\infty\quad \text{for some $q\in(\frac{n}{1+2s},n)$}\,. $$
Let  $\{u_k\}_{k\in\mathbb{N}}\subset H^1(G,|z|^a\de\mathbf{x})\cap L^\infty(G)$ satisfying $\sup_k\|u_k\|_{L^\infty(G)}< \infty$, and  such that each $u_k$ weakly solves  
\begin{equation}\label{equinG}
 \begin{cases}
{\rm div}\big(z^{a}\nabla u_k)= 0 & \text{in $G$}\,,\\[8pt]
\displaystyle d_s\boldsymbol{\partial}^{(2s)}_z u_k=\frac{1}{\varepsilon^{2s}_k}W^\prime(u_k) -f_k & \text{on $\partial^0G$}\,. 
\end{cases}
\end{equation}
If  $\sup_k \mathbf{F}_{\varepsilon_k}(u_k,G)<\infty$, then there exist a (not relabeled) subsequence, $u_*\in H^1(G,|z|^a\de\mathbf{x})$ and an open subset $E_*\subset\partial^0G$ such that  $u_*=\chi_{E_*}-\chi_{\partial^0G\setminus E_*}$ on $\partial^0G$,   $u_k\rightharpoonup u_*$ weakly in $H^1(G,|z|^a\de\mathbf{x})$, and $u_k\to u_*$ strongly in $H^1_{{\rm loc}}(G\cup\partial^0G,|z|^a\de\mathbf{x})$ as $k\to\infty$. In addition, 
\vskip5pt

\begin{itemize}[leftmargin=22pt]
\item[ \rm  (i)] $\varepsilon^{-2s}_kW(u_k)\to 0$ in $L^1_{\rm loc}(\partial^0G)$; 
\vskip5pt

\item[\rm  (ii)] $u_k\to u_*$ in $C^0_{\rm loc}(\partial^0G\setminus\partial E_*)$; 
\vskip5pt 

\item[\rm  (iii)] if $\sup_k\|f_k\|_{L^\infty(\partial^0G)}<\infty$, then $u_k\to u_*$ in $C^{0,\alpha}_{\rm loc}(\partial^0G\setminus\partial E_*)$ for every $\alpha\in(0,\boldsymbol{\beta}_*)$ with $\boldsymbol{\beta}_*$  given by Lemma \ref{prereg};
\vskip5pt

\item[\rm  (iv)] if $\sup_k\|f_k\|_{C^{0,1}(\partial^0G)}<\infty$, then $u_k\to u_*$ in $C^{1,\alpha}_{\rm loc}(\partial^0G\setminus\partial E_*)$ for every $\alpha\in(0,\boldsymbol{\beta}_*)$;
\vskip5pt

\item[\rm  (v)] for each $t\in(-1,1)$, the level set $L^t_k:=\{u_k=t\}$ converges locally uniformly  in $\partial^0G$ to $\partial E_*\cap \partial^0 G$, i.e., for every compact set $K\subset \partial^0G$ and every $r>0$,
$$L^t_k\cap K\subset \mathscr{T}_r(\partial E_*\cap \partial^0 G) \quad\text{and}\quad
\partial E_*\cap K\subset \mathscr{T}_r(L_k^t\cap  \partial^0 G)\,, $$
whenever $k$ is large enough;
\item[\rm  (vi)] if $f_k\rightharpoonup f_*$ weakly in $W^{1,q}(\partial^0G)$, then the function $u_*$ satisfies 
$$\delta{\bf E}\big(u_*,G\cup\partial^0G\big)[{\bf X}]=\int_{\partial^0G} u_*\, {\rm div}(f_*X)\,\de x$$
for every vector field $\mathbf{X}=(X,\mathbf{X}_{n+1})\in C^1(\overline G;\R^{n+1})$ compactly supported in $G\cup \partial^0G$ such that $\mathbf{X}_{n+1}=0$ on $\partial^0G$. 
\end{itemize}
\end{theorem}
\vskip5pt

We have divided the proof of this theorem in several steps according to the following subsections.

%%%%%%%%%%%%%%%%%%%%%%%%%%%%%%%%%%%%%%%%%%%%%%%%%%%%%%%

\subsection{Energy monotonicity and the clearing-out property}
 
 %%%%%%%%%%%%%%%%%%%%%%%%%%%%%%%%%%%%%%%%%%%%%%%%%%%%%%%

 In this subsection, we prove two of the main ingredients, an energy monotonicty, and a clearing-out property reminiscent of Ginzburg-Landau theories. We start with the fundamental {\sl monotonicity formula}. 

 \begin{lemma}\label{monotform}
Let $q\in(\frac{n}{1+2s},n)$, $R>0$, and $\eps>0$. Given $f\in  C^{0,1}(D_{R})$, let $u_\eps \in  H^1(B_R^+,|z|^a\de \mathbf{x})\cap L^\infty(B^+_R)$ 
 be a weak solution  of 
\begin{equation}\label{pfff}
\begin{cases}
{\rm div}(z^{a}\nabla u_\eps)= 0 & \text{in $B_R^+$}\,,\\[5pt]
\displaystyle d_s\boldsymbol{\partial}_z^{(2s)} u_\eps=\frac{1}{\varepsilon^{2s}}W^\prime(u_\eps)-f & \text{on $D_R$}\,.
\end{cases}
\end{equation} 
There exists a constant $\mathbf{c}_{n,q}>0$ (depending only on $n$ and~$q$) such that 
for every point $\mathbf{x}_0=(x_0,0)\in D_R\times\{0\}$, the function $r\in(0,R-|\mathbf{x}_0|\,]\mapsto \boldsymbol{\Theta}^\eps_{u_\eps}(f,x_0,r)$ defined by 
$$ \boldsymbol{\Theta}^\eps_{u_\eps}(f,x_0,r):=\frac{1}{r^{n-2s}} \mathbf{E}_\eps\big(u_\eps,B_r^+(\mathbf{x}_0)\big)+\mathbf{c}_{n,q}\|u_\eps\|_{L^\infty(D_R)}\int_0^rt^{\theta_q-1}\|f\|_{\dot W^{1,q}(D_t(x_0))}\,\de t$$
with $\theta_q:=
1+2s-n/q$, is non-decreasing.
\end{lemma}

\begin{remark}
In the statement  above, $\|f\|_{\dot W^{1,q}(A)}$ denotes the following  {\it $W^{1,q}$-homogeneous norm} of $f$ in $A$,
$$ \|f\|_{\dot W^{1,q}(A)}:=\|f\|_{L^{q^*}(A)}+\|\nabla f\|_{L^q(A)}\,,$$
where $q^*:=nq/(n-q)$. 
\end{remark}

\begin{proof}[Proof of Lemma \ref{monotform}]
Without loss of generality we may assume that $x_0=0$.  
By Theorem~\ref{regint} the function $r\in(0,R)\mapsto \mathbf{E}_\eps(u_\eps,B_r^+)$ is of class $C^1$, and then it is enough to seek for a constant $L$ such that for  $r\in(0,R)$, 
$$-\frac{(n-2s)}{r^{n+1-2s}} \mathbf{E}_\eps(u_\eps,B_r^+)+\frac{1}{r^{n-2s}} \frac{\de}{\de r} \mathbf{E}_\varepsilon(u_\eps,B_r^+)+ L r^{\theta_q-1}\|f\|_{\dot W^{1,q}(D_r)}\geq 0\,,$$
or equivalently, 
\begin{equation}\label{diffineq}
(n-2s) \mathbf{E}_\eps(u_\eps,B_r^+)-r \frac{\de}{\de r} \mathbf{E}_\varepsilon(u_\eps,B_r^+)\leq L r^{n+1-n/q}\|f\|_{\dot W^{1,q}(D_r)}\,.
\end{equation}
Note that for   $r\in(0,R)$, 
\begin{equation}\label{derenergbr}
\frac{\de}{\de r}\mathbf{E}_\varepsilon(u_\eps,B_r^+)=\frac{d_s}{2}\int_{\partial^+B_r} z^{a}\big|\nabla u_\eps \big|^2\,\de \mathscr{H}^n +\frac{1}{\eps^{2s}}\int_{\partial D_r}W(u_\eps)\,\de \mathscr{H}^{n-1}\,.
\end{equation}
To prove \eqref{diffineq}, we first consider an arbitrary even function $\eta\in C^\infty(\R;[0,1])$ with compact support in $(-R,R)$. Using the vector field $\mathbf{X}(\mathbf{x}):=\eta(|\mathbf{x}|)\mathbf{x}$ in Corollary \ref{statAC} and formula \eqref{formulafirstvarbound}, we find that
\begin{multline}\label{tim1843}
\frac{(n-2s)d_s}{2}\int_{B_R^+}z^{a}|\nabla u_\eps|^2\eta(|\mathbf{x}|)\,\de \mathbf{x}+ \frac{d_s}{2}\int_{B_R^+}z^{a}|\nabla u_\eps|^2\eta^\prime(|\mathbf{x}|)|\mathbf{x}|\,\de \mathbf{x} \\
-d_s\int_{B_R^+} z^a\Big|\frac{\mathbf{x}}{|\mathbf{x}|}\cdot\nabla u_\eps\Big|^2 \eta^\prime(|\mathbf{x}|)|\mathbf{x}|\,\de \mathbf{x}
+\frac{n}{\varepsilon^{2s}}\int_{D_R}W(u_\eps)\,\eta(|x|)\,\de x\\
+\frac{1}{\varepsilon^{2s}}\int_{D_R}W(u_\eps)\,\eta^\prime(|x|)|x|\,\de x\\ 
=\int_{D_R} \big(nf+x\cdot\nabla f\big)u_\eps \eta(|x|)\,\de x +\int_{D_R} fu_\eps\eta^\prime(|x|)|x|\,\de x\,.
\end{multline}
Given $r\in(0,R)$, we can consider a sequence $\{\eta_k\}_{k\in\mathbb{N}}$ of functions as above such that $\eta_k$ converges weakly* in $BV$ as $k\to\infty$ to the characteristic function of the interval $[-r,r]$. Using such sequences $\{\eta_k\}_{k\in\mathbb{N}}$ as test functions in \eqref{tim1843} and letting $k\to\infty$, we infer that   
\begin{multline}\label{condmonot}
(n-2s) \mathbf{E}_\eps(u_\eps,B_r^+)-r \frac{\de}{\de r} \mathbf{E}_\varepsilon(u_\eps,B_r^+) + d_sr\int_{\partial^+ B_r}z^a\Big|\frac{\mathbf{x}}{|\mathbf{x}|}\cdot\nabla u_\eps\Big|^2\,\de \mathscr{H}^n\\
+ \frac{2s}{\varepsilon^{2s}}\int_{D_t}W(u_\eps)\,\de x=\int_{D_r} \big(nf+x\cdot\nabla f\big)u_\eps  \,\de x - r\int_{\partial D_r} fu_\eps \,\de \mathscr{H}^{n-1}\,.
\end{multline}
Therefore, 
\begin{equation}\label{124514}
(n-2s) \mathbf{E}_\eps(u_\eps,B_r^+)-r \frac{\de}{\de r} \mathbf{E}_\varepsilon(u_\eps,B_r^+)\\
\leq \|u_\eps\|_{L^\infty(D_R)} I(r) \,,
\end{equation}
with 
\begin{equation}\label{definitionImonot}
I(r):=\int_{D_r} |f|+r|\nabla f|\,\de x+r\int_{\partial D_r}|f|\,\de\mathscr{H}^{n-1}\,.
\end{equation}
By Sobolev embedding and trace inequality, we have  
\begin{equation}\label{1245142}
I(r)\leq \mathbf{c}_{n,q}\,r^{n+1-\frac{n}{q}} \|f\|_{\dot W^{1,q}(D_{r})} \,,
\end{equation}
for a constant $\mathbf{c}_{n,q}$ depending only on $n$ and $q$.  Combining \eqref{124514} and \eqref{1245142} leads to \eqref{diffineq}, with $L=\mathbf{c}_{n,q}\|u_\eps\|_{L^\infty(D_R)}$. 
\end{proof}

\begin{lemma}\label{clear2}
Let $q\in(\frac{n}{1+2s},n)$. Given $b\geq 1$ and $T\geq 0$, there exists a non-decreasing function ${\boldsymbol{\eta}}_{b,T}:(0,1)\to (0,\infty)$ 
depending only $n$, $s$, $b$, $T$, and $W$, such that the following holds. Let $R\in(0,1]$, $\eps\in(0,R)$, and $f\in C^{0,1}(D_R)$ such that $\eps^{2s}\|f\|_{L^\infty(D_R)}\leq T$. 
If $u_\eps\in H^1(B_R^+,|z|^a\de \mathbf{x})\cap L^\infty(B^+_R)$ is a weak solution of  \eqref{pfff} satisfying $\|u_\eps\|_{L^\infty(B^+_R)}\leq b$, and for some $\delta\in(0,1)$, 
\begin{equation}\label{smallenerg}
\boldsymbol{\Theta}^\eps_{u_\eps}(f,0,R)\leq \boldsymbol{\eta}_{b,T}(\delta)\,,
\end{equation}
then $\big||u_\eps|-1\big|\leq \delta$  on  ${D}_{R/2}$. 
\end{lemma}

\begin{proof}
{\it Step 1.} We assume in this first step that $\varepsilon\geq R/2$. 
We claim that we can find a constant $\widetilde{\boldsymbol{\eta}}_{b,T}(\delta)>0$ depending only on $\delta$, $n$, $s$, $b$, $T$, and $W$, such that the condition 
$\boldsymbol{\Theta}^\eps_{u_\eps}(f,0,R) \leq \widetilde{\boldsymbol{\eta}}_{b,T}(\delta)$
implies $ \big||u_\eps|-1\big|\leq \delta$ in $\overline B^+_{R/2}$. To this purpose, we consider the rescaled function $\widetilde u_\varepsilon(\mathbf{x}):=u_\varepsilon(R \mathbf{x}) $, 
which satisfies
$$\begin{cases}
{\rm div}(z^{a} \nabla \widetilde u_\varepsilon) = 0 & \text{in $B_1^+$}\,,\\[8pt]
\displaystyle d_s\boldsymbol{\partial}_z^{(2s)} \widetilde u_\varepsilon =\frac{R^{2s}}{\eps^{2s}}W^\prime(\widetilde u_\eps)-f_R & \text{on $D_1$}\,,
\end{cases}$$
with $\eps/R\in[ 1/2,1)$, and $f_R(x):=R^{2s}f(Rx)$ satisfying 
$$\|f_R\|_{L^\infty(D_1)}\leq 2^{2s}\varepsilon^{2s}\|f\|_{L^\infty(D_R)} \leq 2^{2s}T\,.$$ 
Since $\|\widetilde u_\eps\|_{L^\infty(B_1^+)}\leq b$, we infer from Remark~\ref{remclear} that 
\begin{equation}\label{holdbdseq}
\|\widetilde u_\eps\|_{ C^{0,\boldsymbol{\beta}_*}(\overline B^+_{1/2})}\leq C_{b,T}\,,
\end{equation} 
for a constant $C_{b,T}$ depending only on $n$, $s$, $b$, $T$, and $W$.   

We now argue by contradiction assuming that for some sequences $\{R_k\}_{k\in\mathbb{N}}\subset(0,1]$, $\{\varepsilon_k\}_{k\in\mathbb{N}}\subset[R_k/2,R_k)$, $\{f_k\}_{k\in\mathbb{N}}\subset C^{0,1}(D_{R_k})$ with $\eps_k^{2s}\|f_k\|_{L^\infty(D_{R_k})}\leq T$, and points $\{\mathbf{x}_k\}_{k\in\mathbb{N}}\subset \overline B^+_{1/2}$, 
the function $\widetilde u_k:=\widetilde u_{\eps_k}$ satisfies 
$$ \big| |\widetilde u_k(\mathbf{x}_k)|-1\big|>\delta \quad\text{for every $k$}\,,$$
and
$$ \mathbf{E}_{\varepsilon_k/R_k}(\widetilde u_k,B_1^+)=\frac{1}{R_k^{n-2s}}\mathbf{E}_{\eps_k}(u_{\eps_k},B_{R_k}^+)\leq \boldsymbol{\Theta}^{\eps_k}_{u_{\eps_k}}(f_k,0,R_k)\to 0\quad \text{as $k\to\infty$}\,.$$
By the Arzel\`a-Ascoli Theorem and \eqref{holdbdseq}, we can find a (not relabeled) subsequence  such that $\widetilde u_k$ converges uniformly on $\overline B^+_{1/2}$. 
Since  $\mathbf{E}_{\varepsilon_k/R_k}(\widetilde u_k,B_1^+)\to 0$, the limit has to be a constant of modulus one. In particular, $|\widetilde u_k|\to 1$  uniformly on $\overline B^+_{1/2}$, which contradicts our assumption  $\big||\widetilde u_k(\mathbf{x}_k)|-1\big|>\delta$. 
\vskip3pt

\noindent{\it Step 2.} Define
$$\boldsymbol{\eta}_{b,T}(\delta):=2^{2s-n}\inf_{t\in[\delta,1)} \widetilde{\boldsymbol{\eta}}_{b,T}(t)\,.$$
Let $\delta\in(0,1)$ and assume that \eqref{smallenerg} holds for $R\in(0, 1]$ and $\eps\in(0,R)$. We fix an arbitrary point $\mathbf{x}_0\in \overline D_{R/2}\times\{0\}$. If $\eps\geq R/2$, then $\big||u_\eps(\mathbf{x}_0)|-1\big|\leq \delta$ by Step~1. If $\eps<R/2$, then $\varepsilon<R-|\mathbf{x}_0|$ and  by Lemma~\ref{monotform} we have 
$$\boldsymbol{\Theta}^\eps_{u_\eps}(f,x_0,\eps)\leq \boldsymbol{\Theta}^\eps_{u_\eps}\big(f,x_0,R-|x_0|\big)\leq 2^{2s-n}\boldsymbol{\Theta}^\eps_{u_\eps}(f,0,R)\,.$$
Our choice of $\boldsymbol{\eta}_{b,T}(\delta)$ then implies $\boldsymbol{\Theta}^\eps_{u_\eps}(f,x_0,\eps)\leq \widetilde{\boldsymbol{\eta}}_{b,T}(\delta)$,  
and  we infer from Step 1 that $\big||u_\eps|-1\big|\leq \delta$ in $\overline B^+_{\eps/2}(\mathbf{x}_0)$.
\end{proof}

\begin{remark}
 By Theorem \ref{regint}, $u_\eps$ is continuous up to $D_R$. Hence the conclusion of Lemma~\ref{clear2} implies that either $|u_\eps-1|\leq \delta$ on $D_{R/2}$, or $|u_\eps+1|\leq \delta$ on $D_{R/2}$. 
\end{remark}

%%%%%%%%%%%%%%%%%%%%%%%%%%%%%%%%%%%%%%%%%%%%%%%%%%%%%%%
 
 \subsection{Small energy compactness}
 
%%%%%%%%%%%%%%%%%%%%%%%%%%%%%%%%%%%%%%%%%%%%%%%%%%%%%%%

Our objective in this subsection is to prove that the small energy assumption \eqref{smallenerg} implies strong compactness in a half ball of smaller radius, and uniform convergence to either $+1$ or $-1$ on the bottom disc. By Lemma~\ref{clear2}, it suffices to prove such compactness assuming that the solution is already very close to $\pm 1$ on the disc. In this situation, the main ingredient to use is  the convexity of the potential $W$ near $\{\pm1\}$ to show the minimality of the solution.  Then compactness can be deduced by classical cut and paste arguments. To quantify the convexity of $W$ near $\{\pm1\}$, we introduce a structural constant $\boldsymbol{\delta}_W\in(0,1/2]$ (whose existence is ensured by assumptions (H1)-(H2)) such that 
\begin{equation}\label{quantconvex}
W^{\prime\prime}(t)\geq \frac{1}{2}\min\big\{W^{\prime\prime}(1),W^{\prime\prime}(-1)\big\}>0 \qquad\text{for $\big||t|-1|\leq \boldsymbol{\delta}_W$}\,.
\end{equation}
In this way, the restriction of $W$ to each interval $I_\kappa:=(\kappa-\boldsymbol{\delta}_W,\kappa+\boldsymbol{\delta}_W)$, $\kappa\in\{\pm 1\}$, 
is (strictly) convex. We now consider the modified potentials defined for $\kappa\in\{\pm1\}$ by 
$$\widetilde W_{\kappa}(t) :=
\begin{cases}
W(t) & \text{for $t\in I_\kappa$}\,,\\
W(\kappa-\boldsymbol{\delta}_W) +W^\prime(\kappa-\boldsymbol{\delta}_W)(t-\kappa+\boldsymbol{\delta}_W) & \text{for $t\leq \kappa-\boldsymbol{\delta}_W$}\,,\\
W(\kappa+\boldsymbol{\delta}_W) +W^\prime(\kappa+\boldsymbol{\delta}_W)(t-\kappa-\boldsymbol{\delta}_W) & \text{for $t\geq \kappa+\boldsymbol{\delta}_W$}\,.
\end{cases}$$
By construction, we have $\widetilde W_\kappa\in C^1(\R)$ and $ \widetilde W_\kappa$  is convex for each $\kappa\in\{\pm1\}$. 

\begin{lemma}\label{minimality}
Let $R>0$, $f\in L^{\infty}(D_R)$, and let $u_\varepsilon\in H^1(B^+_R,|z|^a\de\mathbf{x})\cap L^p(D_R)$ be a weak solution of \eqref{pfff}.  
If $|u_\varepsilon - \kappa|\leq \boldsymbol{\delta}_W$ on $D_{R}$ with $\kappa\in\{\pm1\}$, then 
\begin{multline*}
{\bf E}(u_\eps,B_R^+)+\frac{1}{\varepsilon^{2s}}\int_{D_{R}}\widetilde W_\kappa(u_\varepsilon)\,\de x-\int_{D_R}fu_\eps\,\de x\\
\leq {\bf E}(w,B_R^+)+\frac{1}{\varepsilon^{2s}}\int_{D_{R}}\widetilde W_\kappa(w)\,\de x -\int_{D_R}fw\,\de x\,,
\end{multline*}
for every $w\in H^1(B^+_R,|z|^a\de\mathbf{x})\cap L^p(D_R)$ such that $w-u_\varepsilon$ is compactly supported in $B^+_{R}\cup D_{R}$. 
\end{lemma}

\begin{proof}
Set $\phi:=w-u_\eps$, so that $\phi$ is compactly supported in $B^+_{R}\cup D_{R}$. By convexity of the potential $\widetilde W_\kappa$, we have
$$ \widetilde W_\kappa(u_\eps+\phi)\geq \widetilde W_\kappa(u_\eps)+\widetilde W_\kappa^\prime(u_\eps)\phi\quad\text{on $D_{R}$}\,.$$
Since $|u_\varepsilon- \kappa|\leq \boldsymbol{\delta}_W$ on $D_{R}$, we have $\widetilde W_\kappa^\prime(u_\eps)=W^\prime(u_\eps)$ on $D_{R}$. Then we derive from equation  \eqref{pfff}, 
\begin{align*}
&{\bf E}(u_\eps+\phi,B_R^+) +\frac{1}{\varepsilon^{2s}}\int_{D_{R}}\widetilde W_\kappa(u_\eps+\phi)\,\de x \\
&\begin{multlined}[t][12cm]
\geq {\bf E}(u_\eps,B_R^+) +\frac{1}{\varepsilon^{2s}}\int_{D_{R}}\widetilde W_\kappa(u_\eps)\,\de x\\
+d_s \int_{B^+_{R}}z^{a}\nabla u_\eps \nabla \phi \,\de \mathbf{x} +\frac{1}{\varepsilon^{2s}}\int_{D_{R}}W^\prime(u_\eps) \phi \,\de x
\end{multlined}\\
&\geq  {\bf E}(u_\eps,B_R^+) +\frac{1}{\varepsilon^{2s}}\int_{D_{R}}\widetilde W_\kappa(u_\eps)\,\de x+\int_{D_R} f\phi\,\de x\,,
\end{align*}
and the lemma is proved. 
\end{proof}

We now prove the announced compactness in energy space under the closeness assumption to $\{\pm 1\}$ on the bottom disc.

\begin{corollary}\label{strongcoro}
Let $R>0$, $\eps_k\downarrow0$ a given sequence, and $\{f_k\}_{k\in\mathbb{N}}\subset L^{\infty}(D_R)$ satisfying $\sup_k\|f_k\|_{L^q(D_R)}<\infty$ for some $q>1$. Let  $\{u_k\}_{k\in\mathbb{N}} \subset H^1(B_R^+,|z|^a\de\mathbf{x})\cap L^\infty(B^+_R)$ satisfying  $|u_k -\kappa|\leq \boldsymbol{\delta}_W$ on $D_{R}$ with $\kappa\in\{\pm1\}$, and such that $u_k$ solves in the weak sense  
\begin{equation}\label{tim1439}
\begin{cases}
{\rm div}(z^{a}\nabla u_k)= 0 & \text{in $B_R^+$}\,,\\[5pt]
\displaystyle d_s\boldsymbol{\partial}_z^{(2s)} u_k=\frac{1}{\varepsilon_k^{2s}}W^\prime(u_k) -f_k & \text{on $D_R$}\,.
\end{cases}
\end{equation}
If $\sup_k \big\{\mathbf{E}_{\varepsilon_k}(u_k,B_R^+) +\|u_k\|_{L^\infty(B_R^+)}\big\}< \infty$, then there exist a (not relabeled) subsequence  and $u_*\in H^1(B_R^+,|z|^a\de\mathbf{x})$ satisfying $u_*=\kappa$ on $D_R$ such that 
\begin{enumerate}
\item[\rm (i)] $u_k\to u_*$ strongly in   $H^1(B_{r}^+,|z|^a\de\mathbf{x})$  for every $r\in(0,R)$; 
\item[\rm (ii)] $\eps^{-2s}_k\int_{D_r}W(u_k)\,\de x\to 0$ for every $r\in(0,R)$.
\end{enumerate} 
\end{corollary}

\begin{proof} We may assume without loss of generality that $R=1$ and $|u_k- 1|\leq \boldsymbol{\delta}_W$  on $D_{1}$ (i.e., $\kappa=+1$). Let us set 
$$M:=\sup_k \big\{\mathbf{E}_{\varepsilon_k}(u_k,B_1^+) +\|u_k\|_{L^\infty(B_1^+)}\big\}\,.$$
From the assumption that $M$ is finite, we first deduce  that the sequence $\{u_k\}_{k\in\mathbb{N}}$ is bounded in $H^1(B_1^+,|z|^a\de\mathbf{x})$. Hence we can find a (not relateled) subsequence such that $u_k\rightharpoonup u_*$ weakly in $H^1(B_1^+,|z|^a\de\mathbf{x})$. 
On the other hand, since $|u_k- 1|\leq \boldsymbol{\delta}_W$ on $D_R$, we infer from \eqref{quantconvex} that 
$$\int_{D_1}|u_k-1|^2\,\de x\leq C \int_{D_1}W(u_k)\,\de x\leq C M\eps_k^{2s} \to 0\,,$$
so that $u_k\to 1$ strongly in $L^{2}(D_1)$, and therefore in $L^{q/(q-1)}(D_1)$.  By continuity of the linear trace operator, it also follows that $u_*=1$ on $D_1$. 

Let us now fix $r\in(0,1)$. We start selecting a subsequence $\{u_{k_j}\}_{j\in\mathbb{N}}$ such that 
$$\limsup_{k\to+\infty}  \mathbf{E}_{\eps_{k}}(u_{k},B_r^+)=\lim_{j\to+\infty}  \mathbf{E}_{\eps_{k_j}}(u_{k_j},B_r^+)\,.$$
For $\theta\in(0,1)$, we set $r_\theta:=1-\theta+\theta r$ and $L_\theta:= r_\theta -r $.  Given an arbitrary integer $m\geq 1$, we define $r_i:=r+i\delta_m$ where $i\in\{0,\ldots, m\}$ and $\delta_m:=L_\theta/m$. 
Since 
$$\sum_{i=0}^{m-1} \mathbf{E}_{\eps_{k_j}}(u_{k_j},B^+_{r_{i+1}}\setminus B^+_{r_i})\leq M\,,  $$
we can find a good index  $i_m\in \{0,\ldots,m-1\}$ and a (not relabeled)  further subsequence of $\{u_{k_j}\}_{j\in\mathbb{N}}$ such that 
$$\mathbf{E}_{\eps_{k_j}}(u_{k_j},B^+_{r_{i_m+1}}\setminus B^+_{r_{i_m}}) \leq \frac{M+1}{m}\quad \forall j\in\N\,.$$
From the weak convergence of $u_{k_j}$ towards $u_*$ and the lower semicontinuity of ${\bf E}$, we deduce that 
$${\bf E}\big(u_*,B^+_{r_{i_m+1}}\setminus B^+_{r_{i_m}}\big) \leq \frac{M+1}{m}\,.$$
Now consider a smooth cut-off function $\chi\in C_c^\infty(B_1,[0,1])$ such that $\chi=1$ in $B_{r_{i_m}}$, $\chi=0$ in $B_1\setminus B_{r_{i_m+1}}$, and satisfying 
$|\nabla\chi|\leq C\delta_m^{-1}$ for a constant $C$ only depending on~$n$. Then define 
$$w_j:=\chi u_* +(1-\chi)u_{k_j}\,, $$
so that $w_j\in H^1(B_1^+,|z|^a\de\mathbf{x})$ and $w_j-u_{k_j}$ is compactly supported in $B^+_{1}\cup D_{1}$. Since $|w_j-1|\leq \boldsymbol{\delta}_W$ on $D_1$, we infer from Lemma \ref{minimality} that 
$$ \mathbf{F}_{\varepsilon_{k_j}}(u_{k_j},B_1^+)\leq \mathbf{F}_{\varepsilon_{k_j}}(w_j,B_1^+)\,,$$
which leads to 
\begin{multline*}
 \mathbf{E}_{\varepsilon_{k_j}}(u_{k_j},B_{r}^+)\leq {\bf E}(u_*,B^+_{r_\theta}) + \mathbf{E}_{\varepsilon_{k_j}}(w_j,B_{r_{i_{m}+1}}^+\setminus B_{r_{i_{m}}}^+)\\
 +\|f_{k_j}\|_{L^q(D_1)}\|u_{k_j}-1\|_{L^{q/(q-1)}(D_1)}\,.
\end{multline*}
Using the convexity of $W(t)$ near $t=1$, we estimate 
\begin{multline*}
\mathbf{E}_{\varepsilon_{k_j}}(w_j,B_{r_{i_{m}+1}}^+\setminus B_{r_{i_{m}}}^+)\leq  {\bf E}\big(u_*,B_{r_{i_{m}+1}}^+\setminus B_{r_{i_{m}}}^+\big)\\
+\mathbf{E}_{\varepsilon_{k_j}}\big(u_{k_j},B_{r_{i_{m}+1}}^+\setminus B_{r_{i_{m}}}^+\big)
+C\delta_m^{-2}\int_{B_{r_{i_{m}+1}}^+\setminus B_{r_{i_{m}}}^+} z^{a}|u_{k_j}-u_*|^2\,\de \mathbf{x}\,.
\end{multline*}
From the compact embedding $H^1(B_1^+,|z|^a\de\mathbf{x})\hookrightarrow L^1(B_1^+)$ and the fact that $|u_{k_j}|\leq M$ in $B_1^+$, we infer that 
 $u_{k_j}\to u_*$ strongly in $L^2(B_1^+,|z|^a\de\mathbf{x})$.  Consequently,  
$$\limsup_{j\to\infty} \mathbf{E}_{\varepsilon_{k_j}}(w_j,B_{r_{i_{m+1}}}^+\setminus B_{r_{i_{m}}}^+)\leq \frac{2(M+1)}{m}\,.$$
Therefore,
$$ \lim_{j\to\infty}\mathbf{E}_{\varepsilon_{k_j}}(u_{k_j},B_{r}^+)\leq {\bf E}(u_*,B_{r_\theta}^+) + \frac{2(M+1)}{m}\,.$$
Finally, letting first $m\to\infty$ and then $\theta\to 1$, we conclude that 
$$\lim_{j\to+\infty} \mathbf{E}_{\varepsilon_{k_j}}(u_{k_j},B_{r}^+)\leq    {\bf E}(u_*,B_{r}^+)\,.$$
On the other hand, $\liminf_j  \mathbf{E}(u_{k_j},B_{r}^+)\geq {\bf E}(u_*,B_{r}^+)$ by lower semicontinuity, and consequently, 
$$\lim_{j\to\infty} \mathbf{E}(u_{k_j},B_{r}^+)= {\bf E}(u_*,B_{r}^+)\quad\text{and}\quad \lim_{j\to\infty} \frac{1}{\eps_{k_j}^{2s}}\int_{D_r}W(u_{k_j})\,\de x=0\,. $$
From the weak convergence of $u_{k_j}$, it classically follows that the sequence $\{u_{k_j}\}_{j\in\mathbb{N}}$ converges strongly in $H^1(B_r^+,|z|^a\de\mathbf{x})$ towards $u_*$.
\end{proof}

\begin{lemma}\label{eas}
If $u_*\in H^1(B_1^+,|z|^a\de\mathbf{x})\cap L^\infty(B_1^+)$ satisfies 
$$\begin{cases}
{\rm div}\big(z^{a}\nabla u_*)=0 & \text{in $B_1^+$}\,,\\
u_*=1 & \text{on $D_1$}\,,
\end{cases}$$
then $u_*\in C^{0,\alpha}_{\rm loc}(B_1^+\cup D_1)$, $\nabla_xu_*\in C^{0,\alpha}_{\rm loc}(B_1^+\cup D_1)$,  and $z^a\partial_zu_*\in C^{0,\alpha}_{\rm loc}(B_1^+\cup D_1)$ for some $\alpha=\alpha(n,s)\in(0,1)$. Moreover, for every $r\in(0,1)$, $\|u_*\|_{C^{0,\alpha}(B_r^+)}$,  $\|\nabla_x u_*\|_{C^{0,\alpha}(B_r^+)}$, and $\|z^a\partial_zu_*\|_{C^{0,\alpha}(B_r^+)}$ only depends $n$, $s$, $r$, and $\|u_*\|_{L^\infty(B_1^+)}$. In particular, 
\begin{equation}\label{samed1250}
\lim_{r\to0}\frac{1}{r^{n-2s}} {\bf E}\big(u_*,B^+_r({\bf x}_0)\big)=0 
\end{equation}
locally uniformly with respect to $\mathbf{x}_0\in D_1\times\{0\}$.
\end{lemma}

\begin{proof}
Considering $u_*-1$ instead of $u_*$, we can assume that $u_*=0$ on $D_1$. Then we extend $u_*$ to the whole ball $B_1$ by odd symmetry, i.e.,  $u_*(x,z):=-u_*(x,-z)$ for $z<0$. 
Since $u_*=0$  on $D_1$, we have $u_*\in H^1(B_1,|z|^a\de\mathbf{x})\cap L^\infty(B_1)$. 
In addition, $u_*$ solves ${\rm div}(|z|^a\nabla u_*)=0$ in the ball $B_1$ (in the weak sense), i.e., 
$$\int_{B_1}|z|^a\nabla u_*\cdot\nabla\phi\,\de \mathbf{x}=0 $$
for all $\phi\in H^1(B_1,|z|^a\de\mathbf{x})$ compactly supported in $B_1$. 
Standard elliptic regularity yields $u_*\in C^\infty(B_1\setminus D_1)$, and for every compact set $K\subset B_1\setminus D_1$, $\|\nabla u_*\|_{L^\infty(K)}$ only depends on $n$, $s$, $K$, and $\|u_*\|_{L^\infty(B_1^+)}$.  
Then the regularity result in \cite{FKS} (see also \cite[Section~3.2]{CS1}) tells us that $u_*\in C_{\rm loc}^{0,\alpha}(B_1)$ for some exponent $\alpha\in (0,1)$ depending only $n$ and $s$. And for $r\in(0,1)$, $\|u_*\|_{C^{0,\alpha}(B_r)}$ only depends on $n$, $s$, $r$, and $\|u_*\|_{L^\infty(B_1^+)}$. 
 By the argument used in the proof of Theorem \ref{regint} (based on finite difference quotients), we show that $\nabla_xu_*\in C^{0,\alpha}_{\rm loc}(B_1)$, and $\|\nabla_xu_*\|_{C^{0,\alpha}(B_r)}$ only depends on $n$, $s$, $r\in(0,1)$, and $\|u_*\|_{L^\infty(B_1^+)}$. 
 
Let us now fix a radius $r\in(0,1)$ and an index $j\in\{1,\ldots,n\}$. We set for $\delta\in(0,1-r)$, 
 $$w_\delta(x,z):=\frac{u_*(x+\delta e_j,z)-u_*(x,z)}{\delta}\,. $$
The function $w_\delta$ belongs to $H^1(B_r,|z|^a\de{\bf x})\cap L^\infty(B_r)$, and it satisfies (in the weak sense)  
$${\rm div}(|z|^a\nabla w_\delta)=0\quad\text{in $B_r$}\,.$$ 
Consider a cut-off $\chi\in C^1_c(B_r)$ such that $\chi\equiv 1$ in $B_{r-\tau}$ for some $\tau\in(0,r)$. Using the test function $\phi=\chi^2w_\delta$, we obtain 
 $$0=\int_{B_r}|z|^a\nabla w_\delta\cdot\nabla\phi\,\de \mathbf{x}= \int_{B_r}|z|^a\chi^2|\nabla w_\delta|^2\,\de \mathbf{x}+ 2\int_{B_r}|z|^a(\chi\nabla w_\delta)\cdot(w_\delta\nabla\chi)\,\de \mathbf{x}\,.$$
From Cauchy-Schwarz Inequality  we  infer that 
 $$ \int_{B_r}|z|^a\chi^2|\nabla w_\delta|^2\,\de \mathbf{x}\leq 4 \int_{B_r}|z|^aw^2_\delta|\nabla\chi|^2\,\de \mathbf{x}\leq C\,,$$
for a constant $C$ independent of $\delta$. Letting $\delta\to 0$, we obtain by lower semicontinuity that 
$$\int_{B_{r-\tau}} |z|^a|\nabla(\partial_j u_*)|^2\,\de \mathbf{x}\leq C\,. $$
In view of the arbitrariness of $\tau$ and $r$, we conclude that $\partial_ju_*\in H^1_{{\rm loc}}(B_1,|z|^a\de\mathbf{x})\cap L^\infty_{\rm loc}(B_1)$.  In addition, $\partial_j u_*$ satisfies 
${\rm div}(|z|^a\nabla (\partial_ju_*))=0$ in $B_1$ (in the weak sense). By the regularity results in \cite{FKS} and the consideration above, we  infer that  
$\nabla_{x}(\partial_j u_*)\in C_{\rm loc}^{0,\alpha}(B_1)$,  and $\|\nabla_x(\partial_ju_*)\|_{C^{0,\alpha}(B_r)}$ only depends on $n$, $s$, $r\in(0,1)$, and $\|u_*\|_{L^\infty(B_1^+)}$  (since $\|\partial_ju_*\|_{L^\infty(B_r)}$ only depends on $n$, $s$, $r$, and $\|u_*\|_{L^\infty(B_1^+)}$). 

From the arbitrariness of $j$, we conclude that  $\Delta_xu_*\in C^{0,\alpha}_{\rm loc}(B_1)$, and  
$\|\Delta_xu_*\|_{C^{0,\alpha}(B_r)} $ only depends on $n$, $s$, $r\in(0,1)$, and $\|u_*\|_{L^\infty(B_1^+)}$. On the other hand, 
$$\partial_z\big(z^{a}\partial_{z}u_*\big)= z^{a}\Delta_x u_*\quad\text{in $B_1^+$}\,.$$
Consequently, given $r\in(0,1)$ and writing 
$$z^a\partial_zu_*(x,z)=r^a\partial_zu_*(x,r)-\int_z^rt^a\Delta_x u_*(x,t)\,\de t $$
for $(x,z)\in B_1^+$ such that $(x,r)\in B_1^+$, we  deduce that $z^{a}\partial_{z}u_*$ is actually H\"older continuous up to $D_1$ for some exponent $\tilde\alpha=\tilde\alpha(n,s)\in(0,1)$ (perhaps smaller then $\alpha$), and $\|z^a\partial_zu_*\|_{C^{0,\tilde \alpha}(B^+_r)}$ only depends on $n$, $s$, $r$, and $\|u\|_{L^\infty(B_1^+)}$. 

Finally, if $\mathbf{x}_0\in D_R\times\{0\}$ for some $R\in(0,1)$,   we now have for $0<r<1/2(1-|\mathbf{x}_0|)$ the estimate $z^{a}|\nabla u_*|\leq C_R$ in $B^+_r(\mathbf{x}_0)$ with a constant $C_R$ independent of $\mathbf{x}_0$ and $r$. Hence,
$$\int_{B^+_r(\mathbf{x}_0)}z^{a}|\nabla u_*|^2\,\de\mathbf{x}\leq C_R\int_{B^+_r(\mathbf{x}_0)}z^{-a}\,\de{\bf x}\leq C_Rr^{n+2s} \,,$$
and \eqref{samed1250} follows.
\end{proof}

Combining Lemma \ref{clear2} with Corollary \ref{strongcoro} leads to the following

\begin{proposition}\label{epsreg}
Let $q\in(\frac{n}{1+2s},n)$, $b\geq 1$, $T>0$, and $\varepsilon_k\downarrow0$ a given sequence. Let $R\in(0,1]$ and $\{f_k\}_{k\in\mathbb{N}}\subset  C^{0,1}(D_R)$ such that 
\begin{equation}\label{bdfw1qdiscR}
\eps_k^{2s}\|f_k\|_{L^\infty(D_R)} +\|f_k\|_{\dot W^{1,q}(D_R)}\leq T\,.
\end{equation}
There exist two constants  
$\boldsymbol{\theta}_{b,T}>0$ and ${\bf R}_{b,T}>0$ (depending only on  $n$, $s$, $q$, $b$, $T$, and $W$) such that the following holds. 
Let $\{u_k\}_{k\in\mathbb{N}} \subset H^1(B_R^+,|z|^a\de\mathbf{x})\cap L^\infty(B_R^+)$ be such that $\|u_k\|_{L^\infty(B_R^+)}\leq b$, and $u_k$ solves \eqref{tim1439} in the weak sense. If $R\leq {\bf R}_{b,T}$ and 
\begin{equation}\label{condliminf} 
\liminf_{k\to\infty} {\bf E}_{\eps_k}(u_k,B_R^+)< \boldsymbol{\theta}_{b,T}R^{n-2s} \,,
\end{equation}
then there exist a (not relabeled) subsequence and $u_*\in H^1(B_R^+,|z|^a\de\mathbf{x})$ satisfying either $u_*=1$ on $D_{R/4}$, or $u_*=-1$ on $D_{R/4}$, such that 
\vskip3pt
\begin{enumerate}
\item[\rm (i)] $u_k\to u_*$ strongly in $H^1(B_{R/4}^+,|z|^a\de\mathbf{x})\,$; 
\vskip3pt
\item[\rm (ii)] $u_k\to u_*$  uniformly on $D_{R/4}\,$;
\vskip3pt
\item[\rm (iii)] $\eps_k^{-2s}\int_{D_{R/4}} W(u_k)\,\de x \to 0\,$. 
 \end{enumerate} 
\end{proposition}

\begin{proof} 
Let $\boldsymbol{\theta}_{b,T}:=\frac{1}{2}{\boldsymbol{\eta}}_{b,T}(\boldsymbol{\delta}_W)$ where the constant $\boldsymbol{\delta}_W$ is given by \eqref{quantconvex}, and ${\boldsymbol{\eta}}_{b,T}$ given by Lemma \ref{clear2}. Then we choose 
$${\bf R}_{b,T}:=\min\left\{1,\left(\frac{\theta_q{\boldsymbol{\eta}}_{b,T}(\boldsymbol{\delta}_W)}{2b\,{\bf c}_{n,q}T}\right)^{1/\theta_q}\right\}\,.$$
If $R\leq {\bf R}_{b,T}$, then the a priori bound \eqref{bdfw1qdiscR} yields 
$${\bf c}_{n,q}\|u_k\|_{L^\infty(B_R^+)}\int_{0}^{R}t^{\theta_q-1}\|f_k\|_{\dot W^{1,q}(D_t(x_j))} \,\de t\leq\frac{1}{2}{\boldsymbol{\eta}}_{b,T}(\boldsymbol{\delta}_W) \,,$$
so that 
\begin{equation}\label{1420bientTD}
 \liminf_{k\to\infty}\boldsymbol{\Theta}^{\eps_k}_{u_k}(f_k,0,R) <  {\boldsymbol{\eta}}_{b,T}(\boldsymbol{\delta}_W)\,.
 \end{equation}
Select a (not relabeled) subsequence which achieves the $\liminf$ in \eqref{1420bientTD}. 
By the uniform energy bound, we can find a (not relabeled) subsequence such that $u_k\rightharpoonup u_*$ weakly in $H^1(B_R^+,|z|^a\de\mathbf{x})$.  From the compact embedding $H^1(B_R^+,|z|^a\de\mathbf{x})\hookrightarrow L^1(B_R^+)$, we deduce that $|u_*|\leq b$ in $B_R^+$. Since $ \boldsymbol{\Theta}^{\eps_k}_{u_k}(f_k,0,R) \leq  \boldsymbol{\theta}_{b,T}$ for $k$ sufficiently large, Lemma~\ref{clear2} shows that $\big||u_k|-1\big|\leq \boldsymbol{\delta}_W$ on $D_{R/2}$ for such $k$'s. Extracting another subsequence if necessary, we can assume without loss of generality that  $|u_k-1|\leq \boldsymbol{\delta}_W$ on the disc $D_{R/2}$. Then Corollary~\ref{strongcoro} yields  $u_*=1$ on $D_{R/2}$, $u_k\to u_*$ strongly in $H^1(B^+_{3R/8},|z|^a\de\mathbf{x})$,  and 
\begin{equation}\label{prevanishpot}
\frac{1}{\eps_k^{2s}}\int_{D_{3R/8}}W(u_k)\,\de x\to 0\,.
\end{equation}
Now fix $\delta\in (0, \boldsymbol{\delta}_W)$ arbitrary. By Lemma \ref{eas}, we can find a radius $r_\delta\leq R/8$ such that 
$${\bf E}\big(u_*,B^+_{r_\delta}(\bar{\mathbf{x}})\big) \leq \frac{{\boldsymbol{\eta}}_{b,T}(\delta)}{3}\,r_\delta^{n-2s}$$
for every $\bar{\mathbf{x}}\in D_{R/4}\times\{0\}$. Then consider a finite covering of $\overline{D}_{R/4}\times\{0\}$ by balls of radius $r_\delta/2$ centered at points of $\overline{D}_{R/4}\times\{0\}$. We denote by $\mathbf{x}_1=(x_1,0),\ldots, \mathbf{x}_L=(x_L,0)$ the centers of these balls. 
From the strong convergence of $\{u_k\}_{k\in\mathbb{N}}$ and \eqref{prevanishpot}, we deduce that for $k$ large enough, 
$$\frac{1}{r_\delta^{n-2s}} \mathbf{E}_{\varepsilon_k}(u_k,B^+_{r_\delta}(\mathbf{x}_j))\leq \frac{{\boldsymbol{\eta}}_{b,T}(\delta)}{2}\quad\forall j\in\{1,\ldots,L\}\,.$$
On the other hand, 
$${\bf c}_{n,q}\|u_k\|_{L^\infty(B_R^+)}\int_{0}^{r_\delta}t^{\theta_q-1}\|f_k\|_{\dot W^{1,q}(D_t(x_j))} \,\de t\leq \frac{b\,{\bf c}_{n,q}}{\theta_q}Tr_\delta^{\theta_q} \,.$$
Hence, choosing a smaller value for $r_\delta$ if necessary, we have 
$$  \boldsymbol{\Theta}^{\eps_k}_{u_k}(f_k,x_j,r_\delta) \leq {\boldsymbol{\eta}}_{b,T}(\delta)\quad\forall j\in\{1,\ldots,L\}\,.$$
Then Lemma \ref{clear2} shows that $|u_k-1|\leq \delta$ in $D_{r_{\delta}/2}(x_j)$ for every $j=1,\ldots,L$. 
Hence $|u_k-1|\leq \delta$ in $D_{R/4}$ whenever $k$ is large enough. 
\end{proof}

We now improve the previous convergence result under stronger assumptions on the sequence $\{f_k\}_{k\in\mathbb{N}}$. 

\begin{proposition}\label{imprcv}
In addition to the conclusions of Proposition \ref{epsreg}, 
\begin{enumerate}
\item[\rm (i)] if $\sup_k\|f_k\|_{L^\infty(D_R)}<\infty$, then $u_k\to u_*$ in $C^{0,\alpha}(D_{R/16})$ for every $\alpha\in(0,\boldsymbol{\beta}_*)$;
\vskip3pt

\item[\rm (ii)] if $\sup_k\|f_k\|_{C^{0,1}(D_R)}<\infty$, then $u_k\to u_*$ in $C^{1,\alpha}(D_{R/32})$ for every $\alpha\in(0,\boldsymbol{\beta}_*)$;
\end{enumerate}
where $\boldsymbol{\beta}_*$  is given by Lemma \ref{prereg}
\end{proposition}

\begin{proof}
{\it Step 1.} We start proving item (i). Assume that $u_*=1$ on $D_{R/4}$. By Proposition~\ref{epsreg}, we have for $k$ large enough, $|u_k-1|\leq \boldsymbol{\delta}_W$ on $D_{R/4}$. We shall prove that 
\begin{equation}\label{speedu}
\|u_k-1\|_{L^\infty(D_{R/8})}\leq C\eps_k^{2s} \,,
\end{equation}
for some constant $C$ independent of $\eps_k$. Note that the conclusion follows from this estimate. Indeed, if holds \eqref{speedu}, then the $C^2$-assumption on $W$ implies that 
$$\big\|W^\prime(u_k)\big\|_{L^\infty(D_{R/8})}\leq C\eps_k^{2s}\,, $$
and we can thus apply Lemma \ref{prereg} to infer that $u_k$ is bounded in $C^{0,\boldsymbol{\beta}_*}(B^+_{R/16})$. 

To prove \eqref{speedu} we proceed as follows. Fix an arbitrary parameter $\eta\in(0,1)$, and consider the nonnegative smooth convex function
$$\psi_\eta(t):=\sqrt{t^2+\eta^2}-\eta\,. $$  
Set $\upsilon_\eta:=\psi_\eta(u_k-1)\in H^1(B^+_{R/4},|z|^a\de{\bf x})\cap L^\infty(B^+_{R/4})$, and we observe that $\upsilon_\eta$ satisfies in the weak sense
$$ \begin{cases}
{\rm div}(z^{a}\nabla \upsilon_\eta)= z^a\psi^{\prime\prime}(u_k-1)|\nabla u_k|^2 & \text{in $B_{R/4}^+$}\,,\\[5pt]
\displaystyle d_s\boldsymbol{\partial}_z^{(2s)} \upsilon_\eta=\frac{\psi^{\prime}(u_k-1)}{\varepsilon_k^{2s}}W^\prime(u_k) -\psi^{\prime}(u_k-1)f_k & \text{on $D_{R/4}$}\,.
\end{cases}$$
On the other hand, \eqref{quantconvex} implies that 
$$(t-1)W^\prime(t)\geq \boldsymbol{\kappa}_W (t-1)^2\quad \text{for $|t-1|\leq \boldsymbol{\delta}_W$}\,,$$
where $\boldsymbol{\kappa}_W:=\frac{1}{2}\min\big\{W^{\prime\prime}(1),W^{\prime\prime}(-1)\big\}>0$. Noticing that $t\psi^{\prime}(t)\geq \psi(t)$ for every $t\in\R$, we thus have 
$$\psi^{\prime}(t-1)W^\prime(t)=\frac{(t-1)\psi^{\prime}(t-1)}{(t-1)^2}\,(t-1)W^{\prime}(t)\geq\boldsymbol{\kappa}_W\psi(t-1) \quad \text{for $|t-1|\leq \boldsymbol{\delta}_W$}\,.$$
Therefore $\upsilon_\eta$ satisfies 
$$ \begin{cases}
{\rm div}(z^{a}\nabla \upsilon_\eta)\geq 0 & \text{in $B_{R/4}^+$}\,,\\[5pt]
\displaystyle d_s\boldsymbol{\partial}_z^{(2s)} \upsilon_\eta\geq \frac{\boldsymbol{\kappa}_W}{\varepsilon_k^{2s}}\,\upsilon_\eta -\|f_k\|_{L^\infty(D_R)} & \text{on $D_{R/4}$}\,.
\end{cases}$$
By \cite[Lemma 3.5]{TVZ} it implies 
$$\|\upsilon_\eta\|_{L^\infty(D_{R/8})}\leq \frac{(1+\|f_k\|_{L^\infty(D_R)})\eps_k^{2s}}{\boldsymbol{\kappa}_W} \sqrt{(1+b)^2+\eta^2}\,. $$
Letting $\eta\to 0$, we deduce that \eqref{speedu} holds with $C=\boldsymbol{\kappa}_W^{-1}(1+b)(1+\sup_k\|f_k\|_{L^\infty(D_R)})$. 
\vskip3pt

\noindent{\it Step 2.} To prove the $C^{1,\alpha}$-convergence, we shall rely on the regularity argument developed in the proof of Theorem \ref{regint} (that we partially reproduce for clarity reason). To simplify the notation, we assume here (without loss of generality) that $R=32$. Fix an arbitrary point $x_0\in \overline D_{1}$, and for $\mathbf{x}=(x,z)\in B^+_{1}\cup D_{1}$ consider the translated function  
$\bar u_k(\mathbf{x}):=u_k(x+x_0,z)$.  For  $h\in D_{1/8}$, $h\not=0$, we set for $\mathbf{x}\in B^+_{7/8}\cup D_{7/8}$, 
$$w_h(\mathbf{x}):=\frac{\bar u_k(x+h,z)-\bar u_k(\mathbf{x})}{|h|^{\boldsymbol{\beta}_*}} \,.$$
By Step 1, we have $\|w_h\|_{L^\infty}(B^+_{7/8})\leq C$ for a constant $C$ independent of $h$ and $\eps_k$. 
Given $\eta\in(0,1)$, we can argue as in Step 1 to infer that the function  $\zeta_\eta:=\psi_\eta(w_h)\in H^1(B^+_{7/8},|z|^a\de{\bf x})\cap L^\infty(B^+_{7/8})$ satisfies 
$$ \begin{cases}
{\rm div}(z^{a}\nabla \zeta_\eta)\geq 0 & \text{in $B_{7/8}^+$}\,,\\[5pt]
\displaystyle d_s\boldsymbol{\partial}_z^{(2s)} \zeta_\eta\geq \frac{\boldsymbol{\kappa}_W}{\varepsilon_k^{2s}}\,\zeta_\eta -\|f_k\|_{C^{0,\boldsymbol{\beta}_*}(D_1)} & \text{on $D_{7/8}$}\,.
\end{cases}$$
Then \cite[Lemma 3.5]{TVZ} yields $\|w_h\|_{L^\infty(D_{7/16})}\leq C\eps_k^{2s}$ once we let $\eta\to 0$, for a constant $C$ independent of $h$ and $\eps_k$. From the equation satisfied by $w_h$, it implies through Lemma~\ref{prereg} that $w_h$ is bounded in $C^{0,\boldsymbol{\beta}_*}(B^+_{7/32})$ independently of $h$ and $\eps_k$.  As a consequence, 
$$\sup_{x\in \overline{D}_{1/16}}  \big|\bar u_k(x+h,z)-2\bar u_k(x,z)+\bar u_k(x-h,z)\big|\leq C|h|^{2{\boldsymbol{\beta}_*}}$$
for every $h\in \overline D_{1/16}$, $z\in[0,1/16]$, and a constant $C$ independent of $h$ and $\eps_k$. At this stage, we can reproduce the iteration scheme of Theorem  \ref{regint} by means of the above argument (relying on  \cite[Lemma 3.5]{TVZ}) to conclude that $\nabla_x u_k$ is bounded in $C^{0,\boldsymbol{\beta}_*}$ in a (uniform in size) neighborhood of $(x_0,0)$. 
\end{proof}

Note that (for later use) the proof above leads to the following estimate on the potential for a right hand side $f$ which is bounded. 

\begin{lemma}\label{potbddfint}
Let $R>0$, $f\in L^{\infty}(D_R)$, and let $u_\varepsilon\in H^1(B^+_R,|z|^a\de\mathbf{x})\cap L^\infty(D_R)$ be a weak solution of \eqref{pfff}.  
If $\big||u_\varepsilon| - 1\big|\leq \boldsymbol{\delta}_W$ on $D_{R}$, then 
$$W(u_\eps)\leq C_W(1+\|f\|_{L^\infty(D_R)})^2(1+\|u_\eps\|_{L^\infty(B_R^+)})^2 \frac{\eps^{4s}}{R^{4s}} \quad\text{on $D_{R/2}$}\,,$$
and 
$$\big|W^\prime(u_\eps)\big|\leq C_W(1+\|f\|_{L^\infty(D_R)})(1+\|u_\eps\|_{L^\infty(B_R^+)}) \frac{\eps^{2s}}{R^{2s}} \quad\text{on $D_{R/2}$}\,,$$
for a constant $C_W>0$ depending only on  the potential $W$. 
\end{lemma}

\begin{proof}
By rescaling equation \eqref{pfff}, it is enough to consider the case $R=1$. Then, observe that $u_\eps \in C^0(B_1^+\cup D_1)$ by Remark \ref{remclear}. Hence, either 
$|u_\varepsilon - 1 |\leq \boldsymbol{\delta}_W$ or $|u_\varepsilon+1|\leq \boldsymbol{\delta}_W$ on the disc $D_{1}$. Without loss of generality, we may assume that the first case occurs. Then the proof of Proposition \ref{imprcv} (Step 1) shows that 
$$|u_\eps-1|\leq  \frac{1}{\boldsymbol{\kappa}_W} (1+\|f\|_{L^\infty(D_1)})(1+\|u_\eps\|_{L^\infty(B_1^+)}) \eps^{2s}\quad\text{on $D_{1/2}$} \,.$$
Expanding $W$ near $t=1$ yields the announced result. 
\end{proof}

%%%%%%%%%%%%%%%%%%%%%%%%%%%%%%%%%%%%%%%%%%%%%%%%%%%%%%
  
\subsection{Proof of Theorem \ref{asymptneum}}

 %%%%%%%%%%%%%%%%%%%%%%%%%%%%%%%%%%%%%%%%%%%%%%%%%%%%%%%

We are now ready to give the proof of Theorem \ref{asymptneum}.

\begin{proof}  
{\it Step 1: Compactness.} Let $b\geq 1$ such that  $b\geq \sup_k\|u_k\|_{L^\infty(G)}$. By the assumptions on $\{u_k\}_{k\in\mathbb{N}}$, we have 
$${\sup_k}\,\mathbf{E}_{\varepsilon_k}(u_k,G)\leq {\sup_k}\big(\mathbf{F}_{\varepsilon_k}(u_k,G)+ b\|f_k\|_{L^1(\partial^0G)}\big)<\infty\,.$$
Hence there is a (not relabeled) subsequence  such that 
$u_k\rightharpoonup u_*$ weakly in $H^1(G,|z|^a\de \mathbf{x})$. By the compact embedding  $H^1(G,|z|^a\de\mathbf{x})\hookrightarrow L^1(G)$, we also have $u_k\to u_*$ strongly in $L^1(G)$. 
Since $|u_k|\leq b$, it implies that $|u_*|\leq b$ in $G$, and  $u_k\to u_*$ strongly in $L^2(G,|z|^a\de\mathbf{x})$. Moreover, 
by equation \eqref{equinG} and standard elliptic regularity, $u_k\to u_*$ in $ C^\ell_{\rm loc}(G)$ for all $\ell\in\mathbb{N}$, so that ${\rm div}\big(z^{a}\nabla u_*)= 0$  in $G$.  On the other hand, the uniform energy bound implies $|u_k|\to 1$ in $L^1(\partial^0G)$, and we infer from the continuity of the trace operator that $|u_*|=1$  on $\partial^0G$. 
\vskip3pt

We now wish to analyse the asymptotic behavior of $u_k$ near $\partial^0G$. For this we consider the measures  
$$\mu_k:= \frac{d_s}{2}z^{a}|\nabla u_k|^2\mathscr{L}^{n+1}\LL G +\frac{1}{\varepsilon^{2s}_k}W(u_k)\mathscr{H}^n \LL\partial^0G\,.$$ 
Since $\sup_k\mu_k(G\cup\partial^0G)<\infty$, we can find a further subsequence such that 
\begin{equation}\label{tim2253}
\mu_k\rightharpoonup \mu:=\frac{d_s}{2}z^{a}|\nabla u_*|^2\mathscr{L}^{n+1}\LL G+\mu_{\rm sing}\,,
\end{equation}
weakly* as Radon measures on $G\cup\partial^0G$ for some finite nonnegative measure $\mu_{\rm sing}$. 
Notice that the local smooth convergence of $u_k$ to $u_*$ in $G$ implies that 
\begin{equation}\label{tim1751}
{\rm spt}(\mu_{\rm sing})\subset \partial^0G
\end{equation}   
(here ${\rm spt}(\mu_{\rm sing})$ denotes the relative support of $\mu_{\rm sing}$ in $G\cup\partial^0G$). 

Since $\partial^0G$ is a Lipschitz domain of $\mathbb{R}^n$, there exits a constant $C$ depending only on $\partial^0G$ such that $\|f_k\|_{\dot W^{1,q}(\partial^0G)}\leq C \|f_k\|_{W^{1,q}(\partial^0G)}$. 
Then we set 
$$T:=\sup_k \bigg((2\eps_k)^{2s}\|f_k\|_{L^\infty(\partial^0G)}+ \|f_k\|_{\dot W^{1,q}(\partial^0G)}\bigg)<\infty\,.$$
Noticing that  
$$\int_\rho^r t^{\theta_q-1}\|f_k\|_{\dot W^{1,q}(D_t(x))}\,\de t\leq \frac{T}{\theta_q}(r^{\theta_q}-\rho^{\theta_q})\,, $$
we can apply Lemma \ref{monotform} to deduce that  
\begin{equation}\label{tim2332}
\rho^{2s-n}\mu_k(B_\rho(\mathbf{x}))+\frac{b\,\mathbf{c}_{n,q}}{\theta_q}T\rho^{\theta_q}\leq r^{2s-n}\mu_k(B_r(\mathbf{x}))+\frac{b\,\mathbf{c}_{n,q}}{\theta_q}Tr^{\theta_q}
\end{equation}
for every $\mathbf{x}\in\partial^0G$ and every $0<\rho<r<\min\big(1,{\rm dist}(\mathbf{x},\partial^+G)\big)$. Therefore,
\begin{equation}\label{tim1213}
\rho^{2s-n}\mu(B_\rho(\mathbf{x}))+\frac{b\,\mathbf{c}_{n,q}}{\theta_q}T\rho^{\theta_q}\leq r^{2s-n}\mu(B_r(\mathbf{x}))+\frac{b\,\mathbf{c}_{n,q}}{\theta_q}Tr^{\theta_q}
\end{equation}
for every $\mathbf{x}\in\partial^0G$ and every $0<\rho<r<\min\big(1,{\rm dist}(\mathbf{x},\partial^+G)\big)$. As a consequence, the $(n-2s)$-dimensional density 
\begin{equation}\label{existdens}
\Theta^{n-2s}(\mu,\mathbf{x}):=\lim_{r\downarrow 0}\, \frac{\mu(B_r(\mathbf{x}))}{{\omega}_{n-2s}r^{n-2s}}
\end{equation}
exists\footnote{Here we have set $\displaystyle\omega_{n-2s}:=\frac{\pi^{\frac{n-2s}{2}}}{\Gamma(1+\frac{n-2s}{2})}$.} 
and is finite at every point $\mathbf{x}\in\partial^0G$. 
Note that \eqref{tim2253} and \eqref{tim2332} yield
\begin{equation}\label{upbddensity}
\Theta^{n-2s}(\mu,\mathbf{x})\leq \frac{C}{\big({\rm dist}(\mathbf{x},\partial^+G)\big)^{n-2s}} \,\sup_{k} \mathbf{E}_{\eps_k}(u_k,G)+\frac{b\,\mathbf{c}_{n,q}}{\theta_q}T({\rm diam}\,\partial^0 G)^{\theta_q}<\infty 
\end{equation}
for all $\mathbf{x}\in\partial^0G$. On the other hand, by the smooth convergence of $u_k$ toward $u_*$ in $G$, 
$$\Theta^{n-2s}(\mu,\mathbf{x})=0\quad\text{for all $x\in G$}\,. $$
In addition, we observe that $\mathbf{x}\in\partial^0G\mapsto \Theta^{n-2s}(\mu,\mathbf{x})$ is upper semicontinuous \footnote{Indeed, assume that ${\bf x}_j\to {\bf x}\in \partial^0G$, and choose a sequence $r_m\downarrow0$ such that $\mu(\partial B_{r_m}({\bf x}))=0$. By  \eqref{tim1213}, we have $\limsup_j\Theta^{n-2s}(\mu,\mathbf{x}_j)\leq \omega^{-1}_{n-2s}r_m^{n-2s}\mu(B_{r_m}({\bf x}))+Cr_m^{\theta_q}$, and the conclusion follows letting $r_m\to0$.}.
\vskip3pt

Next we define the concentration set
\begin{multline}\label{defconcentrset}
\Sigma:=\bigg\{\mathbf{x}\in \partial^0G :  \inf_r\Big\{ \liminf_{k\to\infty}\, r^{2s-n}\mu_k(B_r(\mathbf{x})) : \\
0<r<\min\big(1,{\rm dist}(\mathbf{x},\partial^+G)\big)\Big\}\geq\boldsymbol{\theta}_{b,T}\bigg\}\,, 
\end{multline}
where $\boldsymbol{\theta}_{b,T}>0$ is the constant given by Proposition~\ref{epsreg}. From \eqref{tim2332} and \eqref{tim1213} we infer that 
\begin{multline*}
\Sigma =\bigg\{\mathbf{x}\in \partial^0G :  \lim_{r\downarrow 0} \,\liminf_{k\to\infty}\, r^{2s-n}\mu_k(B_r(\mathbf{x})) \geq\boldsymbol{\theta}_{b,T}\bigg\} \\
 = \bigg\{\mathbf{x}\in \partial^0G :  \lim_{r\downarrow 0}\, r^{2s-n}\mu(B_r(\mathbf{x})) \geq \boldsymbol{\theta}_{b,T}\bigg\} \,,$$
\end{multline*} 
and consequently, 
\begin{equation}\label{tim1721}
\Sigma=\bigg\{\mathbf{x}\in \partial^0G: \Theta^{n-2s}(\mu,\mathbf{x})\geq \frac{\boldsymbol{\theta}_{b,T}}{{\omega}_{n-2s}} \bigg\} \,. 
\end{equation}
In particular, $\Sigma$ is a relatively closed subset of $\partial^0G$ since $\Theta^{n-2s}(\mu,\cdot)$ is upper semicontinuous. 
Moreover, by a well known property of  densities (see e.g. \cite[Theorem~2.56]{AFP}), we have
\begin{equation}\label{train}
 \frac{\boldsymbol{\theta}_{b,T}}{{\omega}_{n-2s}}\mathscr{H}^{n-2s}(\Sigma)  \leq \mu(\Sigma)<\infty\,.
 \end{equation}
 On the other hand, it follows from \eqref{upbddensity} and \cite[Theorem~2.56]{AFP} that $\mu_{\rm sing}\LL \Sigma$ is absolutely continuous with respect to $\mathscr{H}^{n-2s}\LL\Sigma$. 

We now claim that ${\rm spt}(\mu_{\rm sing})\subset\Sigma$. Indeed, for $\mathbf{x}_0\in \partial^0G\setminus\Sigma$, we can find a radius 
$$0<r <\min\Big\{\mathbf{R}_{b,T},{\rm dist}(\mathbf{x}_0,\partial^+G\cup\Sigma)\Big\}$$ 
(with $\mathbf{R}_{b,T}$ given by Proposition~\ref{epsreg})
 such  that $r^{2s-n}\mu(B_r(\mathbf{x}_0))< \boldsymbol{\theta}_{b,T}$ and $\mu(\partial B_r(\mathbf{x}_0))=0$. Then 
$$\lim_{k\to\infty} \mathbf{E}_{\eps_k}(u_k,B^+_r(\mathbf{x}_0))=\mu(B_r(\mathbf{x}_0))< \boldsymbol{\theta}_{b,T}r^{n-2s}\,,$$
and we deduce from Proposition \ref{epsreg} that $\mu_{\rm sing}(B_{r/4}(\mathbf{x}_0))=0$. Hence 
$$\mu_{\rm sing}(\partial^0G\setminus \Sigma)=0\,,$$ 
and thus $\mu_{\rm sing}$ is supported by $\Sigma$. In conclusion, we thus proved that $\mu_{\rm sing}$ is absolutely continuous with respect to the Radon measure $\mathscr{H}^{n-2s}\LL\Sigma$. 

We are now ready to show that $\mu_{\rm sing}\equiv 0$. We argue by contradiction assuming that $\mu_{\rm sing}(\Sigma)>0$. By \cite[Corollary 3.2.3]{Zi}, we can find a Borel subset $\widetilde \Sigma\subset\Sigma$ such that $\mathscr{H}^{n-2s}(\Sigma\setminus\widetilde\Sigma)=0$ and 
$$\lim_{r\downarrow 0} \frac{1}{r^{n-2s}}{\bf E}\big(u_*, B^+_r(\mathbf{x}_0)\big)=0\quad\text{for every $\mathbf{x}_0\in\widetilde\Sigma$}\,.$$
Then $\mu_{\rm sing}(\widetilde\Sigma)=\mu_{\rm sing}(\Sigma)>0$. 
Moreover, by our choice of $\widetilde\Sigma$, the density 
$$\Theta^{n-2s}(\mu_{\rm sing},\mathbf{x}_0):=\lim_{r\downarrow 0}\, \frac{\mu_{\rm sing}(B_r(\mathbf{x}_0))}{{\omega}_{n-2s}r^{n-2s}}$$ 
exists at every $\mathbf{x}_0\in \widetilde\Sigma$, and 
$$\Theta^{n-2s}(\mu_{\rm sing},\mathbf{x}_0)=\Theta^{n-2s}(\mu,\mathbf{x}_0)\in(0,\infty)\,. $$
By Marstrand's Theorem (see e.g. \cite[Theorem 14.10]{Matti}), it implies that $(n-2s)$ is an integer, which is an obvious contradiction. Hence $\mu_{\rm sing}\equiv 0$. 

Note that \eqref{train}  now yields $\mathscr{H}^{n-2s}(\Sigma)=0$. 
Moreover, we infer from \eqref{tim2253} that for every admissible open set  $G'$ such that  $\overline{G'}\subset G\cup\partial^0G$, 
$${\bf E}(u_*,G^\prime) 
\leq \liminf_{k\to\infty} {\bf E}(u_k,G^\prime) \\
\leq \lim_{k\to\infty}\mathbf{E}_{\eps_k}(u_k,G')= {\bf E}(u_*,G^\prime)\,.$$
Therefore $u_k\to u_*$ strongly in $H^1_{{\rm loc}}(G\cup\partial^0G,|z|^a\de\mathbf{x})$, and $\eps_k^{-2s}W(u_k)\to 0$ in $L^1_{\rm loc}(\partial^0G)$. 
\vskip3pt

\noindent{\it Step 2: Uniform convergence.} Let us define
$$E^+:=\Big\{\mathbf{x}=(x,0)\in\partial^0G : \text{$u_*=1$  a.e. on $D_r(x)$ for some $r\in(0,{\rm dist}(\mathbf{x},\partial^+G))$}\Big \} \,,$$
and 
$$E^-:=\Big\{\mathbf{x}=(x,0)\in\partial^0G : \text{$u_*=-1$   a.e. on $D_r(x)$ for some $r\in(0,{\rm dist}(\mathbf{x},\partial^+G))$}\Big \} \,.$$
By construction, $E^+$ and $E^-$ are disjoint relatively open subsets of $\partial^0G$. 

We claim that $E^\pm\cap \Sigma=\emptyset$. Indeed, assume for instance that  $\mathbf{x}_0=(x_0,0)\in E^+$. Then we can find $r>0$ such that $u_*=1$   on $D_r(x_0)$.  By Lemma \ref{eas} we have 
$$\Theta^{n-2s}(\mu,\mathbf{x}_0)=\lim_{\rho\to 0} \frac{1}{\rho^{n-2s}}{\bf E}\big(u_*,  B^+_\rho(\mathbf{x}_0)\big)=0\,,$$
whence $\mathbf{x}_0\not \in \Sigma$. 

Next we claim that $\partial^0G=E^+\cup\Sigma\cup E^-$. Indeed, if $\mathbf{x}_0=(x_0,0)\in\partial^0G\setminus \Sigma$, then we can find a radius $r>0$ such that  
$\lim_k \mathbf{E}_{\eps_k}(u_k,B^+_{r}(\mathbf{x}_0))< \boldsymbol{\theta}_{b,T}r^{n-2s}$. By Proposition~\ref{epsreg}, either $u_k\to 1$ or $u_k\to -1$ uniformly in $D_{r/4}(x_0)$. Therefore, either $u_*=1$ or $u_*=-1$  on  $D_{r/4}(x_0)$. Hence $\mathbf{x}_0\in E^+\cup E^-$. 

Since $\mathscr{L}^n(\Sigma)=0$, it implies in particular that 
$$u_*=\chi_{E^+}-\chi_{\partial^0G\setminus E^+}\text{ on $\partial^0G$}\,. $$

Now we show that 
$$\partial E^+\cap\partial^0G=\Sigma=\partial E^-\cap\partial^0G\,.$$ 
Indeed, if $\mathbf{x}_0=(x_0,0)\in \partial E^+\cap\partial^0G$, then $D_r(x_0)\cap E^+\not=\emptyset$ for every $r>0$. Since $E^+$ is open, $D_r(x_0)\cap E^+$ contains a small disc for every $r>0$. Thus $D_r(x_0)\not\subset E^-$ for every $r>0$, and thus  $x_0\in\Sigma$.  This shows that $\partial E^+\cap\partial^0G\subset\Sigma$.  The other way around, if $x_0\in\Sigma$, then $x_0\not\in E^-$. Thus   $\mathscr{L}^n(\{u_*=-1\}\cap D_r(x_0))< \mathscr{L}^n(D_r(x_0))$ for every $r>0$. Since $\mathscr{L}^n(\Sigma)=0$, we deduce that for every $r>0$ there exists $x\in E^+\cap D_r(x_0)$. Hence $\Sigma\subset \partial E^+\cap\partial^0G$. 
\vskip3pt

We claim that $u_k\to \pm 1$ locally uniformly in $E^\pm$ (respectively). We only show that $u_k\to 1$ locally uniformly in $E^+$, the other case being completely analogous. Fix an arbitrary compact set $K\subset E^+$. By Lemma \ref{eas}, we can find a radius $r_K\leq \min\big\{{\rm dist}(K,\partial E^+), \mathbf{R}_{b,T}\big\}$ such that 
$${\bf E}\big(u_*,B^+_{r_K}(\bar{\mathbf{x}})\big) <\boldsymbol{\theta}_{b,T}r_K^{n-2s}$$
for every $\bar{\mathbf{x}}\in K\times\{0\}$. Then we deduce from Step 1 that 
$$\lim_{k\to\infty} \mathbf{E}_{\eps_k}\big(u_k,B^+_{r_K}(\bar{\mathbf{x}})\big) < \boldsymbol{\theta}_{b,T}r_K^{n-2s}$$ 
for every $\bar{\mathbf{x}}\in K\times\{0\}$. By Proposition \ref{epsreg} and a standard covering argument, it implies that $u_k\to u_*=1$ uniformly on $K$.  
Then items (iii) and (iv) follow from Proposition~\ref{imprcv}. 
\vskip3pt

\noindent{\it Step 3: Convergence of level sets.} We now prove (v). We fix $t\in(-1,1)$, a compact set $K\subset\partial^0G$, and a radius $r>0$. First, from (iii) we deduce that $|u_k|\to 1$ uniformly on $K\setminus \mathscr{T}_r(\Sigma)$. Therefore, $L_k^t\cap K\subset \mathscr{T}_r(\Sigma)$ for $k$ large enough. Then we consider a covering of $\Sigma\cap K$ made by finitely many discs $D_{r/2}(x_1),\ldots, D_{r/2}(x_J)$ (included in $\partial^0G$, choosing a smaller radius if necessary). Then, for each $j$ we can find a point $x^+_j\in D_{r/2}(x_j)\cap E^+$ and a point $x^-_j\in D_{r/2}(x_j)\cap E^-$. From (ii) we infer that for $k$ large enough, 
$$u_k(x^+_j)\geq 1/2(1+t) \quad\text{and}\quad u_k(x^-_j)\leq 1/2(-1+t)\qquad \forall j\in\{1,\ldots,J\}\,.$$
Then, by the mean value theorem, for $k$ large enough we can find for each $j$ a point $x_j^k\in[x_j^-,x_j]\cup[x_j,x_j^+]\subset D_{r/2}(x_j)$ such that $u_k(x_j^k)=t$. Now, if $x$ is an arbitrary point in $\Sigma\cap K$, then $x\in D_{r/2}(x_{j_x})$ for some $j_x\in\{1,\ldots,J\}$, and thus $|x-x_{j_x}^k|\leq |x-x_{j_x}|+|x_{j_x}-x_{j_x}^k|<r$. Hence $\Sigma\cap K\subset \mathscr{T}_r(L_k^t)$ whenever $k$ is sufficiently large.
\vskip3pt

\noindent{\it Step 4: Proof of {\rm (vi)}.} Let $\mathbf{X}=(X,\mathbf{X}_{n+1})\in C^1(\overline G;\R^{n+1})$ be a compactly supported  vector field in $G\cup \partial^0G$ such that $\mathbf{X}_{n+1}=0$ on $\partial^0G$.
By Corollary \ref{statAC}, we have 
$$\delta{\bf E}\big(u_k,G\cup\partial^0G\big)[{\bf X}]
+\frac{1}{\eps_k^{2s}}\int_{\partial^0 G}W(u_k)\,{\rm div}X\,\de x=\int_{\partial^0G} u_k\,{\rm div}(f_kX)\,\de x\,.$$
From formula \eqref{formulafirstvarbound} and the convergences established in Step 1, we can pass to the limit $k\to\infty$ in this identity to infer that 
$$\delta{\bf E}\big(u_*,G\cup\partial^0G\big)[{\bf X}]=\int_{\partial^0G} u_*{\rm div}(fX)\,\de x\,,$$
and the proof is complete. 
\end{proof}

 %%%%%%%%%%%%%%%%%%%%%%%%%%%%%%%%%%%%%%%%%%%%%%%%%%%%%%%
 %%%%%%%%%%%%%%%%%%%%%%%%%%%%%%%%%%%%%%%%%%%%%%%%%%%%%%%
 														 
\section{Asymptotics for the fractional Allen-Cahn equation}\label{FGLasymp}   		 
														 
%%%%%%%%%%%%%%%%%%%%%%%%%%%%%%%%%%%%%%%%%%%%%%%%%%%%%%%%
 %%%%%%%%%%%%%%%%%%%%%%%%%%%%%%%%%%%%%%%%%%%%%%%%%%%%%%%%

 The object of this section is to prove a general convergence result as $\varepsilon\downarrow0$ for the fractional equation \eqref{fracalllcahnf}. As we already explained, we  rely on the results obtained in Theorem \ref{asymptneum} for the degenerate equation with boundary reaction. In Section \ref{imprsect}, we will improve some of the convergences below under stronger assumptions on the sequence of right hand sides $\{f_k\}_{k\in\mathbb{N}}$.

 \begin{theorem}\label{main1}
 Let $\Omega$ be a smooth bounded open set, and $\eps_k\downarrow 0$ a given sequence. 
 Let $\{g_k\}_{k\in\N}\subset C^{0,1}_{\rm loc}(\R^n)$ be such that $\sup_k\|g_k\|_{L^\infty(\R^n\setminus\Omega)}<\infty$ and $g_k\to g$ in $L^1_{\rm loc}(\R^n\setminus \Omega)$ for a function $g$ satisfying $|g|=1$  a.e. in~$\R^n\setminus\Omega\,$.  Let $\{f_k\}_{k\in\mathbb{N}}\subset C^{0,1}(\Omega)$  
satisfying 
$$\sup_k\big(\eps_k^{2s}\|f_k\|_{L^\infty(\Omega)} + \|f_k\|_{W^{1,q}(\Omega)}\big)<\infty\quad \text{for some $n/(1+2s)<q<n$}\,,$$
and such that $f_k \rightharpoonup f$ weakly in  $W^{1,q}(\Omega)$. 
 Let $\{v_k\}_{k\in\N}\subset H_{g_k}^s(\Omega)\cap L^p(\Omega)$ be a sequence such that $v_k$ weakly solves
 \begin{equation}\label{fracalllcahnf}
 \begin{cases}
 \displaystyle  (-\Delta)^{s} v_k+\frac{1}{\varepsilon_k^{2s}}W^\prime(v_k) =f_k & \text{in $\Omega$}\,,\\
  v_k=g_k & \text{in $\R^n\setminus \Omega$}\,.
\end{cases}
\end{equation}
 If $\sup_k \mathcal{F}_{\varepsilon_k}(v_k,\Omega)<\infty$, then there exist a (not relabeled) subsequence and a Borel set $E_*\subset \R^n$ of finite $2s$-perimeter in $\Omega$ such that $v_k\to v_*:=\chi_{E_*}-\chi_{\R^n\setminus E_*}$ strongly in $H^s_{\rm loc}(\Omega)$ and $L^2_{\rm loc}(\R^n)$. Moreover, $E_*\cap \Omega$ is an open set, and 
 \begin{equation}\label{meancurveq}
\delta P_{2s}(E_*,\Omega)[X]=\frac{1}{\gamma_{n,s}}\int_{E_*\cap\Omega}{\rm div}(fX)\,\de x \quad\text{for every $X\in C_c^1(\Omega;\R^n)$}\,.
 \end{equation}
In addition, for every smooth open subset $\Omega'\subset\Omega$ such that $\overline{\Omega'}\subset\Omega$, 
 \begin{itemize}[leftmargin=22pt]
\item[ \rm  (i)] $ \mathcal{E}(v_k,\Omega') \to 2\gamma_{n,s}P_{2s}(E_*,\Omega')$; 
\vskip5pt

\item[ \rm  (ii)] $\displaystyle\frac{1}{\varepsilon_k^{2s}}W(v_k)\to 0$ in $L^1(\Omega^\prime)$; 
\vskip5pt

\item[\rm  (iii)] $\displaystyle f_k(x)-\frac{1}{\varepsilon_k^{2s}}W^\prime(v_k(x))\to \left(\frac{\gamma_{n,s}}{2}\int_{\R^n}\frac{|v_*(x)-v_*(y)|^2}{|x-y|^{n+2s}}\,\de y\right)v_*(x)$ strongly in  $H^{-s}(\Omega^\prime)$; 
\vskip5pt 

\item[\rm  (iv)] $v_k\to v_*$ in $C^0_{\rm loc}(\Omega\setminus\partial E_*)$;
\vskip5pt

\item[\rm  (v)] if $\sup_k\|f_k\|_{L^\infty(\Omega)}<\infty$, then $v_k\to v_*$ in $C^{0,\alpha}_{\rm loc}(\Omega\setminus\partial E_*)$ for every $\alpha\in(0,\boldsymbol{\beta}_*)$
with $\boldsymbol{\beta}_*$  given by Lemma \ref{prereg};
\vskip5pt 

\item[\rm  (vi)] if $\sup_k\|f_k\|_{C^{0,1}(\Omega)}<\infty$, then $v_k\to v_*$ in $C^{1,\alpha}_{\rm loc}(\Omega\setminus\partial E_*)$ for every $\alpha\in(0,\boldsymbol{\beta}_*)$;
\vskip5pt 

\item[\rm  (vii)] for each $\delta\in(-1,1)$, the level set $L^\delta_k:=\{v_k=\delta\}$ converges locally uniformly  in $\Omega$ to $\partial E_*\cap \Omega$, i.e., for every compact set $K\subset \Omega$ and every $r>0$,
$$L^\delta_k\cap K\subset \mathscr{T}_r(\partial E_*\cap \Omega) \quad \text{and}\quad \partial E_*\cap K\subset \mathscr{T}_r(L_k^\delta\cap \Omega)$$
whenever $k$ is large enough. 
\end{itemize}
 \end{theorem}

\begin{proof}
{\it Step 1.} First we recall that, under the assumptions of the theorem, we have proved in Section \ref{FractAC} that $v_k\in C^{1,\boldsymbol{\beta}_*}_{\rm loc}(\Omega)\cap C^0(\R^n)$ and $\sup_k\|v_k\|_{L^\infty(\R^n)}<\infty$. Then the assumption $\sup_k \mathcal{F}_{\varepsilon_k}(v_k,\Omega)<\infty$ clearly implies  $\sup_k \mathcal{E}_{\varepsilon_k}(v_k,\Omega)<\infty$. In turn, Lemma~\ref{adminHchap} shows that the sequence $\{v_k\}_{k\in\mathbb{N}}$ is bounded in $L^2(\R^n,\mathfrak{m})$, where the measure $\mathfrak{m}$ is defined in \eqref{defmeasm}. Therefore, we can find a (not relabeled) subsequence and $v_*\in L^2(\R^n,\mathfrak{m})$ such that $v_k\rightharpoonup v_*$ weakly in $L^2(\R^n,\mathfrak{m})$. In particular, $v_k\rightharpoonup v_*$ weakly in $L^2_{\rm loc}(\R^n)$. On the other hand, the uniform energy bound shows that  $|v_k|\to 1$ in $L^1(\Omega)$, and $\{v_k\}_{k\in\mathbb{N}}$ is bounded in $H^{s}(\Omega)$. Hence  $v_k\rightharpoonup v_*$ weakly in $H^{s}(\Omega)$, and from the compact embedding $H^{s}(\Omega)\hookrightarrow L^2(\Omega)$, it implies that $v_k\to v_*$ strongly in $L^2(\Omega)$. By assumption we have $g_k\to g$ in $L^1_{\rm loc}(\R^n\setminus\Omega)$ and $\sup_k\|g_k\|_{L^\infty(\R^n\setminus\Omega)}<\infty$, so that $g_k\to g$ in $L^2_{\rm loc}(\R^n\setminus\Omega)$. Since $v_k=g_k$ in $\R^n\setminus \Omega$, we conclude that $v_*=g$ in $\R^n\setminus\Omega$ and $v_k\to v_*$ strongly in $L^2_{\rm loc}(\R^n)$. Extracting a further subsequence if necessary, we may assume that $v_k\to v_*$ a.e. in $\R^n$. Since $|g|=1$ a.e. in $\R^n$, we derive that  $|v_*|=1$ a.e. in $\R^n$. Hence we can find a Borel set $F\subset \R^n$ such that  
$$v_*= \chi_F-\chi_{\R^n\setminus F} \text{ a.e. in $\R^n$.}$$ 
Moreover, we easily infer from Fatou's lemma that
\begin{equation}\label{liminf}
\mathcal{E}(v_*,\Omega)\leq \liminf_{k\to\infty} \mathcal{E}(v_k,\Omega) <\infty\,.
\end{equation}

We end this first step showing that $v_k^\e\rightharpoonup v_*^\e$ weakly in $H^1_{{\rm loc}}(\R^{n+1}_+\cup\Omega,|z|^a\de \mathbf{x}) $. Indeed, we start deducing from Lemma~\ref{context} that $v_k^\e\rightharpoonup v_*^\e$ weakly in $L^2_{{\rm loc}}(\overline{\R^{n+1}_+}, |z|^a\de \mathbf{x})$.  On the other hand, the uniform energy bound together with Lemma \ref{hatH1/2toH1} and standard elliptic estimates shows that $\{v_k^\e\}_{k\in\mathbb{N}}$ is bounded in $H^1_{{\rm loc}}(\R^{n+1}_+\cup\Omega,|z|^a\de \mathbf{x})$, whence the announced weak convergence. 
\vskip3pt

\noindent{\it Step 2.} Let us now consider an increasing  sequence $\{G_l\}_{l\in\N}$ of  bounded admissible open sets such that  $\overline{\partial^0G_l}\subset\Omega$ for every $l\in\N$, $\cup_l G_l=\R^{n+1}_+$, and $\cup_l\partial^0G_l=\Omega$. By \eqref{bdlinftyext}, 
Step 1, and the results in Section \ref{FractAC},  $v_k^\e\in H^1(G_l,|z|^a\de \mathbf{x})\cap L^\infty(G_l)$  satisfies $\sup_k\|v_k^\e\|_{L^\infty(G_l)}\leq \sup_k\|v_k\|_{L^\infty(\R^n)}<\infty$, and each $v_k$ solves 
$$ \begin{cases}
{\rm div}(z^a\nabla v^\e_k)= 0 & \text{in $G_l$}\,,\\[8pt]
\displaystyle d_s\boldsymbol{\partial}^{(2s)}_z v^\e_k=\frac{1}{\varepsilon^{2s}_k}W^\prime(v^\e_k) -f & \text{on $\partial^0 G_l$}\,,  
\end{cases}
$$
for every $l\in\N$. In addition, $\sup_k {\bf E}_{\eps_k}(v_k^\e,G_l)<\infty$ for every $l\in\N$, still by Step 1. Therefore, we can find a further subsequence such that the conclusions of Theorem~\ref{asymptneum}  hold in every $G_l$, and $v_*^\e$ is the limiting function in each $G_l$ by Step 1. In particular,  $v_k^\e\to v_*^\e$ strongly in $H^1_{{\rm loc}}(\R^{n+1}_+\cup\Omega,|z|^a\de \mathbf{x})$.

For each $l\in\mathbb{N}$, denote by $E_l$ the limiting open subset of $\partial^0G_l$ provided by Theorem~\ref{asymptneum}, and observe that $E_l=E_{l+1}\cap \partial^0G_l$ for every $l\in\N$ (see the proof of Theorem~\ref{asymptneum}, Step 2). Then we define $E_\Omega:=\cup_l E_l$, so that $E_\Omega$ is an open subset of $\Omega$,   
$E_l=E_\Omega\cap  \partial^0G_l$ for every $l\in\N$, and $v_*=\chi_{E_\Omega}- \chi_{\Omega\setminus E_\Omega}$ a.e. in $\Omega$.   Setting 
$$E_*:=(F\setminus\Omega)\cup E_\Omega \,,$$
it follows that $v_*=\chi_{E_*}- \chi_{\R^n\setminus E_*}$ a.e. in $\R^n$. In particular, $E_*$ has finite $2s$-perimeter in $\Omega$ since 
$$\mathcal{E}(v_*,\Omega)=2\gamma_{n,s}P_{2s}(E_*,\Omega)\,.$$ 
Finally, the conclusions of  Theorem~\ref{asymptneum}  in each $G_l$ clearly imply the announced results stated in (ii), (iv), (v), (vi), and (vii).
\vskip3pt

\noindent{\it Step 3.} Now we show items (i), (iii), and the strong convergence of $v_k$ in $H^s_{\rm loc}(\Omega)$. To this purpose, we fix a smooth open set $\Omega'\subset\Omega$ such that $\overline{\Omega'}\subset\Omega$. Setting for an arbitrary function $v\in\widehat H^s(\Omega)$,
$${\rm e}_s\big(v(x),\Omega\big):=\frac{\gamma_{n,s}}{2}\int_\Omega\frac{|v(x)-v(y)|^2}{|x-y|^{n+2s}}\,\de y+\gamma_{n,s} \int_{\R^n\setminus\Omega}\frac{|v(x)-v(y)|^2}{|x-y|^{n+2s}}\,\de y\,,$$
we claim that 
$${\rm e}_s(v_k,\Omega)\,\mathscr{L}^n\LL\Omega \rightharpoonup {\rm e}_s(v_*,\Omega)\,\mathscr{L}^n\LL\Omega$$
weakly* as Radon measures on $\Omega$. Indeed, by the uniform energy bound, we can extract a subsequence such that 
${\rm e}_s(v_k,\Omega)\,\mathscr{L}^n\LL\Omega \mathop{\rightharpoonup}\limits^* \nu$ for some finite Radon measure $\nu$ on $\Omega$. Then 
we fix $\varphi\in\mathscr{D}(\Omega)$ arbitrary. Notice that 
  \begin{align*}
  \int_\Omega {\rm e}_s(v_k,\Omega)\varphi\,\de x      = &\,  \big\langle  (-\Delta)^s  v_k,\varphi v_k\big\rangle_\Omega  \\
  &   -\frac{\gamma_{n,s}}{2}\iint_{\Omega\times\Omega} \frac{(v_k(x)-v_k(y)) v_k(y)(\varphi(x)-\varphi(y))}{|x-y|^{n+2s}}\,\de x\de y\\
 & -\gamma_{n,s}\iint_{\Omega\times\Omega^c} \frac{(v_k(x)-v_k(y)) v_k(y)\varphi(x)}{|x-y|^{n+2s}}\,\de x\de y\\
=:&\,I_k-II_k-III_k\,.
  \end{align*}
We consider  a function $\Phi\in  C^\infty(\overline{\R^{n+1}_+})$ compactly supported in $G\cup\partial^0G$ for some bounded admissible open set $G\subset\R^{n+1}_+$
such that  $\overline{\partial^0G}\subset\Omega$ and   $\Phi_{|\R^n}=\varphi$. 
Since $\varphi v_k\in H^s_{00}(\Omega)$ and $\Phi v_k^\e\in H^1(G,|z|^a\de \mathbf{x})$ is compactly supported in $G\cup\partial^0G$,  Lemma~\ref{repnormderfraclap}  yields 
$$\big\langle  (-\Delta)^{s}  v_k,\varphi v_k\big\rangle_\Omega
= d_s\int_{G}z^a|\nabla v_k^\e|^2\Phi\,\de \mathbf{x}+d_s\int_{G}z^a\nabla v_k^\e\cdot(v_k^\e \nabla\Phi)\,\de \mathbf{x}\,.$$
Since $v_k^\e \to v_*^\e$ strongly in $H^1(G,|z|^a\de{\bf x})$, we obtain 
\begin{multline*}
\big\langle  (-\Delta)^{s}  v_k,\varphi v_k\big\rangle_\Omega\mathop{\longrightarrow}\limits_{k\to\infty} d_s \int_{G}z^a|\nabla v_*^\e|^2\Phi\,\de {\bf x}+d_s\int_{G}z^a\nabla v_*^\e\cdot(v_*^\e \nabla\Phi)\,\de {\bf x}\\
 =d_s\int_{\R^{n+1}_+}z^a\nabla v_*^\e\cdot\nabla(\Phi v_*^\e)\,\de {\bf x}\,.
\end{multline*}
By  Lemma \ref{repnormderfraclap} again, we have thus proved that 
\begin{equation}\label{tim1555}
\big\langle  (-\Delta)^{s}  v_k,\varphi v_k\big\rangle_\Omega\mathop{\longrightarrow}\limits_{k\to\infty} \big\langle  (-\Delta)^{s}  v_*,\varphi v_*\big\rangle_\Omega\,.
\end{equation}
On the other hand, we  easily deduce by dominated convergence that 
\begin{equation}\label{tim1627}
II_k \to \frac{\gamma_{n,s}}{2}\iint_{\Omega\times\Omega} \frac{(v_*(x)-v_*(y)) v_*(y)(\varphi(x)-\varphi(y))}{|x-y|^{n+2s}}\,\de x\de y
\end{equation}
and
\begin{equation}\label{tim1628}
III_k\to  \gamma_{n,s}\iint_{\Omega\times\Omega^c} \frac{(v_*(x)-v_*(y)) v_*(y)\varphi(x)}{|x-y|^{n+2s}}\,\de x\de y
\end{equation}
as $k\to\infty$. Gathering \eqref{tim1555}, \eqref{tim1627}, and \eqref{tim1628} leads to 
$$\int_\Omega {\rm e}_s(v_k,\Omega)\varphi\,\de x \mathop{\longrightarrow}\limits_{k\to\infty} \int_\Omega {\rm e}_s(v_*,\Omega)\varphi\,\de x
\,,$$
and thus $\nu={\rm e}_s(v_*,\Omega)\,\mathscr{L}^n\LL\Omega$ by the arbitrariness of $\varphi$. 

Since $\nu(\partial\Omega')=0$, we now derive that 
\begin{equation}\label{convdensit}
\int_{\Omega'} {\rm e}_s(v_k,\Omega)\,\de x \to \int_{\Omega'} {\rm e}_s(v_*,\Omega)\,\de x\,.
\end{equation}
Then, since $\Omega'$ is smooth and bounded, it has finite $2s$-perimeter in $\R^n$, and thus  
\begin{equation}\label{cpasdim}
\int_{\Omega'}\int_{\Omega\setminus\Omega'}\frac{1}{|x-y|^{n+2s}}\,\de x\de y \leq  \int_{\Omega'}\int_{\R^n\setminus\Omega'}\frac{1}{|x-y|^{n+2s}}\,\de x\de y=P_{2s}(\Omega',\R^n)<\infty\,.
\end{equation}
It now follows by dominated convergence and \eqref{convdensit} that 
\begin{multline}\label{convenerg2020}
\mathcal{E}(v_k,\Omega')= \frac{1}{2}\int_{\Omega'} {\rm e}_s(v_k,\Omega)\,\de x+\frac{\gamma_{n,s}}{4}\int_{\Omega'}\int_{\Omega\setminus\Omega'}\frac{|v_k(x)-v_k(y)|^2}{|x-y|^{n+2s}}\,\de x\de y \\
\mathop{\longrightarrow}\limits_{k\to\infty}  \frac{1}{2}\int_{\Omega'} {\rm e}_s(v_*,\Omega)\,\de x+\frac{\gamma_{n,s}}{4}\int_{\Omega'}\int_{\Omega\setminus\Omega'}\frac{|v_*(x)-v_*(y)|^2}{|x-y|^{n+2s}}\,\de x\de y\\
= \mathcal{E}(v_*,\Omega')=2\gamma_{n,s}P_{2s}(E_*,\Omega')\,.
\end{multline}
Using  \eqref{convdensit} again, the same argument shows that 
$$[v_k]^2_{H^s(\Omega')} \to [v_*]^2_{H^s(\Omega')}\,,$$
and thus $v_k\to v_*$ strongly in $H^s(\Omega')$, since we already know that $v_k\rightharpoonup v_*$ weakly in $H^s(\Omega')$. In turn, the strong convergence in $H^s(\Omega')$ and \eqref{cpasdim} easily imply $\big\langle(-\Delta)^sv_k,v_*\big\rangle_{\Omega^\prime}\to \big\langle(-\Delta)^sv_*,v_*\big\rangle_{\Omega}=2\mathcal{E}(v_*,\Omega^\prime)$ by dominated convergence. Consequently, 
$$\mathcal{E}(v_k-v_*,\Omega')=\mathcal{E}(v_k,\Omega')+\mathcal{E}(v_*,\Omega')-\big\langle(-\Delta)^sv_k,v_*\big\rangle_{\Omega^\prime}\longrightarrow 0\,. $$
Next we infer from \eqref{estinormH-1/2fraclap} that $(-\Delta)^s v_k \to (-\Delta)^s v_*$ strongly in $H^{-s}(\Omega^\prime)$.  

Then, fix some $\varphi\in\mathscr{D}(\Omega^\prime)$. Since $v_*^2=1$, we have the identity
\begin{equation}\label{identdim}
\big(v_*(x)-v_*(y))(\varphi(x)-\varphi(y)\big) = \frac{1}{2}|v_*(x)-v_*(y)|^2\big(v_*(x)\varphi(x)+v_*(y)\varphi(y)\big)\,,
\end{equation}
that we may insert in \eqref{deffraclap} to obtain
\begin{equation}\label{identdim2}
\big\langle (-\Delta)^s v_*,\varphi\big\rangle_{\Omega^\prime}=\int_{\Omega^{\prime}} \left(\frac{\gamma_{n,s}}{2}\int_{\R^n}\frac{|v_*(x)-v_*(y)|^2}{|x-y|^{n+2s}}\,\de y\right)v_*(x)\varphi(x)\,\de x\,.
\end{equation}
Using this equation and \eqref{fracalllcahnf}, item (iii) follows. 
\vskip3pt

\noindent{\it Step 4.} Now it only remains to show that $E_*$ satisfies \eqref{meancurveq}. Let $X\in C^1(\R^n;\R^n)$ compactly supported in $\Omega$, and $\mathbf{X}=(\mathbf{X}_1,\ldots,\mathbf{X}_{n+1})\in C^1(\overline{\R^{n+1}_+};\mathbb{R}^{n+1})$ compactly supported in $\R^{n+1}_+\cup\Omega$ satisfying $\mathbf{X}=(X,0)$ on $\Omega$. Setting $\{\phi_t\}_{t\in\R}$ to be the flow on $\R^n$ generated by $X$, we notice that 
\begin{equation}\label{identflowEenerg}
P_{2s}\big(\phi_t(E_*) , \Omega\big) = \frac{1}{2\gamma_{n,s}}\mathcal{E}(v_*\circ\phi_{-t},\Omega)\,.
\end{equation}
Since the support of $\mathbf{X}$ is contained in $G_l\cup\partial^0G_l$ for $l$ large enough, we can apply (vi) in Theorem \ref{asymptneum}. In view of Remark \ref{varuptobound} and \eqref{identflowEenerg}, we obtain 
\begin{multline*}
\delta P_{2s}(E_*,\Omega)[X]=\frac{1}{2\gamma_{n,s}}\delta\mathcal{E}(v_*,\Omega)[X]\\
=\frac{1}{2\gamma_{n,s}}\delta{\bf E}\big(v^\e_*,G_l\cup\partial^0G_l\big)[{\bf X}]= \frac{1}{2\gamma_{n,s}}\int_\Omega v_*\,{\rm div}(fX)\,\de x\\
=\frac{1}{\gamma_{n,s}}\int_{E_*\cap \Omega} {\rm div}(fX)\,\de x\,,  
\end{multline*}
by the divergence theorem, and the proof is complete. 
\end{proof}

%%%%%%%%%%%%%%%%%%%%%%%%%%%%%%%%%%%%%%%%%%%%%%%%%%%%%%%%
%%%%%%%%%%%%%%%%%%%%%%%%%%%%%%%%%%%%%%%%%%%%%%%%%%%%%%%%

\section{Surfaces of prescribed nonlocal mean curvature}\label{GHSNMC}

%%%%%%%%%%%%%%%%%%%%%%%%%%%%%%%%%%%%%%%%%%%%%%%%%%%%%%%%
%%%%%%%%%%%%%%%%%%%%%%%%%%%%%%%%%%%%%%%%%%%%%%%%%%%%%%%%

In this section, we investigate regularity properties in a Lipschitz bounded open set $\Omega\subset\R^n$ of a (Borel) set $E\subset \R^n$ which is a {\it weak solution} in $\Omega$ of the prescribed nonlocal $2s$-mean curvature equation  
\begin{equation}\label{eqcurvsec6}
{\rm H}^{(2s)}_{\partial E}=\frac{1}{\gamma_{n,s}} f\quad \text{on $\partial E\cap \Omega$}\,, 
\end{equation}
where $f$ is a given Sobolev function in $W^{1,q}(\Omega)$ with $q\in(\frac{n}{1+2s},n)$. The notion of weak solution corresponds to the following weak formulation of \eqref{eqcurvsec6}: 

\begin{definition}
A set $E\subset \R^n$ is a weak solution of \eqref{eqcurvsec6} if $P_{2s}(E,\Omega)<\infty$ and  
$$\delta P_{2s}(E,\Omega)[X]=\frac{1}{\gamma_{n,s}}\int_{E\cap\Omega}{\rm div}(fX)\,\de x \qquad \forall X\in C_c^1(\Omega;\R^n)\,.$$
\end{definition}

Introducing the ``phase function" $v_E:=\chi_E-\chi_{\R^n\setminus E}\in\widehat H(\Omega)$, this equation rewrites (as in the proof of Theorem \ref{main1}, Step 4)
\begin{equation}\label{campdev1847}
\delta\mathcal{E}(v_E,\Omega)[X]=\int_\Omega v_E\, {\rm div}(fX)\,\de x\qquad\forall X\in C_c^1(\Omega;\R^n)\,.
\end{equation}
As we already did for the fractional Allen-Cahn equation, we rely on the fractional harmonic extension $(v_E)^\e$ defined in  \eqref{poisson} which satisfies 
\begin{equation}\label{campdev2}
\begin{cases}
{\rm div}(z^a\nabla (v_E)^\e)=0 & \text{in $\R^{n+1}_+$}\,,\\
|(v_E)^\e|\leq 1 & \text{in $\R^{n+1}_+$}\,,\\
|(v_E)^\e|=1 & \text{on $\R^n$}\,,
\end{cases}
\end{equation}
and (by Remark \ref{varuptobound} and \eqref{campdev1847})
\begin{equation}\label{stateeq1458}
\delta{\bf E}\big((v_E)^\e,G\cup\partial^0G\big)[{\bf X}]= \int_{\partial^0G} (v_E)^\e\,{\rm div}(fX)\,\de x 
\end{equation}
for every vector field ${\bf X}=(X,{\bf X}_{n+1})\in C^1(\overline G;\R^{n+1})$ compactly supported in $G\cup\partial^0G$ satisfying ${\bf X}_{n+1}=0$ on $\partial^0 G$, whenever $G\subset\R^{n+1}_+$ is an admissible bounded open set such that $\overline{\partial^0G}\subset \Omega$.

Similarly to Section \ref{epsregtitle}, instead of  investigating only the regularity of $(v_E)^\e$  from \eqref{campdev2} and \eqref{stateeq1458}, we deal with the following more general situation. We consider an admissible bounded  open set $G\subset \R^{n+1}_+$ and a  function $u\in H^{1}(G,|z|^a\de{\bf x})\cap L^\infty(G)$ satisfying  
\begin{equation}\label{startequsurf}
\begin{cases}
{\rm div}(z^a\nabla u) =0 & \text{in $G$}\,,\\
|u| \leq b  & \text{in $G$}\,,\\
|u|=1 & \text{on $\partial^0 G$}\,,
\end{cases}
\end{equation}
for a given parameter $b\geq 1$ (whose importance will only appear in Section \ref{imprsect}), and 
\begin{equation}\label{critnonlocexteq}
\delta{\bf E}\big(u,G\cup\partial^0G\big)[{\bf X}]= \int_{\partial^0G} u\,{\rm div}(fX)\,\de x\,,
\end{equation}
where, again, $f$ belongs to $W^{1,q}(\partial^0G)$ with $q\in(\frac{n}{1+2s},n)$. 

Regularity estimates on the function $u$ at the boundary $\partial^0G$  will be our main concern in this section. The application to weak solutions of \eqref{eqcurvsec6} is the object of the very last subsection with some specific results.

%%%%%%%%%%%%%%%%%%%%%%%%%%%%%%%%%%%%%%%%%%%%%%%%%%%%%%%%

\subsection{Energy monotonicity and clearing-out}\label{subres1} 

%%%%%%%%%%%%%%%%%%%%%%%%%%%%%%%%%%%%%%%%%%%%%%%%%%%%%%%%

In this subsection, we consider an arbitrary solution $u\in H^{1}(G,|z|^a\de{\bf x})\cap L^\infty(G)$ of \eqref{startequsurf}-\eqref{critnonlocexteq}. We begin with the fundamental monotonicity formula involving the following density function:  
for a point $\mathbf{x}_0=(x_0,0)\in \partial^0G$ and $r>0$ such that $\overline B_r^+({\bf x}_0)\subset G$, we set  
$$\boldsymbol{\Theta}_u(f,x_0,r):=\frac{1}{r^{n-2s}}\mathbf{E}\big(u,B^+_r({\bf x}_0)\big)+ {\bf c}_{n,q}\,b\int_{0}^rt^{\theta_q-1}\|f\|_{\dot W^{1,q}(D_t(x_0))}\,\de t\,,$$
where the constants $\theta_q$ and $ {\bf c}_{n,q}$ are given by Lemma \ref{monotform}.

\begin{lemma}\label{monotformsurf}
For every $\mathbf{x}_0=(x_0,0)\in \partial^0G$ and $r>\rho>0$ such that $\overline B_r^+({\bf x}_0)\subset G$,  
$$\boldsymbol{\Theta}_u(f,x_0,r)- \boldsymbol{\Theta}_u(f,x_0,\rho)\geq 
d_s\int_{B^+_r({\bf x}_0)\setminus B^+_\rho({\bf x}_0)} z^a\frac{|(\mathbf{x}-{\bf x}_0)\cdot \nabla u|^2}{|\mathbf{x}-{\bf x}_0|^{n+2-2s}}\,\de{\bf x}\,.$$
Moreover, equality holds if $f=0$.  
\end{lemma}

\begin{proof}
We proceed exactly as in the proof of Lemma \ref{monotform}, assuming without loss of generality that $x_0=0$. Using \eqref{critnonlocexteq} and formula \eqref{formulafirstvarbound}, we infer that 
$$(n-2s) \mathbf{E}(u,B_r^+)-r \frac{\de}{\de r} \mathbf{E}(u,B_r^+)+ d_sr\int_{\partial^+ B_r}z^a\Big|\frac{\mathbf{x}}{|\mathbf{x}|}\cdot\nabla u\Big|^2\,\de \mathscr{H}^n
\leq b I(r) \,,$$
since $\|u\|_{L^\infty(\partial^0G)}\leq b$, where $I(r)$ is given by \eqref{definitionImonot}. Note that equality actually holds for $f=0$. In view of \eqref{1245142}, dividing by $r^{n+1-2s}$ and integrating the resulting inequality (or equality if $f=0$), the conclusion follows.  
\end{proof}

\begin{corollary}\label{directcorlmonot}
For every $\mathbf{x}=(x,0)\in \partial^0G\times\{0\}$, the limits 
$$\boldsymbol{\Theta}_u(x):=\lim_{r\downarrow 0} \boldsymbol{\Theta}_u(f,x,r)=\lim_{r\downarrow 0}  \frac{1}{r^{n-2s}}\mathbf{E}\big(u,B^+_r({\bf x}_0)\big)$$
exist, and the function $\boldsymbol{\Theta}_u:\partial^0G \to[0,\infty)$ is upper semicontinuous. In addition, 
\begin{equation}\label{cdv1117}
\boldsymbol{\Theta}_u(f,x_0,r)- \boldsymbol{\Theta}_u(x_0)\geq d_s\int_{B^+_r({\bf x}_0)} z^a\frac{|(\mathbf{x}-{\bf x}_0)\cdot \nabla u|^2}{|\mathbf{x}-{\bf x}_0|^{n+2-2s}}\,\de{\bf x}\,,
\end{equation}
and equality holds if $f=0$. 
\end{corollary}

\begin{proof}
The existence of first limit defining $\boldsymbol{\Theta}_u(x)$ is of course a direct consequence of the monotonicity of the density function established in Lemma \ref{monotformsurf}. Existence and equality  for the second one follows from the existence of the first one and the estimate
$$\int_{0}^rt^{\theta_q-1}\|f\|_{\dot W^{1,q}(D_t(x_0))}\,\de t\leq  \frac{\|f\|_{\dot W^{1,q}(\partial^0G)}}{\theta_q}\,r^{\theta_q}\,.$$
Then $\boldsymbol{\Theta}_u$ is upper semicontinuous as a pointwise limit of a decreasing family of continuous functions. Finally, letting $\rho\to 0$ in Lemma \ref{monotformsurf} yields  \eqref{cdv1117}.  
\end{proof}

We continue with the following clearing-out property which can be seen as a {\sl small-energy} regularity result. 

\begin{lemma}\label{epsregminsurf}
There exist  a constant ${\boldsymbol{\eta}}_0>0$ (depending only on $n$ and $s$)  such that the following holds. For 
$\mathbf{x}_0=(x_0,0)\in \partial^0G$ and $r >0$ such that $\overline B_r^+({\bf x}_0)\subset G$,  the condition 
$$\boldsymbol{\Theta}_u(f,x_0,r)\leq {\boldsymbol{\eta}}_0$$ 
implies that  either $u=1$ on $D_{r/2}(x_0)$, or $u=-1$ on $D_{r/2}(x_0)$. 
\end{lemma}

\begin{proof}
Let us fix some ${\bf y}=(y,0)\in D_{r/2}(x_0)\times\{0\}$. By Lemma \ref{monotformsurf}, for $0<\rho<r/2$, 
$$\boldsymbol{\Theta}_u(y,\rho) \leq \boldsymbol{\Theta}_u(y,r/2)\leq 2^{n-2s} \boldsymbol{\Theta}_u(x_0,r)\leq 2^{n-2s}{\boldsymbol{\eta}}_0\,.$$
By the Poincar\'e inequality in Lemma \ref{poincare}, we deduce that 
$$A_\rho(y):=\frac{1}{\rho^n}\int_{D_\rho(y)}\big|u-[u]_{y,\rho}\big|\,\de x\leq  2^{n/2-s}\boldsymbol{\lambda}_{n,s}\sqrt{{\boldsymbol{\eta}}_0}\,,$$
where $[u]_{y,\rho}$ denotes the average of $u$ over $D_\rho(y)$. 
Since $|u|=1$ on $\partial^0G$, we can find a Borel subset $E\subset\partial^0G$ such that $u=\chi_E-\chi_{\partial^0G\setminus E}$ a.e. on $\partial^0G$. Then, 
$$A_\rho(y)= 4\omega_n\left(1-\frac{|E\cap D_\rho(y)|}{|D_\rho|}\right)\frac{|E\cap D_\rho(y)|}{|D_\rho|}\,.$$
Choosing 
$${\boldsymbol{\eta}}_0:=\frac{9\omega_n^2}{2^{n+4-2s}\boldsymbol{\lambda}^2_{n,s}} $$
leads to  $A_\rho(y)\leq 3\omega_n/4$.  In turn, this inequality implies 
$$|E\cap D_\rho(y)|/|D_\rho| \in [0,1/4]\cup[3/4,1]\,.$$ 
Since the function $(y,\rho)\in D_{r/2}(x_0)\times(0,r/2)\mapsto |E\cap D_\rho(y)|/|D_\rho|$ is continuous, we infer that either 
$$\frac{|E\cap D_\rho(y)|}{|D_\rho|} \in [0,1/4] \quad\text{for every $y\in D_{r/2}(x_0)$ and every $0<\rho<r/2$}\,,$$
or 
\begin{equation}\label{density1}
\frac{|E\cap D_\rho(y)|}{|D_\rho|} \in [3/4,1] \quad\text{for every $y\in D_{r/2}(x_0)$ and every $0<\rho<r/2$}\,.
\end{equation}
Now assume that \eqref{density1} holds (the other case being analogous). Then, by the Lebesgue differentiation theorem, we deduce that a.e. $y\in D_{r/2}(x_0)$ is a point of density $1$ for $E$. Consequently, $u=1$ a.e. on $D_{r/2}(x_0)$, and the lemma is proved. 
\end{proof}

\begin{corollary}\label{openset}
For every $(x,0)\in\partial^0G$, either $\boldsymbol{\Theta}_u(x)=0$ or $\boldsymbol{\Theta}_u(x)\geq {\boldsymbol{\eta}}_0$. As a consequence,  there is an open subset $E_u\subset\partial^0G$ such that $\partial E_u\cap\partial^0G=\big\{\boldsymbol{\Theta}_u\geq {\boldsymbol{\eta}}_0\big\}$ and 
$$u=\chi_{E_u}-\chi_{\partial^0G\setminus E_u} \quad\text{a.e. on $\partial^0G$}\,.$$ 
\end{corollary}

\begin{proof}
The alternative $\boldsymbol{\Theta}_u(x)=0$ or $\boldsymbol{\Theta}_u(x)\geq {\boldsymbol{\eta}}_0$ is a direct consequence of Lemma \ref{epsregminsurf} together with Lemma \ref{eas}. By upper semicontinuity of $\boldsymbol{\Theta}_u$, the set $\Sigma:=\{\boldsymbol{\Theta}_u\geq {\boldsymbol{\eta}}_0\}$ is relatively closed in $\partial^0G$, and  
$$E_u :=\Big\{\mathbf{x}=(x,0)\in\partial^0G : \text{$u=1$  on $D_r(x)$ for some $r\in(0,{\rm dist}(\mathbf{x},\partial^+G))$}\Big \}$$
is open and disjoint from $\Sigma$. Arguing as in the proof of Theorem \ref{asymptneum}, Step 4, we obtain that $u=\chi_{E_u}-\chi_{\partial^0G\setminus E_u}$ a.e. on  $\partial^0G$, and 
$\partial E_u\cap\partial^0G=\Sigma$. 
\end{proof}

\begin{remark}
By \cite[Corollary 3.2.3]{Zi}, we also have $\mathscr{H}^{n-2s}(\partial E_u\cap \partial^0G)=0$.  We will improve this a priori estimate later on. 
\end{remark}

%%%%%%%%%%%%%%%%%%%%%%%%%%%%%%%%%%%%%%%%%%%%%%%%%%%%%%%%

\subsection{Compactness}\label{subsectcompsurf}

%%%%%%%%%%%%%%%%%%%%%%%%%%%%%%%%%%%%%%%%%%%%%%%%%%%%%%%%

In this subsection, we are dealing with compactness issues for sequences  $\{u_k\}_{k\in\mathbb{N}}\subset H^{1}(G,|z|^a\de{\bf x})\cap L^\infty(G)$ satisfying 
$$\begin{cases}
{\rm div}(z^a\nabla u_k) =0 & \text{in $G$}\,,\\
|u_k|\leq b & \text{in $G$}\,,\\
|u_k|=1 & \text{on $\partial^0 G$}\,,
\end{cases}$$
and 
\begin{equation}\label{critnonlocexteqcoucou}
\delta{\bf E}\big(u_k,G\cup\partial^0G\big)[{\bf X}]=  \int_{\partial^0G} u_k\,{\rm div}(f_kX)\,\de x\,,
\end{equation}
for some $f_k\in W^{1,q}(\partial^0G)$  with $q\in(\frac{n}{1+2s},n)$, and a parameter $b\geq 1$ independent of $k$. 

\begin{theorem}\label{compactminsurf}
If $\sup_k{\bf E}(u_k,G)+\|f_k\|_{W^{1,q}(\partial^0G)}<\infty$, then there exist a (not relabeled) subsequence and a function $u\in H^1(G,|z|^a\de{\bf x})\cap L^\infty(G)$ satisfying \eqref{startequsurf} such that $u_k\rightharpoonup u$ weakly in $H^1(G,|z|^a\de{\bf x})$, and $u_k\to u$ 
strongly in $H^1_{\rm loc}(G\cup\partial^0G,|z|^a\de{\bf x})$. In addition, if $f_k\rightharpoonup f$ weakly in $W^{1,q}(\partial^0G)$, then $u$ satisfies \eqref{critnonlocexteq}. 
\end{theorem}

\begin{proof}
Since the argument essentially follows the proof of Theorem \ref{asymptneum} (Step 1), we only sketch the main points. 
First, by assumption on the energy,  we can find a subsequence and $u\in H^1(G,|z|^a\de{\bf x})\cap L^\infty(G)$ satisfying \eqref{startequsurf}  such that 
$u_k\rightharpoonup u$ weakly in $H^1(G,|z|^a\de{\bf x})$ and strongly in $H^1_{\rm loc}(G,|z|^a\de{\bf x})$ . 
Consider the sequence of measures $\mu_k:=\frac{d_s}{2}z^a|\nabla u_k|^2 \mathscr{L}^{n+1}\LL G$ 
which admits a weakly* convergent (not relabeled) subsequence towards a limiting measure $\mu= \frac{d_s}{2}z^a|\nabla u|^2 \mathscr{L}^{n+1}\LL G +\mu_{\rm sing}$ with ${\rm spt}(\mu_{\rm sing})\subset\partial^0G$. From Lemma \ref{monotformsurf}, we infer that $\mu$ satisfies the monotonicity inequality \eqref{tim1213} with $T=\sup_k\|f_k\|_{\dot W^{1,q}(\partial^0G)}$. As a consequence, the density $\Theta^{n-2s}(\mu,{\bf x})$  (as defined in \eqref{existdens}) exists, is finite for every ${\bf x}\in\partial^0G$, and defines an upper semicontinuous function on $\partial^0G$. We define the concentration set $\Sigma$ as in \eqref{defconcentrset} with $\boldsymbol{\theta}_{b,T}$ replaced by $\boldsymbol{\eta}_0/2$.  
Then $\Sigma=\big\{\Theta^{n-2s}(\mu,\cdot)\geq \boldsymbol{\eta}_0/(2\omega_{n-2s}) \big\}\subset\partial^0G$, and $\mathscr{H}^{n-2s}(\Sigma)$ is finite. We continue exactly as Theorem \ref{asymptneum} to show that $\mu_{\rm sing}$ is absolutely continuous with respect $\mathscr{H}^{n-2s}\res\Sigma$, and that $\Theta^{n-2s}(\mu_{\rm sing},{\bf x})\in[0,\infty)$ exists at $\mathscr{H}^{n-2s}$-a.e. ${\bf x}\in \Sigma$. By Marstrand's Theorem, we must have $\mu_{\rm sing}\equiv 0$. In other words, $u_k\to u$ strongly in $H^1_{\rm loc}(G\cup\partial^0G,|z|^a\de{\bf x})$. In view of \eqref{formulafirstvarbound}, if $f_k\rightharpoonup f$ weakly in $W^{1,q}(\partial^0G)$, this strong convergence allows us to pass to the limit $k\to\infty$ in \eqref{critnonlocexteqcoucou} and obtain \eqref{critnonlocexteq}. 
\end{proof}

\begin{remark}
If $u_k\to u$ strongly in $H^1_{\rm loc}(G\cup\partial^0G,|z|^a\de{\bf x})$, 
$f_k\to f$ strongly in $W^{1,q}(\partial^0G)$, $x_k\to x$ and $r_k\to r>0$, then $\boldsymbol{\Theta}_{u_k}(f_k,x_k,r_k)\to\boldsymbol{\Theta}_{u}(f,x,r)$. 
\end{remark}

\begin{lemma}\label{updenstheta}
In addition to the conclusion of Theorem \ref{compactminsurf}, if $\{x_k\}_{k\in\mathbb{N}}\subset\partial^0 G$ is a sequence converging to $x\in\partial^0G$, then 
$$ \limsup_{k\to\infty} \,\boldsymbol{\Theta}_{u_k}(x_k)\leq \boldsymbol{\Theta}_{u}(x)\,.$$
\end{lemma}

\begin{proof}
Assume for simplicity that $x=0$. Applying Corollary \ref{directcorlmonot}, we obtain for $r>0$ sufficiently small and $r_k:=|x_k|<r$, 
$$ \boldsymbol{\Theta}_{u_k}(x_k) \leq  \boldsymbol{\Theta}_{u_k}(f_k,x_k,r) \leq \frac{1}{r^{n-2s}} {\bf E}(u_k,B^+_{r+r_k}) +Tr^{\theta_q}\,, $$
with $T:=( {\bf c}_{n,q}b/\theta_q)\sup_k\|f_k\|_{\dot W^{1,q}(\partial^0G)}<\infty$. 
Since $r_k\to0$ and $u_k$ converges strongly to $u$ in $H^1(B^+_{2r},|z|^a\de{\bf x})$, we have $ {\bf E}(u_k,B^+_{r+r_k}) \to  {\bf E}(u,B^+_{r}) $. Hence
$$\limsup_{k\to\infty} \, \boldsymbol{\Theta}_{u_k}(x_k) \leq   \frac{1}{r^{n-2s}} {\bf E}(u,B^+_{r}) +Tr^{\theta_q} \,.$$
Letting $r\downarrow0$ now leads to the conclusion.
\end{proof}

\begin{corollary}\label{Hausdcvbd}
In addition to the conclusion of Theorem \ref{compactminsurf}, the boundaries $\partial E_{u_k}\cap\partial^0G$ converge locally uniformly in $\partial^0G$ to $\partial E_u\cap\partial^0G$, i.e., for every compact subset $K\subset\partial^0G$ and every $r>0$, 
$$\partial E_{u_k}\cap K \subset\mathscr{T}_r(\partial E_u\cap \partial^0G) \quad\text{and}\quad \partial E_{u}\cap K\subset \mathscr{T}_r(\partial E_{u_k}\cap \partial^0G) $$
for $k$ large enough. 
\end{corollary}

\begin{proof}
We start proving the first inclusion. By Corollary \ref{openset}, $\boldsymbol{\Theta}_{u}(x)=0$ for every point $x\in K\setminus \mathscr{T}_r(\partial E_u\cap \partial^0G)$. Since $\boldsymbol{\Theta}_{u_k}$ is upper semicontinuous, we can find  a point $x_k\in  K\setminus \mathscr{T}_r(\partial E_u\cap \partial^0G)$ such that 
$$\boldsymbol{\Theta}_{u_k}(x_k)= \sup_{x\in K\setminus \mathscr{T}_r(\partial E_u\cap \partial^0G)}\boldsymbol{\Theta}_{u_k}(x)\,.$$
Then select a subsequence $\{k_j\}_{j\in\mathbb{N}}$ such that $\lim_j\boldsymbol{\Theta}_{u_{k_j}}(x_{k_j})=\limsup_k\boldsymbol{\Theta}_{u_k}(x_k)$. Extracting a further subsequence if necessary, we can assume that $x_{k_j}\to x_*\in K\setminus \mathscr{T}_r(\partial E_u\cap \partial^0G)$. Since $\boldsymbol{\Theta}_{u}(x_*)=0$, we infer from Lemma \ref{updenstheta} that 
$\limsup_k\boldsymbol{\Theta}_{u_k}(x_k)=0$. Consequently, $\boldsymbol{\Theta}_{u_k}(x_k)\leq {\boldsymbol{\eta}}_0/2$ for $k$ large enough, and Corollary  \ref{openset} shows that, for such integers $k$, $(\partial E_{u_k}\cap K)\setminus \mathscr{T}_r(\partial E_u\cap \partial^0G)=\emptyset$.  

To prove the second inclusion, we consider a covering of $\partial E_u\cap K$ made by finitely many discs $D_{r/2}(x_1),\ldots, D_{r/2}(x_J)$ (included in $\partial^0G$, choosing a smaller radius if necessary). Then, for each $j$, we can find a point $x^+_j\in D_{r/2}(x_j)\cap E_u$ and a point $x^-_j\in D_{r/2}(x_j)\setminus \overline E_u$. Since $D_{r/2}(x_j)\cap E_u$ and $D_{r/2}(x_j)\setminus \overline E_u$ are open sets, we can find a radius $\varrho>0$ such that $D_{2\varrho}(x_j^+)\subset D_{r/2}(x_j)\cap E_u$ and $D_{2\varrho}(x_j^-)\subset  D_{r/2}(x_j)\setminus \overline E_u$ for each $j\in\{1,\ldots,J\}$. Hence, $u=\pm 1$ on $D_{2\varrho}(x_j^\pm)$  for each $j\in\{1,\ldots,J\}$.
In particular, $\boldsymbol{\Theta}_{u}(x)=0$ for every $x\in \overline D_{\varrho}(x_j^\pm)$ and each $j\in\{1,\ldots,J\}$. Arguing as before (for the first inclusion), 
we infer from Lemma \ref{updenstheta} that 
$$\lim_{k\to\infty}\Big(\sup_{x\in \overline D_{\varrho}(x_j^\pm)}\boldsymbol{\Theta}_{u_k}(x)\Big)=0 \qquad\forall j\in\{1,\ldots,J\}\,. $$
Then Corollary  \ref{openset} implies that $\boldsymbol{\Theta}_{u_k}(x)=0$ for every $x\in  \overline D_{\varrho}(x_j^\pm)$ and $j\in\{1,\ldots,J\}$, whenever $k$ is large enough. Since each $D_{\varrho}(x_j^\pm)$ is connected, we must have either $u_k=+1$ or $u_k=-1$ on $D_{\varrho}(x_j^\pm)$ (otherwise $D_{\varrho}(x_j^\pm)$ could be written as the disjoint union of two non empty open sets). On the other hand, $u_k\to u$ in $L^1(D_{\varrho}(x_j^\pm))$ by Remark \ref{trace}, and we conclude that $u_k=u=\pm 1$ on $D_{\varrho}(x_j^\pm)$ for each  $j\in\{1,\ldots,J\}$, whenever $k$ is large enough. Hence, $D_{r/2}(x_j)\cap E_{u_k}\not=\emptyset$ and $D_{r/2}(x_j)\setminus \overline E_{u_k}\not=\emptyset$, and we have thus proved that $\partial E_{u_k}\cap D_{r/2}(x_j)\not=\emptyset$  for each  $j\in\{1,\ldots,J\}$, whenever $k$ is large enough. Therefore, $\partial E_u\cap K\subset \bigcup_j D_{r/2}(x_j)\subset \mathscr{T}_r(\partial E_{u_k}\cap\partial^0G)$ for $k$  sufficiently large. 
\end{proof}

%%%%%%%%%%%%%%%%%%%%%%%%%%%%%%%%%%%%%%%%%%%%%%%%%%%%%%%%

\subsection{Tangent maps}\label{secttangmap}

%%%%%%%%%%%%%%%%%%%%%%%%%%%%%%%%%%%%%%%%%%%%%%%%%%%%%%%%

We now return back a given solution $u\in H^1(G,|z|^a\de {\bf x})\cap L^\infty(G)$ of \eqref{startequsurf} and \eqref{critnonlocexteq}, and we apply the results of Subsection \ref{subsectcompsurf} to define the so-called ``tangent maps" of $u$ at a given point. To this purpose, we 
fix the point of study ${\bf x_0}=(x_0,0)\in\partial^0 G$ and a reference radius $\rho_0>0$ such that $B^+_{\rho_0}({\bf x}_0)\subset G$. We introduce the rescaled functions 
\begin{equation}\label{defrescmap}
u_{x_0,\rho}({\bf x}):=u({\bf x}_0+\rho{\bf x})\quad\text{and} \quad f_{x_0,\rho}(x):=f(x_0+\rho x)\,,
\end{equation}
which are defined for $0<\rho<\rho_0/r$, ${\bf x}\in B^+_r$ and $x\in D_r$, respectively. Changing variables, we observe that 
\begin{equation}\label{1912tang}
\boldsymbol{\Theta}_{u_{x_0,\rho}}(\rho^{2s}f_{x_0,\rho}, 0, r)
=\boldsymbol{\Theta}_{u}(f,x_0,\rho r)\,.
\end{equation}
This identity, together with Lemma \ref{monotformsurf}, leads to
\begin{multline}\label{1912tangbd}
\frac{1}{r^{n-2s}}{\bf E}(u_{x_0,\rho},B_r^+) 
\leq \boldsymbol{\Theta}_{u}(f,x_0,\rho r)\leq \boldsymbol{\Theta}_u(f,x_0,\rho_0)\\
\leq \frac{1}{\rho_0^{n-2s}}{\bf E}(u,G)+\frac{{\bf c}_{n,q}b \,\rho_0^{\theta_q}}{\theta_q}\|f\|_{\dot W^{1,q}(\partial^0G)}\,.
\end{multline}
Given a sequence $\rho_k\to0$, we deduce that  
\begin{equation}\label{enbdresc}
\limsup_{k\to\infty}{\bf E}(u_{x_0,\rho_k},B_r^+) <\infty \quad\text{for every $r>0$}\,.
\end{equation} 
As a consequence of Theorem \ref{compactminsurf}, we have the following 

\begin{lemma}\label{tangmapstat}
Every sequence $\rho_k\to0$ admits a subsequence $\{\rho^\prime_k\}_{k\in\mathbb{N}}$ such that  $u_{x_0,\rho^\prime_k}\to \varphi$ strongly in $H^1(B_r^+,|z|^a\de{\bf x})$ for every $r>0$, where $\varphi$ satisfies 
\begin{equation}\label{eqtangmap}
\begin{cases}
{\rm div}(z^a\nabla\varphi)=0 & \text{in $\R^{n+1}_+$}\,,\\
|\varphi|\leq b & \text{in $\R^{n+1}_+$}\,,\\
|\varphi|=1 & \text{on $\R^n$}\,,
\end{cases}
\end{equation}
and for each $r>0$, 
\begin{equation}\label{stattangmap}
\delta{\bf E}(\varphi,B_r^+\cup D_r)[{\bf X}]=0
\end{equation}
for every vector field ${\bf X}=(X,{\bf X}_{n+1})\in C^1(\overline B_r^+,\R^{n+1})$ compactly supported in $B_R^+\cup D_r$ such that ${\bf X}_{n+1}=0$ on $D_r$. 
\end{lemma}

\begin{proof}
In view of \eqref{enbdresc}, Theorem \ref{compactminsurf} yields the announced convergence and \eqref{eqtangmap}. Then observe that $u_{x_0,\rho}$ satisfies 
$$\delta{\bf E}(u_{x_0,\rho},B_r^+\cup D_r)[{\bf X}]=\int_{D_r} u_{x_0,\rho}\,{\rm div}(\rho^{2s} f_{x_0,\rho}X)\,\de x\,.$$
Rescaling variables, we obtain 
$$\|\rho^{2s}f_{x_0,\rho}\|_{\dot W^{1,q}(D_r)}=\rho^{\theta_q} \|f\|_{\dot W^{1,q}(D_{\rho r}(x_0))}\mathop{\longrightarrow}\limits_{\rho\to 0} 0\,.$$
Hence $\rho^{2s}f_{x_0,\rho}\to 0$ strongly in $W^{1,q}(D_r)$, and the conclusion follows from Theorem~\ref{compactminsurf}. 
\end{proof}

\begin{definition}
Every function $\varphi$ obtained by this process will be referred to as {\it tangent map of $u$ at the point $x_0$}.  The family of all tangent maps of $u$ at $x_0$ will be denoted by 
$T_{x_0}(u)$. 
\end{definition}

\begin{lemma}\label{homo2203}
If $\varphi\in T_{x_0}(u)$, then  
$$\boldsymbol{\Theta}_\varphi(0,0,r)=\boldsymbol{\Theta}_\varphi(0)= \boldsymbol{\Theta}_u(x_0)\quad\forall r>0\,,$$
and $\varphi$ is $0$-homogeneous, i.e., $\varphi(\lambda {\bf x})=\varphi({\bf x})$ for every $\lambda>0$ and every ${\bf x}\in\R^{n+1}_+$. 
\end{lemma}

\begin{proof}
From the strong convergence of $u_{x_0,\rho^\prime_k}$ toward $\varphi$ and the identity in \eqref{1912tang}, we first infer that
$$ \boldsymbol{\Theta}_\varphi(0,0,r)=\lim_{k\to\infty} \boldsymbol{\Theta}_{u_{x_0,\rho^\prime_k}}\big((\rho^\prime_k)^{2s}f_{x_0,\rho_k^\prime},0,r\big)=\boldsymbol{\Theta}_u(x_0)\quad\forall r>0\,.$$
Then the monotonicity formula in Lemma \ref{monotformsurf} applied to $\varphi$ implies that ${\bf x}\cdot \nabla\varphi({\bf x})=0$ for every ${\bf x}\in\R^{n+1}_+$, and the conclusion  follows.
\end{proof}

%%%%%%%%%%%%%%%%%%%%%%%%%%%%%%%%%%%%%%%%%%%%%%%%%%%%%%%%

\subsection{Homogeneous solutions}  

%%%%%%%%%%%%%%%%%%%%%%%%%%%%%%%%%%%%%%%%%%%%%%%%%%%%%%%%

In view of Lemma \ref{homo2203}, the study of tangent maps leads to  the study of $0$-homogeneous solutions, which is the purpose of this subsection. We start with the following observation. 

\begin{lemma}\label{lemhomogsol}
Let $\varphi\in H^1(B_1^+,|z|^a\de {\bf x})\cap L^\infty(B_1^+)$  be a solution of 
\begin{equation}\label{eqcone}
\begin{cases}
{\rm div}(z^a\nabla\varphi)=0 & \text{in $B_1^+$}\,,\\
|\varphi|\leq b & \text{in $B_1^+$}\,,\\
|\varphi|=1 & \text{on $D_1$}\,,
\end{cases}
\end{equation}
for some constant $b\geq 1$. Assume that there exists  $f\in W^{1,q}(D_1)$ with $n/(1+2s)<q<n$ such that 
\begin{equation}\label{statinB1}
\delta{\bf E}\big(\varphi,B_1^+\cup D_1\big)[{\bf X}]= \int_{D_1} \varphi\,{\rm div}(fX)\,\de x\,,
\end{equation}
for every vector field ${\bf X}=(X,{\bf X}_{n+1})\in C^1(\overline B_1^+,\R^{n+1})$ compactly supported in $B_1^+\cup D_1$ such that ${\bf X}_{n+1}=0$ on $D_1$. 
If $\boldsymbol{\Theta}_{\varphi}(f,0,1)=\boldsymbol{\Theta}_{\varphi}(0) $, then $\varphi$ is $0$-homogeneous and $f=0$. 
\end{lemma}

\begin{proof}
As in the proof Lemma \ref{homo2203}, Corollary \ref{directcorlmonot} applied at $x_0=0$ leads to the homogeneity of $\varphi$. In turn, the homogeneity of $\varphi$ implies that 
$T_0(\varphi)=\{\varphi\}$, and the conclusion  follows from Lemma \ref{tangmapstat}. 
\end{proof}

\begin{definition}
We say that a function $\varphi\in L^1_{\rm loc}(\R^{n+1}_+)$ is a {\it nonlocal stationary cone} if $\varphi$ is $0$-homogeneous, $\varphi\in H^1(B_1^+,|z|^a\de {\bf x})\cap L^\infty(B_1^+)$, and $\varphi$ satisfies \eqref{eqtangmap}-\eqref{stattangmap} (for some constant $b\geq 1$). 
\end{definition}

Summing up the results of the previous subsection, tangent maps to a solution of \eqref{startequsurf}-\eqref{critnonlocexteq} are thus nonlocal stationary cones. We shall present in details the main properties of those ``cones''. We start with the following lemma explaining somehow the terminology. 

\begin{lemma}\label{relattangmap}
If $\varphi$ is a nonlocal stationary cone, then there is an open cone $\mathcal{C}_\varphi \subset \R^n$ such that 
$$\varphi=\big(\chi_{\mathcal{C}_\varphi}-\chi_{\R^n\setminus\mathcal{C}_\varphi}\big)^\e\,, $$
as defined in \eqref{poisson}. In particular, $|\varphi|\leq 1$ in $\R^{n+1}_+$. 
\end{lemma}

\begin{proof}
By Corollary \ref{openset}, there is an open set $\mathcal{C}_\varphi \subset \R^n$ such that $\varphi=\chi_{\mathcal{C}_\varphi}-\chi_{\R^n\setminus\mathcal{C}_\varphi}$ a.e. on $\R^n$. Since  $\varphi$ is $0$-homogeneous, we easily infer that  $\mathcal{C}_\varphi $ is an open cone. We  set 
$$w:=\varphi- \big(\chi_{\mathcal{C}_\varphi}-\chi_{\R^n\setminus\mathcal{C}_\varphi}\big)^\e\,.$$ 
Since $w$ is $0$-homogeneous, $w\in H^1_{\rm loc}\big(\overline{\mathbb{R}^{n+1}_+},|z|^a\de{\bf x}\big)\cap L^\infty(\mathbb{R}^{n+1}_+)$ with $\|w\|_{L^\infty(\R^{n+1}_+)}\leq 1+\|\varphi\|_{L^\infty(\R^{n+1}_+)}$, and $w$ satisfies
$$\begin{cases} 
{\rm div}(z^a\nabla w)=0& \text{in $\R^{n+1}_+$}\,,\\
w=0 & \text{on $\partial \R^{n+1}_+$}\,. 
\end{cases}$$
Note that, as in the proof of Lemma \ref{eas},  $w$ and $z^a\partial_z w$ are H\"older continuous up to $\partial\R^{n+1}_+$, and smooth in $\R^{n+1}_+$ by  elliptic regularity. Since $w$ is bounded, the Liouville type theorem in \cite[Corollary 3.5]{CS1} tells us that $w\equiv 0$. 
\end{proof}

\begin{remark}
If  $\varphi$ is a nonlocal stationary cone, then $ \boldsymbol{\Theta}_\varphi(\lambda y)=\boldsymbol{\Theta}_\varphi(y)$ for every $y\in\R^n\setminus\{0\}$ and $\lambda>0$. Indeed, by homogeneity of $\varphi $  we have for  each $\rho>0$, 
$$\boldsymbol{\Theta}_\varphi(0,\lambda y,\rho) = \boldsymbol{\Theta}_\varphi(0,y,\rho/\lambda)\,,$$
and the assertion follows letting $\rho\to 0$. 
\end{remark}

\begin{lemma}\label{lemSphi}
Let $\varphi$ be a nonlocal stationary cone. Then, 
$$ \boldsymbol{\Theta}_\varphi(y)\leq \boldsymbol{\Theta}_\varphi(0)\quad \forall y\in\R^n\,.$$
In addition, the set 
$$S(\varphi):=\Big\{y\in\R^n:  \boldsymbol{\Theta}_\varphi(y)=\boldsymbol{\Theta}_\varphi(0)\Big\} $$
is a linear subspace of $\R^n$, and $\varphi({\bf x}+{\bf y})=\varphi({\bf x})$ for every ${\bf y}\in S(\varphi)\times\{0\}$ and ${\bf x}\in\R^{n+1}_+$. 
\end{lemma}

\begin{proof}
By Corollary \ref{directcorlmonot}, we have for every ${\bf y}=(y,0)\in\partial\R^{n+1}_+$ and every $\rho>0$, 
\begin{equation}\label{2124}
\boldsymbol{\Theta}_\varphi(y)+d_s\int_{B_\rho^+({\bf y})} z^a\frac{|(\mathbf{x}-{\bf y})\cdot \nabla \varphi({\bf x})|^2}{|\mathbf{x}-{\bf y}|^{n+2-2s}}\,\de{\bf x} = \boldsymbol{\Theta}_\varphi(0,y,\rho)\,.
\end{equation}
On the other hand, by homogeneity of $\varphi$, 
$$\boldsymbol{\Theta}_\varphi(0,y,\rho) \leq \frac{(\rho+|z|)^{n-2s}}{\rho^{n-2s}}\,\boldsymbol{\Theta}_\varphi(0,0,\rho+|y|) =  \frac{(\rho+|y|)^{n-2s}}{\rho^{n-2s}}\,\boldsymbol{\Theta}_\varphi(0)\,.  $$
Inserting this inequality in \eqref{2124} and letting $\rho\to\infty$, we deduce that 
$$ \boldsymbol{\Theta}_\varphi(y)+d_s\int_{\R^{n+1}_+} z^a\frac{|(\mathbf{x}-{\bf y})\cdot \nabla \varphi({\bf x})|^2}{|\mathbf{x}-{\bf y}|^{n+2-2s}}\,\de{\bf x} \leq \boldsymbol{\Theta}_\varphi(0)\,.$$
Next, if $ \boldsymbol{\Theta}_\varphi(y)= \boldsymbol{\Theta}_\varphi(0)$, then $(\mathbf{x}-{\bf y})\cdot \nabla \varphi({\bf x})=0$ for every ${\bf x}\in\R^{n+1}_+$. By homogeneity of $\varphi$, we deduce that ${\bf y}\cdot \nabla \varphi({\bf x})=0$ for every ${\bf x}\in\R^{n+1}_+$, i.e, 
\begin{equation}\label{transinv}
\varphi({\bf x}+{\bf y})=\varphi({\bf x}) \quad \forall {\bf x}\in\R^{n+1}_+\,.
\end{equation}
The other way around, if ${\bf y}=(y,0)$ satisfies \eqref{transinv}, then $(\mathbf{x}-{\bf y})\cdot \nabla \varphi({\bf x})=0$ for every ${\bf x}\in\R^{n+1}_+$ (using again the homogeneity of $\varphi$). By \eqref{2124} and \eqref{transinv},  it implies that for each radius $\rho>0$, 
$$\boldsymbol{\Theta}_\varphi(y)= \boldsymbol{\Theta}_\varphi(0,y,\rho)= \boldsymbol{\Theta}_\varphi(0,0,\rho)=\boldsymbol{\Theta}_\varphi(0)\,,$$
and thus $y\in S(\varphi)$. From \eqref{transinv} it now follows that $S(\varphi)$ is a linear subspace of $\R^n$. 
\end{proof}

\begin{remark}\label{remSphisubbound}
If  $\varphi$ is a non constant nonlocal stationary cone, then $ \boldsymbol{\Theta}_\varphi(0)>0$ by Lemma~\ref{epsregminsurf}. In turn, we infer from Corollary \ref{openset} that  $S(\varphi)\subset \partial \mathcal{C}_\varphi$.
\end{remark}

\begin{remark}\label{spinedimn}
If  $\varphi$ is a nonlocal stationary cone such that ${\rm dim}\,S(\varphi)=n$, then either $\mathcal{C}_\varphi=\R^n$ or $\mathcal{C}_\varphi=\emptyset$, i.e., either $\varphi=1$ or $\varphi=-1$, respectively. As a consequence, if $\varphi\in T_{x_0}(u)$ for some solution $u$  of \eqref{startequsurf}-\eqref{critnonlocexteq}, then $ \boldsymbol{\Theta}_u(x_0)=\boldsymbol{\Theta}_\varphi(0)=0$. Now Corollary \ref{openset} yields $x_0\not\in \partial E_u\cap\partial^0G$. In other words,
$$x_0\in\partial E_u\cap\partial^0G \Longleftrightarrow\,{\rm dim}\,S(\varphi)\leq n-1 \text{ for all $\varphi\in T_{x_0}(u)$}\,.$$
\end{remark}

\begin{remark}\label{remhalfspa}
If  $\varphi$ is a nonlocal stationary cone such that ${\rm dim}\,S(\varphi)=n-1$, then $\mathcal{C}_\varphi$ is a half-space. Indeed, up to a rotation, we may assume that $S(\varphi)=\{0\}\times \R^{n-1}$, and Lemma~\ref{lemSphi} yields $\varphi({\bf x})=\varphi(x_1,z)$ for all ${\bf x}=(x_1,\ldots,x_n,z)\in\R^{n+1}_+$. As a consequence, either  $\mathcal{C}_\varphi=\{x_1>0\}$ or  $\mathcal{C}_\varphi=\{x_1<0\}$. 
\end{remark}

In view of the remark above, we introduce the half-space $P_1\subset\R^n$ defined by 
\begin{equation}\label{defP1}
P_1:=\{x_1>0\}\,,
\end{equation}
and its extension to $\R^{n+1}_+$, $\varphi_{\rm ref}:=(\chi_{P_1}-\chi_{\R^n\setminus P_1})^\e$. Then we set 
\begin{equation}\label{defthetaplane}
\boldsymbol{\theta}_{n,s}:=\frac{d_s}{2}\int_{B_1^+}z^a|\nabla\varphi_{\rm ref}|^2\,\de{\bf x}\,.
\end{equation}

\begin{lemma}\label{energhalfplane}
If $\varphi$ is a non constant nonlocal stationary cone, then $ \boldsymbol{\Theta}_\varphi(0)\geq \boldsymbol{\theta}_{n,s}$. In addition, equality holds if and only if $\mathcal{C}_\varphi$ is an open half-space.  
\end{lemma}

\begin{proof}
{\it Step 1.} Since $\varphi$ is not trivial, by Corollary  \ref{openset}, Remark \ref{remSphisubbound}, and Lemma \ref{lemSphi}, we can find a point $y\in \mathbb{S}^{n-1}$ such that $\boldsymbol{\Theta}_\varphi(0)\geq \boldsymbol{\Theta}_\varphi(y)>0$. Rotating coordinates if necessary, we may assume that $y=e_n$, where $(e_1,\ldots,e_n)$ denotes the canonical basis of $\R^n$.  Let $\psi_n$ 
be a tangent map of $\varphi$ at $e_n$. We claim that $\psi_n$ is independent of the $x_n$-variable, i.e., $\partial_{x_n} \psi_n({\bf x})=0$ for every ${\bf x}\in \R^{n+1}_+$.  To prove this claim, we consider a sequence of radii $\rho_k\downarrow0$ such that $\varphi_{e_n,\rho_k}\to \psi_n$ strongly in $H^1(B_r^+,|z|^a\de{\bf x})$ for every $r>0$. By homogeneity of $\varphi$, we have for every ${\bf x}=(x,z)\in\R^{n+1}_+$, 
$$\partial_{x_n}\varphi_{e_n,\rho_k}({\bf x})= -\rho_k^2{\bf x}\cdot \nabla \varphi(e_n+\rho_k x,\rho_k z)= -\rho_k{\bf x}\cdot \nabla \varphi_{e_n,\rho_k}({\bf x})\,.$$
Consequently, 
$$\frac{1}{r^{n-2s}}\int_{B_r^+}z^a|\partial_{x_n}\varphi_{e_n,\rho_k}|^2\,\de{\bf x} \leq r^2\rho_k^2 \boldsymbol{\Theta}_\varphi(0,e_n,r\rho_k)\mathop{\longrightarrow}\limits_{k\to\infty}0\,,$$
and the claim follows. As a consequence, $\mathcal{C}_{\psi_n}=\mathcal{C}_{n-1}\times\R$ for some open cone $\mathcal{C}_{n-1}\subset\R^{n-1}$, and $\boldsymbol{\Theta}_{\psi_n}(0)=\boldsymbol{\Theta}_\varphi(y)>0$. Since $\psi_n$ is not trivial, we can now find a point $y\in \mathbb{S}^{n-2}\times\{0\}$ such that $\boldsymbol{\Theta}_\varphi(0)\geq {\Theta}_{\psi_n}(y)>0$. Rotating coordinates if necessary, we may assume that $y=e_{n-1}$, and we consider a tangent map $\psi_{n-1}$ of $\psi_n$ at $e_{n-1}$. Then, such a tangent map is independent of the $x_n$ and $x_{n-1}$ variables. Iterating this process, we produce for each $k\in\{n-1,\ldots,2\}$, a non trivial tangent map $\psi_k$ of $\psi_{k+1}$ at $e_k$  such that $\mathcal{C}_{\psi_k}=\mathcal{C}_{k-1}\times\R^{n+1-k}$ for some open cone  $\mathcal{C}_{k-1}\subset\R^{k-1}$, and $\boldsymbol{\Theta}_{\varphi}(0)\geq\boldsymbol{\Theta}_{\psi_k}(0)>0$.   At the last step of the process (i.e., $k=2$), we have either $\mathcal{C}_{1}=(0,+\infty)$ or $\mathcal{C}_{1}=(-\infty,0)$. In other words, either 
$\mathcal{C}_{\psi_2} =\{x_1>0\}$ or $\mathcal{C}_{\psi_2} =\{x_1<0\}$. Without loss of generality, we may assume that $\mathcal{C}_{\psi_2} =\{x_1>0\}$. Then, by Corollary  \ref{relattangmap} we have $\psi_{2}=(\chi_{P_1}-\chi_{\R^n\setminus P_1})^\e$ where $P_1$ is the reference half space \eqref{defP1}. By Lemma \ref{homo2203}, we conclude that $\boldsymbol{\Theta}_{\psi_2}(0)=\boldsymbol{\theta}_{n,s}$, and we have thus proved that $\boldsymbol{\Theta}_{\varphi}(0)\geq \boldsymbol{\theta}_{n,s}$.  
\vskip3pt

\noindent{\it Step 2.} Assume that $ \boldsymbol{\Theta}_\varphi(0)= \boldsymbol{\theta}_{n,s}$. From Step 1, Corollary \ref{openset}, and Lemma \ref{lemSphi} we infer that $ \boldsymbol{\Theta}_\varphi(x)= \boldsymbol{\theta}_{n,s}$ for every $x\in\partial \mathcal{C}_\varphi$. In view of Remark \ref{remSphisubbound}, it leads to $S(\varphi)=\partial \mathcal{C}_\varphi$. Since $\varphi$ is not trivial, we must have ${\rm dim}\,S(\varphi)=n-1$, and Remark \ref{remhalfspa} tells us that $\mathcal{C}_\varphi$ is a half-space. 
\end{proof}

For a constant $\Lambda\geq 0$ and $j\in\{0,\ldots,n\}$, we now introduce the following class of nonlocal stationary cones 
$$\mathscr{C}_j(\Lambda):=\Big\{\text{nonlocal stationary cones $\varphi$ such that ${\rm dim}\,S(\varphi)\geq j$ and $\boldsymbol{\Theta}_\varphi(0)\leq\Lambda$}\Big\}\,. $$
Note that $\mathscr{C}_{j+1}(\Lambda)\subset \mathscr{C}_j(\Lambda)$, and $\mathscr{C}_n(\Lambda)=\{+1,-1\}$ by Remark \ref{spinedimn}.

\begin{lemma}\label{compactcone}
For each $j\in\{0,\ldots,n\}$ and $r>0$, the set $\big\{\varphi_{| B_r^+}: \varphi\in\mathscr{C}_j(\Lambda)\big\} $ is strongly compact in $H^1(B_r^+,|z|^a\de{\bf x})$. 
\end{lemma}

\begin{proof}
By homogeneity, it is enough to consider the case $r=1$. Let $\{\varphi_k\}_{k\in\mathbb{N}}\subset \mathscr{C}_j(\Lambda)$ be an arbitrary sequence. Still by homogeneity, we have $\boldsymbol{\Theta}_{\varphi_k}(0,0,2)=\boldsymbol{\Theta}_{\varphi_k}(0)\leq \Lambda$, so that 
$${\bf E}(\varphi_k,B^+_{2})\leq 2^{n-2s}\Lambda\,. $$
Since $|\varphi_k|\leq 1$ by Lemma \ref{relattangmap}, we can apply Theorem \ref{compactminsurf} to find a (not relabeled) subsequence   such that $\varphi_k\to\psi$ strongly in $H^1(B_1^+,|z|^a\de{\bf x})$ for a function $\psi$ satisfying \eqref{eqcone}-\eqref{statinB1} with $f=0$ and $b=1$. Then we deduce from Lemma \ref{updenstheta} that 
$$\boldsymbol{\Theta}_{\psi}(0)\geq \lim_{k\to\infty}\boldsymbol{\Theta}_{\varphi_k}(0)=\lim_{k\to\infty}\boldsymbol{\Theta}_{\varphi_k}(0,0,1) =\boldsymbol{\Theta}_{\psi}(0,0,1)\,.$$
In turn, Corollary \ref{directcorlmonot} shows that $\boldsymbol{\Theta}_{\psi}(0)=\boldsymbol{\Theta}_{\psi}(0,0,1)$, and thus $\psi$ is $0$-homogeneous by Lemma \ref{lemhomogsol}, and $\boldsymbol{\Theta}_{\psi}(0)=\lim_k\boldsymbol{\Theta}_{\varphi_k}(0)\leq \Lambda$. Consequently, $\psi$ is a nonlocal stationary cone, and it remains to show that  ${\rm dim}\,S(\psi)\geq j$. 

Extracting a further subsequence if necessary, we may assume that ${\rm dim}\,S(\varphi_k)$ is a constant integer $d\geq j$, and $S(\varphi_k)\to V$ in the Grassmannian  $G(d,n)$ of all $d$-dimensional linear subspaces of $\mathbb{R}^n$. For an arbitrary $y\in V\cap D_1$, there exists a sequence $\{y_k\}_{k\in\mathbb{N}}\subset D_1$ such that $y_k\in S(\varphi_k)$ and $y_k\to y$.  By Lemma \ref{updenstheta}, we have 
$$\boldsymbol{\Theta}_{\psi}(y)\geq \lim_{k\to\infty} \boldsymbol{\Theta}_{\varphi_k}(y_k)=\lim_{k\to\infty} \boldsymbol{\Theta}_{\varphi_k}(0)=\boldsymbol{\Theta}_{\psi}(0)\,,$$
and we deduce from Lemma \ref{lemSphi} that $y\in S(\psi)$. Therefore $V\subset S(\psi)$, and in particular ${\rm dim}\,S(\psi)\geq d$. 
\end{proof}

%%%%%%%%%%%%%%%%%%%%%%%%%%%%%%%%%%%%%%%%%%%%%%%%%%%%%%%%

\subsection{Quantitative stratification}\label{subsectstrat}  

%%%%%%%%%%%%%%%%%%%%%%%%%%%%%%%%%%%%%%%%%%%%%%%%%%%%%%%%

In this subsection, we are back again to the analysis  of the function $u\in H^1(G,|z|^a\de{\bf x})\cap L^\infty(G)$ solving \eqref{startequsurf}-\eqref{critnonlocexteq}.  
We are interested in regularity properties of the open subset $E_u\subset \partial^0G$ satisfying $u=\chi_{E_u}-\chi_{\partial^0 G\setminus E_u}$ on $\partial^0G$ (provided by Corollary~\ref{openset}). To this purpose, we introduced the following (standard)  stratification of the singular set of $u$, 
$${\rm Sing}^j(u):=\Big\{x\in\partial^0G: {\rm dim}\,S(\varphi)\leq j \text{ for all $\varphi\in T_x(u)$}\Big\}\,,\quad j=0,\ldots,n-1\,.$$ 
Obviously, 
$${\rm Sing}^j(u)\subset {\rm Sing}^{j+1}(u)\,, $$
and by Remark \ref{spinedimn}, 
\begin{equation}\label{identboundtopstrat}
\partial E_u\cap\partial^0G =  {\rm Sing}^{n-1}(u)\,.
\end{equation}
We also introduce the ``regular part'' $\Sigma_{\rm reg}(u)$ of $\partial E_u\cap\partial^0G$, 
$$\Sigma_{\rm reg}(u):= \big(\partial E_u\cap\partial^0G\big)\setminus  {\rm Sing}^{n-2}(u)\,.$$
The terminology {\sl regular part} is motivated by the following proposition showing that all blow-up limits of $\partial E_u$ at points of  $\Sigma_{\rm reg}(u)$ are hyperplanes. 

\begin{proposition}\label{regblowupset}
For every $x\in \Sigma_{\rm reg}(u)$ and $\varphi\in T_x(u)$, we have ${\rm dim}\,S(\varphi)=n-1$. In particular, if $x\in \Sigma_{\rm reg}(u)$, then every sequence $\rho_k\downarrow 0$ admits a subsequence $\{\rho^\prime_k\}_{k\in\mathbb{N}}$ and a half space $P\subset\mathbb{R}^n$, with $0\in\partial P$, such that  the rescaled boundaries 
$$\partial E_k:= (\partial E\cap\partial^0G -x)/ \rho^\prime_k$$ 
converge locally uniformly to the hyperplane $\partial P$, i.e., for every compact set $K\subset \R^n$ and every $r>0$,
$$\partial E_k\cap K\subset \mathscr{T}_r(\partial P) \quad\text{and}\quad  \partial P\cap K\subset  \mathscr{T}_r(\partial E_k)$$
whenever $k$ is large enough.
\end{proposition}

\begin{proof}
By the very  definition of $\Sigma_{\rm reg}(u)$ and \eqref{identboundtopstrat}, if $x\in \Sigma_{\rm reg}(u)$, then there exists $\varphi_0\in  T_x(u)$ such that ${\rm dim}\,S(\varphi_0)=n-1$.  By Lemma~\ref{homo2203} and Remark \ref{remhalfspa}, we have $\boldsymbol{\Theta}_{u}(x)=\boldsymbol{\Theta}_{\varphi_0}(0)=\boldsymbol{\theta}_{n,s}$ as defined in  \eqref{defthetaplane}. 

Let $\rho_k\downarrow 0$ be an arbitrary sequence. By the results in Subsection \ref{secttangmap}, there exists a subsequence $\{\rho^\prime_k\}_{k\in\mathbb{N}}$ such that $u_{x,\rho^\prime_k}\to \varphi$ strongly in $H^1(B_r^+,|z|^a\de{\bf x})$ for every $r>0$, for some nonlocal stationary cone $\varphi\in T_x(u)$ satisfying $\boldsymbol{\Theta}_{\varphi}(0)=\boldsymbol{\Theta}_{u}(0)=\boldsymbol{\theta}_{n,s}$.  By Lemma~\ref{energhalfplane}, there is an open half-space $P\subset \mathbb{R}^n$, with $0\in \partial P$, such that  $\varphi=(\chi_P-\chi_{\mathbb{R}^n\setminus P})^\e$. Then the conclusion follows from Corollary \ref{Hausdcvbd}. 
\end{proof}

We are now ready to prove one of the main result of this section: the optimal estimate for the dimension of $\partial E_u\cap\Omega$ (here, ${\rm dim}_{\mathscr{M}} $ denotes the Minkowski dimension). 

\begin{theorem}\label{dimMink}
We have ${\rm dim}_{\mathscr{M}} (\partial E_u\cap\Omega^\prime) = n-1$ for every open subset $\Omega^\prime\subset \partial^0G$ such that $\overline{\Omega^\prime}\subset\partial^0G$ and $\partial E_u\cap\Omega^\prime\neq\emptyset$. In addition, ${\rm dim}_{\mathscr{H}}{\rm Sing}^j(u)\leq j$ for each $j\in\{1,\ldots,n-2\}$, and ${\rm Sing}^0(u)$ is countable. 
\end{theorem}

We will prove Theorem \ref{dimMink} usnig the abstract stratification principle of \cite{FMS}, originally introduced in \cite{CheegNab}. To fit the setting of \cite{FMS}, we first need to introduce some notations. 

For a radius $r>0$, we set 
\begin{equation}\label{defOmegar}
\Omega^{r}:=\big\{x\in\R^n: B^+_{2r}((x,0))\subset G\big\}\,.
\end{equation}
In what follows, we {\it fix  three constants} $r_0> 0$, $H_0\geq0$, and $\Lambda_0\geq 0$ such that 
\begin{equation}\label{controlcurv}
 \|f\|_{\dot W^{1,q}(\partial^0G)}\leq H_0\,, 
 \end{equation}
 and 
\begin{equation}\label{controldensit}
\sup\Big\{\boldsymbol{\Theta}_u(f,x,\rho) : x\in\Omega^{r_0}\,,\;0<\rho\leq r_0\Big\}\leq \Lambda_0\,. 
\end{equation}
Note that the supremum above is indeed finite  by \eqref{1912tangbd}, and for $0<\rho\leq r_0$, 
$$\boldsymbol{\Theta}_u(f,x,\rho)\leq \frac{1}{r_0^{n-2s}}{\bf E}(u,G)+\frac{{\bf c}_{n,q}b\,({\rm diam}\,\partial^0G)^{\theta_q}}{\theta_q}H_0\quad\forall x\in\Omega^{r_0} \,.$$

For each $j\in \{0,\ldots,n\}$, $\rho\in(0,r_0)$, $x_0\in\Omega^{r_0}$ and ${\bf x}_0=(x_0,0)$, we now introduce the function ${\bf d}_j(\cdot,x_0,\rho): H^1(B^+_\rho({\bf x}_0),|z|^a\de{\bf x})\to [0,\infty)$
defined by
$${\bf d}_j(v,x_0,\rho):=\inf\Big\{\|v_{x_0,\rho}-\varphi\|_{L^1(B_1^+)}: \varphi\in\mathscr{C}_j(\Lambda_0)\Big\} \,,$$
where $v_{x_0,\rho}({\bf x}):=v({\bf x}_0+\rho {\bf x})$. Note that the infimum above is well defined by Remark \ref{trace}, and it is always achieved by Lemma~\ref{compactcone}. Moreover, 
$${\bf d}_0(\cdot,x_0,\rho)\leq {\bf d}_1(\cdot,x_0,\rho)\leq \ldots\leq {\bf d}_{n}(\cdot,x_0,\rho)\,,$$
and
$${\bf d}_n(v,x_0,\rho):=\min\Big\{ \|v_{x_0,\rho}-1\|_{L^1(B_1^+)},\|v_{x_0,\rho}+1\|_{L^1(B_1^+)} \Big\}\,. $$
Observe that each functional ${\bf d}_j(\cdot,x_0,\rho)$ is a (rescaled) $L^1$-distance function, and consequently they are $\rho^{-n}$-Lipschitz functions with respect to the $L^1(B_\rho^+({\bf x}_0))$-norm. In particular, each functional ${\bf d}_j(\cdot,x_0,\rho)$ is continuous with respect to strong convergence in $L^1(B_\rho^+({\bf x}_0))$. 
\vskip3pt

In the terminology of \cite[Section 2.1]{FMS}, the functions $\boldsymbol{\Theta}_u(f,\cdot,\rho)$ and ${\bf d}_j(u,\cdot,\cdot)$ will play the roles of {\sl density function} and {\sl control functions},  respectively (thanks to Lemma \ref{monotformsurf}). We need to show that they satisfy the structural assumptions of  \cite[Section 2.2]{FMS}. This is the purpose of the following lemmas.

\begin{lemma}\label{gapdistlemma}
There exists a constant 
$$\delta_0(r_0)=\delta_0(r_0,H_0,\Lambda_0,b,n,s,q)\in(0,1)$$ 
(independent of $u$ and $f$) such that for every for every $x\in\Omega^{r_0}$ and $\rho\in(0,r_0)$, 
$$\boldsymbol{\Theta}_u(x)>0\quad\Longrightarrow\quad  {\bf d}_n(u,x,\rho) \geq\delta_0 \,.$$
\end{lemma}

\begin{proof}
Assume by contradiction that there exist sequences of functions $\{(u_k,f_k)\}_{k\in\mathbb{N}}$ solving \eqref{startequsurf}-\eqref{critnonlocexteq} and satisfying \eqref{controlcurv}-\eqref{controldensit}, points $\{x_k\}_{k\in\mathbb{N}}\subset\Omega^{r_0}$, and radii $\{\rho_k\}_{k\in\mathbb{N}}\subset (0,r_0)$ such that $\boldsymbol{\Theta}_{u_k}(x_k)>0$ and ${\bf d}_n(u_k,x_k,\rho_k) \leq 2^{-k}$. 

We continue with a general first step that we shall use again in the sequel. 
\vskip3pt
 
\noindent{\it Step 1, general compactness.} We consider the rescaled maps $\widetilde u_k:=(u_k)_{x_k,\rho_k}$ and $\widetilde f_k:=\rho_k^{2s}(f_k)_{x_k,\rho_k}$ as defined in \eqref{defrescmap}. By our choice of $\Lambda_0$, a simple change of variables yields   
$$\boldsymbol{\Theta}_{\widetilde u_k}(\widetilde f_k,0,1)\leq \Lambda_0 \quad\text{and}\quad \|\widetilde f_k\|_{\dot W^{1,q}(D_1)}\leq r_0^{\theta_q}H_0\,.$$
By Theorem \ref{compactminsurf}, we can find a (not relabeled) subsequence such that $\widetilde u_k\to u_*$ weakly in $H^1(B_1^+,|z|^a\de {\bf x})$ and strongly in $H^1(B_r^+,|z|^a\de {\bf x})$ for every $0<r<1$, and $\widetilde f_k\rightharpoonup f_*$ weakly in $W^{1,q}(D_1)$, where $(u_*,f_*)$ satisfies \eqref{eqcone}-\eqref{statinB1}. By Remark \ref{trace}, $\widetilde u_k\to u_*$ strongly in $L^1(B_1^+)$, so that 
$${\bf d}_j(\widetilde u_k,0,1)\to {\bf d}_j(u_*,0,1)\quad\text{for each $j\in\{0,\ldots,n\}$}\,.$$ 
In addition, by lower semicontinuity of ${\bf E}(\cdot, B_1^+)$ and Fatou's lemma, we have 
 \begin{equation}\label{lscthetw1q}
 \boldsymbol{\Theta}_{u}(f,0,1)\leq \liminf_{k\to\infty}\boldsymbol{\Theta}_{\widetilde u_k}(\widetilde f_k,0,1)\leq\Lambda_0\quad\text{and}\quad \|f\|_{\dot W^{1,q}(D_1)}\leq r_0^{\theta_q} H_0\,.
 \end{equation}
\vskip3pt

\noindent{\it Step 2, conclusion.} Since ${\bf d}_n(\widetilde u_k,0,1)\leq 2^{-k}$, we have ${\bf d}_n(u_*,0,1)=0$. In other words, either $u_*=1$ or $u_*=-1$, and consequently 
$\boldsymbol{\Theta}_{u_*}(0)=0$. On the other hand, by Corollary~\ref{openset}, $\boldsymbol{\Theta}_{\widetilde u_k}(0)=\boldsymbol{\Theta}_{u_k}(0)\geq {\boldsymbol{\eta}}_0>0$. 
Then Lemma \ref{updenstheta} yields $\boldsymbol{\Theta}_{u_*}(0)\geq \limsup_k \boldsymbol{\Theta}_{u_k}(0)>0$, which contradicts $\boldsymbol{\Theta}_{u_*}(0)=0$. 
\end{proof}

\begin{lemma}\label{lemme2stratsurf}
For every $\delta>0$, there exist constants 
$$\eta_1(\delta,r_0)=\eta_1(\delta,r_0,H_0,\Lambda_0,b,n,s,q)\in (0,1/4)$$ 
and   
$$\lambda_1(\delta,r_0)=\lambda_1(\delta,r_0,H_0,\Lambda_0,b,n,s,q)\in (0,1/4)$$ 
(independent of $u$ and $f$) such that for every $x\in\Omega^{r_0}$ and $\rho\in(0,r_0)$,   
$$\boldsymbol{\Theta}_u(f,x,\rho)-\boldsymbol{\Theta}_u(f,x,\lambda_1\rho) \leq \eta_1 \quad \Longrightarrow\quad {\bf d}_0(u,x,\rho) \leq\delta \,.$$ 
\end{lemma}

\begin{proof}
Assume by contradiction that for some $\delta>0$, there exist sequences of functions $\{(u_k,f_k)\}_{k\in\mathbb{N}}$ solving \eqref{startequsurf}-\eqref{critnonlocexteq} and satisfying \eqref{controlcurv}-\eqref{controldensit}, points $\{x_k\}_{k\in\mathbb{N}}\subset\Omega^{r_0}$, and radii $\{\rho_k\}_{k\in\mathbb{N}}\subset (0,r_0)$ such that  
$$\boldsymbol{\Theta}_{u_k}(f_k,x_k,\rho_k)-\boldsymbol{\Theta}_{u_k}(f_k,x_k,\lambda_k\rho_k) \leq 2^{-k} \quad \text{and}\quad {\bf d}_0(u_k,x_k,\rho_k) \geq\delta \,,$$
where $\lambda_k\to0$ as $k\to\infty$. 
We consider the rescaled maps $\widetilde u_k:=(u_k)_{x_k,\rho_k}$ and $\widetilde f_k:=\rho_k^{2s}(f_k)_{x_k,\rho_k}$ as defined in \eqref{defrescmap}, so that 
$$\boldsymbol{\Theta}_{\widetilde u_k}(\widetilde f_k,0,1)-\boldsymbol{\Theta}_{\widetilde u_k}(\widetilde f_k,0,\lambda_k) \leq 2^{-k} \quad \text{and}\quad {\bf d}_0(\widetilde u_k,0,1) \geq\delta \,. $$ 
Then we apply Step 1 in the proof of Lemma \ref{gapdistlemma} to find a (not relabeled) sequence along which $\widetilde u_k$ and $\widetilde f_k$ converge to $u_*$ and $f_*$, respectively. As consequence of the established convergences, we first deduce that   ${\bf d}_0(u_*,0,1)\geq \delta$. 

On the other hand, by Lemma \ref{monotformsurf} we can estimate for $0<r<1$ and $k$ large enough  (in such a way that $\lambda_k<r$), 
\begin{multline*}
\boldsymbol{\Theta}_{\widetilde u_k}(\widetilde f_k,0,1)-  \frac{1}{r^{n-2s}}{\bf E}(\widetilde u_k,B_r^+)-\frac{{\bf c}_{n,q}b\, r_0^{\theta_q}}{\theta_q}H_0r^{\theta_q}\\
\leq \boldsymbol{\Theta}_{\widetilde u_k}(\widetilde f_k,0,1)-\boldsymbol{\Theta}_{\widetilde u_k}(\widetilde f_k,0,r) \leq 2^{-k}\,. 
\end{multline*}
Using \eqref{lscthetw1q} and the strong convergence of $\widetilde u_k$ in $H^1(B_r^+,|z|^a\de{\bf x})$, we can let $k\to\infty$ to deduce that 
$$  \boldsymbol{\Theta}_{u_*}(f_*,0,1)-  \frac{1}{r^{n-2s}}{\bf E}(u_*,B_r^+)\leq \frac{{\bf c}_{n,q} b\,r_0^{\theta_q}}{\theta_q}H_0r^{\theta_q}\,. $$
Letting $r\to 0$, we infer from  Corollary \ref{directcorlmonot} that  $\boldsymbol{\Theta}_{u_*}(f_*,0,1)=\boldsymbol{\Theta}_{u_*}(0)$. By Lemma~\ref{lemhomogsol}, $f_*=0$ and $u_*$ is a nonlocal  stationary cone. Moreover, \eqref{lscthetw1q} yields the estimate $\boldsymbol{\Theta}_{u_*}(0)\leq \Lambda_0$, so that $u_*\in\mathscr{C}_0(\Lambda_0)$. Hence ${\bf d}_0(u_*,0,1)=0$, which contradicts the previous estimate ${\bf d}_0(u_*,0,1)\geq \delta$. 
\end{proof}

\begin{lemma}\label{condiistrat}
For every $\delta,\tau\in(0,1)$, there exists a constant 
$$\eta_2(\delta,\tau,r_0)=\eta_2(\delta,\tau,r_0,H_0,\Lambda_0,b,n,s,q)\in (0,\delta\,]$$
(independent of $u$ and $f$) such that the following holds for every $\rho\in(0,r_0/5)$ and $x\in\Omega^{r_0}$. If 
$${\bf d}_j(u,x,4\rho) \leq\eta_2\quad \text{and}\quad {\bf d}_{j+1}(u,x,4\rho)\geq \delta\,,$$
hold for some $j\in\{0,\ldots,n-1\}$, then there exists a $j$-dimensional linear subspace $V\subset\R^n$ for which
$${\bf d}_0(u,y,4\rho)> \eta_2 \quad \forall y\in D_{\rho}(x)\setminus \mathscr{T}_{\tau \rho}(x+V)\,.$$ 
\end{lemma}

\begin{proof}
The proof is again by contradiction. Assume that for some $\delta,\tau\in(0,1)$ and some $j\in\{0,\ldots,n-1\}$, 
there exist sequences of functions $\{(u_k,f_k)\}_{k\in\mathbb{N}}$ solving \eqref{startequsurf}-\eqref{critnonlocexteq} and satisfying \eqref{controlcurv}-\eqref{controldensit}, points $\{x_k\}_{k\in\mathbb{N}}\subset\Omega^{r_0}$, and radii $\{\rho_k\}_{k\in\mathbb{N}}\subset (0,r_0/5)$ such that 
$${\bf d}_j(u_k,x_k,4\rho_k) \leq 2^{-k}\quad \text{and}\quad {\bf d}_{j+1}(u_k,x_k,4\rho_k)\geq \delta\,,$$
and such that the conclusion of the lemma does not hold. Now we consider  the rescaled functions $\widetilde u_k:=(u_k)_{x_k,4\rho_k}$ and $\widetilde f_k:=(4\rho_k)^{2s}(f_k)_{x_k,4\rho_k}$.
\vskip3pt

\noindent{\it Step 1.} For each $k$, we select $\varphi_k\in\mathscr{C}_j(\Lambda_0)$ such that $\|\widetilde u_k-\varphi_k\|_{L^1(B_1^+)}\leq 2^{-k}$ (which is possible by Lemma~\ref{compactcone}). Since 
\begin{equation}\label{20147fin}
{\bf d}_{j+1}(\varphi_k,0,1)\geq {\bf d}_{j+1}(\widetilde u_k,0,1)-\|\widetilde u_k-\varphi_k\|_{L^1(B_1^+)}\geq \delta-2^{-k}\,, 
\end{equation}
we infer that ${\rm dim}\,S(\varphi_k)=j$ for $k$ large enough. Extracting a (not relabeled) subsequence and rotating coordinates if necessary, we may assume that $S(\varphi_k)=V$ for some fixed linear subspace $V$ of dimension $j$.  Then, by Lemma \ref{compactcone} we can find a further (not relabeled) subsequence such that $\varphi_k\to \varphi$ strongly in $H^1(B_r^+,|z|^a\de {\bf x})$ for every $r>0$ and some $\varphi\in \mathscr{C}_j(\Lambda_0)$. In particular,  
$$\boldsymbol{\Theta}_\varphi(0)=\boldsymbol{\Theta}_\varphi(0,0,1)=\lim_{k\to\infty} \boldsymbol{\Theta}_{\varphi_k}(0,0,1)=\lim_{k\to\infty} \boldsymbol{\Theta}_{\varphi_k}(0)\,. $$
On the other hand, by Lemma \ref{updenstheta},
$$\boldsymbol{\Theta}_\varphi(y)\geq \lim_{k\to\infty} \boldsymbol{\Theta}_{\varphi_k}(y)=\lim_{k\to\infty} \boldsymbol{\Theta}_{\varphi_k}(0)= \boldsymbol{\Theta}_\varphi(0)\quad\forall y\in V\,. $$
Therefore, $V\subset S(\varphi)$ by Lemma \ref{lemSphi}. But letting $k\to \infty$ in \eqref{20147fin}, we deduce that  ${\bf d}_{j+1}(\varphi,0,1)\geq \delta$, and thus $S(\varphi)=V$. 
Since the conclusion of the lemma does not hold, for each $k$ we can find a point $y_k\in D_{1/4}\setminus \mathscr{T}_{\tau/4}(V)$ such that 
${\bf d}_0(\widetilde u_k,y_k,1)\to 0$ as $k\to\infty$.  
\vskip3pt

\noindent{\it Step 2.} Consider the translated function $\widehat u_k:=(\widetilde u_k)_{y_k,1}$, and select 
 $\psi_k\in\mathscr{C}_0(\Lambda_0)$ such that $\|\widehat u_k-\psi_k\|_{L^1(B_1^+)}= {\bf d}_0(\widetilde u_k,y_k,1)\to 0$. By Lemma \ref{compactcone} and Remark \ref{trace}, we can find a further (not relabeled) subsequence such that $\psi_k\to \psi$ strongly in $L^1(B_1^+)$ for  some $\psi\in \mathscr{C}_0(\Lambda_0)$. Then $\widehat u_k\to \psi$ strongly in $L^1(B_1^+)$. Now we extract a further (not relabeled) subsequence such that $y_k\to y_*$ for some $y_*\in \overline D_{1/4}\setminus \mathscr{T}_{\tau/4}(V)$. Observe that 
\begin{multline*}
\|\psi-\varphi_{y_k,1}\|_{L^1(B^+_{3/4})} \leq \|\psi-\widehat u_k\|_{L^1(B^+_{3/4})}+ \|(\widetilde u_k)_{y_k,1}-\varphi_{y_k,1}\|_{L^1(B^+_{3/4})}\\
 \leq  \|\psi-\widehat u_k\|_{L^1(B^+_{1})}+ \|\widetilde u_k-\varphi\|_{L^1(B^+_{1})}\,.
\end{multline*}
By continuity of translations in $L^1$, and since $\widetilde u_k\to \varphi$, we infer that 
$$\|\psi-\varphi_{y_*,1}\|_{L^1(B^+_{3/4})}=\lim_{k\to\infty} \|\psi-\varphi_{y_k,1}\|_{L^1(B^+_{3/4})} =0\,.$$
In other words, $\psi=\varphi_{y_*,1}$ on $B^+_{3/4}$.  As a consequence, setting ${\bf y}_*:=(y_*,0)$, for every ${\bf x}\in B^+_{1/2}$ and $t\in(0,1)$, 
$$\varphi\big({\bf x}+t({\bf y}_*-{\bf x})\big) =\psi\big((1-t){\bf x}+(t-1){\bf y}_*\big)=\psi({\bf y}_*-{\bf x})\,.$$ 
Differentiating first this identity with respect to $t$, and then letting $t\to 0$, we discover that $0=({\bf y}_*-{\bf x})\cdot\nabla\varphi({\bf x})={\bf y}_*\cdot\nabla\varphi({\bf x})$ for every ${\bf x}\in B^+_{1/2}$. By homogeneity of $\varphi$, it implies that ${\bf y}_*\cdot\nabla\varphi({\bf x})=0$ for every ${\bf x}\in\R^{n+1}_+$. Arguing as in the proof of Lemma \ref{lemSphi}, we deduce that 
$y_*\in S(\varphi)=V$, which contradicts the fact that  $y_*\in \overline D_{1/4}\setminus \mathscr{T}_{\tau/4}(V)$. 
\end{proof}

We finally prove the following corollary whose importance will be revealed in Section \ref{imprsect}. 

\begin{corollary}\label{corstrat3surf}
For every $\delta,\tau\in(0,1)$, there exists a constant 
$$\eta_3(\delta,\tau,r_0)=\eta_3(\delta,\tau,r_0,H_0,\Lambda_0,b,n,s,q)\in (0,\delta]$$
(independent of $u$ and $f$) such that for every $\rho\in(0,r_0/5]$ and $x\in\Omega^{r_0}$, the conditions
$${\bf d}_0(u,x,4\rho) \leq\eta_3\quad \text{and}\quad {\bf d}_{n}(u,x,4\rho)\geq \delta\,,$$
imply the existence of a linear subspace $V\subset\R^n$, with ${\rm dim}\,V\leq n-1$, for which
$${\bf d}_0(u,y,4\rho)> \eta_3 \quad \forall y\in D_{\rho}(x)\setminus \mathscr{T}_{\tau \rho}(x+V)\,.$$ 
\end{corollary}

\begin{proof}
We argue by induction on the dimension $j\in\{1,\ldots,n\}$ assuming that there exists a constant $\eta_{*,j}(\delta,\tau,r_0)\in (0,\delta]$ 
such that for every $\rho\in(0,r_0/5]$ and $x\in\Omega^{r_0}$, the conditions
$${\bf d}_0(u,x,4\rho) \leq\eta_{*,j}\quad \text{and}\quad {\bf d}_{j}(u,x,4\rho)\geq \delta\,,$$
imply the existence of a linear subspace $V$, with ${\rm dim}\,V\leq j-1$, for which
$${\bf d}_0(u,y,4\rho)> \eta_{*,j} \quad \forall y\in D_{\rho}(x)\setminus \mathscr{T}_{\tau \rho}(x+V)\,.$$ 
By Lemma \ref{condiistrat} this property holds for $j=1$ with $\eta_{*,1}(\delta,\tau)=\eta_2(\delta,\tau)$. 

Now we assume that the property holds at step $j$, and we prove that it also holds at step $j+1$. To this purpose, we choose 
$$ \eta_{*,j+1}(\delta,\tau,r_0):=\eta_{*,j}\big( \eta_{*,j}(\delta,\tau,r_0) ,\tau,r_0\big)\in \big(0, \eta_{*,j}(\delta,\tau)\big]\subset (0,\delta] \,.$$
Then we distinguish two cases. 

\noindent {\it Case 1).} If ${\bf d}_{j}(u,x,4\rho)\leq \eta_{*,j}$, then ${\bf d}_{j}(u,x,4\rho)\leq \eta_{2}$ and 
we can apply Lemma \ref{condiistrat}  to find the required linear subspace $V$ of dimension $j=(j+1)-1$. 

\noindent {\it Case 2).} If ${\bf d}_{j}(u,x,4\rho)>\eta_{*,j}$, then we apply the induction hypothesis to find the required linear subspace $V$ of dimension less than $j-1$. 

Now the conclusion follows for $\eta_3(\delta,\tau,r_0):= \eta_{*,n}(\delta,\tau,r_0)$.
\end{proof}

We now introduce the so-called singular strata of $u$. For $\delta\in(0,1)$, $0<r\leq r_0$, and $j\in\{0,\ldots,n-1\}$, we set 
$$\mathcal{S}^j_{r_0,r,\delta}(u):=\Big\{x\in\Omega^{r_0}: \boldsymbol{\Theta}_u(x)>0\text{ and }{\bf d}_{j+1}(u,x,\rho)\geq\delta\text{ for all } r\leq\rho\leq r_0\Big\}\,,$$
$$\mathcal{S}^j_{r_0,\delta}(u):=\bigcap_{0<r\leq r_0} \mathcal{S}^j_{r_0,r,\delta}(u)\quad\text{and}\quad \mathcal{S}^j_{r_0}(u):=\bigcup_{0<\delta<1}\mathcal{S}^j_{r_0,\delta}(u)\,.$$
According to \cite{FMS}, we have the following result. 

\begin{theorem}\label{volsingsrat}
For every $\kappa_0\in(0,1)$, there exists a constant 
$$C=C(\kappa_0,r_0,H_0,\Lambda_0,b,n,s,q)>0$$ 
such that 
\begin{equation}\label{pff1710}
\mathscr{L}^n\Big(\mathscr{T}_r\big(\mathcal{S}^{n-1}_{r_0}(u)\big)\Big)\leq Cr^{1-\kappa_0}  \quad\forall r\in(0,r_0)\,. 
\end{equation}
In addition, ${\rm dim}_{\mathscr{H}}\big( \mathcal{S}^j_{r_0}(u)\big)\leq j$ for each $j\in\{1,\ldots,n-2\}$, and $ \mathcal{S}^0_{r_0}(u)$ is countable. 
\end{theorem}

\begin{proof}
By Lemma \ref{lemme2stratsurf} and Lemma \ref{condiistrat}, the functions $\boldsymbol{\Theta}_u(f,\cdot,\cdot)$ and ${\bf d}_j(u,\cdot,\cdot)$ satisfy the assumptions in \cite[Section 2.2]{FMS}. Then the dimension estimate on  $\mathcal{S}^j_{r_0}(u)$ for each $j\in\{1,\ldots,n-2\}$, and the fact that $ \mathcal{S}^0_{r_0}(u)$ is countable, follow from \cite[Theorem 2.3]{FMS}. 

According to Lemma \ref{gapdistlemma}, $\mathcal{S}^{n-1}_{r_0,\delta}(u)=\mathcal{S}^{n-1}_{r_0,\delta_0(r_0)}(u)$ for every $\delta\in(0,\delta_0(r_0)]$. Since the sets $\mathcal{S}^{n-1}_{r_0,\delta}(u)$ are decreasing with respect to $\delta$, we deduce that $\mathcal{S}^{n-1}_{r_0}(u)=\mathcal{S}^{n-1}_{r_0,\delta_0(r_0)}(u)$. Then, estimate \eqref{pff1710} follows from \cite[Theorem 2.2]{FMS}. 
\end{proof}

\begin{proof}[Proof of Theorem \ref{dimMink}]
We  choose $r_0>0$ in such a way that $\Omega^\prime\subset\Omega^{r_0}$. By Corollary \ref{openset} and Lemma \ref{gapdistlemma}, we have $\partial E_u\cap\Omega^\prime\subset \mathcal{S}^{n-1}_{r_0}(u)$. According to estimate \eqref{pff1710}, for every $\alpha\in(0,1)$ there exists a constant $C=C(\alpha,r_0)$ such that 
\begin{equation}
\mathscr{L}^n\big(\mathscr{T}_r(\partial E_u\cap\Omega^\prime)\big)\leq Cr^\alpha \quad\forall r\in(0,r_0)\,.
\end{equation}
Hence, 
$$\limsup_{r\downarrow 0}\left( n-\frac{\log\Big(\mathscr{L}^n\big(\mathscr{T}_r(\partial E_u\cap\Omega^\prime)\big)\Big)}{\log r}\right)\leq n-\alpha\quad\forall\alpha\in(0,1)\,, $$
and we obtain that the upper Minkowski dimension $\overline{\rm dim}_{\mathscr{M}}(\partial E_u\cap\Omega^\prime)$ is less than $n-1$. On the other hand, 
since $E_u\cap \Omega^\prime$ is a not empty open subset of $\Omega^\prime$, distinct from $\Omega^\prime$, we have ${\rm dim}_{\mathscr{H}}(\partial E_u\cap \Omega^\prime)\geq n-1$.  Since the lower Minkowski dimension $\underline{\rm dim}_{\mathscr{M}}(\partial E_u\cap\Omega^\prime)$ is greater than than the Hausdorff dimension, we conclude that ${\rm dim}_{\mathscr{M}}(\partial E_u\cap\Omega^\prime)=n-1$. 

To complete the proof, we show that 
\begin{equation}\label{identstrata}
{\rm Sing}^j(u)\cap\Omega^{r_0}\subset \mathcal{S}^j_{r_0}(u)\quad \text{for each } j\in\{0,\ldots,n-2\}\,,
\end{equation}
so that the conclusion follows from Theorem \ref{volsingsrat} (letting $r_0\to0$ along a decreasing sequence). To prove \eqref{identstrata}, 
we argue by contradiction assuming that there exists a point $x\in {\rm Sing}^j(u)\cap\Omega^{r_0}\setminus \mathcal{S}^j_{r_0}(u)$. Then, $x\not\in \mathcal{S}^{j}_{r_0,2^{-k}}(u)$ for every $k\in\mathbb{N}$. Hence, for each $k\in\mathbb{N}$, there exists a radius $r_k\in(0,r_0]$ such that $x\not\in  \mathcal{S}^{j}_{r_0,r_k,2^{-k}}(u)$, and therefore a radius $\rho_k\in[r_k,r_0]$ such that ${\bf d}_{j+1}(u,x,\rho_k)<2^{-k}$. Now we extract a (not relabeled) subsequence such that $\rho_k\to \rho_*$ for some $\rho_*\in[0,r_0]$. We distinguish the two following cases: 
\vskip3pt

\noindent{\it Case 1).} If $\rho_*=0$, then we can extract a further subsequence such that $(u)_{x,\rho_k}\to\varphi$ strongly in $H^1(B_1^+,|z|^a\de{\bf x})$ for some $\varphi\in T_x(u)$ (by Lemma \ref{tangmapstat}). In addition, 
$$ {\bf d}_{j+1}(\varphi,0,1)=\lim_{k\to\infty} {\bf d}_{j+1}(u_{x,\rho_k},0,1)=\lim_{k\to\infty} {\bf d}_{j+1}(u,x,\rho_k)=0\,,$$
so that $\varphi\in\mathscr{C}_{j+1}(\Lambda_0)$. Then ${\rm dim}\,S(\varphi)\geq j+1$ which contradicts $x\in {\rm Sing}^j(u)$. 
\vskip3pt

\noindent{\it Case 2).} If $\rho_*>0$,  then 
$${\bf d}_{j+1}(u_{x,\rho_*},0,1)=\lim_{k\to\infty} {\bf d}_{j+1}(u_{x,\rho_k},0,1)=\lim_{k\to\infty} {\bf d}_{j+1}(u,x,\rho_k)=0\,.$$
Hence there exists $\varphi\in\mathscr{C}_{j+1}(\Lambda_0)$ such that $u_{x,\rho_*}=\varphi$ on  $B^+_{1}$. Clearly, it implies that $T_x(u)=\{\varphi\}$, which contradicts  $x\in {\rm Sing}^j(u)$ as in Case 1). 
\end{proof}

%%%%%%%%%%%%%%%%%%%%%%%%%%%%%%%%%%%%%%%%%%%%%%%%%%%%%%%%

\subsection{Application to the prescribed nonlocal mean curvature equation}\label{complrestprescrNMMC}

%%%%%%%%%%%%%%%%%%%%%%%%%%%%%%%%%%%%%%%%%%%%%%%%%%%%%%%%

In this subsection, we apply the previous results to a weak solution $E\subset \R^n$ of the prescribed nonlocal $2s$-mean curvature equation \eqref{eqcurvsec6}. In order to do so, we may consider an increasing sequence of admissible bounded open sets $\{G_l\}_{l\in\mathbb{N}}$ such that $\overline{\partial^0G_l}\subset \Omega$, $\bigcup_l G_l=\R^{n+1}_+$, and $\bigcup_l\partial^0G_l=\Omega$. In view of \eqref{campdev2}-\eqref{stateeq1458}, we can apply to the  extended function $(v_E)^\e$ the different results from Subsection \ref{subres1} to Subsection \ref{subsectstrat} to reach the following main conclusions:

\begin{enumerate}
\item The set $E\cap\Omega$ is essentially open. More precisely, $\mathscr{L}^n\big((E\cap\Omega)\triangle E_{(v_E)^\e}\big)=0$ where $E_{(v_E)^\e}\subset\Omega$ is the open set provided by Corollary \ref{openset}. From now on, we will  identify the set $E\cap\Omega$ with $E_{(v_E)^\e}$. 
\item ${\rm dim}_{\mathscr{M}}(\partial E\cap\Omega^\prime)\leq n-1$ for every open subset $\Omega^\prime$ such that $\overline{\Omega^\prime}\subset\Omega$ (with equality if if $\partial E\cap\Omega^\prime$ is not empty).
\item There is a subset $\Sigma_{\rm sing}\subset \partial E\cap\Omega$ with ${\rm dim}_{\mathscr{H}}\Sigma_{\rm sing}\leq n-2$ (countable if $n=2$) such that the following holds: 
if $x_0\in  (\partial E\cap\Omega)\setminus \Sigma_{\rm sing}$, then every sequence $\rho_k\downarrow 0$ admits a (not relabeled) subsequence such that 
\begin{itemize}
\item $E_k:=(E-x_0)/\rho_k\to P$ in $L^1_{\rm loc}(\R^n)$ for some half-space $P\subset \R^n$, $0\in \partial P$;  
\item $\partial E_k$ converges locally uniformly to the hyperplane $\partial P$, i.e.,  for every compact set $K\subset \R^n$ and every $r>0$, $\partial E_k\cap K\subset \mathscr{T}_r(\partial P)$ and $ \partial P\cap K\subset  \mathscr{T}_r(\partial E_k)$ whenever $k$ is large enough.
\end{itemize}
\end{enumerate}

\begin{remark}
In the case of minimizing nonlocal minimal surfaces (i.e., solutions of \eqref{defminimizingnonlocminsurf}), or  minimizing solutions of \eqref{eqcurvsec6}  for $f\not=0$ (i.e., solutions of \eqref{nonlocalminf}), the set $\Sigma_{\rm sing}$ is a closed subset of $\partial E\cap\Omega$, and  $(\partial E\cap\Omega)\setminus \Sigma_{\rm sing}$ is locally the graph of a $C^{1,\alpha}$ function (at least), see \cite{CRS,CapGui}. The minimality condition allows one to prove that equation \eqref{eqcurvsec6} holds in a suitable viscosity sense. This is a key point to prove the {\sl improvement of flatness} of \cite{CRS}. Combined with property (3) above, it  leads to the regularity at points of $(\partial E\cap\Omega)\setminus \Sigma_{\rm sing}$. The improvement of flatness property also implies the existence of a constant $\delta>0$ such that $ \boldsymbol{\Theta}_\varphi(0)\geq \boldsymbol{\theta}_{n,s}+\delta$ for every {\sl minimizing} nonlocal cone $\varphi$ such that ${\rm dim}\,S(\varphi)\leq n-2$, and the closeness of  $\Sigma_{\rm sing}$ can be deduced from the upper semicontinuity of the density function $ \boldsymbol{\Theta}$. In the stationary case, it remains unclear whether or not an improvement of flatness holds. It is even unclear if this there an energy gap between hyperplanes and other nonlocal stationary cones. 
\end{remark}

\begin{remark}
In the minimizing case, we have the improved estimate ${\rm dim}_{\mathscr{H}}\Sigma_{\rm sing}\leq n-3$ as shown in \cite{SV1bis}. In the stationary case,  
the estimate ${\rm dim}_{\mathscr{H}}\Sigma_{\rm sing}\leq n-2$ is  optimal. Indeed, in the plane $\R^2$, the boundary of the open set $E:=\{x_1x_2>0\}$ is an entire stationary nonlocal minimal surface with $\Sigma_{\rm sing}=\{0\}$.
\end{remark}

Our objective for the rest of this subsection is to show that the Minkowski dimension estimate on $\partial E\cap\Omega$ leads to the following higher regularity result. 

\begin{theorem}\label{improvperim}
For every $s^\prime\in(0,1/2)$ and every open subset $\Omega^\prime\subset\Omega$ such that $\overline{\Omega^\prime}\subset \Omega$, 
$$P_{2s^\prime}(E,\Omega^\prime)<\infty\,.$$
\end{theorem}

The proof of Theorem \ref{improvperim} (postponed to this end of the subsection) rests on a general regularity result, which might be of independent interest.  

\begin{proposition}\label{keypropimpr}
Let $v\in \widehat H^s(\Omega)$ be such that $v\in L^{\infty}_{\rm loc}(\Omega)$ and $(-\Delta)^s v\in L^{\bar p}_{\rm loc}(\Omega)$ for  some exponent $\bar p\in(1,\infty)$. Then, for every $s^\prime\in(0,s)$ and every open subsets $\Omega^\prime,\Omega^{\prime\prime}$ of $\Omega$ such that $\overline{\Omega^{\prime\prime}}\subset \Omega^\prime$ and $\overline{\Omega^\prime}\subset \Omega$, 
\begin{equation}\label{sameditroptard}
\left(\iint_{\Omega^{\prime\prime}\times \Omega^{\prime\prime}} \frac{|v(x)-v(y)|^{\bar p}}{|x-y|^{n+2s^\prime\bar p}}\,\de x\de y\right)^{1/\bar p}\leq C\Big(\|(-\Delta)^s v\|_{L^{\bar p}(\Omega^\prime)}+\|v\|_{L^{\infty}(\Omega^\prime)}\Big)\,,
\end{equation}
for some constant $C=C(n,s,s^\prime,\bar p,\Omega^{\prime},\Omega^{\prime\prime})$ independent of $v$.  
\end{proposition}

\begin{proof}
{\it Step 1.} We fix a cut-off function $\zeta\in C_c^\infty(\Omega^\prime;[0,1])$ such that $\zeta=1$ in $\Omega^{\prime\prime}$. Define $w:=\zeta v$, and notice that  $w\in H^s_{00}(\Omega^\prime)\cap L^\infty(\Omega^\prime)$. In particular, $(-\Delta)^sw\in H^{-s}(\R^n)$. The objective of this first step is to show that $(-\Delta)^sw\in L^{\bar p}(\R^n)$ with 
\begin{equation}\label{samed1620}
 \|(-\Delta)^sw\|_{L^{\bar p}(\R^n)}\leq C\big(\|(-\Delta)^s v\|_{L^{\bar p}(\Omega^\prime)}+\|v\|_{L^\infty(\Omega^\prime)}\big)\,,
\end{equation}
for some constant $C$ independent of $v$. 

We start writing for $\varphi\in\mathscr{D}(\R^n)$,
\begin{multline*}
\langle (-\Delta)^sw,\varphi\rangle=\frac{\gamma_{n,s}}{2}\iint_{\Omega^\prime\times \Omega^\prime}\frac{\big(w(x)-w(y)\big)(\varphi(x)-\varphi(y))}{|x-y|^{n+2s}}\,\de x\de y\\
+ \gamma_{n,s}\iint_{\Omega^\prime\times (\R^n\setminus \Omega^\prime)}\frac{\big(w(x)-w(y)\big)(\varphi(x)-\varphi(y))}{|x-y|^{n+2s}}\,\de x\de y\,.
\end{multline*}
Since ${\rm spt}\,w\subset \Omega^\prime$, we have 
$$\langle (-\Delta)^sw,\varphi\rangle=\frac{\gamma_{n,s}}{2}\iint_{\Omega^\prime\times \Omega^\prime}\frac{\big(w(x)-w(y)\big)(\varphi(x)-\varphi(y))}{|x-y|^{n+2s}}\,\de x\de y+\int_{\R^n}g_1\varphi\,\de x\,,$$
where 
$$g_1(x):= \gamma_{n,s}\chi_{\Omega^\prime}(x)\zeta(x)v(x)\int_{\R^n\setminus \Omega^\prime}\frac{\de y}{|x-y|^{n+2s}}-\gamma_{n,s}\chi_{\R^n\setminus \Omega^\prime}(x)\int_{\Omega^\prime}\frac{\zeta(y)v(y)}{|x-y|^{n+2s}}\,\de y\,,$$
and $g_1\in L^{\bar p}(\R^n)\cap L^\infty(\R^n)$. Now we write 
\begin{multline*}
\big(w(x)-w(y)\big)\big(\varphi(x)-\varphi(y)\big)=\big(v(x)-v(y)\big)\big(\zeta(x)\varphi(x)-\zeta(y)\varphi(y)\big)\\ 
+v(y)\big(\zeta(x)-\zeta(y)\big)\varphi(x)-v(x)\big(\zeta(x)-\zeta(y)\big)\varphi(y)
\end{multline*}
to realize that 
\begin{multline*}
\frac{\gamma_{n,s}}{2}\iint_{\Omega^\prime\times\Omega^\prime}\frac{\big(w(x)-w(y)\big)(\varphi(x)-\varphi(y))}{|x-y|^{n+2s}}\,\de x\de y\\
= \frac{\gamma_{n,s}}{2}\iint_{\Omega^\prime\times\Omega^\prime}\frac{\big(v(x)-v(y)\big)(\zeta(x)\varphi(x)-\zeta(y)\varphi(y))}{|x-y|^{n+2s}}\,\de x\de y\\
+\gamma_{n,s}\iint_{\Omega^\prime\times\Omega^\prime}\frac{v(y)\big(\zeta(x)-\zeta(y)\big)\varphi(x)}{|x-y|^{n+2s}}\,\de x\de y\,.
\end{multline*}
Consquently,
\begin{multline*}
\langle (-\Delta)^sw,\varphi\rangle= \frac{\gamma_{n,s}}{2}\iint_{\Omega^\prime\times \Omega^\prime}\frac{\big(v(x)-v(y)\big)(\zeta(x)\varphi(x)-\zeta(y)\varphi(y))}{|x-y|^{n+2s}}\,\de x\de y\\
+\int_{\R^n}(g_1+g_2)\varphi\,\de x\,,
\end{multline*}
where 
$$g_2(x):= \gamma_{n,s}\chi_{\Omega^\prime}(x)\int_{\Omega^\prime}\frac{v(y)\big(\zeta(x)-\zeta(y)\big)}{|x-y|^{n+2s}}\,\de y\,, $$
and $g_2\in L^{\bar p}(\R^n)\cap L^\infty(\R^n)$.  Since $\zeta\varphi\in\mathscr{D}(\Omega^\prime)$, we have
\begin{multline*}
\langle (-\Delta)^sw,\varphi\rangle=\langle (-\Delta)^sv,\zeta\varphi\rangle_{\Omega^\prime} -\gamma_{n,s}\iint_{\Omega^\prime\times(\R^n\setminus \Omega^\prime)}\frac{\big(v(x)-v(y)\big)\zeta(x)\varphi(x)}{|x-y|^{n+2s}}\,\de x\de y\\
+\int_{\R^n}(g_1+g_2)\varphi\,\de x\,,
\end{multline*}
so that 
$$\langle (-\Delta)^sw,\varphi\rangle=\langle (-\Delta)^sv,\zeta\varphi\rangle_{\Omega^\prime}+\int_{\R^n}(g_1+g_2+g_3)\varphi\,\de x\,,$$
where 
$$g_3(x):=-\gamma_{n,s}\zeta(x)\int_{\R^n\setminus \Omega^\prime}\frac{v(x)-v(y)}{|x-y|^{n+2s}}\,\de y\,, $$
and $g_3\in L^{\bar p}(\R^n)\cap L^\infty(\R^n)$ (recall that ${\rm spt}\,\zeta\subset \Omega^\prime$). By assumption, there exists $g_4\in L^{\bar p}(\Omega^\prime)$ such that $\langle (-\Delta)^sv,\psi\rangle_{\Omega^\prime}=\int_{\Omega^\prime}g_4\psi\,\de x$ for all $\psi\in\mathscr{D}(\Omega^\prime)$. Extending $g_4$ by zero outside $\Omega^\prime$, we conclude that 
$$\langle (-\Delta)^sw,\varphi\rangle=\int_{\R^n}g\varphi\,\de x\,, $$
with $g:=g_1+g_2+g_3+\zeta g_4\in L^{\bar p}(\R^n)$. Clearly, $\|g\|_{L^{\bar p}(\R^n)}\leq C\big(\|(-\Delta v\|_{L^{\bar p}(\Omega^\prime)}+\|v\|_{L^\infty(\Omega^\prime)}\big)$ for some constant $C$ independent of $v$,  and \eqref{samed1620} is proved.  
\vskip3pt

\noindent{\it Step 2.} We now claim that $(I-\Delta)^sw\in L^{\bar p}(\R^n)$ with 
\begin{equation}\label{dimmidi}
\big\|(I-\Delta)^sw\big\|_{L^{\bar p}(\R^n)}\leq C\big(\|(-\Delta v\|_{L^{\bar p}(\Omega^\prime)}+\|v\|_{L^\infty(\Omega^\prime)}\big)\,,  
\end{equation}
for some constant $C$ independent of $v$. 
Indeed, by \cite[proof of Lemma 2, Section 3.2]{Stein} there exists $\Phi_s\in L^1(\R^n)$ such that 
$$\big(1+4\pi^2|\xi|^2\big)^s=1+\widehat \Phi_s(\xi)+(2\pi |\xi|)^{2s} + (2\pi |\xi|)^{2s}\widehat \Phi_s(\xi)\,,$$
where $\widehat \Phi_s$ denotes the Fourier transform of $\Phi_s$. Since $\big(1+4\pi^2|\xi|^2\big)^s$ is the symbol of $(I-\Delta)^s$ in Fourier space,   we infer that 
 $$(I-\Delta)^sw=w+\Phi_s*w+g+\Phi_s*g\in L^{\bar p}(\R^n)\,,$$
 and \eqref{dimmidi} follows. 
\vskip3pt

\noindent{\it Step 3.} By Step 2, the function $w$ belongs to the Bessel potential space $\mathscr{L}^{\bar p}_{2s}(\R^n)$.  According to \cite[Section 2.3.5]{Triebel}, 
$\mathscr{L}^{\bar p}_{2s}(\R^n)$ coincides with the Triebel-Lizorkin space $F^{2s}_{p,2}(\R^n)$ (notice that $\mathscr{L}^{\bar p}_{2s}(\R^n)$ is denoted by $H^{2s}_{\bar p}(\R^n)$ in \cite{Triebel}). Then we use the continuous embeddings between Triebel-Lizorkin spaces and Besov spaces (recall that  $s^\prime<s$)
$$F^{2s}_{\bar p,2}(\R^n)\subset B^{2s}_{\bar p,\max(\bar p,2)}(\R^n)\subset B^{2s^\prime}_{\bar p,\bar p}(\R^n)\,,$$
see \cite[Proposition 2, p. 47]{Triebel}, to deduce that $w$ belongs to the Besov space $B^{2s^\prime}_{\bar p,\bar p}(\R^n)$. Recalling that $B^{2s^\prime}_{\bar p,\bar p}(\R^n)=W^{2s^\prime,\bar p}(\R^n)$ (the Slobodeckij-Sobolev space, see \cite[Section 2.3.5]{Triebel}), we have thus proved that 
\begin{multline*}
\|w\|_{W^{2s^\prime,\bar p}(\R^n)}:=\|w\|_{L^{\bar p}(\R^n)}+\left(\iint_{\R^n\times\R^n}\frac{|w(x)-w(y)|^{\bar p}}{|x-y|^{n+2s^\prime\bar p}}\,\de x\de y\right)^{1/\bar p}\\ 
\leq C\big(\|(-\Delta)^s v\|_{L^{\bar p}(\Omega^\prime)}+\|v\|_{L^\infty(\Omega^\prime)}\big)\,,
\end{multline*}
for some constant $C$ independent of $v$. Since $w=v$ on $\Omega^{\prime\prime}$, this estimate implies~\eqref{sameditroptard}.
\end{proof}

We continue with a simple observation (that we already made implicitly  during the proof of  Theorem \ref{main1}).

\begin{lemma}\label{lemeqharmmap}
Let $F\subset\R^n$ be a Borel set such that $P_{2s}(F,\Omega)<\infty$. Then the function $v_F:=\chi_F-\chi_{\R^n\setminus F}$ belongs to $\widehat H^s(\Omega)$, and $(-\Delta)^sv_F\in L^1(\Omega)$ with 
$$ (-\Delta)^sv_F(x)=\left(\frac{\gamma_{n,s}}{2}\int_{\R^n}\frac{|v_F(x)-v_F(y)|^2}{|x-y|^{n+2s}}\,\de y\right)v_F(x)\quad \text{for a.e. $x\in \Omega$}\,.$$
\end{lemma}

\begin{proof}
Argue as in \eqref{identdim}-\eqref{identdim2}. 
\end{proof}

Back to our original set $E$, we combine Lemma \ref{lemeqharmmap} with the estimate on the Minkowski dimension to obtain 

\begin{proposition}\label{imprintsharp}
We have $(-\Delta)^sv_E\in L^{\bar p}_{\rm loc}(\Omega)$ for every $\bar p<1/2s$.
\end{proposition}

\begin{proof}
Let us fix two open subsets $\Omega^\prime,\Omega^{\prime\prime}$ of $\Omega$ such that  $\overline{\Omega^{\prime\prime}}\subset\Omega^{\prime}$ and $\overline{\Omega^\prime}\subset\Omega$. By Lemma~\ref{lemeqharmmap}, we have $(-\Delta)^sv_E\in L^{1}(\Omega^\prime)$. 
We claim that 
\begin{equation}\label{dim1819tard}
\big|(-\Delta)^sv_E(x)|\leq \frac{C(\Omega^\prime,\Omega^{\prime\prime})}{{\rm dist}(x,\partial E\cap\Omega^\prime)^{2s}}\quad\text{for a.e. $x\in\Omega^{\prime\prime}\setminus\partial E$}\,, 
\end{equation}
for some constant $C(\Omega^\prime,\Omega^{\prime\prime})$ independent of $E$. For $x\in \Omega^{\prime\prime}\setminus \partial E$, we set 
$$r_x:=\frac{1}{2}\min\left({\rm dist}(x,\partial E\cap\Omega^\prime), \min\Big\{|z-y|: z\in\overline{\Omega^{\prime\prime}}\,,\;y\in\R^n\setminus\Omega^\prime\Big\}\right)\,.$$
Since $D_{r_x}(x)\subset \Omega^\prime\setminus \partial E$, we can deduce from Lemma~\ref{lemeqharmmap} that
$$\big|(-\Delta)^sv_E(x)|\leq 2\gamma_{n,s}\int_{\R^n\setminus D_{r_x}(x)}\frac{1}{|x-y|^{n+2s}}\,\de y =\frac{C_{n,s}}{(r_x)^{2s}}\,,$$
and \eqref{dim1819tard} follows.

Let us now fix an exponent $\alpha\in(2s\bar p, 1)$. Since ${\rm dim}_{\mathscr{M}}(\partial E\cap\Omega^\prime)\leq n-1$, we can find a radius $R_\alpha\in(0,1)$ such that $\mathscr{L}^n(\mathscr{T}_r(\partial E\cap \Omega^\prime))\leq r^\alpha$ for every $r\in(0,2R_\alpha)$. Then, we estimate for an arbitrary integer $k\geq 1$, 
\begin{multline*}
\int_{\Omega^{\prime\prime}\setminus\mathscr{T}_{2^{-k}R_\alpha}(\partial E\cap\Omega^\prime)}\big|(-\Delta)^sv_E|^{\bar p}\,\de x\leq 
\int_{\Omega^{\prime\prime}\setminus\mathscr{T}_{R_\alpha}(\partial E\cap\Omega^\prime)}\big|(-\Delta)^sv_E|^{\bar p}\,\de x\\
+\sum_{j=0}^{k-1}\int_{\Omega^{\prime\prime}\cap\mathscr{A}_j}\big|(-\Delta)^sv_E|^{\bar p}\,\de x\,.
\end{multline*}
where we have set $\mathscr{A}_j:=\mathscr{T}_{2^{-j}R_\alpha}(\partial E\cap\Omega^\prime)\setminus\mathscr{T}_{2^{-(j+1)}R_\alpha}(\partial E\cap\Omega^\prime)$. 
Inserting \eqref{dim1819tard}, we derive 
$$\int_{\Omega^{\prime\prime}\setminus\mathscr{T}_{2^{-k}R_\alpha}(\partial E\cap\Omega^\prime)}\big|(-\Delta)^sv_E|^{\bar p}\,\de x\leq CR_\alpha^{-2s\bar p}\left(1+\sum_{j=0}^{\infty}\frac{1}{2^{(\alpha-2s\bar p)j}} \right) <\infty\,.$$
Letting $k\to\infty$, we can now conclude by dominated convergence. 
\end{proof}

\begin{proof}[Proof of Theorem \ref{improvperim}]
Fix two open subsets $\Omega^\prime,\Omega^{\prime\prime}$ of $\Omega$ such that  $\overline{\Omega^{\prime\prime}}\subset\Omega^{\prime}$ and $\overline{\Omega^\prime}\subset\Omega$. We choose a number $\theta>2$ such that $ \max(s,s^\prime)<1/\theta$. We set 
$\bar p:=1/(\theta s)<1/2s$, and $\bar s:=s^\prime/\bar p<s$. By Proposition \ref{imprintsharp}, $(-\Delta)^sv_E\in L^{\bar p}_{\rm loc}(\Omega)$, and in turn, Proposition \ref{keypropimpr} yields  
$$\iint_{E\cap\Omega^\prime\times(\Omega^\prime\setminus E)}\frac{1}{|x-y|^{n+2s^\prime}}\,\de x\de y=\frac{1}{2^{\bar p+1}}\iint_{\Omega^\prime\times\Omega^\prime}\frac{|v_E(x)-v_E(y)|^{\bar p}}{|x-y|^{n+2\bar s\bar p}}\,\de x\de y<\infty\,.$$
Then we observe that 
$$P_{2s^\prime}(E,\Omega^{\prime\prime})\leq \iint_{E\cap\Omega^\prime\times(\Omega^\prime\setminus E)}\frac{1}{|x-y|^{n+2s^\prime}}\,\de x\de y+C\,,$$
for a constant $C$ depending only $\Omega^\prime$ and $\Omega^{\prime\prime}$, $n$, and $s^\prime$. 
\end{proof}

%%%%%%%%%%%%%%%%%%%%%%%%%%%%%%%%%%%%%%%%%%%%%%%%%%%%%%%%
%%%%%%%%%%%%%%%%%%%%%%%%%%%%%%%%%%%%%%%%%%%%%%%%%%%%%%%%

\section{Volume of transition sets and  improved estimates}\label{imprsect}

%%%%%%%%%%%%%%%%%%%%%%%%%%%%%%%%%%%%%%%%%%%%%%%%%%%%%%%%
%%%%%%%%%%%%%%%%%%%%%%%%%%%%%%%%%%%%%%%%%%%%%%%%%%%%%%%%

In this section, we apply the quantitative stratification principle of the previous section in order to improve the convergence results of  Theorem \ref{main1}. In few words, we obtain a quantitative volume estimate on the transition set (i.e., where the solution takes values close to zero). This estimate, combined with Lemma~\ref{potbddfint}, leads to further estimates on the potential  in the case where $f_\eps$ is uniformly bounded. We stress again that the general framework of \cite{FMS} does not apply {\sl stricto sensu} to Allen-Cahn type equations, and non trivial adjustments  need to be made. As before, we start with estimates on the degenerate boundary Allen-Cahn equation. 

%%%%%%%%%%%%%%%%%%%%%%%%%%%%%%%%%%%%%%%%%%%%%%%%%%%%%%%%

\subsection{Quantitative estimates for the Allen-Cahn boundary equation}\label{subsectstratAC}

%%%%%%%%%%%%%%%%%%%%%%%%%%%%%%%%%%%%%%%%%%%%%%%%%%%%%%%%

In this subsection, we are considering a bounded admissible open set $G\subset \R^{n+1}_+$,  $\eps\in(0,1)$, and a weak solution $u_\eps\in H^{1}(G,|z|^a\de{\bf x})\cap L^\infty(G)$ of  
\begin{equation}\label{eqACagain}
\begin{cases}
{\rm div}(z^a\nabla u_\eps) =0 & \text{in $G$}\,,\\[8pt]
\displaystyle d_s\boldsymbol{\partial}^{(2s)}_z u_\eps=\frac{1}{\varepsilon^{2s}}W^\prime(u_\eps) -f_\eps & \text{on $\partial^0 G$}\,,
\end{cases}
\end{equation}
for some given function $f_\eps\in C^{0,1}(\partial^0G)$. We {\it fix constants} $r_0> 0$, $b\geq 1$, $q\in(\frac{n}{1+2s},n)$, $H_0\geq0$, and $\Lambda_0\geq 0$ such that 
\begin{equation}\label{controlLinftyueps}
 \|u_\eps\|_{L^\infty(G)}\leq b\,, 
 \end{equation}
\begin{equation}\label{controlcurveps}
 \eps^{2s}\|f_\eps\|_{L^\infty(\partial^0G)}+\|f_\eps\|_{\dot W^{1,q}(\partial^0G)}\leq H_0\,, 
 \end{equation}
 and 
\begin{equation}\label{controldensiteps}
\sup\Big\{\boldsymbol{\Theta}^{\eps}_{u_\eps}(f_\eps,x,\rho) : x\in\Omega^{r_0}\,,\;0<\rho\leq r_0\Big\}\leq \Lambda_0\,, 
\end{equation}
where the domain $\Omega^{r_0}$ is defined in \eqref{defOmegar}. 
\vskip3pt

Our aforementioned volume estimate is the following theorem, cornerstone of the  section.

\begin{theorem}\label{volesti}
For each $\alpha\in(0,1)$, there exist   ${\bf k_*}={\bf k_*}(\alpha,r_0,H_0,\Lambda_0,W,b,n,s,q)>0$  
and $C=C(\alpha,r_0,H_0,\Lambda_0,W,{\rm diam}(\partial^0G),b,n,s,q)$ such that  
\begin{equation}\label{volestieq}
\mathscr{L}^n\Big(\mathscr{T}_r\big(\{|u_\eps|<1-\boldsymbol{\delta}_W\}\cap \Omega^{r_0}\big)\Big)\leq Cr^\alpha\qquad\forall r\in({\bf k_*}\eps,r_0)\,, 
\end{equation}
where $\boldsymbol{\delta}_W\in(0,1/2]$ is given by \eqref{quantconvex}.
\end{theorem}

The proof of Theorem \ref{volesti} follows in some sense the lines of  \cite[Theorem 2.2]{FMS} with a different set of structural assumptions adjusted to our setting. Since the solution $u_\varepsilon$ is smooth, there is of course no singular set, and no strict  analogue to \cite[Theorem 2.2]{FMS}. However, if we don't look at $u_\eps$ at too small scales, then the transition set  $\{|u_\eps|<1-\boldsymbol{\delta}_W\}$ can play the role of singular set. As one may guess, the threshold scale is $\eps$, explaining the restriction on the admissible radii in \eqref{volestieq}. The same  threshold appears of course in our ``structural assumptions'', provided by Lemmas \ref{lemme1strateps}, \ref{lemme2strateps}, and \ref{lemme3strateps} below.

\begin{lemma}\label{lemme1strateps}
There exist constants 
$$\widetilde\delta_0(r_0)=\delta_0(r_0,H_0,\Lambda_0,W,b,n,s,q)\in(0,1)$$ 
and 
$${\bf k}_0(r_0)={\bf k}_0(r_0,H_0,\Lambda_0,W,b,n,s,q) \geq 1$$ 
(independent of $\eps$, $u_\eps$, and $f_\eps$)   that for every $x\in\Omega^{r_0}$ and 
$\rho\in(0,r_0)$,
$$ |u_\eps(x,0)|<1-\boldsymbol{\delta}_W \quad\text{and}\quad {\bf k}_0\eps\leq \rho\quad\Longrightarrow\quad {\bf d}_n(u_\eps,x,\rho) \geq\widetilde\delta_0 \,.$$
\end{lemma}

\begin{proof}
Assume by contradiction that there exist sequences $\{\eps_k\}_{k\in\mathbb{N}}\subset(0,1)$, $\{(u_k,f_k)\}_{k\in\mathbb{N}}$ satisfying \eqref{eqACagain}-\eqref{controlLinftyueps}-\eqref{controlcurveps}-\eqref{controldensiteps}, points $\{x_k\}_{k\in\mathbb{N}}\subset\Omega^{r_0}$, and radii $\{\rho_k\}_{k\in\mathbb{N}}\subset(0,r_0)$ such that  $|u_k(x_k,0)|<1-\boldsymbol{\delta}_W$, $\eps_k/\rho_k\leq 2^{-k}$, and 
${\bf d}_n(u_k,x_k,\rho_k)\to 0$.  

Setting $\widetilde \eps_k:=\eps_k/\rho_k$,  consider the rescaled maps $\widetilde u_k:=(u_k)_{x_k,\rho_k}$ and $\widetilde f_k:=\rho_k^{2s}(f_k)_{x_k,\rho_k}$ as defined in \eqref{defrescmap}. Rescaling variables, we derive that 
\begin{equation}\label{rescACstrat} 
\begin{cases}
{\rm div}(z^a\nabla \widetilde u_k) =0 & \text{in $B_1^+$}\,,\\[8pt]
\displaystyle d_s\boldsymbol{\partial}^{(2s)}_z \widetilde u_k=\frac{1}{(\widetilde\varepsilon_k)^{2s}}W^\prime(\widetilde u_k) -\widetilde f_k & \text{on $D_1$}\,,
\end{cases}
\end{equation}
and
\begin{equation}\label{contrrescACstrat}
\| \widetilde u_k\|_{L^\infty(B_1^+)}\leq b\,,\quad   (\widetilde\eps_k)^{2s}\|\widetilde f_k\|_{L^\infty(D_1)}\leq H_0\,,\quad \|\widetilde f_k\|_{\dot W^{1,q}(D_1)}\leq r_0^{\theta_q}H_0\,,
\end{equation}
as well as
\begin{equation}\label{contrdensrescACstrat}
\boldsymbol{\Theta}^{\widetilde \eps_k}_{\widetilde u_k}(\widetilde f_k,0,1) =\boldsymbol{\Theta}^{\eps_k}_{u_k}(f_k,x_k,\rho_k)\leq \Lambda_0\,. 
\end{equation}
By Theorem \ref{asymptneum}, we can find a (not relabeled) subsequence such that $\widetilde u_k\to u_*$ weakly in $H^1(B_1^+,|z|^a\de{\bf x})$ and strongly in $H^1(B_r^+,|z|^a\de{\bf x})$ for every $r\in(0,1)$. Then $\widetilde u_k\to u_*$ strongly in $L^1(D_r)$ for every $r\in(0,1)$ by Remark \ref{trace}. 
On the other hand, ${\bf d}_n(\widetilde u_k,0,1)={\bf d}_n(u_k,x_k,\rho_k)\to 0$, so that either $u_*=1$ or $u_*=-1$ on $D_1$. Without loss of generality, we may assume that  $u_*=1$ on $D_1$. Then Theorem \ref{asymptneum} tells us that $\widetilde u_k\to 1$ uniformly on $D_r$ for every $r\in(0,1)$. In particular $\widetilde u_k(0)\to 1$ which contradicts our assumption $|\widetilde u_k(0)|=|u_k(x_k,0)|<1-\boldsymbol{\delta}_W$. 
\end{proof}

\begin{lemma}\label{lemme2strateps}
For every $\delta>0$, there exist constants 
$$\widetilde\eta_1(\delta,r_0)=\widetilde\eta_1(\delta,r_0,H_0,\Lambda_0,W,b,n,s,q)\in (0,1/4)\,,$$ 
$$\widetilde \lambda_1(\delta,r_0)=\widetilde\lambda_1(\delta,r_0,H_0,\Lambda_0,W,b,n,s,q)\in (0,1/4)\,,$$ 
and 
$${\bf k}_1(\delta,r_0)={\bf k}_1(\delta,r_0,H_0,\Lambda_0,W,b,n,s,q) \geq 1$$ 
(independent of $u_\eps$ and $f_\eps$) such that for every 
$\rho\in(0,r_0/5)$ and $x\in\Omega^{r_0}$,   
$$\boldsymbol{\Theta}^\eps_{u_\eps}(f_\eps,x,\rho)-\boldsymbol{\Theta}^\eps_{u_\eps}(f_\eps,x,\widetilde\lambda_1\rho) \leq \widetilde\eta_1\;\text{ and }\; {\bf k}_1\eps\leq \rho \quad \Longrightarrow\quad {\bf d}_0(u_\eps,x,\rho) \leq\delta \,.$$ 
\end{lemma}

\begin{proof}
We choose 
$$\widetilde\eta_1(\delta,r_0):=\frac{1}{2}\eta_1(\delta/2,2/5,r_0^{\theta_q}H_0,\Lambda_0,b,n,s,q)\,, $$
where $\eta_1$ is given by Lemma \ref{lemme2stratsurf}. Then we argue again by contradiction assuming that for some constant $\delta>0$, there exist sequences $\{\eps_k\}_{k\in\mathbb{N}}\subset(0,1)$, $\{(u_k,f_k)\}_{k\in\mathbb{N}}$ satisfying \eqref{eqACagain}-\eqref{controlLinftyueps}-\eqref{controlcurveps}-\eqref{controldensiteps}, points $\{x_k\}_{k\in\mathbb{N}}\subset\Omega^{r_0}$, radii $\{\rho_k\}_{k\in\mathbb{N}}\subset(0,r_0/5)$, and $\lambda_k\to 0$  such that $\eps_k/\rho_k\leq 2^{-k}$, 
$$\boldsymbol{\Theta}^{\eps_k}_{u_k}(f_k,x_k,\rho_k)-\boldsymbol{\Theta}^{\eps_k}_{u_k}(f_k,x_k,\lambda_k\rho_k) \leq \widetilde\eta_1\,,\quad\text{and}\quad {\bf d}_0(u_k,x_k,\rho_k) \geq\delta\,.$$ 
Next we proceed as in the proof of Lemma \ref{lemme1strateps} rescaling variables as $\widetilde \eps_k:=\eps_k/(5\rho_k)$, $\widetilde u_k:=(u_k)_{x_k,5\rho_k}$, and $\widetilde f_k:=(5\rho_k)^{2s}(f_k)_{x_k,5\rho_k}$. Then, \eqref{rescACstrat}, \eqref{contrrescACstrat}, and \eqref{contrdensrescACstrat}  hold, as well as 
\begin{equation}\label{1451jhu}
 \sup\Big\{\boldsymbol{\Theta}^{\widetilde \eps_k}_{\widetilde u_k}(\widetilde f_k,x,\rho) : x\in D_{1/5}\,,\;0<\rho\leq 2/5\Big\}\leq \Lambda_0\,.
 \end{equation}
Now our assumptions lead to 
$$\boldsymbol{\Theta}^{\widetilde \eps_k}_{\widetilde u_k}(\widetilde f_k,0,1/5)-\boldsymbol{\Theta}^{\widetilde \eps_k}_{\widetilde u_k}(\widetilde f_k,0,\lambda_k/5) \leq \widetilde\eta_1\,,\quad\text{and}\quad {\bf d}_0(\widetilde u_k,0,1/5) \geq\delta\,.$$ 
By Theorem \ref{asymptneum}, we can find a (not relabeled) subsequence such that $\widetilde u_k\to u_*$ strongly in $H^1(B_r^+,|z|^a\de{\bf x})$  for every $r\in(0,1)$, and $\widetilde f_k\rightharpoonup f_*$ in $W^{1,q}(D_1)$, where the pair $(u_*,f_*)$ solves  \eqref{eqcone}-\eqref{statinB1}. Note that, by lower semicontinuity, we have $\|f_*\|_{\dot W^{1,q}(D_1)}\leq r_0^{\theta_q}H_0$. In addition, by Theorem  \ref{asymptneum} and Fatou's lemma, we deduce from \eqref{1451jhu} that
\begin{equation}\label{peuplusdutotu}
\sup\Big\{\boldsymbol{\Theta}_{u_*}(f_*,x,\rho) : x\in D_{1/5}\,,\;0<\rho\leq 2/5\Big\}\leq \Lambda_0\,.
\end{equation}

By means of Lemma \ref{monotform}, we now estimate for $0<r<1/5$ and $k$ large enough  (in such a way that $\lambda_k<r$), 
\begin{multline*}
\boldsymbol{\Theta}^{\widetilde\eps_k}_{\widetilde u_k}(\widetilde f_k,0,1/5)-  \frac{1}{r^{n-2s}}{\bf E}_{\widetilde\eps_k}(\widetilde u_k,B_r^+)-\frac{{\bf c}_{n,q}b\, r_0^{\theta_q}}{\theta_q}H_0r^{\theta_q}\\
\leq \boldsymbol{\Theta}_{\widetilde u_k}(\widetilde f_k,0,1/5)-\boldsymbol{\Theta}_{\widetilde u_k}(\widetilde f_k,0,r) \leq \widetilde\eta_1\,. 
\end{multline*}
Using Theorem \ref{asymptneum}, we can let $k\to\infty$ in this inequality to derive
\begin{multline*}
\boldsymbol{\Theta}_{u_*}(f_*,0,1/5)- \boldsymbol{\Theta}_{u_*}(f_*,0,r)\leq \boldsymbol{\Theta}_{u_*}(f_*,0,1/5)-  \frac{1}{r^{n-2s}}{\bf E}(u_*,B_r^+)\\
\leq \widetilde\eta_1+\frac{{\bf c}_{n,q}b\, r_0^{\theta_q}}{\theta_q}H_0r^{\theta_q}\,.
\end{multline*}
Choosing $r$ small enough in such a way that 
$$ \frac{{\bf c}_{n,q}b\, r_0^{\theta_q}}{\theta_q}H_0r^{\theta_q}\leq \widetilde\eta_1\quad\text{and} \quad r\leq \frac{1}{5}
\lambda_1(\delta/2,2/5,r_0^{\theta_q}H_0,\Lambda_0,b,n,s,q)\,, $$
where $\lambda_1$ is given Lemma \ref{lemme2stratsurf}, we infer from Lemma \ref{monotformsurf} that
$$\boldsymbol{\Theta}_{u_*}(f_*,0,1/5)- \boldsymbol{\Theta}_{u_*}(f_*,0,\lambda_1/5)\leq 2 \widetilde\eta_1=\eta_1\,.$$
Then Lemma \ref{lemme2stratsurf} yields $ {\bf d}_0(u_*,0,1/5) \leq\delta/2$.  On the other hand, by Remark \ref{trace}, $\widetilde u_k\to u_*$ in $L^1(D_{1/5})$, and thus 
$ {\bf d}_0(u_*,0,1/5) =\lim_k{\bf d}_0(\widetilde u_k,0,1/5) \geq\delta$, contradiction. 
\end{proof}

\begin{lemma}\label{lemme3strateps}
For every $\delta,\tau\in(0,1)$, there exist two constants 
$$\widetilde \eta_2(\delta,\tau,r_0)=\widetilde \eta_2(\delta,\tau,r_0,H_0,\Lambda_0,W,b,n,s,q)\in (0,\delta]$$
and 
$${\bf k}_2(\delta,\tau,r_0)= {\bf k}_2(\delta,\tau,r_0,H_0,\Lambda_0,W,b,n,s,q)\geq 1$$
(independent of $u_\eps$ and $f_\eps$) such that for every $\rho\in(0,r_0/25)$ and $x\in\Omega^{r_0}$, the conditions
$${\bf k}_2\eps\leq \rho\,,\quad{\bf d}_0(u_\eps,x,4\rho) \leq\widetilde \eta_2\quad \text{and}\quad {\bf d}_{n}(u_\eps,x,4\rho)\geq \delta\,,$$
imply the existence of a linear subspace $V\subset\R^n$, with ${\rm dim}\,V\leq n-1$, for which
$${\bf d}_0(u_\eps,y,4\rho)> \widetilde\eta_2 \quad \forall y\in D_{\rho}(x)\setminus \mathscr{T}_{\tau \rho}(x+V)\,.$$ 
\end{lemma}

\begin{proof}
We choose 
$$\widetilde\eta_2(\delta,\tau,r_0):=\frac{1}{2}\eta_3(\delta,\tau,2/5,r_0^{\theta_q}H_0,\Lambda_0,b,n,s,q)\,, $$
where $\eta_3$ is given by Corollary \ref{corstrat3surf}. We still argue by contradiction assuming that for some constants $\delta,\tau\in(0,1)$, there exist sequences $\{\eps_k\}_{k\in\mathbb{N}}\subset(0,1)$, $\{(u_k,f_k)\}_{k\in\mathbb{N}}$ satisfying \eqref{eqACagain}-\eqref{controlLinftyueps}-\eqref{controlcurveps}-\eqref{controldensiteps}, points $\{x_k\}_{k\in\mathbb{N}}\subset\Omega^{r_0}$, and radii $\{\rho_k\}_{k\in\mathbb{N}}\subset(0,r_0/25)$  such that 
$$\eps_k/\rho_k\leq 2^{-k}\,,\quad{\bf d}_0(u_k,x_k,4\rho_k) \leq\widetilde \eta_2\quad \text{and}\quad {\bf d}_{n}(u_k,x_k,4\rho_k)\geq \delta\,,$$
and such that the conclusion of the lemma fails. 

Once again, we rescale variables setting $\widetilde \eps_k:=\eps_k/(25\rho_k)$, $\widetilde u_k:=(u_k)_{x_k,25\rho_k}$, and $\widetilde f_k:=(25\rho_k)^{2s}(f_k)_{x_k,25\rho_k}$, so that \eqref{rescACstrat}, \eqref{contrrescACstrat}, \eqref{contrdensrescACstrat}, and \eqref{1451jhu} hold. Then we reproduce the proof of Lemma \ref{lemme2strateps} to find a (not relabeled) subsequence along which $(\widetilde u_k,\widetilde f_k)$ converges to some limiting pair $(u_*,f_*)$ solving \eqref{eqcone}-\eqref{statinB1}, and satisfying \eqref{contrrescACstrat}-\eqref{contrdensrescACstrat}-\eqref{peuplusdutotu}. In particular, $\widetilde u_k\to u_*$ strongly in $L^1(D_{1/5})$. As a consequence,  
$${\bf d}_0(u_*,0,4/25) \leq\widetilde \eta_2\quad\text{and}\quad {\bf d}_{n}(u_*,0,4/25)\geq \delta\,.$$ 
By Corollary \ref{corstrat3surf}, there exists a linear subspace $V\subset \R^n$, with ${\rm dim}\,V\leq n-1$, such that 
\begin{equation}\label{1825Yanhouse}
{\bf d}_0(u_*,y,4/25)> \eta_3 \quad \forall y\in D_{1/25}\setminus \mathscr{T}_{\tau /25}(V)\,.
\end{equation}
Since the conclusion of the lemma does not hold, we can find for each integer $k$ a point $y_k\in  D_{1/25}\setminus \mathscr{T}_{\tau /25}(V)$ such that 
${\bf d}_0(\widetilde u_k,y_k,4/25)\leq\widetilde\eta_2$.  Then extract a further subsequence such that $y_k\to y_*$ for some $y_*\in \overline D_{1/25}\setminus \mathscr{T}_{\tau /25}(V)$. Noticing that 
$$\|(u_*)_{y_*,1}-(\widetilde u_k)_{y_k,1} \|_{L^1(D_{4/25})} \leq \|(u_*)_{y_*,1}-(u_*)_{y_k,1} \|_{L^1(D_{4/25})}+ \|u_*-\widetilde u_k \|_{L^1(D_{1/5})}\,,$$
by continuity of translations in $L^1$, we have  $\|(u_*)_{y_*,1}-(\widetilde u_k)_{y_k,1} \|_{L^1(D_{4/25})}\to 0$.  Consequently, 
\begin{multline*}
{\bf d}_0(u_*,y_*,4/25)={\bf d}_0\big((u_*)_{y_*,1},0,4/25\big)\\
=\lim_{k\to\infty}{\bf d}_0\big((\widetilde u_k)_{y_k,1},0,4/25)=\lim_{k\to\infty}{\bf d}_0(\widetilde u_k,y_k,4/25) \,, 
\end{multline*}
and thus ${\bf d}_0(u_*,y_*,4/25)\leq \widetilde\eta_2$.  However \eqref{1825Yanhouse} yields  ${\bf d}_0(u_*,y_*,4/25)\geq \eta_3=2 \widetilde\eta_2$, contradiction.
\end{proof}

\begin{proof}[Proof of Theorem \ref{volesti}]
For $0<r\leq r_0$, we consider the set 
$$\mathcal{S}^\eps_{r_0,r}:=\bigg\{x\in\Omega^{r_0} : {\bf d}_n(u_\eps,x,\rho)\geq \widetilde\delta_0(r_0)\;\;\forall \,r\leq \rho\leq r_0\bigg\} \,,$$
where  $ \widetilde\delta_0(r_0)>0$ is given by Lemma \ref{lemme1strateps}. We fix the exponent $\alpha\in(0,1)$, and we set $\kappa_0:=1-\alpha\in(0,1)$. 

We will prove that there exist constants ${\bf k_*}={\bf k_*}(\kappa_0,r_0,H_0,\Lambda_0,W,b,n,s,q)\geq {\bf k}_0(r_0)$  
and $C=C(\kappa_0,r_0,H_0,\Lambda_0,W,b,n,s,q)$ such that  
\begin{equation}\label{estvolstarteps}
\mathscr{L}^n\big(\mathscr{T}_r(\mathcal{S}^\eps_{r_0,r})\big)\leq C r^{1-\kappa_0} \quad\forall r\in({\bf k}_*\eps,r_0)\,,
\end{equation}
where $ {\bf k}_0(r_0)$ is  given by Lemma  \ref{lemme1strateps}. Note that, since ${\bf k_*}\geq {\bf k}_0(r_0)$, we have 
$$\{|u_\eps|<1-\boldsymbol{\delta}_W\}\cap\Omega^{r_0}\subset \mathcal{S}^\eps_{r_0,r} \quad\forall r\in ({\bf k}_*\eps,r_0)\,,$$
by Lemma  \ref{lemme1strateps}. In other words, estimates \eqref{estvolstarteps} implies Theorem \ref{volesti}. 

Now the proof follows closely the arguments in \cite[proof of Theorem~2.2]{FMS} once adjusted to our setting, but  for the sake of clarity we partially reproduce it.  

We start fixing a number $\tau=\tau(\kappa_0,n)\in(0,1)$ such that $\tau^{\kappa_0/2}\leq 20^{-n}$. 
We consider the following constants according to Lemma \ref{lemme1strateps}, Lemma \ref{lemme2strateps}, and Lemma~\ref{lemme3strateps}:
\begin{enumerate}
\item[(i)] $\widetilde\eta_2:=\widetilde\eta_2\big(\widetilde\delta_0(r_0),\tau,r_0\big)$ and ${\bf k}_2:={\bf k}_2\big(\widetilde\delta_0(r_0),\tau,r_0\big)$; 
\item[(ii)] $\widetilde\eta_1:=\widetilde\eta_1\big(\widetilde\eta_2,r_0\big)$, $\widetilde\lambda_1:=\widetilde\lambda_1\big(\widetilde\eta_2,r_0\big)$, and ${\bf k}_1:={\bf k}_1\big(\widetilde\eta_2,r_0\big)$; 
\item[(iii)] ${\bf k}_3:=\max\{{\bf k}_0(r_0), {\bf k}_1, {\bf k}_2\}$.   
\end{enumerate}
Next we fix an integer $q_0\in\mathbb{N}$ such that $\tau^{q_0}\leq \widetilde\lambda_1$, and we set $M:=\lfloor q_0\Lambda_0/ \widetilde\eta_1\rfloor$ (the integer part of).  Set 
$p_0:=q_0+M+1$ and define 
$$\boldsymbol{\eps}_0:=\min\left\{1,\frac{r_0\tau^{p_0+1}}{25{\bf k}_3}\right\}\,,\quad {\bf k}_*:=\frac{{\bf k}_3}{\tau}\,.$$
Without loss of generality, we may assume that $\eps\in(0,\boldsymbol{\eps}_0)$ (since \eqref{estvolstarteps} is straightforward for $\eps\geq \boldsymbol{\eps}_0$). Let ${\bf k}_4={\bf k}_4(\eps)$ be defined by the relation $r_0\tau^{{\bf k}_4|\log\eps|}=25{\bf k}_3\eps$, and set 
$p_1=p_1(\eps):=\lfloor {\bf k}_4|\log\eps|\rfloor$ (the integer part of).  Note that our choice of $\boldsymbol{\eps}_0$ and ${\bf k}_*$ insures that $p_1\geq p_0+1$ and ${\bf k}_3\eps\leq\frac{r_0\tau^{p_1}}{25}\leq {\bf k}_*\eps$.
\vskip3pt

\noindent{\it Step 1. Reduction to $\tau$-adic radii.} We argue exactly as in \cite[Proof of Theorem 2.2, Step~1]{FMS} to show that  it suffices to prove \eqref{estvolstarteps} for each radius $r$ of the form  $r=\frac{r_0\tau^k}{25}$ for an integer $k$ satisfying $p_0\leq k\leq p_1$. 
\vskip3pt

\noindent{\it Step 2. Selection of good scales.} We fix an integer $k$ with $p_0\leq k\leq p_1$ and set $r=\frac{r_0\tau^k}{25}$.  
For an arbitrary $x\in\Omega^{r_0}$, we have 
\begin{multline*}
\sum_{l=q_0}^k \boldsymbol{\Theta}^\eps_{u_\eps}(f_\eps,x,4r_0\tau^{l})-\boldsymbol{\Theta}^\eps_{u_\eps}(f_\eps,x,4r_0\tau^{l+q_0})\\  
=\sum_{l=q_0}^k \sum_{i=l}^{l+q_0-1}\boldsymbol{\Theta}^\eps_{u_\eps}(f_\eps,x,4r_0\tau^{i})-\boldsymbol{\Theta}^\eps_{u_\eps}(f_\eps,x,4r_0\tau^{i+1})  \\
\leq q_0 \sum_{l=q_0}^{k+q_0-1}\boldsymbol{\Theta}^\eps_{u_\eps}(f_\eps,x,4r_0\tau^{l})-\boldsymbol{\Theta}^\eps_{u_\eps}(f_\eps,x,4r_0\tau^{l+1})\,,
\end{multline*}
and thus
$$\sum_{l=q_0}^k \boldsymbol{\Theta}^\eps_{u_\eps}(f_\eps,x,4r_0\tau^{l})-\boldsymbol{\Theta}^\eps_{u_\eps}(f_\eps,x,4r_0\tau^{l+q_0}) \leq q_0 \boldsymbol{\Theta}^\eps_{u_\eps}(f_\eps,x,4r_0\tau^{q_0}) \leq  q_0\Lambda_0\,. $$
Hence there exists a (possibly empty) subset $A(x)\subset\{q_0,\ldots,k\}$ with ${\rm Card}(A(x))\leq M$ such that  for every $l\in\{q_0,\ldots,k\}\setminus A(x)$, 
\begin{equation}\label{jhu1333}
\boldsymbol{\Theta}^\eps_{u_\eps}(f_\eps,x,4r_0\tau^{l})-\boldsymbol{\Theta}^\eps_{u_\eps}(f_\eps,x,4r_0\tau^{l+q_0})\leq \widetilde\eta_1\,.
\end{equation}
Next define $\mathfrak{A}:=\{A\subset\{q_0,\ldots,k\} : {\rm Card}(A)= M\}$, and set for $A\in\mathfrak{A}$,  
$$\mathcal{S}_A:= \bigg\{x\in\mathcal{S}^\eps_{r_0,r}:  \text{ \eqref{jhu1333} holds for each } l\in\{q_0,\ldots,k\}\setminus A\bigg\}\,.$$
By our previous discussion, we have $\mathcal{S}^\eps_{r_0,r}\subset \bigcup_{A\in\mathfrak{A}} \mathcal{S}_A$. 
In the next step, we shall prove that for any $A\in\mathfrak{A}$, 
\begin{equation}\label{Yanikoumouk}
\mathscr{L}^n\big(\mathscr{T}_r(\mathcal{S}_A)\big) \leq C r^{1-\kappa_0/2}\,.
\end{equation}
Since ${\rm Card}(\mathfrak{A})  \leq k^M\leq C|\log r|^M$, the conclusion follows from this estimate, i.e., 
$$\mathscr{L}^n\big( \mathscr{T}_r(\mathcal{S}^\eps_{r_0,r})\big) \leq\sum_{A\in\mathfrak{A}} \mathscr{L}^n\big(\mathscr{T}_r(\mathcal{S}_A)\big)\leq  C |\log r|^M  r^{1-\kappa_0/2}\leq C  r^{1-\kappa_0}\,,$$
for some constants $C=C(\kappa_0,r_0,H_0,\Lambda_0,W, {\rm diam}(\partial^0G),b,n,s,q)$. 
\vskip3pt

\noindent{\it Step 3. Proof of \eqref{Yanikoumouk}.} Again we follow \cite[Proof of Theorem 2.2, Step~3]{FMS}. We first consider a finite cover of $\mathscr{T}_{r_0\tau^{q_0}/25}(\mathcal{S}_A)$ made of discs $\{D_{r_0\tau^{q_0}}(x_{i,q_0})\}_{i\in I_{q_0}}$ with $x_{i,q_0}\in\mathcal{S}_A$, and 
$${\rm Card}(I_{q_0})\leq 5^n\tau^{-nq_0}r_0^{-n}({\rm diam}(\partial^0G)+1)^n\,. $$
We  argue now by iteration on the integer $j\in\{q_0+1,\ldots,k\}$, assuming that  we already have a cover $\{D_{r_0\tau^{j-1}}(x_{i,j-1})\}_{i\in I_{j-1}}$ of $\mathscr{T}_{r_0\tau^{j-1}/25}(\mathcal{S}_A)$ such that $x_{i,j-1}\in\mathcal{A}$. 
We select the next cover $\{D_{r_0\tau^{j}}(x_{i,j})\}_{i\in I_{j}}$ (still centered at points of $\mathcal{S}_A$) of $\mathscr{T}_{r_0\tau^{j}/25}(\mathcal{S}_A)$ according to the following two cases: $j-1\in A$ or $j-1\not\in A$. 

\noindent {\it Case 1)} If $j-1\in A$, then we proceed exactly as in  \cite[Proof of Theorem 2.2, Step~3, Case~(a)]{FMS} to produce the new cover $\{D_{r_0\tau^{j}}(x_{i,j})\}_{i\in I_{j}}$  in such a way that 
$${\rm Card}(I_{j})\leq 20^n{\rm Card}(I_{j-1})\tau^{-n}\,.  $$

\noindent {\it Case 2)} If $j-1\not\in A$, then \eqref{jhu1333} holds with $l=j-1$. By our choice of $q_0$ and Lemma~\ref{monotform}, we infer that 
\begin{multline*}
\boldsymbol{\Theta}^\eps_{u_\eps}(f_\eps,x_i,4r_0\tau^{j-1})-\boldsymbol{\Theta}^\eps_{u_\eps}(f_\eps,x_i,4r_0\widetilde\lambda_1\tau^{j-1})\\
\leq \boldsymbol{\Theta}^\eps_{u_\eps}(f_\eps,x_i,4r_0\tau^{j-1})-\boldsymbol{\Theta}^\eps_{u_\eps}(f_\eps,x_i,4r_0\tau^{j-1+q_0})\leq \widetilde\eta_1\quad\forall x\in\mathcal{S}_A\,.
\end{multline*}
Then Lemma \ref{lemme2strateps} yields ${\bf d}_0(u_\eps,x,4r_0\tau^{j-1}) \leq \widetilde\eta_2$ for every $x\in\mathcal{S}_A$. On the other hand, by the  definition of $\mathcal{S}_A$ we have ${\bf d}_n(u_\eps,x,4r_0\tau^{j-1}) \geq \widetilde\delta_0$ for every $x\in\mathcal{S}_A$. Applying Lemma~\ref{lemme3strateps} at each point $x_{i,j-1}$, we infer that for each $i\in I_{j-1}$, there is a linear subspace $V_i$, with ${\rm dim}\,V_i\leq n-1$, such that  $\mathcal{S}_A\cap D_{r_0\tau^{j-1}}(x_{i,j-1})\subset \mathscr{T}_{r_0\tau^j}(x_{i,j-1}+V_i)$. 
From this inclusion, we estimate for each $i\in I_{j-1}$, 
$$\mathscr{L}^n\bigg(\mathscr{T}_{r_0\tau^j}\big(\mathcal{S}_A\cap D_{r_0\tau^{j-1}}(x_{i,j-1})\big)\bigg)\leq 2^{n+1}\omega_{n-1}r^n_0\tau^{nj-n+1} \,.$$
By the covering lemma in \cite[Lemma 3.2]{FMS}), we can find a cover of $\mathscr{T}_{r_0\tau^j/25}(\mathcal{S}_A)$ by discs 
$\{D_{r_0\tau^{j}}(x_{i,j})\}_{i\in I_{j}}$ centered on $\mathcal{S}_A$  such that 
$$ {\rm Card}(I_{j})\leq 10^n\frac{2\omega_{n-1}}{\omega_n}{\rm Card}(I_{j-1})  \tau^{-(n-1)}\leq  20^n{\rm Card}(I_{j-1})  \tau^{-(n-1)} \,.$$
The iteration procedure stops at $j=k$, and it yields a cover $\{D_{r_0\tau^{k}}(x_{i,k})\}_{i\in I_{k}}$ of $\mathscr{T}_{r}(\mathcal{S}_A)$. Collecting the estimates from Case 1 and Case 2 (and using ${\rm Card}\,A=M$), we derive
\begin{multline*}
 {\rm Card}(I_{k})\leq   5^n\tau^{-nq_0}r_0^{-n}({\rm diam}(\partial^0G)+1)^n(20^n\tau^{-n})^M\left(20^n\tau^{-(n-1)}\right)^{k-q_0-M} \\
 \leq C\tau^{-k(n-1+\kappa_0/2)}\,,
 \end{multline*}
where $C$ depends on the announced parameters (recall that $\tau^{\kappa_0/2}\leq 20^{-n}$). Consequently, 
$$\mathscr{L}^n\big(\mathscr{T}_r(\mathcal{S}_A)\big) \leq\omega_n {\rm Card}(I_{k})r^n\leq C \tau^{k(1-\kappa_0/2)}\leq C r^{1-\kappa_0/2}\,, $$
and the proof is complete.  
\end{proof}

\begin{corollary}\label{corpotquant}
For  every $\alpha\in(0,1)$,  
$$\int_{\Omega^{2r_0}}W(u_\eps)\,\de x\leq C \eps^{\min(4s,\alpha)} \,,$$
for some constant $C=C(\alpha,r_0,\|f_\eps\|_{L^\infty(\partial^0G)},H_0,\Lambda_0,W,b,n,s,q)$.
\end{corollary}

\begin{proof} Without loss of generality, we may assume that $\alpha\not=4s$. 
We use the notation of the proof of Theorem  \ref{volesti}, and we assume (without loss of generality) that $\eps\in(0,\boldsymbol{\eps}_0)$.  
Let us set $\mathcal{V}_\eps:=\{|u_\eps|<1-\boldsymbol{\delta}_W\}$, and $\rho_k:=\frac{r_0\tau^{k}}{25}$ for $k\in\mathbb{N}$. Notice that 
$$\rho_{p_1(\eps)-1}\in({\bf k}_*\varepsilon,{\bf k}_*\tau^{-1}\eps)\,. $$
Hence, by Theorem  \ref{volesti}, we have 
\begin{equation}\label{olala2251}
\mathscr{L}^n\left(\mathscr{T}_{\rho_k}(\mathcal{V}_\eps\cap\Omega^{r_0}\right)\leq C \rho_k^{\alpha}\leq C \tau^{\alpha k}\quad\text{for $k=0,\ldots,p_1(\eps)-1$}\,,
\end{equation}
where the constant $C$ may depend on the announced parameters. In particular, 
\begin{equation}\label{toto2253}
\int_{\mathscr{T}_{\rho_{p_1(\eps)-1}}(\mathcal{V}_\eps\cap\Omega^{r_0})}W(u_\eps)\,\de x\leq C\|W\|_{L^\infty(-b,b)}\rho_{p_1(\eps)-1}^\alpha\leq C\eps^\alpha\,.
\end{equation}
On the other hand, by Lemma \ref{potbddfint}, we have 
\begin{equation}\label{titi2254}
W\big(u_\eps(x,0)\big)\leq \frac{C\eps^{4s}}{\big({\rm dist}(x,\mathcal{V}_\eps)\big)^{4s}}\leq \frac{C\eps^{4s}}{\big({\rm dist}(x,\mathcal{V}_\eps\cap\Omega^{r_0})\big)^{4s}} \quad\text{in $\Omega^{2r_0}\setminus \mathcal{V}_\eps$}\,.
\end{equation}
Writing $\mathscr{A}_k:=\big(\mathscr{T}_{\rho_{k-1}}(\mathcal{V}_\eps\cap\Omega^{r_0})\setminus\big(\mathscr{T}_{\rho_k}(\mathcal{V}_\eps\cap\Omega^{r_0})\big)$, 
we have
\begin{multline*}
\int_{\Omega^{2r_0}}W(u_\eps)\,\de x  = 
\int_{\Omega^{2r_0}\setminus\mathscr{T}_{\rho_0}(\mathcal{V}_\eps)}W(u_\eps)\,\de x+\int_{\mathscr{T}_{\rho_{p_1(\eps)-1}}(\mathcal{V}_\eps\cap\Omega^{r_0})\cap\Omega^{2r_0}}W(u_\eps)\,\de x \\
+\sum_{k=1}^{p_1(\eps)-1}\int_{\mathscr{A}_k\cap\Omega^{2r_0}}W(u_\eps)\,\de x\,,
 \end{multline*}
 We may now estimate by \eqref{olala2251}, \eqref{toto2253}, and \eqref{titi2254}, 
 \begin{equation}\label{tata2256}
\int_{\Omega^{2r_0}}W(u_\eps)\,\de x \leq C\left(\eps^{4s}+ \eps^\alpha+\eps^{4s}\sum_{k=1}^{p_1(\eps)-1}\tau^{k(\alpha -4s)}\right)\,.
\end{equation}
If $\alpha>4s$, then $\sum_{k\geq 1}\tau^{k(\alpha-4s)}<\infty$, and the result is proved. If $\alpha<4s$, then 
$$ \sum_{k=1}^{p_1(\eps)-1}\tau^{k(\alpha -4s)}\leq C\tau^{p_1(\eps)(\alpha-4s)}\leq C\eps^{\alpha-4s}\,.$$
Inserting this estimate in \eqref{tata2256} still yields the announced result. 
\end{proof}

\begin{corollary}\label{corpotquantbis}
For every $\bar p <1/2s$, 
$$\left\|W^\prime(u_\eps)\right\|_{L^{\bar p}(\Omega^{2r_0})}\leq C\eps^{2s} \,,$$
for some constant $C=C(\bar p,r_0,\|f_\eps\|_{L^\infty(\partial^0G)},H_0,\Lambda_0,W,b,n,s,q)$. 
\end{corollary}

\begin{proof}
We proceed as in the proof Corollary \ref{corpotquant}, using $\alpha\in(2s\bar p,1)$. Keeping the same notations, we first derive as in \eqref{toto2253}, 
\begin{equation}\label{toto2253bis}
\int_{\mathscr{T}_{\rho_{p_1(\eps)-1}}(\mathcal{V}_\eps\cap\Omega^{r_0})}\big|W^\prime(u_\eps)\big|^{\bar p}\,\de x \leq C\eps^\alpha\,.
\end{equation}
Then Lemma \ref{potbddfint} yields, 
\begin{equation}\label{titi2254bis}
\big|W\big(u_\eps(x,0)\big)\big|\leq \frac{C\eps^{2s}}{\big({\rm dist}(x,\mathcal{V}_\eps\cap\Omega^{r_0})\big)^{2s}} \quad\text{in $\Omega^{2r_0}\setminus \mathcal{V}_\eps$}\,.
\end{equation}
Writing 
\begin{multline*}
\int_{\Omega^{2r_0}}\big|W^\prime (u_\eps)\big|^{\bar p}\,\de x  = 
\int_{\Omega^{2r_0}\setminus\mathscr{T}_{\rho_0}(\mathcal{V}_\eps)}\big|W^\prime(u_\eps)\big|^{\bar p}\,\de x\\
+\int_{\mathscr{T}_{\rho_{p_1(\eps)-1}}(\mathcal{V}_\eps\cap\Omega^{r_0})\cap\Omega^{2r_0}}\big|W^\prime(u_\eps)\big|^{\bar p}\,\de x 
+\sum_{k=1}^{p_1(\eps)-1}\int_{\mathscr{A}_k\cap\Omega^{2r_0}}\big|W^\prime(u_\eps)\big|^{\bar p}\,\de x\,,
 \end{multline*}
 we estimate by means of \eqref{olala2251}, \eqref{toto2253bis}, and \eqref{titi2254bis}, 
 $$\int_{\Omega^{2r_0}}\big|W^\prime (u_\eps)\big|^{\bar p}\,\de x \leq  C\left(\eps^{2s\bar p}+ \eps^\alpha+\eps^{2s\bar p}\sum_{k=1}^{p_1(\eps)-1}\tau^{k(\alpha -2s\bar p)}\right)\leq C\eps^{2s\bar p}\,,$$
 and the proof is complete. 
\end{proof}

%%%%%%%%%%%%%%%%%%%%%%%%%%%%%%%%%%%%%%%%%%%%%%%%%%%%%%%%

\subsection{Application to the fractional Allen-Cahn equation}

%%%%%%%%%%%%%%%%%%%%%%%%%%%%%%%%%%%%%%%%%%%%%%%%%%%%%%%%

Applying the estimates obtained in the previous section to the fractional Allen-Cahn equation, we obtain the following improvement of Theorem \ref{main1}. Together with Theorem \ref{main1}, it completes the proof 
of Theorem \ref{main1new} in the special case $f=0$. 

\begin{theorem}\label{thmlastone}
In addition to Theorem \ref{main1}, if $\sup_k\|f_k\|_{L^\infty(\Omega)}<\infty$, then for every open subset $\Omega^\prime\subset \Omega$ such that $\overline{\Omega^\prime}\subset \Omega$, 
\begin{enumerate}
\item[\rm (i)]  $v_k\to v_*$ strongly in $H^{s^\prime}(\Omega^\prime)$ for every $s^\prime\in\big(0,\min(2s,1/2)\big)$;
\vskip5pt

\item[\rm (ii)] $\int_{\Omega^\prime}W(v_k)\,\de x=O\big(\eps_k^{\min(4s,\alpha)}\big)$ for every $\alpha\in(0,1)$;  
\vskip5pt

\item[\rm (iii)] $\displaystyle f_k(x)-\frac{1}{\eps_k^{2s}}W(v_k(x))\rightharpoonup\left( \frac{\gamma_{n,s}}{2}\int_{\R^n}\frac{|v_*(x)-v_*(y)|^2}{|x-y|^{n+2s}}\,\de y\right) v_*(x)$ weakly in $L^{\bar p}(\Omega^\prime)$ 
\vskip3pt

\noindent for every $\bar p<1/2s$.  
\end{enumerate}
\end{theorem}

\begin{proof}
The proof departs from the end of the proof of Theorem \ref{main1}. We apply the results of Subsection \ref{subsectstratAC} to the extended function $v_k^\e$. Then items (ii) and (iii) are straightforward consequences of Corollaries \ref{corpotquant} and \ref{corpotquantbis} (together with item (iii) in Theorem \ref{main1}). 

Let us now fix an open subset $\Omega^{\prime\prime}\subset \Omega^\prime$ with  Lipschitz boundary such that $\overline{\Omega^{\prime\prime}}\subset\Omega^\prime$. 
Since $s^\prime<2s$, we can find a number $\theta>\max(2,1/2s)$ such that $ \max(s,s^\prime)<1/\theta$. We set 
$\bar p:=1/(\theta s)<\min(1/2s,2)$, and $\bar s:=s^\prime/\bar p<s$. Since $\{f_k\}_{k\in\mathbb{N}}$ is assumed to be bounded in $L^\infty(\Omega)$, we infer from item (iii) that $\{(-\Delta)^sv_k\}_{k\in\mathbb{N}}$ remains bounded in $L^{\bar p}(\Omega^\prime)$.  On the other hand, we already proved that $\{v_k\}_{k\in\mathbb{N}}$ remains bounded in $L^\infty(\R^n)$. Hence Proposition~\ref{keypropimpr} shows that 
\begin{multline*}
\iint_{\Omega^{\prime\prime}\times\Omega^{\prime\prime}}\frac{|v_k(x)-v_k(y)|^{2}}{|x-y|^{n+2s^\prime}}\,\de x\de y\\
\leq 2^{2-\bar p}\|v_k\|^{2-\bar p}_{L^\infty(\R^n)} \iint_{\Omega^{\prime\prime}\times\Omega^{\prime\prime}}\frac{|v_k(x)-v_k(y)|^{\bar p}}{|x-y|^{n+2\bar s\bar p}}\,\de x\de y\leq C\,, 
\end{multline*}
for some constant $C$ independent of $k$. The  sequence $\{v_k\}_{k\in\mathbb{N}}$ is thus bounded in $H^{s^\prime}(\Omega^{\prime\prime})$. Finally, for an arbitrary $s^{\prime\prime}\in(0,s^\prime)$, the embedding $H^{s^{\prime\prime}}(\Omega^{\prime\prime})\subset H^{s^\prime}(\Omega^{\prime})$ is compact, and consequently $\{v_k\}_{k\in\mathbb{N}}$ is strongly relatively compact in 
$H^{s^{\prime\prime}}(\Omega^{\prime\prime})$ which proves (i). 
\end{proof}

%%%%%%%%%%%%%%%%%%%%%%%%%%%%%%%%%%%%%%%%%%%%%%%%%%%%%%%%
%%%%%%%%%%%%%%%%%%%%%%%%%%%%%%%%%%%%%%%%%%%%%%%%%%%%%%%%

\vskip15pt

 \noindent{\bf Aknowledgements.} 
Part of this article has been written while V.M. was visiting the Department of Mathematics at Johns Hopkins University.  
V.M. is supported by the {\sl Agence Nationale de la Recherche} through the projects  ANR-12-BS01-0014-01 and  ANR-14-CE25-0009-01.

%=======================
% BIBLIOGRAPHY AND INDEX
%=======================


\begin{thebibliography}{60}
%
\bibitem{AbaVal} {\sc N. Abatangelo, E. Valdinoci} : A notion of nonlocal curvature, 
{\it Numer. Funct. Anal. Optim.}  {\bf 35} (2014), 793--815.
%
%\vskip3pt
\bibitem{ALBBEL} {\sc G. Alberti, G. Bellettini} : A non-local anisotropic model for phase transitions: asymptotic behaviour of rescaled energies, {\it European J. Appl. Math.} {\bf 9} (1998), 261--284. 
%
\bibitem{ABScr} {\sc G. Alberti, G. Bouchitt{\'e}, P. Seppecher} : Un r\'esultat de perturbations singuli\`eres 
avec la norme $H^{1/2}$, {\it C. R. Acad. Sci. Paris S\'er. I Math.} {\bf 319} (1994), 333--338. 
%
%\vskip3pt
\bibitem{ABS} {\sc G. Alberti, G. Bouchitt{\'e}, P. Seppecher} : Phase transition with the line-tension effect,
{\it Arch. Rational Mech. Anal.} {\bf144} (1998), 1--46.
%
%\vskip3pt
\bibitem{All} {\sc W.K. Allard} : On the first variation of a varifold, {\it Ann. of Math.} {\bf 95} (1972), 417--491.
%
%\vskip3pt
\bibitem{ADPM} {\sc L. Ambrosio, G. De Philippis,  L. Martinazzi} : Gamma-convergence of nonlocal perimeter functionals, {\it Manuscripta Math.} {\bf 134} (2011), 377--403. 
%
%\vskip3pt
\bibitem{AFP} {\sc L. Ambrosio, N. Fusco, D. Pallara} : {\it Functions of Bounded Variation and Free Discontinuity Problems}, Oxford University Press, New York (2000).
%
%\vskip3pt
\bibitem{BFV} {\sc B. Barrios Barrera, A. Figalli, E. Valdinoci} : Bootstrap regularity for integro-differential operators, and its application to nonlocal minimal surfaces, {\it Ann. Scuola Norm. Sup. Pisa Cl. Sci.} Vol. XIII (2014), 609--639. 
%
%\vskip3pt 
\bibitem{CC} {\sc X. Cabr\'e, E. Cinti} : Energy estimates and 1-{D} symmetry for nonlinear equations
              involving the half-{L}aplacian, {\it Discrete Contin. Dyn. Syst.}  {\bf 28} (2010), 1179--1206. 
%
%\vskip3pt               
\bibitem{CC2} {\sc X. Cabr\'e, E. Cinti} :  Sharp energy estimates for nonlinear fractional diffusion equations,  {\it Calc. Var. Partial Differential Equations} {\bf 49} (2014), 233--269.           
%
\vskip3pt               
\bibitem{CFSMT} {\sc X. Cabr\'e, M.M. Fall, J. Sol\`a-Morales, T. Weth} :  Curves and surfaces with constant nonlocal mean curvature: meeting Alexandrov and Delaunay,  to appear in {\it J. Reine Angew. Math. (Crelle's Journal)}, preprint \texttt{arXiv:1503.00469}.
%
%\vskip3pt               
\bibitem{CFW} {\sc X. Cabr\'e, M.M. Fall, T. Weth} :  Delaunay hypersurfaces with constant nonlocal mean curvature,  preprint \texttt{arXiv:1602.02623}. 
%
%\vskip3pt 
\bibitem{CS1} {\sc X. Cabr\'e, Y. Sire} : Nonlinear equations for fractional Laplacians I: Regularity, maximum principles, and Hamiltonian estimates, {\it Ann. Inst. H. Poincar\'e Anal. Non Lin\'eaire} {\bf 31} (2014), 23--53.
%
%\vskip3pt 
\bibitem{CS2} {\sc X. Cabr\'e, Y. Sire} : Nonlinear equations for fractional Laplacians II: existence, uniqueness, and qualitative properties of solutions, {\it Trans. Amer. Math. Soc}, {\bf 367} (2015), 911--941. 
%
%\vskip3pt
\bibitem{CSM} {\sc X. Cabr\'e, J. Sol\`a-Morales} : Layer solutions in a halfspace for boundary reactions, {\it Comm. Pure Appl. Math.} {\bf 58} (2005), 1678--1732.
%
%\vskip3pt
\bibitem{CRS} {\sc L.A. Caffarelli, J.M. Roquejoffre, O. Savin} : Nonlocal minimal surfaces, {\it Comm. Pure Appl. Math.} {\bf 63} (2010), 1111--1144.
%
%\vskip3pt
\bibitem{CaffSil} {\sc L.A. Caffarelli, L. Silvestre} : An extension problem related to the fractional Laplacian, {\it Comm. Partial Differential Equations} {\bf 32} (2007), 1245--1260.
%
%\vskip3pt
\bibitem{CapGui} {\sc M.C. Caputo, N. Guillen} : Regularity for non-local almost minimal boundaries and applications, preprint \texttt{arXiv:1003.2470}.
%
%\vskip3pt
\bibitem{CheegNab} {\sc J. Cheeger, A. Naber} : Quantitative stratification and the regularity of harmonic maps and minimal currents, {\it Comm. Pure Appl. Math.} {\bf 66} (2013), 965--990.
%
%\vskip3pt
\bibitem{cinserval} {\sc E. Cinti, J. Serra, E. Valdinoci} : Quantitative flatness results and $BV$-estimates for stable nonlocal minimal surfaces, preprint \texttt{arXiv:1602.00540}.
%
%\vskip3pt
\bibitem{DaRi} {\sc F. Da Lio, T. Rivi{\`e}re} : Three-term commutator estimates and the regularity of $\frac12$-harmonic maps into spheres, {\it Anal. PDE}  {\bf 4} (2011), 149--190. 
%
%\vskip3pt
\bibitem{Dav} {\sc J. D\' avila} : On an open question about functions of bounded variation, {\it Calc. Var. Partial Differential Equations}  {\bf 15} (2002), 519--527. 
%
%\vskip3pt
\bibitem{DDPDV} {\sc J. D\' avila, M. del Pino, S. Dipierro, E. Valdinoci} : Nonlocal Delaunay surfaces, {\it Nonlinear Anal.} {\bf 137} (2016), 357--380.
%
%\vskip3pt
\bibitem{DDPW} {\sc J. D\' avila, M. del Pino, J. Wei} : Nonlocal $s$-minimal surfaces and Lawson cones, to appear in {\it J. Differential Geom.}, preprint \texttt{arXiv:1402.4173}. 
%
%\vskip3pt
 \bibitem{EvGa} {\sc L.C. Evans, R.F. Gariepy} : {\it Measure theory and fine properties of functions}, Studies in Advanced Mathematics, CRC Press,  Boca Raton FL (1992). 
%
%\vskip3pt
 \bibitem{FKS} {\sc E.B. Fabes, C.E. Kenig, R.P. Serapioni} : The local regularity of solutions of degenerate elliptic equations, {\it Comm. Partial Differential Equations} {\bf 7} (1982), 77--116.  
%
%\vskip3pt
\bibitem{FFMMM} {\sc A. Figalli, N. Fusco, F. Maggi, V. Millot, M. Morini} : Isoperimetry and stability properties of balls with respect to nonlocal energies, 
{\it Comm. Math. Phys.}  {\bf  336} (2015), 441--507.
%
%\vskip3pt
\bibitem{FV} {\sc A. Figalli, E. Valdinoci} : Regularity and Bernstein-type results for nonlocal minimal surfaces, to appear in {\it J. Reine Angew. Math. (Crelle's Journal)}, preprint \texttt{arXiv:1307.0234}. 
%
%\vskip3pt
\bibitem{FMS} {\sc M. Focardi, A. Marchese, E. Spadaro} : Improved estimate of the singular set of Dir-minimizing $Q$-valued functions via an abstract regularity result, {\it J. Funct. Anal.} {\bf 268} (2015), 3290-3325.
%
%\vskip3pt
\bibitem{GM} {\sc A. Garroni, S. M{\"u}ller} : A variational model for dislocations in the line tension limit, {\it Arch. Ration. Mech. Anal.}  {\bf 181} (2006), 535--578. 
%
%\vskip3pt
\bibitem{Gurt} {\sc  M.E. Gurtin} : On a theory of phase transitions with interfacial energy, {\it Arch. Ration. Mech. Anal.}  {\bf 87} (1984), 187--212. 
%
%\vskip3pt
\bibitem{G} {\sc P. Grisvard} : {\it Elliptic problems in nonsmooth domains}, Monographs and Studies in Mathematics, Pitman, Boston MA (1985). 
%
%\vskip3pt
\bibitem{HT} {\sc J.E. Hutchinson, Y. Tonegawa} : Convergence of phase interfaces in the van der {W}aals-{C}ahn-{H}illiard theory,
{\it Calc. Var. Partial Differential Equations} {\bf 10} (2000), 49--84.
%
%\vskip3pt
\bibitem{Im} {\sc C. Imbert} :  Level set approach for fractional mean curvature flows, {\it Interfaces Free Bound.} {\bf 11} (2009), 153--176. 
%
%\vskip3pt
\bibitem{ImSoug} {\sc C. Imbert, P.E. Souganidis} :  Phase field theory for fractional reaction-diffusion equations and applications, 
preprint \texttt{arXiv:0907.5524}.
%
%\vskip3pt
\bibitem{LW1} {\sc F.H. Lin, C.Y. Wang} : Harmonic and quasi-harmonic spheres, {\it Comm. Anal. Geom.}  {\bf 7} (1999), 397--429. 
%
%\vskip3pt
\bibitem{LW2} {\sc F.H. Lin, C.Y. Wang} : Harmonic and quasi-harmonic spheres. {II}, {\it Comm. Anal. Geom.} {\bf 10} (2002), 341--375. 
%
%\vskip3pt
\bibitem{Matti} {\sc P. Mattila} : {\it Geometry of sets and measures in Euclidean spaces}, Cambridge Studies in Advanced Mathematics {\bf 44}, Cambridge University Press (1995).
%
%\vskip3pt
\bibitem{MilSir} {\sc V. Millot, Y. Sire} : On a fractional Ginzburg-Landau equation and $1/2$-harmonic maps into spheres, {\it Arch. Rational Mech. Anal.} {\bf 215} (2015), 125--210.
%
%\vskip3pt
\bibitem{Mod} {\sc L. Modica} : The gradient theory of phase transitions and the minimal interface criterion,
 {\it Arch. Rational Mech. Anal.}  {\bf 98} (1987), 123--142. 
%
%\vskip3pt
\bibitem{ModMort} {\sc L. Modica, S. Mortola} : Un esempio di $\Gamma$-convergenza, {\it Boll. Un. Mat. Ital. B} {\bf 14} (1977), 285--299.
%
%\vskip3pt
\bibitem{PSV} {\sc G. Palatucci, O. Savin, E. Valdinoci} : Local and global minimizers for a variational energy involving a fractional norm, {\it Ann. Mat. Pura Appl.}  {\bf 192} (2013),  673-€"-718.
%
%\vskip3pt
\bibitem{Riv} {\sc T. Rivi{\`e}re} : Everywhere discontinuous harmonic maps into spheres, {\it Acta Math.} {\bf 175} (1995), 197--226.
%
%\vskip3pt
\bibitem{SVold} {\sc O. Savin, E. Valdinoci} : Density estimates for a nonlocal variational model via the Sobolev inequality, {\it SIAM J. Math. Anal.}  {\bf 43} (2011), 2675--2687.
%
\vskip3pt
\bibitem{SV1} {\sc O. Savin, E. Valdinoci} : $\Gamma$-convergence for nonlocal phase transitions,
{\it Ann. Inst. H. Poincar\'e Anal. Non Lin\'eaire} {\bf 29} (2012), 479--500. 
%
%\vskip3pt
\bibitem{SV1bis} {\sc O. Savin, E. Valdinoci} : Regularity of nonlocal minimal cones in dimension 2, {\it Calc. Var. Partial Differential Equations} {\bf 48} (2013), 33--39. 
%
%\vskip3pt
\bibitem{SV2} {\sc O. Savin, E. Valdinoci} : Density estimates for a variational model driven by the Gagliardo norm, {\it J. Math. Pures Appl.} {\bf 101} (2014), 1€"--26. 
%
%\vskip3pt
\bibitem{Scha} {\sc R. Sch\"atzle} : Hypersurfaces with mean curvature given by an ambiant Sobolev function, {\it J. Differential Geom.} {\bf 58} (2001),   371--420.

%\vskip3pt              
\bibitem{SerVal} {\sc R. Servadei, E. Valdinoci} : Weak and viscosity solutions of the fractional Laplace equation, {\it Publ. Mat.} {\bf 58} (2014), 133--154. 
%
%\vskip3pt      
\bibitem{Sim2} {\sc L. Simon} : {\it Lectures on geometric measure theory}, Proceedings of the Centre for Mathematical Analysis, Australian National University, Centre for Mathematical Analysis, Canberra  (1983).
%
%\vskip3pt      
\bibitem{Sim} {\sc L. Simon} : {\it Theorems on regularity and singularity of energy minimizing maps}, Lectures in Mathematics ETH Z\"urich,
Birkh\"auser Verlag, Basel (1996).
%
%\vskip3pt     
\bibitem{SV} {\sc Y. Sire, E. Valdinoci} : Fractional {L}aplacian phase transitions and boundary
              reactions: a geometric inequality and a symmetry result, {\it J. Funct. Anal.} {\bf 256} (2009), 1842--1864. 
%
%\vskip3pt
\bibitem{Stein} {\sc E. M. Stein} : {\it Singular integrals and differentiability properties of
              functions}, Princeton Mathematical Series {\bf 30}, Princeton University Press, Princeton N.J. (1970). 
%
%\vskip3pt
\bibitem{Ster} {\sc P. Sternberg} : The effect of a singular perturbation onnonconvex variational problems, {\it Arch. Rational Mech. Anal.} {\bf 101} (1988), 209--260.
%
%\vskip3pt
\bibitem{TVZ} {\sc S. Terracini, G. Verzini, A. Zilio} :  Uniform H\"older regularity with small exponent in competition-fractional diffusion systems, {\it Discrete Contin. Dyn. Syst.} {\bf 34} (2014), 2669--2691. 
%
%\vskip3pt
\bibitem{Ton1} {\sc Y. Tonegawa} : Phase field model with a variable chemical potential, 
{\it Proc. R. Soc. Edinburgh Sect. A} {\bf 132} (2002), 993--1019.
%
%\vskip3pt
\bibitem{Ton2} {\sc Y. Tonegawa} : A diffuse interface whose chemical potential lies in a Sobolev space, 
{\it Ann. Scuola Norm. Sup. Pisa CL. Sci.} Vol. IV (2005), 487--510.
%
%\vskip3pt
\bibitem{Triebel} {\sc H. Triebel} : {\it Theory of function spaces}, Monographs in Mathematics, Birkh\"auser Verlag, Basel (1983).
%
%\vskip3pt
\bibitem{W} {\sc C.Y. Wang} : Limits of solutions to the generalized {G}inzburg-{L}andau
              functional, {\it Comm. Partial Differential Equations} {\bf 27} (2002), 877--906. 
%
%\vskip3pt
\bibitem{Zi} {\sc W.P. Ziemer} : {\it Weakly differentiable functions}, Graduate Texts in Mathematics, Springer-Verlag, New York (1989). 
%
\end{thebibliography}
\end{document}